\documentclass[11pt,a4]{article}
\usepackage{graphicx}
\usepackage{amssymb}

\newtheorem{lemma}{{\bf Lemma}}[section]
\newtheorem{prop}[lemma]{{\bf Proposition}}
\newtheorem{cor}[lemma]{{\bf Corollary}}
\newtheorem{theorem}[lemma]{{\bf Theorem}}
\newtheorem{remark}[lemma]{{\bf Remark}}
\newtheorem{example}[lemma]{{\bf Example}}
\newtheorem{defn}[lemma]{{\bf Definition}}
\newenvironment{proof}%
{\begin{trivlist}\item[]{\bf Proof} }%
{\hfill$\Box$\end{trivlist}}

\newcommand{\trsub}[1]{\mbox{$\mbox{\rm tr}_{_{#1}}$}}
\newcommand{\mma}{multi--matrix algebra}
\newcommand{\Diprod}[2]{\left\langle{}#1,#2\right\rangle_{tr_{_{D}}}}
\newcommand{\csq}{commuting square}

\newcommand{\Rijklamsq}[3]{\frac{\!\sqrt{R_{#1}(\!\lambda\!)R_{#2}(\!\lambda\!)}}{R_{#3}(\lambda)}}
\newcommand{\Rijksq}[3]{\frac{\sqrt{R_{#1}R_{#2}}}{R_{#3}}}
\newcommand{\Rnkvot}[2]{\frac{R_{#1}}{R_{#2}}}
\newcommand{\Rnt}[1]{R_{#1}(t)}

\newcommand{\firestar}[4]{\mbox{S$(#1$,$#2$,$#3$,$#4)$}}
\newcommand{\deltasub}[1]{\mbox{$\delta_{_{#1}}$}}

\newcommand{\delsqi}[1]{\mbox{$\delta_{_{#1}}^{2}$}}      
\newcommand{\alialj}[2]{\mbox{$\alpha_{_{#1}}\alpha_{_{#2}} $}}
\newcommand{\uij}[2]{\mbox{$u_{_{#1#2}}$}}
\newcommand{\sqalialj}[2]{\mbox{$\sqrt{\alpha_{_{#1}}\alpha_{_{#2}} }$}}
\newcommand{\vij}[2]{\mbox{$v_{_{#1}}^{#2}$}}

\newcommand{\rijkl}[4]{\mbox{$\frac{R_{#1}R_{#2}}{R_{#3}R_{#4}}$}}
\newcommand{\rikl}[3]{\mbox{$\frac{R_{#1}}{R_{#2}R_{#3}}$}}
\newcommand{\rijl}[3]{\mbox{$\frac{R_{#1}R_{#2}}{R_{#3}}$}}
\newcommand{\rik}[2]{\mbox{$\frac{R_{#1}}{R_{#2}}$}}
\newcommand{\ri}[1]{\mbox{$R_{#1}$}}
\newcommand{\sqrijkl}[4]{\mbox{$\sqrt{\frac{R_{#1}R_{#2}}{R_{#3}R_{#4}}}$}}
\newcommand{\sqrikl}[3]{\mbox{$\sqrt{\frac{R_{#1}}{R_{#2}R_{#3}}}$}}

\newcommand{\ca}{\mbox{${\cal A}$}}
\newcommand{\cb}{\mbox{${\cal B}$}}
\newcommand{\cc}{\mbox{${\cal C}$}}
\newcommand{\cd}{\mbox{${\cal D}$}}

\newcommand{\sqdelal}[2]{\mbox{$\sqrt{\frac{\lambda\delta_{1}^{^{2}}}{\alpha_{_{#1}}\alpha_{_{#2}}}}$}}
\newcommand{\sqalikl}[3]{\mbox{$\sqrt{\frac{\alpha_{_{#1}}}{\alpha_{_{#2}}\alpha_{_{#3}}}}$}}
\newcommand{\sqalijkl}[4]{\mbox{$\sqrt{\frac{\alpha_{_{#1}}\alpha_{_{#2}}}{\alpha_{_{#3}}\alpha_{_{#4}}}}$}}
\newcommand{\sqlamdel}[2]{\mbox{$\sqrt{\frac{\lambda\delta_{1}^{^{2}}}{\delta_{_{#1}}\delta_{_{#2}}}}$}}
\newcommand{\sqdelijkl}[4]{\mbox{$\sqrt{\frac{\delta_{_{#1}}\delta_{_{#2}}}{\delta_{_{#3}}\delta_{_{#4}}}}$}}

\newcommand{\gam}[1]{\mbox{$\gamma_{_{#1}}$}}
\newcommand{\al}[1]{\mbox{$\alpha_{_{#1}}$}}
\newcommand{\del}[1]{\mbox{$\delta_{_{#1}}$}}
\newcommand{\lam}[1]{\mbox{$\lambda_{_{#1}}$}}
\newcommand{\eij}[2]{\mbox{$e_{_{#1#2}}$}}

\newcommand{\eijeij}[2]{\mbox{$e_{_{#1#2}}\otimes{}\overline{e}_{_{#1#2}}$}}
\newcommand{\xx}[1]{\mbox{$x_{_{#1}}\otimes{}\overline{x}_{_{#1}}$}}
\newcommand{\yy}[1]{\mbox{$y_{_{#1}}\otimes{}\overline{y}_{_{#1}}$}}
\newcommand{\ee}[1]{\mbox{$e_{_{#1}}\otimes{}\overline{e}_{_{#1}}$}}
\newcommand{\aij}[2]{\mbox{$a_{_{#1#2}}$}}
\newcommand{\qij}[2]{\mbox{$q_{_{#1#2}}$}}
\newcommand{\qijs}[2]{\mbox{$q_{_{#1#2}}^{s}$}}
\newcommand{\qa}{\mbox{$q^{a}$}}

\newcommand{\eefrac}[2]{\frac{e^{#1x}-e^{-#1x}}{e^{#2x}-e^{-#2x}}}
\newcommand{\trestar}[3]{\mbox{S$(#1$,$#2$,$#3)$}}
\newcommand{\Rnlam}[2]{R_{#1}(\lambda_{_{#2}})}
\newcommand{\Rnlamkvot}[3]{\frac{R_{#1}(\lambda_{_{#3}})}{R_{#2}(\lambda_{_{#3}})}}

\newcommand{\edg}[2]{\mbox{$(\{#1\},\{#2\})$}}
\newcommand{\tinyedg}[2]{\mbox{{\tiny $(\{#1\},\{#2\})$}}}
\newcommand{\edgu}[2]{\mbox{$(\{#1\},\{#2\})_{_{u}}$}}
\newcommand{\expp}[1]{\mbox{$e^{^{i(#1)x}}-e^{^{-i(#1)x}}$}}
\newcommand{\edgv}[2]{\mbox{$(\{#1\},\{#2\})_{_{v}}$}}
\newcommand{\ta}[4]{\mbox{\tiny{$
\frac{\beta_{_{#1}}\beta_{_{#2}}}{\beta_{_{#3}}\beta_{_{#4}}}$}}}
\newcommand{\taa}[4]{\mbox{\scriptsize{$
\sqrt{\frac{\alpha_{_{#1}}\alpha_{_{#2}}}{\alpha_{_{#3}}\alpha_{_{#4}}}}$}}}

\newcommand{\AenexAnul}{\mbox{$\langle{}A_{1},e_{A_{0}}\rangle{}$}} 
\newcommand{\BenexBnul}{\mbox{$\langle{}B_{1},e_{B_{0}}\rangle{}$}}
\newcommand{\BexA}{\mbox{$\langle{}B,e_{A}\rangle{}$}} 
\newcommand{\AenexBnul}{\mbox{$\{A_{1},e_{B_{0}}\}''$}} 
\newcommand{\eAnul}{\mbox{$e_{A_{0}}$}} 
 
\newcommand{\eBnul}{\mbox{$e_{B_{0}}$}} 
 
\newcommand{\eA}{\mbox{$e_{A}$}} 
 
\newcommand{\EAnul}{\mbox{$E_{A_{0}}$}} 
\newcommand{\EAen}{\mbox{$E_{A_{1}}$}} 
\newcommand{\EBnul}{\mbox{$E_{B_{0}}$}} 
 
\newcommand{\EBn}{\mbox{$E_{B_{n}}$}}
\newcommand{\EA}{\mbox{$E_{A}$}} 
 
\newcommand{\trAen}{\mbox{$tr_{A_{1}}$}} 
\newcommand{\trAto}{\mbox{$tr_{A_{2}}$}} 
\newcommand{\trAn}{\mbox{$tr_{A_{n}}$}}
\newcommand{\trBen}{\mbox{$tr_{B_{1}}$}} 
\newcommand{\trBto}{\mbox{$tr_{B_{2}}$}} 
\newcommand{\trBn}{\mbox{$tr_{B_{n}}$}} 
\newcommand{\trA}{\mbox{$tr_{A}$}} 
\newcommand{\trB}{\mbox{$tr_{B}$}} 
\newcommand{\LAen}{\mbox{$L^{2}(A_{1},tr_{A_{1}})$}} 
 
\newcommand{\LBen}{\mbox{$L^{2}(B_{1},tr_{B_{1}})$}} 
 
\newcommand{\LA}{\mbox{$L^{2}(A,tr_{A})$}} 
\newcommand{\LB}{\mbox{$L^{2}(B,tr_{B})$}} 
\newcommand{\restrLAen}{\mbox{$|_{L^{2}(A_{1},tr_{A_{1}})}$}} 
\newcommand{\restrLAto}{\mbox{$|_{L^{2}(A_{2},tr_{A_{2}})}$}}

\newcommand{\norm}[1]{\mbox{$\,\|\,#1\,\|\,$}}
\newcommand{\etnorm}[1]{\mbox{$\,\|\,#1\,\|_{_{_{1}}}\,$}}
\newcommand{\tonorm}[1]{\mbox{$\,\|\,#1\,\|_{_{_{2}}}\,$}}
\newcommand{\tonormsq}[1]{\mbox{$\,\|\,#1\,\|_{_{_{2}}}^{^{2}}\,$}}
\newcommand{\GtG}{\mbox{$G^{t}G{}$}}
\newcommand{\GGt}{\mbox{$GG^{t}{}$}}
\newcommand{\lGGt}{\mbox{$\frac{1}{\lambda^{2}}GG^{t}{}$}}
\newcommand{\lampow}[1]{\mbox{$\lambda^{#1}$}}
\newcommand{\etalow}[1]{\mbox{$\eta_{_{_{#1}}}$}}
\newcommand{\chapter}[1]{{\bf{}\Large{}#1}}

\newcommand{\ro}[1]{\mbox{$\rho^{#1}$}}
\newcommand{\efrac}[1]{\frac{e^{#1x}-e^{-#1x}}{e^{x}-e^{-x}}}

\renewcommand{\theequation}{\roman{equation}}

\title{Commuting Squares\\and\\Index for Subfactors\\\mbox{\normalsize (Ph.D Thesis)}}
\author{John Kehlet Schou}
\date{December 1990}
\begin{document}
\maketitle
\pagenumbering{roman}
\setcounter{equation}{0}
\tableofcontents
\newpage
\setcounter{secnumdepth}{0}
\section{Summary}
The index for subfactors was introduced by V.\ Jones in \cite{jones}. In this paper he proves that the only values of the index for a subfactor, $A,$ of the hyperfinite $II_{1}-$factor, $R,$ in the interval $[1,4)$ are the numbers
\[\{4\cos^{2}\mbox{$\frac{\pi}{n}$}\;|\;n=3,4,\ldots\}.\] The values above 4 of the index for irreducible subfators of $R$ are not known today, but recently S.\ Popa \cite{popa1}  proved, that all real  numbers above 4 can be the index of a irreducible pair of non-hyperfinite factors. 

It seems likely, that the values of the index for irreducible subfactors of $R,$ are related to finite or infinite graphs, in the sense that the index values has to be the square of the norm  of such  graphs \cite{popa2}. The close connection with graphs is also reflected in the present work, where all the given examples of the index, arise as the square of the norm of a finite or infinite graph.

The first real breakthrough in constructing values of the index for irreducible subfactors of $R,$ came with H.\ Wenzl's work \cite{Wen2}, where he uses periodic ladders of inclusions of multi-matrix algebras to construct subfactors
of $R.$ He also gives an easy way to determine the index for the constructed
pair, and a very useful  criterion for determining irreducibility of the pair.
Wenzl constructs the following values of the index
\[\frac{\sin^{2}\frac{k\pi}{l}}{\sin^{2}\frac{\pi}{l}},\;\;\;\;l\geq{}3,\;
2\leq{}k\leq{}l-2.\]

A key ingredient in building periodic ladders of multi-matrix   algebras, is the notion of a commuting square (see \cite{HGJ} chapter 4), which consist of four multi-matrix algebras $A,B,C$ and $D,$ included in each other via inclusion
matrices $G,H,K$ and $L$ \[\begin{array}{lcl}
      C & \subset_{L} & D \\
       \cup_{K} &\,&\cup_{H}\\
      A & \subset_{G} & B  \end{array}  \] together with a faithful trace
$\mbox{tr}_{_{D}}$ on $D,$ such that
$E_{_{A}}=E_{_{B}}E_{_{C}}=E_{_{C}}E_{_{B}},$ where $E_{_{X}}$ denotes the
unique trace preserving conditional expectation of $D$ onto $X,\;X=A,B,C.$

The present work is divided into 6 chapters.\footnote{Chapter 1 and chapter 2 are joint work with Prof. Uffe Haagerup, Dept. of Mathematics and Computer Science,
Odense University.} In the beginning of chapter 1 we give  a characterization of commuting squares of
multi-matrix algebras, the {\em bi-unitary condition}, which also occurs in the
unpublished work of A.\ Ocneanu, in a slightly different setting. If the inclusion
matrices satisfy $HL^{t}=G^{t}K$ in addition to the necessary condition
$GH=KL,$ the above mentioned ladder construction will, under certain mild extra
assumptions on $G,H,K$ and $L,$ produce an irreducible subfactor of the
Hyperfinite $II_{_{1}}-$factor of index $\|H\|^{2}=\|K\|^{2}.$ The main part
of chapter 1 is used to determine which matrices, $G,$ can be used to build a
commuting square of the form 
\begin{equation} \label{summ*1}
\begin{array}{lcl}
      C & \subset_{nG^{t}} & D \\
       \cup_{G} &\,&\cup_{G^{t}}\\
      A & \subset_{nG} & B  \end{array}  \;\;\;\;\;\;n\in{\Bbb N} 
\end{equation}
under the assumption that the resulting index values, $\|G\|^{2},$ should be in
the interval $(4,5).$ The lowest value obtained this way is 
\[\frac{5+\sqrt{13}}{2}\approx{}4.302\]

Following a suggestion by A.\ Ocneanu \cite{Ocn}, we, in chapter 2, obtain
index values closer to 4, by considering commuting squares of the form
\begin{equation} \label{summ*2}
\begin{array}{lcl}
      C & \subset_{G^{t}G-I} & D \\
       \cup_{G} &\,&\cup_{G}\\
      A & \subset_{GG^{t}-I} & B  \end{array}
\end{equation}
Moreover we determine which matrices, $G,$ that can occur in (\ref{summ*2}),
under the assumption that the associated graph $\Gamma_{_{G}},$ has the form of
a star with three rays.

In chapter 2 we also bring a
presentation of a construction of a commuting square based on the graph
$E_{_{10}}.$ The example was originally conceived by A.\ Ocneanu \cite{Ocn}, and
is particularly interesting, since the index value obtained from this graph, is
the lowest (above 4) which can be obtained as $\|G\|^{2},$ for $G$ a matrix
with non-negative integer entries.

In chapter 3 we bring a construction of commuting squares, based on the Dynkin
diagrams $A_{_{l}}.$ The index values constructed this way are
\[\frac{\sin^{2}\frac{k\pi}{l}}{\sin^{2}\frac{\pi}{l}},\;\;\;\;l\geq{}3,\;
2\leq{}k\leq{}l-2,\]
i.e. the same as Wenzl constructed in \cite{Wen2}.

In chapter 4 we define the notion of an infinite dimensional multi-matrix
algebra, and show that the theory for building subfactors of $R,$ from commuting
squares of multi-matrix algebras, can be generalized to commuting squares of
infinite dimensional multi-matrix algebras. We also prove that Wenzl's irreducibility criterion is still valid in this setting.

In chapter 5 we look at some infinite graphs defined by J. Shearer
\cite{shearer}. For each $\lambda>\sqrt{2+\sqrt{5}},$ Shearer defines an
infinite graph, $\Gamma_{_{\!\!\lambda}},$ with largest eigenvalue $\lambda.$ We
show that the Perron--Frobenius vector of these graphs is summable, which
implies that these graphs might be used to define the inclusions $H$ and $K$ of
a commuting square of infinite dimensional multi-matrix algebras. For the only
obvious choice of the other inclusions, $G$ and $L,$ it is shown that a commuting
square of this form implies that $\Gamma_{_{\!\!\lambda}}$ is eventually periodic.
This is not a property of $\Gamma_{_{\!\!\lambda}}$ in general, so there is no
simple way to build a commuting square of infinite dimensional multi-matrix 
algebras from  $\Gamma_{_{\!\!\lambda}}$ for general $\lambda.$

In chapter 6 we look at another class of infinite graphs, $T(1,n,\infty),$ defined
by A.\ Hoffmann in \cite{Hof}. We show, that the value of the index, which might
be constructed using this graph, is not obtainable from a construction on any
finite graph, and we construct commuting squares of infinite dimensional
multi-matrix algebras  based on $T(1,2,\infty),$ $T(1,3,\infty)$ and
$T(1,4,\infty)$. Unfortunately there is no general pattern in these
constructions, which could show us how to do the construction for general $n.$\newpage{}
\section{Sammenfatning (Danish Summary)}
\noindent{}Index for delfaktorer blev
introduceret af V.\ Jones i \cite{jones}. I denne artikel bevises det, at de
eneste v\ae{}rdier af indexet for en irreducibel delfaktor af $R,$ den
Hyperendelige $II_{1}-$faktor, i intervallet $[1,4)$ er tallene
\[\{4\cos^{2}\mbox{$\frac{\pi}{n}$}\;|\;n=3,4,\ldots\}.\] 
M\ae{}ngden af 
v\ae{}rdier over 4, som  index for en irreducibel delfaktor af den
Hyperendelige   $II_{1}-$faktor kan antage, er til dato ukendt, men for nylig
viste S.\ Popa \cite{popa1}, at alle reelle tal over 4 kan antages af index
for par af ikke Hyperendelige irreducible faktorer.

Det virker sandsynligt, at der er en sn\ae{}ver sammenh\ae{}ng mellem
v\ae{}rdierne af index for par af irreducible Hyperendelige $II_{1}-$faktorer
og endelige/uendelige grafer, forst\aa{}et s\aa{}ledes, at indexv\ae{}rdierne
er kvadratet p\aa{} normer af  adjacensmatricer for  s\aa{}danne grafer 
\cite{popa2}. Denne sammenh\ae{}ng er ogs\aa{} synlig i dette arbejde, idet
alle de konstruerede v\ae{}rdier af index kommer til verden som kvadratet
p\aa{} normen af en endelig eller uendelig graf.

Det f\o{}rste  gennembrud i konstruktionen af index for irreducible
delfaktorer af $R$ kom med H.\ Wenzls artikel  \cite{Wen2}. Her benyttes
periodiske  ``ladders'' af inklusioner af multi-matrix algebraer til at
konstruere irreducible delfaktorer af $R.$ Der bestemmes ligeledes en nem
m\aa{}de hvorp\aa{} index for de konstruerede delfaktorer kan beregnes, samt et
yderst anvendeligt kriterium til at bevise  irreducibilitet af af de
konstruerede delfaktorer. Wenzl bestemmer  f\o{}lgende v\ae{}rier af index
\[\frac{\sin^{2}\frac{k\pi}{l}}{\sin^{2}\frac{\pi}{l}},\;\;\;\;l\geq{}3,\;
2\leq{}k\leq{}l-2.\]

Et af de v\ae{}sentligste redskaber i konstruktionen af periodiske ladders
af multi-matrix algebraer er begrebet ``commuting squares''. Disse best\aa{}r if\o{}lge 
\cite{HGJ} kapitel 4 af fire multi-matrix algebraer $A,B,C$ og $D,$
indlejret i hinanden via inklusionsmatricerne $G,H,K$ og $L$
\[\begin{array}{lcl}
      C & \subset_{L} & D \\
       \cup_{K} &\,&\cup_{H}\\
      A & \subset_{G} & B  \end{array}  \] samt et tro spor,
$\mbox{tr}_{_{D}},$ p\aa{} $D,$ s\aa{}ledes at
$E_{_{A}}=E_{_{B}}E_{_{C}}=E_{_{C}}E_{_{B}},$ hvor $E_{_{X}}$ betegner den
entydige ``conditional expectation'' af  $D$ p\aa{} $X,\;X=A,B,C.$

N\ae{}rv\ae{}rende arbejde er delt i 6 kapitler.\footnote{Kapitel 1 og kapitel 2
er f\ae{}lles arbejde med Prof. Uffe Haagerup, Institut for Matematik og Datalogi,  Odense Universitet.} I starten af kapitel 1 gives en karakterisation af
commuting squares af multi-matrix algebraer, den s\aa{}kaldte
{\em bi-unit\ae{}re betingelse}, der endvidere er en del af A.\ Ocneanus, endnu
ikke offentliggjorte, arbejde, hvor den forekommer i en lidt anden
sammenh\ae{}ng. 

\noindent{}Hvis inklusionsmatricerne opfylder $HL^{t}=G^{t}K,$ 
ud over den n\o{}dvendige betingelse, $GH=KL,$ vil den ovenfor n\ae{}vnte
ladder konstruktion, med nogle f\aa{}, ikke specielt restriktive,
antagelser om  $G,H,K$ and $L,$ resultere i en irreducibel delfaktor af den
Hyperendelige $II_{1}-$faktor med index $\|H\|^{2}=\|K\|^{2}.$ 

Hovedparten af kapitel 1 bruges til at bestemme hvilke matricer, $G,$ der kan
bruges til konstruktion af commuting squares p\aa{} formen 
\begin{equation} \label{samm*1} \begin{array}{lcl}
      C & \subset_{nG^{t}} & D \\
       \cup_{G} &\,&\cup_{G^{t}}\\
      A & \subset_{nG} & B  \end{array}  \;\;\;\;\;\;n\in{\Bbb N} 
\end{equation}
hvor det yderligere foruds\ae{}ttes, at den resulterende indexv\ae{}rdi skal tilh\o{}re
intervallet $(4,5).$ Den mindste indexv\ae{}rdi, der fremkommer p\aa{} denne
m\aa{}de, er
\[\frac{5+\sqrt{13}}{2}\approx{}4.302\]

Idet en ide af  A.\ Ocneanu \cite{Ocn}  f\o{}lges op, konstrueres  i
kapitel 2 v\ae{}rdier af index, der ligger meget t\ae{}ttere p\aa{} 4, ved at
betragte commuting squares p\aa{} formen
 \begin{equation} \label{samm*2}
\begin{array}{lcl}
      C & \subset_{G^{t}G-I} & D \\
       \cup_{G} &\,&\cup_{G}\\
      A & \subset_{GG^{t}-I} & B  \end{array}
\end{equation}
Desuden bestemmes  hvilke matricer, $G,$ der kan benyttes i (\ref{samm*2}),
under antagelse af, at de korresponderende grafer, $\Gamma_{_{G}},$  har
form som en stjerne med tre str\aa{}ler.

I kapitel 2 bringes ligeledes en pr\ae{}sentation af en konstruktion, der
oprindelig skyldes A.\ Ocneanu \cite{Ocn}. Denne konstruktion er specielt
interessant, idet den resulterer i den mindst mulige v\ae{}rdi af index (over 4), der
kan fremkomme som  $\|G\|^{2}$ for en matrix, $G,$ med ikke-negative heltallige
koefficienter.

I kapitel 3 bringes en konstruktion af commuting squares baseret p\aa{}
Dynkin diagrammerne $A_{_{l}}.$ Indexv\ae{}rdierne konstrueret p\aa{} denne
m\aa{}de er 
\[\frac{\sin^{2}\frac{k\pi}{l}}{\sin^{2}\frac{\pi}{l}},\;\;\;\;l\geq{}3,\;
2\leq{}k\leq{}l-2,\] i.e. de samme v\ae{}rdier som Wenzl konstruerede i 
\cite{Wen2}.

I kapitel 4 defineres begrebet en uendeligdimensional multi-matrix algebra, og
det vises at teorien for at konstruere delfaktorer af $R,$ p\aa{} grundlag af
commuting squares af multi-matrix algebraer, kan generaliseres til at
g\ae{}lde for commuting squares af uendeligdimensionale multi-matrix
algebraer. Endvidere vises det, at Wenzls kriterium til at bevise
irreducibilitet af det konstruerede par af Hyperendelige $II_{1}-$faktorer,
ogs\aa{} g\ae{}lder i dette tilf\ae{}lde.

I kapitel 5 betragtes nogle uendelige grafer, defineret af J. Shearer
\cite{shearer}. For ethvert $\lambda>\sqrt{2+\sqrt{5}}$ definerer Shearer
en  uendelig graf, $\Gamma_{_{\!\!\lambda}},$ med st\o{}rste egenv\ae{}rdi 
$\lambda.$ Her bevises det, at den tilh\o{}rende positive egenvektor er
 summabel, hvilket betyder, at disse grafer m\aa{}ske kan benyttes til at
definere inklusionerne $H$ og $K$ i en commuting square af uendeligdimensionale multi-matrix algebraer. For det eneste oplagte valg af inklusioner
$G$ og $L$ vises det, at eksistensen af en s\aa{}dan commuting square
implicerer, at grafen $\Gamma_{_{\!\!\lambda}}$ er periodisk fra et vist trin. En
s\aa{}dan periodicitet er ikke en egenskab som $\Gamma_{_{\!\!\lambda}}$ generelt
er i besiddelse af, s\aa{} der er ikke nogen nem m\aa{}de at konstruere en
commuting square af uendeligdimensionale multi-matrix algebraer ud fra
$\Gamma_{_{\!\!\lambda}}$ for et generelt $\lambda.$

I kapitel 6 betragtes en klasse af uendelige grafer, $T(1,n,\infty),$
defineret af  A.\ Hoffmann i \cite{Hof}. Det vises, at de v\ae{}rdier af index,
der eventuelt kan konstrueres ud fra disse grafer, ikke kan konstrueres p\aa{}
grundlag af nogen endelig graf. Herefter konstrueres eksempler p\aa{} grundlag
af graferne $T(1,2,\infty),$ $T(1,3,\infty)$ og $T(1,4,\infty)$. Desv\ae{}rre
afsl\o{}rer der sig ikke noget generelt m\o{}nster i disse konstruktioner, der
kunne give en ide til en konstruktion for vilk\aa{}rligt $n.$\newpage{}
\noindent{\large{}Acknowledgement}\\[1cm]

\noindent{}I would like to thank my adviser and co-author on the two first chapters of the present work, Prof. Uffe Haagerup, for his continuing interest in what I have been trying to do, and for his many valuable suggestions.\newpage{}
\renewcommand{\theequation}{\thesection.\arabic{equation}}
\pagenumbering{arabic}
\setcounter{secnumdepth}{2}
\part{Finite Dimensional Commuting Squares}

\chapter{The Simplest Possible Commuting Squares}
\label{chapter1}

\setcounter{equation}{0}

\section{Ocneanu's Bi--unitary Condition and Symmetric Commuting Squares}
Following \cite{HGJ}  chapter IV, a commuting square is a set of four finite von Neumann algebras $A,$ $B,$ $C$ and $D,$ nested in each other by $A\subset{}B\subset{}D$ and $A\subset{}C\subset{}D,$ together with a normal faithful tracial state $\trsub{D}$ on $D,$ such that the unique trace preserving conditional expectations $E_{_{\!\!A}},$ $E_{_{\!\!B}}$ and $E_{_{\!\!C}}$ of $D$ onto $A,$ $B$ resp. $C$ satisfy
\[E_{_{\!\!A}}=E_{_{\!\!B}}E_{_{\!\!C}}=E_{_{\!\!C}}E_{_{\!\!B}}.\]
If $A,$ $B,$ $C$ and $D$ are \mma{}s, the inclusions $A\subset{}B\subset{}D$ and $A\subset{}C\subset{}D$ are given by inclusion matrices $G,$ $H,$ $K,$ and $L$ and we will write
\[\begin{array}{lcl}
      C & \subset_{L} & D \\
       \cup_{K} &\,&\cup_{H}\\
      A & \subset_{G} & B.\end{array}  \]
For our purposes we will need the following characterization of a commuting square of \mma{}s.\label{fire1page1}
\begin{lemma}
\label{fire1lem11}
Let $A,$ $B,$ $C$ and $D$ be \mma{}s $A\subset{}B\subset{}D,$  $A\subset{}C\subset{}D$ and let \trsub{D} be a trace on $D.$ Set $\Diprod{d_{_{1}}}{d_{_{2}}}=\trsub{D}(d_{_{1}}^{}d_{_{2}}^{*}),$ $d_{_{1}},d_{_{2}}\in{}D.$ Then the following conditions are equivalent
\begin{enumerate}
\item{}$A,$ $B,$ $C$ and $D$ form a commuting square with respect to the the trace \trsub{D} on $D.$
\item{}$A,$ $A^{\perp}\cap{}B$ and $A^{\perp}\cap{}C$ are orthogonal with respect to the inner product $\Diprod{\cdot}{\cdot}.$
\item{}$\Diprod{E_{_{\!\!A}}(b)}{E_{_{\!\!A}}(c)}=\Diprod{b}{c}$ for all $b\in{}B$ and all $c\in{}C.$
\end{enumerate}
\end{lemma}
\begin{proof}
Let \trsub{A}, \trsub{B} and \trsub{C} denote the restriction of \trsub{D} to $A,$ $B$ resp. $C.$

\noindent{}$1\Rightarrow{}2$: Let $c\in{}C$ and $b\in{}A^{\perp}\cap{}B,$ then
\[\Diprod{b}{c}=\trsub{C}(E_{_{\!\!C}}(bc^{*}))=\trsub{C}(E_{_{\!\!C}}(b)c^{*})=\trsub{C}(E_{_{\!\!A}}(b)c^{*})=0.\]
Hence $C\perp(A^{\perp}\cap{}B),$ and we get: $A^{\perp},$ $A^{\perp}\cap{}B$ and $A^{\perp}\cap{}C$ are orthogonal.

\noindent{}$2\Rightarrow{}1$: Assume 2. Let $e_{_{1}},e_{_{2}}$ and $e_{_{3}}$ be the projections on the orthogonal subspaces $A^{\perp},$ $A^{\perp}\cap{}B$ and $A^{\perp}\cap{}C.$ Then
\[E_{_{\!\!B}}E_{_{\!\!C}}=(e_{_{1}}+e_{_{2}})(e_{_{1}}+e_{_{3}})=e_{_{1}}=e_{_{A}}\]
\[E_{_{\!\!C}}E_{_{\!\!B}}=(e_{_{1}}+e_{_{3}})(e_{_{1}}+e_{_{2}})=e_{_{1}}=e_{_{A}}\]

\noindent{}$2\Leftrightarrow{}3$: Let $b\in{}B$ and $c\in{}C,$ and decompose
\[b=b_{_{1}}+b_{_{2}},\;\;\;c=c_{_{1}}+c_{_{2}},\;\;\;b_{_{1}},c_{_{1}}\in{}A,\;\;\;b_{_{2}}\in{}A^{\perp}\cap{}B,\;\;\;c_{_{2}}\in{}A^{\perp}\cap{}C.\]
Then $E_{_{\!\!A}}(b)=b_{_{1}}$ and $E_{_{\!\!A}}(c)=c_{_{1}}.$ Moreover
\begin{equation}
\label{fire1lem1eqn1}
\Diprod{b}{c}=\Diprod{b_{_{1}}}{c_{_{1}}}+\Diprod{b_{_{1}}}{c_{_{2}}}+\Diprod{b_{_{2}}}{c_{_{1}}}+\Diprod{b_{_{2}}}{c_{_{2}}}.
\end{equation}
Assume 2, then the last three terms of (\ref{fire1lem1eqn1}) vanish, and we get 3.

\noindent{}Assume 3, then for $b\in{}A^{\perp}\cap{}B$ and $c\in{}A^{\perp}\cap{}C$
\[\Diprod{b}{c}=\Diprod{E_{_{\!\!A}}(b)}{E_{_{\!\!A}}(c)}=0.\]
This proves 2.
\end{proof}
\begin{lemma}
\label{fire1lem12}
Let 
\begin{equation}
\label{fire1lem12*}
\begin{array}{lcl}
      C & \subset_{L} & D \\
       \cup_{K} &\,&\cup_{H}\\
      A & \subset_{G} & B  \end{array}  
\end{equation}
be a \csq{} of \mma{}s with respect to a trace \trsub{D} on $D,$ and let $e\in{}A$ be an abelian projection with central support 1. Then
\begin{equation}
\label{fire1lem12**}
\begin{array}{ccc}
      eCe & \subset &eDe \\
       \cup &\,&\cup\\
      eAe & \subset &eBe  \end{array}  
\end{equation}
form a \csq{} with the same inclusion matrices as in (\ref{fire1lem12*}) with respect to the trace $\frac{1}{\mbox{\tiny tr}_{_{D}}(e)}\trsub{D}$ on $eDe.$ Moreover $eAe$ is abelian.
\end{lemma}
\begin{proof}
It is clear that (\ref{fire1lem12**}) is a commuting square with respect to  $\frac{1}{\mbox{\tiny tr}_{_{D}}(e)}\trsub{D},$ because $E_{_{\!\!A}},$ $E_{_{\!\!B}}$ and $E_{_{\!\!C}}$ map $eDe$ onto $eAe,$ $eBe$ resp. $eCe.$ Since $e$ has central support 1 in $A,$ the least upper bound of $\{ueu^{*}\;|\;u\in{}A\mbox{ unitary}\}$ is 1. Hence $e$ also has central support 1 in $B,$ $C$ and $D.$ Therefore the map $z\mapsto{}ze$ is an isomorphism of the center ${\cal Z}(A)$ (resp. ${\cal Z}(B),$ ${\cal Z}(C),$ ${\cal Z}(D)$) onto the center ${\cal Z}(eAe)$ (resp. ${\cal Z}(eBe),$ ${\cal Z}(eCe),$ ${\cal Z}(eDe)$), and (up to these isomorphisms) the inclusion matrices of (\ref{fire1lem12**}) are the same as those of (\ref{fire1lem12*}), because any minimal projection in $A$ (resp. $B,$ $C,$ $D$) is equivalent to a projection dominated by $e,$ since $e$ has central support 1 in all four algebras.
\end{proof}
\begin{remark}{\rm 
Note that lemma \ref{fire1lem12} tells us, that a construction of a \csq{} of \mma{} with given inclusion matrices, need only be concerned with an abelian algebra defining the smallest algebra, $A,$ of the \csq{}.
}\end{remark}
\begin{lemma}
\label{fire1lem13}
Let 
\[A\subset_{G}B\subset_{H}D\subset{}B({\cal H})\] 
be \mma{}s with the commutant of $D,$ $D',$ abelian, and let ${\cal K}$ be a Hilbert space with $\dim({\cal H})=\dim({\cal K}).$ For  $U\in{}B({\cal H},{\cal K})$ a unitary matrix, we put $A_{_{1}}=UAU^{*},$ $B_{_{1}}=UBU^{*}$  and $D_{_{1}}=UDU^{*},$ then 
\[A_{_{1}}\subset_{G}B_{_{1}}\subset_{H}D_{_{1}}\subset{}B({\cal H}),\] 
and  $D_{_{1}}'$ is abelian.
\end{lemma}
\begin{proof}
Trivial.
\end{proof}
\begin{lemma}
\label{fire1lem14}
Let ${\cal H},{\cal K}$ be finite dimensional Hilbert spaces. Let $A,D\subset{}B({\cal H})$ and\linebreak{}
$A_{_{1}},D_{_{1}}\subset{}B({\cal K})$ be four \mma{}s, such that $A\subset{}D$ and $A_{_{1}}\subset{}D_{_{1}},$ and such that the two inclusions have the same inclusion matrix, $G,$ with respect to given isomorphisms $\Phi:{\cal Z}(A)\rightarrow{}{\cal Z}(A_{_{1}})$ and $\Psi:{\cal Z}(D)\rightarrow{}{\cal Z}(D_{_{1}}).$ If furthermore $A,A_{_{1}},D'$ and $D_{_{1}}'$ are abelian, then there is a unitary $U\in{}B({\cal H},{\cal K}),$ such that
\[UAU^{*}=A_{_{1}}\;\;\;\mbox{ and }\;\;\;UDU^{*}=D_{_{1}}\]
and such that $U$ implements the given isomorphisms  $\Phi:{\cal Z}(A)\rightarrow{}{\cal Z}(A_{_{1}})$ and $\Psi:{\cal Z}(D)\rightarrow{}{\cal Z}(D_{_{1}}).$
\end{lemma}
\begin{proof}
Since $A$ and $A_{_{1}}$ are abelian, $\Phi$ is an isomorphism of $A$ onto $A_{_{1}}.$ Since the inclusion matrices of $A\subset{}D$ and $A_{_{1}}\subset{}D_{_{1}}$ are the same it follows from \cite{bratteli}, that there is an isomorphism, $\Lambda,$ of $D$ onto $D_{_{1}},$ such that $\Lambda(A)=A_{_{1}},$  $\Lambda|_{_{A}}=\Phi$ and $\Lambda|_{_{{\cal Z}(D)}}=\Psi.$ But since $D$ and $D_{_{1}}$ are type I von Neumann algebras with abelian commutants, $\Lambda$ is implemented by a unitary $U\in{}B({\cal H},{\cal K}).$ (See \cite{dixmier}, chap. III, $\S$ 3, sect. 2).
\end{proof}
\begin{cor}
\label{fire1cor5}
If $A\subset_{G}B\subset_{H}D\subset{}B({\cal H})$ and  $A_{_{1}}\subset_{K}C\subset_{L}D_{_{1}}\subset{}B({\cal K})$ are \mma{}s, such that $A,A_{_{1}},D'$ and $D_{_{1}}'$ are abelian and $GH=KL$ with respect to given isomorphisms $\Phi:{\cal Z}(A)\rightarrow{}{\cal Z}(A_{_{1}})$ and $\Psi:{\cal Z}(D)\rightarrow{}{\cal Z}(D_{_{1}}),$ then there exists a unitary $U\in{}B({\cal H},{\cal K}),$ such that
\begin{equation}
\label{fire1cor15eqn1}
UA_{_{1}}U^{*}=A,\;\;\;UD_{_{1}}U^{*}=D,\;\;\;A\subset_{K}UCU^{*}\subset_{L}D
\end{equation}
and
\begin{equation}
\label{fire1cor15eqn2}
\mbox{ad}(U)|_{_{{\cal Z}(A)}}=\Phi^{^{-1}},\;\;\;ad(U)|_{_{{\cal Z}(D)}}=\Psi^{^{-1}}.
\end{equation}
\end{cor}
\begin{proof}
Since $GH=KL$ and $A,A_{_{1}},D'$ and $D_{_{1}}'$ are abelian, lemma \ref{fire1lem14} produces a unitary $U\in{}B({\cal H},{\cal K}),$ such that $UA_{_{1}}U^{*}=A,$ $UD_{_{1}}U^{*}=D$ and such that (\ref{fire1cor15eqn2}) holds.

\noindent{}By lemma \ref{fire1lem13} we get the assertion 
\[A\subset_{K}UCU^{*}\subset_{L}D.\]
\end{proof}
We will now turn to the path model, which will allow us to build squares
\[\begin{array}{lcl}
      C & \subset_{L} & D \\
       \cup_{K} &\,&\cup_{H}\\
      A & \subset_{G} & B  \end{array}  \]
of \mma{}s with given inclusion matrices $G,$ $H,$ $K$ and $L.$

\noindent{}Let $G\in{}M_{_{nm}}({\Bbb Z}),$ $H\in{}M_{_{mq}}({\Bbb Z}),$ $K\in{}M_{_{np}}({\Bbb Z})$ and $L\in{}M_{_{pq}}({\Bbb Z})$ be matrices with non--negative entries, such that
\[GH=KL\]
and let $\Gamma_{_{G}},$ $\Gamma_{_{H}},$ $\Gamma_{_{K}}$ and $\Gamma_{_{L}}$ be the corresponding bi--partite graphs, i.e. the graphs with adjacency matrices\[\left(\begin{array}{cc}0& G\\ G^{t}& 0\end{array}\right),\;\;\;
\left(\begin{array}{cc}0&H \\ H^{t}& 0\end{array}\right),\;\;\;
\left(\begin{array}{cc}0& K\\ K^{t}& 0\end{array}\right),\;\;\;
\left(\begin{array}{cc}0& L\\ L^{t}& 0\end{array}\right).\]
The Bratteli diagram for $A\subset{}B\subset{}D$ should be of the form
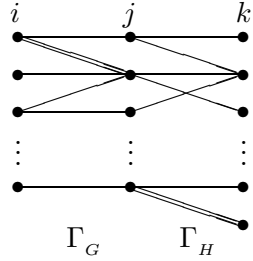
\begin{center}\label{graphadjoin}
\setlength{\unitlength}{1cm}
\begin{picture}(2,4)

\put(-0.1,3.2){$i$}
\put(1.4,3.2){$j$}
\put(2.9,3.2){$k$}
\put(0.65,0.2){$\Gamma_{_{G}}$}
\put(2.15,0.2){$\Gamma_{_{H}}$}

\put(0,1){\circle*{0.15}}
\put(-0.05,1.3){$\vdots$}
\put(0,2){\circle*{0.15}}
\put(0,2.5){\circle*{0.15}}
\put(0,3){\circle*{0.15}}

\put(1.5,1){\circle*{0.15}}
\put(1.45,1.3){$\vdots$}
\put(1.5,2){\circle*{0.15}}
\put(1.5,2.5){\circle*{0.15}}
\put(1.5,3){\circle*{0.15}}

\put(3,0.5){\circle*{0.15}}
\put(3,1){\circle*{0.15}}
\put(2.95,1.3){$\vdots$}
\put(3,2){\circle*{0.15}}
\put(3,2.5){\circle*{0.15}}
\put(3,3){\circle*{0.15}}

\put(0,3.){\line(1,0){1.5}}
\put(0,3.025){\line(3,-1){1.5}}
\put(0,2.975){\line(3,-1){1.5}}
\put(0,2.5){\line(1,0){1,5}}
\put(0,2){\line(3,1){1.5}}
\put(0,2){\line(1,0){1.5}}
\put(0,1){\line(1,0){1.5}}

\put(1.5,3){\line(1,0){1.5}}
\put(1.5,3){\line(3,-1){1.5}}
\put(1.5,2.5){\line(1,0){1.5}}
\put(1.5,2.5){\line(3,-1){1.5}}
\put(1.5,2){\line(3,1){1.5}}
\put(1.5,1){\line(1,0){1.5}}
\put(1.5,0.975){\line(3,-1){1.5}}
\put(1.5,1.025){\line(3,-1){1.5}}
\end{picture}
\end{center}
where the three columns have $n,$ $m$ and $q$ vertices respectively. The paths from the left--hand column to the right--hand column are labeled by
\[{\cal S}=\left\{(i,j,k,\rho,\sigma)\;|\;G_{_{ij}}H_{_{jk}}\neq{}0,\;1\leq{}\rho\leq{}G_{_{ij}},\;1\leq\sigma\leq{}H_{_{jk}}\right\},\]
where $\rho$ (resp. $\sigma$) labels the edges joining the same pair of vertices $(i,j)$ (resp. $(j,k)$) in case of multiple edges.

\noindent{}Let ${\cal H}$ be the Hilbert space of dimension $|{\cal S}|$ with orthonormal basis
\[\left\{\xi_{_{ijk}}^{(\rho,\sigma)}\;|\;(i,j,k,\rho,\sigma)\in{\cal S}\right\}.\]
For $x,y\in{\cal H}$ we let $x\otimes{}\overline{y}$ denote the rank one operator on ${\cal H}$ given by
\[(x\otimes{}\overline{y})(z)=(z,y)x,\;\;\; \mbox{ for } z\in{}{\cal H}\]
Set
\[p_{_{i}}=\sum_{\stackrel{j,k,\rho,\sigma}{\mbox{\tiny $(i,j,k,\rho,\sigma)\in{\cal S}$}}}\!\!\!\!\!\!\xi^{(\rho,\sigma)}_{_{i,j,k}}\otimes{}\overline{\xi}^{(\rho,\sigma)}_{_{i,j,k}},\;\;\;i=1,\ldots,n,\]
\[q_{_{j}}=\sum_{\stackrel{i,k,\rho,\sigma}{\mbox{\tiny$(i,j,k,\rho,\sigma)\in{\cal S}$}}}\!\!\!\!\!\!\xi^{(\rho,\sigma)}_{_{i,j,k}}\otimes{}\overline{\xi}^{(\rho,\sigma)}_{_{i,j,k}},\;\;\;j=1,\ldots,m\]
and
\[r_{_{k}}=\sum_{\stackrel{i,j,\rho,\sigma}{\mbox{\tiny$(i,j,k,\rho,\sigma)\in{\cal S}$}}}\!\!\!\!\!\!\xi^{(\rho,\sigma)}_{_{i,j,k}}\otimes{}\overline{\xi}^{(\rho,\sigma)}_{_{i,j,k}},\;\;\;k=1,\ldots,q.\]
Then the $p_{_{i}}'$s (resp. the $q_{_{j}}'$s and $r_{_{k}}'$s) are orthogonal projections with sum 1. For fixed $j$ the operators
\[f^{(j)}_{_{(i,\rho)(i',\rho')}}=\sum_{\stackrel{k,\sigma}{\stackrel{\mbox{\tiny$(i,j,k,\rho,\sigma)\in{\cal S}$}}{\mbox{\tiny$(i',j,k,\rho',\sigma)\in{\cal S}$}}}}\!\!\!\!\!\!\xi^{(\rho,\sigma)}_{_{i,j,k}}\otimes{}\overline{\xi}^{(\rho',\sigma)}_{_{i',j,k}}\]
form a set of matrix units  for a full matrix algebra $B_{_{j}}$ with unit 
\[\sum_{i,\rho}f^{(j)}_{_{(i,\rho)(i,\rho)}}=q_{_{j}}\]
and for fixed $k$ the operators 
\[g^{(k)}_{_{(i,j,\rho,\sigma)(i',j',\rho',\sigma')}}=\xi^{(\rho,\sigma)}_{_{i,j,k}}\otimes{}\overline{\xi}^{(\rho',\sigma')}_{_{i',j',k}}\]
form a set of matrix units for a full matrix algebra $D_{_{k}}$ with unit
\[\sum_{i,j,\rho,\sigma}g^{(k)}_{_{(i,j,\rho,\sigma)(i,j,\rho,\sigma)}}=r_{_{k}}.\]
Set
\[\begin{array}{lcl}
A&=&\bigoplus_{i}{\Bbb C}p_{_{i}},\\[0.2cm]
B&=&\bigoplus_{j}B_{_{j}},\\[0.2cm]
D&=&\bigoplus_{k}D_{_{k}},
\end{array}\]
then one easily checks that $A\subset{}B\subset{}D$ with the inclusion matrices
 $A\subset_{G}B$ and $B\subset_{H}D.$ Moreover $A$ and  the commutant, $D',$ of $D$ are abelian algebras.

In the same way we can build algebras $A_{_{1}}\subset{}C_{_{1}}\subset{}D_{_{1}}$ in $B({\cal H}_{_{1}})$ with inclusion matrices $A_{_{1}}\subset_{K}B_{_{1}}$ and $B_{_{1}}\subset_{L}D_{_{1}},$ such that $A_{_{1}}$ and  $D_{_{1}}'$ are abelian algebras. ${\cal H}_{_{1}}$ is the Hilbert space with orthonormal basis
\[\left\{\eta_{_{ilk}}^{(\phi,\psi)}\;|\;(i,l,k,\phi,\psi)\in{\cal T}\right\}\]
where
\[{\cal T}=\left\{(i,l,k,\phi,\psi)\;|\;K_{_{il}}L_{_{lk}}\neq{}0,\;1\leq{}\phi\leq{}K_{_{il}},\;1\leq{}\psi\leq{}L_{_{lk}}\right\}.\]
Moreover $A_{_{1}}=\bigoplus_{i}{\Bbb C}p_{_{i}}^{1},$ where
\[p_{_{i}}^{1}=\sum_{\stackrel{l,k,\phi,\psi}{\mbox{\tiny $(i,l,k,\phi,\psi)\in{\cal T}$}}}\!\!\!\!\!\!\eta^{(\phi,\psi)}_{_{i,l,k}}\otimes{}\overline{\eta}^{(\phi,\psi)}_{_{i,l,k}},\;\;\;i=1,\ldots,n\]
and $C_{_{1}}=\bigoplus_{l}C_{_{l}}^{1},$ $D_{_{1}}=\bigoplus_{k}D_{_{k}}^{1},$ where the minimal central projections of $C_{_{1}}$ and $D_{_{1}}$ are given by
\[s_{_{l}}^{1}=\sum_{\stackrel{i,k,\phi,\psi}{\mbox{\tiny $(i,l,k,\phi,\psi)\in{\cal T}$}}}\!\!\!\!\!\!\eta^{(\phi,\psi)}_{_{i,l,k}}\otimes{}\overline{\eta}^{(\phi,\psi)}_{_{i,l,k}},\;\;\;l=1,\ldots,p\]
\[r_{_{k}}^{1}=\sum_{\stackrel{i,l,\phi,\psi}{\mbox{\tiny $(i,l,k,\phi,\psi)\in{\cal T}$}}}\!\!\!\!\!\!\eta^{(\phi,\psi)}_{_{i,l,k}}\otimes{}\overline{\eta}^{(\phi,\psi)}_{_{i,l,k}},\;\;\;k=1,\ldots,q\]
respectively. A set of matrix units for $C_{_{l}}^{1}$ is given by
\[h^{1(l)}_{_{(i,\phi)(i',\phi')}}=\sum_{\stackrel{k,\psi}{\stackrel{\mbox{\tiny$(i,l,k,\phi,\psi)\in{\cal T}$}}{\mbox{\tiny$(i',l,k,\phi',\psi)\in{\cal T}$}}}}
\!\!\!\!\!\!\eta^{(\phi,\psi)}_{_{i,l,k}}\otimes{}\overline{\eta}^{(\phi',\psi)}_{_{i',l,k}}.\]
By corollary \ref{fire1cor5} there is a unitary $U\in{}B({\cal H},{\cal H}_{_{1}})$ such that
\begin{equation}
\label{fire1path*1}
U^{*}A_{_{1}}U=A,\;\;\;\;\;U^{*}D_{_{1}}U=D
\end{equation}
and
\begin{equation}
\label{fire1path*2}
\begin{array}{lcll}
U^{*}p_{_{i}}^{1}U&=&p_{_{i}},\;\;\;\;i=1,\ldots{}n\\[0.3cm]
U^{*}r_{_{k}}^{1}U&=&r_{_{k}},\;\;\;\;k=1,\ldots{}q.
\end{array}\end{equation}
Moreover, by lemma \ref{fire1lem13} for any unitary $U\in{}B({\cal H},{\cal H}_{_{1}})$ satisfying (\ref{fire1path*1}) and (\ref{fire1path*2}),
\begin{equation}
\label{fire1path*3}
\begin{array}{ccc}
      U^{*}C_{_{1}}U & \subset & D \\
       \cup &\,&\cup\\
      A & \subset & B  
\end{array}\end{equation}
is a square (not necessarily commuting) of \mma{}s with the given inclusion matrices $G,$ $H,$ $K$ and $L.$ Furthermore $A$ and $D'$ are abelian.

Next we shall find a necessary and sufficient condition on $U,$ for which (\ref{fire1path*3}) is a {\em commuting} square with respect to a given faithful trace \trsub{D} on $D.$

Note first that (\ref{fire1path*2}) implies (\ref{fire1path*1}) because
\[A=\mbox{span}\{p_{_{i}}\;|\;i=1,\ldots,n\},\;\;\;\;\;\;A_{_{1}}=\mbox{span}\{p_{_{i}}^{1}\;|\;i=1,\ldots,n\}\]
\[D'=\mbox{span}\{r_{_{k}}\;|\;k=1,\ldots,q\},\;\;\;\;\;\;D_{_{1}}'=\mbox{span}\{r_{_{k}}^{1}\;|\;k=1,\ldots,q\}\]
Assume that $U\in{}B({\cal H},{\cal K})$ is a unitary which satisfies (\ref{fire1path*2}). Since $p_{_{i}}$ is the projection on
\[\mbox{span}\left\{\xi^{(\rho,\sigma)}_{_{i,j,k}}\;|\;(i,j,k,\rho,\sigma)\in{\cal S},\;(i\;\mbox{fixed})\right\}\]
and $p^{1}_{_{i}}$ is the projection on
\[\mbox{span}\left\{\eta^{(\phi,\psi)}_{_{i,l,k}}\;|\;(i,l,k,\phi,\psi)\in{\cal T},\;(i\;\mbox{fixed})\right\}\]
the condition $U^{*}p_{_{i}}^{1}U=p_{_{i}}$ implies that
\begin{equation}
\label{fire1019}
\left(\xi^{(\rho,\sigma)}_{_{i,j,k}},\eta^{(\phi,\psi)}_{_{i',l,k'}}\right)=0,\;\;\;\;\;\mbox{ when } i\neq{}i',
\end{equation}
and similarly $U^{*}q_{_{k}}^{1}U=q_{_{k}}$ implies that
\begin{equation}
\label{fire10110}
\left(\xi^{(\rho,\sigma)}_{_{i,j,k}},\eta^{(\phi,\psi)}_{_{i',l,k'}}\right)=0,\;\;\;\;\;\mbox{ when } k\neq{}k'.
\end{equation}
Hence the matrix, $u,$ of $U$ with respect to the $\xi-$basis of ${\cal H}$ and the $\eta-$basis of ${\cal K},$ can be decomposed as a direct sum of unitary blocks
\[u=\bigoplus_{(i,k)}u^{(i,k)},\]
where the summation runs over all pairs $(i,k)$ for which $(GH)_{_{i,k}}=(KL)_{_{ik}}\neq{}0,$ and each block is given by
\[u^{(i,k)}=\left(
u_{(j,\rho,\sigma)(l,\phi,\psi)}^{(i,k)}\right)_{\stackrel{\mbox{\tiny$(i,j,k,\rho,\sigma)\in{\cal S}$}}{\mbox{\tiny$(i,l,k,\phi,\psi)\in{\cal T}$}}},\]
where 
\[u_{(j,\rho,\sigma)(l,\phi,\psi)}^{(i,k)}=\left(U\xi^{(\rho,\sigma)}_{_{i,j,k}},\eta^{(\phi,\psi)}_{_{i,l,k}}\right)\]
Note that each $u^{(i,k)}$ is a square matrix, with  $(GH)_{_{ik}}=(KL)_{_{ik}}$ rows and columns.

\noindent{}Conversely, if $U\in{}B({\cal H},{\cal K})$ has a direct summand decomposition as described above, then (\ref{fire1019}) and (\ref{fire10110}) hold. Thus $U$ maps $p_{_{i}}({\cal H})$ onto $p^{1}_{_{i}}({\cal K})$ and $q_{_{k}}({\cal H})$ onto $q^{1}_{_{k}}({\cal K}),$ so (\ref{fire1path*2}) holds.

Assume in the following, that $U\in{}B({\cal H},{\cal K})$ satisfies (\ref{fire1path*2}). Let $\alpha_{_{i}},$ $\beta_{_{j}},$ $\gamma_{_{l}}$ and $\delta_{_{k}}$ be the trace--weights on $A,$ $B,$ $C=U^{*}C_{_{1}}U$ and $D$ respectively.

\noindent{}For $d\in{}D$ 
\[E_{_{\!\!A}}(d)=\sum_{i}\frac{\Diprod{d}{p_{_{i}}}}{\Diprod{p_{_{i}}}{p_{_{i}}}}p_{_{i}}\]
$\Diprod{p_{_{i}}}{p_{_{i}}}=\trsub{D}(p_{_{i}})=\alpha_{_{i}},$ since $p_{_{i}}$ is a minimal projection in $A_{_{i}}.$ Hence
\[E_{_{\!\!A}}(d)=\sum_{i}\mbox{$\frac{1}{\alpha_{_{i}}}$}\trsub{D}(dp_{_{i}})p_{_{i}}.\]
$\trsub{D}(f^{(j)}_{_{(i,\rho)(i',\rho')}}p_{_{i''}})=0$ unless $i=i'=i''$ and $\rho=\rho',$ and since $f^{(j)}_{_{(i,\rho)(i,\rho)}}\leq{}p_{_{i}}$ we get
\[\trsub{D}(f^{(j)}_{_{(i,\rho)(i,\rho)}}p_{_{i}})=\trsub{D}(f^{(j)}_{_{(i,\rho)(i,\rho)}})=\beta_{_{j}},\]
because $f^{(j)}_{_{(i,\rho)(i,\rho)}}$ is a minimal projection in $B_{_{j}}.$ We now have
\begin{equation}
\label{fire10111}
E_{_{\!\!A}}(f^{(j)}_{_{(i,\rho)(i',\rho')}})=\left\{\begin{array}{cl}
\frac{\beta_{_{j}}}{\alpha_{_{i}}}p_{_{i}}&\mbox{ if } i=i'\mbox{ and }\rho=\rho'\\[0.3cm]
0&\mbox{ otherwise.}
\end{array}\right.
\end{equation}
Similarly 
\[E_{_{\!\!A_{_{1}}}}(h^{1(l)}_{_{(i,\phi)(i',\phi')}})=\left\{\begin{array}{cl}
\frac{\gamma_{_{l}}}{\alpha_{_{i}}}p^{1}_{_{i}}&\mbox{ if } i=i'\mbox{ and }\phi=\phi'\\[0.3cm]
0&\mbox{ otherwise,}
\end{array}\right.\]
or equivalently
\begin{equation}
\label{fire10112}
E_{_{\!\!A}}(h^{(l)}_{_{(i,\phi)(i',\phi')}})=\left\{\begin{array}{cl}
\frac{\gamma_{_{l}}}{\alpha_{_{i}}}p_{_{i}}&\mbox{ if } i=i'\mbox{ and }\phi=\phi'\\[0.3cm]
0&\mbox{ otherwise,}
\end{array}\right.
\end{equation}
where we set 
\[h^{(l)}_{_{(i,\phi)(i',\phi')}}=U^{*}h^{1(l)}_{_{(i,\phi)(i',\phi')}}U.\]
Note that $C=\bigoplus_{l}C_{_{l}},$ where $C_{_{l}}=U^{*}C_{_{l}}^{1}U$ and for fixed $l,$ $h^{(l)}_{_{(i,\phi)(i',\phi')}}$ form a set of matrix units for $C_{_{l}}.$

\noindent{}By (\ref{fire10111}) and (\ref{fire10112}) we get
\begin{equation}
\label{fire10113}
\trsub{D}(E_{_{\!\!A}}(f^{(j)}_{_{(i,\rho)(i',\rho')}})E_{_{\!\!A}}(h^{(l)}_{_{(i'',\phi)(i''',\phi')}}))=
\left\{\begin{array}{cl}
\frac{\beta_{_{j}}\gamma_{_{l}}}{\alpha_{_{i}}}&\mbox{ if } i=i'=i''=i''',\;\;\rho=\rho'\mbox{ and }\phi=\phi'\\[0.3cm]
0&\mbox{ otherwise.}
\end{array}\right.
\end{equation}
We will now compute
\[\trsub{D}(f^{(j)}_{_{(i,\rho)(i',\rho')}}h^{(l)}_{_{(i'',\phi)(i''',\phi')}})=\sum_{k}\delta_{_{k}}\mbox{Tr}(f^{(j)}_{_{(i,\rho)(i',\rho')}}h^{(l)}_{_{(i'',\phi)(i''',\phi')}}r_{_{k}}),\]
where Tr is the usual trace on each of the full matrix algebras $D_{_{k}},$ $k=1,\ldots,q.$

\noindent{}Note that for $x,y,z,v\in{}{\cal H}$
\begin{equation}
\label{fire10114}
\mbox{Tr}((x\otimes\overline{y})(z\otimes\overline{v}))=(z,y)\mbox{Tr}(x\otimes\overline{v})=(z,y)(x,v).
\end{equation}
We now get
\[f^{(j)}_{_{(i,\rho)(i',\rho')}}r_{_{k}}=
\left\{\begin{array}{cl}
\sum_{\sigma=1}^{H_{_{jk}}}\xi^{(\rho,\sigma)}_{_{i,j,k}}\otimes{}\overline{\xi}^{(\rho',\sigma)}_{_{i',j,k}}&\mbox{ if } H_{_{jk}}\neq{}0\\[0.3cm]
0&\mbox{ otherwise,}
\end{array}\right.\]
\[r_{_{k}}h^{(l)}_{_{(i,\phi)(i',\phi')}}=
\left\{\begin{array}{cl}
\sum_{\psi=1}^{L_{_{lk}}}U^{*}\eta^{(\phi,\psi)}_{_{i,l,k}}\otimes{}\overline{U^{*}\eta}^{(\phi',\psi)}_{_{i',l,k}}&\mbox{ if } L_{_{lk}}\neq{}0\\[0.3cm]
0&\mbox{ otherwise.}
\end{array}\right.\]
Therefore
\[\mbox{Tr}(f^{(j)}_{_{(i,\rho)(i',\rho')}}h^{(l)}_{_{(i'',\phi)(i''',\phi')}}r_{_{k}})=\mbox{Tr}(f^{(j)}_{_{(i,\rho)(i',\rho')}}r_{_{k}}r_{_{k}}h^{(l)}_{_{(i'',\phi)(i''',\phi')}})=\]
\[\sum_{\sigma=1}^{H_{_{jk}}}\sum_{\psi=1}^{L_{_{lk}}}\mbox{Tr}((\xi^{(\rho,\sigma)}_{_{i,j,k}}\otimes{}\overline{\xi}^{(\rho',\sigma)}_{_{i',j,k}})(U^{*}\eta^{(\phi,\psi)}_{_{i'',l,k}}\otimes{}\overline{U^{*}\eta}^{(\phi',\psi)}_{_{i''',l,k}}))=\]
\[\sum_{\sigma=1}^{H_{_{jk}}}\sum_{\psi=1}^{L_{_{lk}}}\left(U^{*}\eta^{(\phi,\psi)}_{_{i'',l,k}},\xi^{(\rho',\sigma)}_{_{i',j,k}}\right)
\left(\xi^{(\rho,\sigma)}_{_{i,j,k}},U^{*}\eta^{(\phi',\psi)}_{_{i''',l,k}}\right)=\]
\[\sum_{\sigma=1}^{H_{_{jk}}}\sum_{\psi=1}^{L_{_{lk}}}\left(U^{*}\eta^{(\phi,\psi)}_{_{i',l,k}},\xi^{(\rho',\sigma)}_{_{i',j,k}}\right)
\left(\xi^{(\rho,\sigma)}_{_{i,j,k}},U^{*}\eta^{(\phi',\psi)}_{_{i,l,k}}\right),\;\;\mbox{ if } i''=i',i=i''' \mbox{ and } 0 \mbox{ otherwise.}\]
Hence
\[\trsub{D}(f^{(j)}_{_{(i,\rho)(i',\rho')}}h^{(l)}_{_{(i'',\phi)(i''',\phi')}})=\sum_{k}\sum_{\sigma=1}^{H_{_{jk}}}\sum_{\psi=1}^{L_{_{lk}}}\delta_{_{k}}u_{_{(j,\rho,\sigma)(l,\phi',\psi)}}^{(i,k)}\overline{u}_{_{(j,\rho',\sigma)(l,\phi,\psi)}}^{(i',k)}\mbox{ if } i''=i',i=i''' \]
and $\trsub{D}(f^{(j)}_{_{(i,\rho)(i',\rho')}}h^{(l)}_{_{(i'',\phi)(i''',\phi')}})=0$ otherwise.

\noindent{}Combining with (\ref{fire10113}) and lemma \ref{fire1lem11} 3, we see that (\ref{fire1path*3}) is a \csq{} if and only if
\begin{equation}
\label{fire10114a}
\sum_{k}\sum_{\sigma=1}^{H_{_{jk}}}\sum_{\psi=1}^{L_{_{lk}}}\delta_{_{k}}u_{_{(j,\rho,\sigma)(l,\phi',\psi)}}^{(i,k)}\overline{u}_{_{(j,\rho',\sigma)(l,\phi,\psi)}}^{(i',k)}=
\left\{\begin{array}{cl}
\frac{\beta_{_{j}}\gamma_{_{l}}}{\alpha_{_{i}}}&\mbox{ if }i=i', \rho=\rho,,\phi=\phi'\\[0.3cm]
0&\mbox{ otherwise.}
\end{array}\right.
\end{equation}
Put
\begin{equation}
\label{fire10115}
v_{_{(i,\rho,\phi)(k,\sigma,\psi)}}^{(j,l)}=
\sqrt{\mbox{$\frac{\alpha_{_{i}}\delta_{_{k}}}{\beta_{_{j}}\gamma_{_{l}}}$}}u_{_{(j,\rho,\sigma)(l,\phi,\psi)}}^{(i,k)}.\end{equation}
Then
\[\sum_{k}\sum_{\sigma=1}^{H_{_{jk}}}\sum_{\psi=1}^{L_{_{lk}}}v_{_{(i,\rho,\phi)(k,\sigma,\psi)}}^{(j,l)}\overline{v}_{_{(i',\rho',\phi')(k,\sigma,\psi)}}^{(j,l)}=\delta_{(i,\rho,\phi)(i',\rho',\phi')},\]
where 
\[\delta_{(i,\rho,\phi)(i',\rho',\phi')}=\left\{\begin{array}{cl}
1&\mbox{ if }(i,\rho,\phi)=(i',\rho',\phi')\\[0.3cm]
0&\mbox{ otherwise.}
\end{array}\right.\]
Let $v$ be the matrix $v=\bigoplus_{(j,l)}v^{(j,l)},$ where
\[v^{(j,l)}=\left(v_{_{(i,\rho,\phi)(k,\sigma,\psi)}}^{(j,l)}\right)_{\stackrel{\mbox{\tiny$(i,j,k,\rho,\sigma)\in{\cal S}$}}{\mbox{\tiny$(i,l,k,\phi,\psi)\in{\cal T}$}}}.\]
Then the commuting square condition (\ref{fire10114a}) is equivalent to: Each summand, $v^{(j,l)},$ satisfies $v^{(j,l)}{v^{(j,l)}}^{*}=1,$ which in turn means, that $v$ is the matrix of an isometry. In particular $v^{(j,l)}$ has  at least as many columns as rows. 

\noindent{}Since
\[\mbox{\# rows\hspace{3ex}}=\sum_{i}\sum_{\rho=1}^{G_{_{ij}}}\sum_{\phi=1}^{K_{_{il}}}1=\sum_{i}G_{_{ij}}K_{_{il}}=(G^{t}K)_{_{jl}},\]
\[\mbox{\# columns} = \sum_{k}\sum_{\sigma=1}^{H_{_{jk}}}\sum_{\psi=1}^{L_{_{lk}}}1=\sum_{k}H_{_{jk}}L_{_{lk}}=(HL^{t})_{_{jl}},\]
a necessary condition for (\ref{fire1path*3}) to be a  commuting square is, that $G^{t}K\leq{}HL^{t}$ (element wise ordering).\label{fire1page2}

\noindent{}All in all  we have proved
\begin{theorem}
\label{fire1thm17}
With the notation introduced previously we have
\begin{description}
\item[1.]{}The square (\ref{fire1path*3}) of \mma{}s
\[\begin{array}{ccl}
      U^{*}C_{_{1}}U & \subset_{L} & D \\
      \;\;\;\cup_{K} &\,&\cup_{H}\\
      A & \subset_{G} & B  
\end{array}\]
has the indicated inclusion matrices if and only if $U\in{}B({\cal H},{\cal K})$ is a unitary for which the matrix, $u,$  with respect to the $\xi-$basis for ${\cal H}$ and the $\eta-$basis for ${\cal K},$ is of the form
\[u=\bigoplus_{(i,k)}u^{(i,k)}\]
where
\[u^{(i,k)}=\left(u_{_{(j,\rho,\sigma)(l,\phi,\psi)}}^{(i,k)}\right)_{\stackrel{\mbox{\tiny$(i,j,k,\rho,\sigma)\in{\cal S}$}}{\mbox{\tiny$(i,l,k,\phi,\psi)\in{\cal T}$}}}.\]
\item[2.]{}The square (\ref{fire1path*3}) is a {\em commuting} square with respect to a given faithful trace, $\trsub{D},$ on $D$  if and only if the matrix $v=\bigoplus_{(j,l)}v^{(j,l)}$ with entries
\[v_{_{(i,\rho,\phi)(k,\sigma,\psi)}}^{(j,l)}=
\sqrt{\mbox{$\frac{\alpha_{_{i}}\delta_{_{k}}}{\beta_{_{j}}\gamma_{_{l}}}$}}u_{_{(j,\rho,\sigma)(l,\phi,\psi)}}^{(i,k)}\]
is an isometry.
\item[3.]{}A necessary condition for (\ref{fire1path*3}) to be an \csq{} is $G^{t}K\leq{}HL^{t}$ (element wise ordering).
\end{description}
\end{theorem}
\begin{defn}
\label{symmetricdefn}
If
\[\begin{array}{lcl}
      C & \subset_{L} & D \\
       \cup_{K} &\,&\cup_{H}\\
      A & \subset_{G} & B  \end{array}  \]
is a \csq{}, with respect to the faithful trace, $\trsub{D},$ on $D,$ such that $G^{t}K=HL^{t}$ in addition to $GH=KL,$ we say that the square is a {\em symmetric} \csq{}
\end{defn}
\begin{remark}\label{biunitcond}{\rm  
In the case (\ref{fire1path*3}) is a symmetric \csq{}, the isometry $v$ in theorem \ref{fire1thm17} 2 becomes a unitary, so in the symmetric case we will refer to the condition in theorem \ref{fire1thm17} 2 as the {\em bi--unitary condition for the pair $(u,v)$}.
}\end{remark}
\begin{theorem}
\label{fire1thm10}
Let $G\in{}M_{_{nm}}({\Bbb Z}),$ $H\in{}M_{_{mq}}({\Bbb Z}),$ $K\in{}M_{_{np}}({\Bbb Z})$ and $L\in{}M_{_{pq}}({\Bbb Z})$ be matrices with non--negative entries, such that
\[GH=KL\;\;\mbox{ and }\;\;G^{t}K=HL^{t}.\]
Then the following conditions are equivalent
\begin{description}
\item[(a)]{}There exists a (symmetric) \csq{}
\[(A\subset{}B\subset{}D,\;\;A\subset{}C\subset{}D,\;\;\trsub{D})\]
of \mma{}s, with inclusion matrices
\[\begin{array}{lcl}
      C & \subset_{L} & D \\
       \cup_{K} &\,&\cup_{H}\\
      A & \subset_{G} & B.  \end{array}  \]
\item[(b)]{}There exists a pair of matrices $(u,v)$ satisfying the bi--unitary condition, i.e.
\[u=\bigoplus_{(i,k)}u^{(i,k)},\;\;\;\;\;v=\bigoplus_{(j,l)}v^{(j,l)}\]
where the direct summands
\[u^{(i,k)}=\left(
u_{(j,\rho,\sigma)(l,\phi,\psi)}^{(i,k)}\right)_{\stackrel{\mbox{\tiny$(i,j,k,\rho,\sigma)\in{\cal S}$}}{\mbox{\tiny$(i,l,k,\phi,\psi)\in{\cal T}$}}},\]
\[v^{(j,l)}=\left(v_{_{(i,\rho,\phi)(k,\sigma,\psi)}}^{(j,l)}\right)_{\stackrel{\mbox{\tiny$(i,j,k,\rho,\sigma)\in{\cal S}$}}{\mbox{\tiny$(i,l,k,\phi,\psi)\in{\cal T}$}}}\]
are unitary matrices and
\begin{equation}
\label{fire10116}v_{_{(i,\rho,\phi)(k,\sigma,\psi)}}^{(j,l)}=
\sqrt{\mbox{$\frac{\alpha_{_{i}}\delta_{_{k}}}{\beta_{_{j}}\gamma_{_{l}}}$}}u_{_{(j,\rho,\sigma)(l,\phi,\psi)}}^{(i,k)}
\end{equation}
Here $\alpha_{_{i}},$ $\beta_{_{j}},$ $\gamma_{_{l}}$ and $\delta_{_{k}}$ are the trace weights on $A,$ $B,$ $C$ resp. $D$ coming from $\trsub{D},$ and the indices $i,j,k,l,\rho,\sigma,\phi$ and $\psi$ are as in theorem \ref{fire1thm17}.
\end{description}
\end{theorem}
\begin{proof}

\noindent{}$(b)\Rightarrow(a)$ follows from theorem \ref{fire1thm17} and remark \ref{biunitcond}.

\noindent{}$(a)\Rightarrow(b).$ Assume $(a).$ Then by reducing with an abelian projection $e$ in $A,$ with central support 1, as in lemma \ref{fire1lem12}, we get a new \csq{},
\begin{equation}
\label{fire10117}
\begin{array}{lcl}
     {\cal C }& \subset_{L} &{\cal D} \\
       \cup_{K} &\,&\cup_{H}\\
     {\cal A} & \subset_{G} & {\cal B } \end{array}  
\end{equation}
of \mma{}s, with the same inclusion matrices, such that ${\cal A}$ is abelian. Note that the reduction with $e$ does not change the factor $\sqrt{\alpha_{_{i}}\delta_{_{k}}\left/\beta_{_{j}}\gamma_{_{l}}\right.}$ in (\ref{fire10116}), because $\alpha_{_{i}},$ $\beta_{_{j}},$ $\gamma_{_{l}}$ and $\delta_{_{k}}$ are all multiplied with the same constant $(\trsub{\cal D}(e))^{-1}.$ 

\noindent{}Next we can represent ${\cal D}$ on a Hilbert space, such that the commutant, ${\cal D}',$ is abelian.

\noindent{}As in the proof of lemma \ref{fire1lem14}, the inclusion ${\cal A}\subset_{G}{\cal B}\subset_{H}{\cal D}$ is spatially isomorphic to any other inclusion of \mma{}s, $A\subset_{G}B\subset_{H}D,$ with the same inclusion matrices, for which $A$ and $D'$ are abelian. In particular it is spatially isomorphic to $A\subset_{G}B\subset_{H}D$ coming from the path construction described previously. Similarly  ${\cal A}\subset_{K}{\cal C}\subset_{L}{\cal D}$ is spatially isomorphic to $A_{_{1}}\subset_{K}C_{_{1}}\subset_{L}D_{_{1}}$ coming from the path construction. Hence (\ref{fire10117}) is spatially isomorphic to (\ref{fire1path*3}) for some unitary $u\in{}B({\cal H},{\cal K}).$ Therefore $(a)\Rightarrow(b)$ follows from theorem \ref{fire1thm17} and remark \ref{biunitcond}.
\end{proof}
\begin{prop}
\label{fire1prop17}
If
\[(A\subset_{G}B\subset_{H}D,\;\;A\subset_{K}C\subset_{L}D,\;\;\trsub{D})\]
is a symmetric \csq{}, such that the Bratteli diagrams $\Gamma_{_{G}},$  $\Gamma_{_{H}},$  $\Gamma_{_{K}}$ and $\Gamma_{_{L}}$ are connected, then
\begin{description}
\item[(I)]{}$\|K\|=\|H\|.$ Moreover $\trsub{D}$ is the Markov trace of the embedding $C\subset{}D,$ and $\trsub{D}|_{_{B}}$ is  the Markov trace of the embedding $A\subset{}B.$
\item[(II)]{}$\|G\|=\|L\|.$ Moreover $\trsub{D}$ is the Markov trace of the embedding $B\subset{}D,$ and $\trsub{D}|_{_{C}}$ is  the Markov trace of the embedding $A\subset{}C.$
\end{description}
\end{prop}
\begin{proof}
By the assumptions $GH=KL$ and $G^{t}K=HL^{t}.$ Let $u=\bigoplus{}u^{(i,k)}$ and $v=\bigoplus{}v^{(j,l)}$ be as in theorem \ref{fire1thm10}. Let
\[N(i,j,k,l,\rho,\sigma,\phi,\psi)=\alpha_{_{i}}\delta_{_{k}}\left|u_{(j,\rho,\sigma)(l,\phi,\psi)}^{(i,k)}\right|^{2}=\beta_{_{j}}\gamma_{_{l}}\left|v_{_{(i,\rho,\phi)(k,\sigma,\psi)}}^{(j,l)}\right|^{2}\]
if $(i,j,k,\rho,\sigma)\in{\cal S}$ and $(i,l,k,\phi,\psi)\in{\cal T},$ and let $N(i,j,k,l,\rho,\sigma,\phi,\psi)=0$ otherwise. Since $u$ and $v$ are unitary we get

\begin{equation}
\label{fire1N1}
\sum_{j}\sum_{\rho=1}^{G_{_{ij}}}\sum_{\sigma=1}^{H_{_{jk}}}N(i,j,k,l,\rho,\sigma,\phi,\psi)=\left\{
\begin{array}{cl}
\alpha_{_{i}}\delta_{_{k}}&\mbox{ if there exists a path } k-l-i\\[0.2cm]
0&\mbox{ otherwise}
\end{array}\right.
\end{equation}

\begin{equation}
\label{fire1N2}
\sum_{l}\sum_{\phi=1}^{K_{_{il}}}\sum_{\psi=1}^{L_{_{lk}}}N(i,j,k,l,\rho,\sigma,\phi,\psi)=\left\{
\begin{array}{cl}
\alpha_{_{i}}\delta_{_{k}}&\mbox{ if there exists  a path } i-j-k\\[0.2cm]
0&\mbox{ otherwise}
\end{array}\right.
\end{equation}

\begin{equation}
\label{fire1N3}
\sum_{i}\sum_{\rho=1}^{G_{_{ij}}}\sum_{\phi=1}^{K_{_{il}}}N(i,j,k,l,\rho,\sigma,\phi,\psi)=\left\{
\begin{array}{cl}
\beta_{_{j}}\gamma_{_{l}}&\mbox{ if there exists a path } j-k-l \\[0.2cm]
0&\mbox{ otherwise}
\end{array}\right.
\end{equation}

\begin{equation}
\label{fire1N4}
\sum_{k}\sum_{\sigma=1}^{H_{_{jk}}}\sum_{\psi=1}^{L_{_{lk}}}N(i,j,k,l,\rho,\sigma,\phi,\psi)=\left\{
\begin{array}{cl}
\beta_{_{j}}\gamma_{_{l}}&\mbox{ if there exists a path } l-i-j\\[0.2cm]
0&\mbox{ otherwise}
\end{array}\right.
\end{equation}
where the term ``path'' refers to paths on the graphs $\Gamma_{_{X}},$ $X=G,H,K,L,$ so f.\ inst. there exists a path $k-i-l$ if and only if $L_{_{lk}}\neq{}0$ and $K_{_{il}}\neq{}0.$

\noindent{}Assume that there is an edge $l-k,$ i.e. $L_{_{lk}}\neq{}0$ By (\ref{fire1N1}) we get
\[\sum_{\stackrel{\mbox{\tiny$i$}}{\mbox{\tiny$K_{_{il}}\neq{}0$}}}\sum_{\phi=1}^{K_{_{il}}}\sum_{j}\sum_{\rho=1}^{G_{_{ij}}}\sum_{\sigma=1}^{H_{_{jk}}}N(i,j,k,l,\rho,\sigma,\phi,\psi)=\sum_{\stackrel{\mbox{\tiny$i$}}{\mbox{\tiny$K_{_{il}}\neq{}0$}}}\sum_{\phi=1}^{K_{_{il}}}\alpha_{_{i}}\delta_{_{k}}= \sum_{i}\alpha_{_{i}}K_{_{il}}\delta_{_{k}}=(K^{t}\alpha)_{_{l}}\delta_{_{k}}.\]
Hence 
\begin{equation}
\label{fire1N*1}
\sum_{j}\sum_{i}\sum_{\sigma=1}^{H_{_{jk}}}\sum_{\rho=1}^{G_{_{ij}}}\sum_{\phi=1}^{K_{_{il}}}N(i,j,k,l,\rho,\sigma,\phi,\psi)=\left\{
\begin{array}{cl}
(K^{t}\alpha)_{_{l}}\delta_{_{k}}&\mbox{ if } L_{_{lk}}\neq{}0\\[0.2cm]
0&\mbox{ otherwise}
\end{array}\right.
\end{equation}
Summing over $j$ and $\sigma$ in (\ref{fire1N3}) gives
\[(\ref{fire1N*1})=\left\{
\begin{array}{cl}
(H^{t}\beta)_{_{k}}\gamma_{_{l}}&\mbox{ if } L_{_{lk}}\neq{}0\\[0.2cm]
0&\mbox{ otherwise,}
\end{array}\right.\]
and we have
\[\frac{(H^{t}\beta)_{_{k}}}{\delta_{_{k}}}=\frac{(K^{t}\alpha)_{_{l}}}{\gamma_{_{l}}}\;\;\mbox{ if } \;\;L_{_{lk}}\neq{}0.\]
Since $\Gamma_{_{L}}$ is connected, we then have
\[\frac{(H^{t}\beta)_{_{k}}}{\delta_{_{k}}}=\frac{(K^{t}\alpha)_{_{l}}}{\gamma_{_{l}}}\;\;\mbox{ for all }\;\; l,k,\]
and hence we can find $\mu>0$ such that
\begin{equation}
\label{fire1N*2}
H^{t}\beta=\mu\delta\;\;\mbox{ and }\;\;K^{t}\alpha=\mu\gamma
\end{equation}
and since 
\[\beta=H\delta\;\;\mbox{ and }\;\;\alpha=K\gamma\]
we have
\begin{equation}
\label{fire1N*3}
H^{t}H\delta=\mu\delta\;\;\mbox{ and }\;\;K^{t}K\gamma=\mu\gamma,
\end{equation}
which shows that $\delta$ is the Perron--Frobenius eigenvector for $H^{t}H,$ and that $\gamma$ is the Perron--Frobenius eigenvector for $K^{t}K,$ both corresponding to the same eigenvalue. Hence $\|H\|=\|K\|.$ Using \cite{HGJ} theorem 2.1.3(i), (\ref{fire1N*2}) and (\ref{fire1N*3}) imply the Markov trace assertions of (I).

\noindent{}The proof of (II) follows the same lines, considering (\ref{fire1N2}) and (\ref{fire1N4}) for fixed $i,j$ such that $G_{_{ij}}\neq{}0.$
\end{proof}
Let
\[\begin{array}{lcl}
      B_{0} & \subset_{L} & B_{1} \\
       \cup_{K} &\,&\cup_{H}\\
      A_{0} & \subset_{G} & A_{1}  \end{array}  \]
be a symmetric \csq{}. By the Markov trace properties of proposition \ref{fire1prop17} and  \cite{HGJ} lemma 4.2.4 and proposition 2.4.1, we can use the fundamental construction to obtain a ladder of \mma{}s
\[\begin{array}{lclclclc}
      B_{0} & \subset_{L} & B_{1} & \subset_{L^{t}} & B_{2}
                                                & \subset_{L} & B_{3} &\cdots\\
       \cup_{K} &\,&\cup_{H}&\,&\cup_{K}&\,&\cup_{H}&\,\\
      A_{0} & \subset_{G} & A_{1}  & \subset_{G^{t}} & A_{2}
                                              & \subset_{G} & A_{3}  &\cdots  
\end{array}  \]
By \cite{HGJ} corollary 4.2.3 each
\[\begin{array}{ccc}
      B_{i} & \subset & B_{i+1} \\
       \cup &\,&\cup\\
      A_{i} & \subset & A_{i+1}  \end{array}  \]
is a commuting square.

\noindent{}Theorem 1.5 and theorem 1.6 of \cite{Wen2} and corollary \ref{cor65} now give
\begin{prop}
\label{fire1prop18}
If 
\[\begin{array}{lcl}
      B_{0} & \subset_{L} & B_{1} \\
       \cup_{K} &\,&\cup_{H}\\
      A_{0} & \subset_{G} & A_{1}  \end{array}  \]
is a symmetric \csq{}, with $\Gamma_{_{G}},$ $\Gamma_{_{H}},$ $\Gamma_{_{K}}$ and $\Gamma_{_{L}}$  connected, and if $\Gamma_{_{H}}$ or $\Gamma_{_{K}}$ has a vertex which is  connected to only one other vertex of valency 1, then there exists an irreducible subfactor, $A,$ of the hyperfinite $II_{_{1}}-$factor, $R,$ such that 
\[[R:A]=\|H\|^{2}=\|K\|^{2}\]
\end{prop}
\newpage{}

\setcounter{equation}{0}

\section{Special Symmetric Commuting Squares}
\label{firesection2}
\setcounter{equation}{0}
In this section we shall look  at which bi-partite connected graphs  $\Gamma$  with $\| \Gamma \|^{2} \in  (4,5)$ can define a commuting square of the form
\begin{equation}
\label{fire2eqn21}
\begin{array}{lcl}
      C & \subset_{nG^{t}} & D \\
       \cup_{G} &\,&\cup_{G^{t}}\\
      A & \subset_{nG} & B  \end{array}  
\end{equation}  
where  G  is the  Bratteli diagram of a bi-partition of  $\Gamma,$ and we  shall compute the bi-unitarity condition for these graphs.

Note that if the above diagram is a commuting square, then it necessarily is a symmetric commuting square. Hence by proposition \ref{fire1prop17} we know that the only tracial weights, which will satisfy the Markov-trace conditions on the inclusions  $A \subseteq{}  B,$ $ B \subseteq{}  D,$ $ A \subseteq{}  C$  and  $C \subseteq{}  D$  are those determined by the  Perron--Frobenius eigenvector of  $G.$ 
We shall now inductively define an  $n'$th degree  polynomial  $R_{_{n}}(t),$ which will be of great help in the discussion to come.

\begin{defn}
\label{firestar2def21}Let  $t \in{}{\Bbb R} $   and put  $R_{_{ 0}}(t) = 1,$ $ R_{_{1 }}(t) = t$  and define inductively  
\[R_{_{n}}(t) = tR_{_{ n-1}}(t) - R_{_{n-2}}(t),\;\;\; n\geq{}2.\]
Note that
\[\Rnt{n}=\left\{\begin{array}{cl}
\frac{\sin((n+1)x)}{\sin(x)}\;\;\;&\mbox{ if }\;\;\;t=2\cos(x),\;\;x\in{}(0,\frac{\pi}{2})\\[0.3cm]
n+1 & \mbox{ if }\;\;\;t=2\\[0.3cm]
\frac{\mbox{\rm \footnotesize sinh}((n+1)x)}{\mbox{\rm \footnotesize sinh}(x)}&\mbox{ if }\;\;\;t=2\mbox{\rm  cosh}(x),\;\;x>0
\end{array}\right.\]
\end{defn}
\begin{remark}\label{fire2extrarem}{\rm If we want to determine whether a symmetric commuting square of the form
\begin{equation}
\label{extraeqn}\begin{array}{lcl}
      C & \subset_{L} & D \\
       \cup_{K} &\,&\cup_{H}\\
      A & \subset_{G} & B  \end{array}  
\end{equation}
exists, the bi-unitary condition tells us that we have to show the existence of the two matrices $u$ and $v$ of (\ref{biunitcond}). The way we will usually proceed to prove or disprove the existence of $u$ and/or $v$ is as follows.
\begin{description}
\item[(a)]Determine all the cycles of length four in the diagram \ref{extraeqn}. These cycles label the entries of $u$ and $v,$ if such matrices  exist, and so we can group these to obtain the labeling of the blocks which $u$ and $v$ must consist of.
\item[(b)]Using:
\begin{enumerate}
\item{}The block structure determined above,
\item{}The transition rule that $u$ and $v$ must satisfy,
\item{}The matrix consisting of the moduli squared of the entries of a unitary matrix has sum of a row or a column equal to 1,
\end{enumerate}
we are, in most cases, able to determine the the values  which the moduli of the entries of $u$ and $v$ must have, if a solution exists.
\item[(c)]Determine whether ``phases'' on each entry of the, in b. determined, matrices of moduli can be found, to make these into the unitary matrices $u$ and $v$.
\end{description}}\end{remark}
\begin{lemma}
\label{fire2lem22}Let  $\Gamma$  be a bi-partite connected graph without multiple edges and without any cycles of length  4. Let  $G$  be the adjacency matrix of a bi-partition of  $ \Gamma.$

\noindent{}Consider the diagram
\[\begin{array}{lcl}
      C & \subset_{nG^{t}} & D \\
       \cup_{G} &\,&\cup_{G^{t}}\\
      A & \subset_{nG} & B  \end{array}\]
with Bratteli diagrams determined by the bi-partition of  $\Gamma.$ The labels of the minimal central projections will be

\begin{center}
\setlength{\unitlength}{1cm}
\begin{picture}(2,2)
\put(0.0,1.5){$\{l\}$}
\put(0.0,0.0){$\{i\}$}
\put(1.5,0.0){$\{j\}$}
\put(1.5,1.5){$\{k\}$} 
\put(0.25,0.5){\line(0,1){0.73}} 
\put(1.75,0.5){\line(0,1){0.73}}

\put(0.62,-0.1){\line(1,0){0.73}}
\put(0.62,0.05){\line(1,0){0.73}}
\put(0.62,0.20){\line(1,0){0.73}}

\put(0.62,1.43){\line(1,0){0.73}}
\put(0.62,1.58){\line(1,0){0.73}}
\put(0.62,1.73){\line(1,0){0.73}}
\end{picture}
\end{center} 

Denote  the unitary part of the  bi-unitary condition, corresponding to $i$  and  $k$   fixed,  by  $u$ and  the part corresponding to   $j$  and  $l$ fixed by  $v.$ If there is a solution to (\ref{fire2eqn21}) for $n$ we have, that for  $i$  and  $k$  fixed (resp.  $j$  and  $l$ fixed) the block   of  $u$  (resp. $v$)  labeled  by paths of the form  $i-j-k-l,$ $j,l$ varying (resp. $i,k$ varying), is proportional to an  $n\times{}n$  unitary, and the proportionality constant is the modulus of the corresponding entry of the matrix discussed in remark \ref{fire2extrarem} (a) and (b),  in the case  n = 1.
\end{lemma}
\begin{proof}
Let   $i,$ $j,$ $k$ and  $l$  be  such that  $G_{_{ij}}G_{_{kj}}G_{_{kl}}G_{_{il}}\neq{}0,$ i.e. there exists a path  $i-j-k -l-i$  in the diagram

\begin{center}
\setlength{\unitlength}{1cm}
\begin{picture}(2,2)
\put(0.0,1.5){$\{l\}$}
\put(0.0,0.0){$\{i\}$}
\put(1.5,0.0){$\{j\}$}
\put(1.5,1.5){$\{k\}$} 
\put(0.25,0.5){\line(0,1){0.73}} 
\put(1.75,0.5){\line(0,1){0.73}}

\put(0.62,-0.1){\line(1,0){0.73}}
\put(0.62,0.05){\line(1,0){0.73}}
\put(0.62,0.20){\line(1,0){0.73}}

\put(0.62,1.43){\line(1,0){0.73}}
\put(0.62,1.58){\line(1,0){0.73}}
\put(0.62,1.73){\line(1,0){0.73}}
\end{picture}
\end{center} 
If   $i$  and  $k$   do not  label the same vertex of  $\Gamma,$ then  $j$  and  $l$ must label the  same vertex of  $\Gamma,$ since otherwise
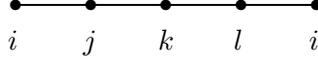
\begin{center}
\setlength{\unitlength}{1cm}
\begin{picture}(5,1)
\put(0.5,1){\line(1,0){4}}
\multiput(0.5,1)(1,0){5}{\circle*{0.15}}
\put(0.4,0.4){$i$}
\put(1.4,0.4){$j$}
\put(2.4,0.4){$k$}
\put(3.4,0.4){$l$}
\put(4.4,0.4){$i$}
\end{picture}
\end{center}
would define a cycle of length  4  on  $\Gamma.$

\noindent{}In  particular the only paths fixing   $i$  and   $k$   are of the form $(ij)_{_{\nu}}(kl)_{_{\mu}},$ where  $(ij)_{_{\nu}}$   resp.  $(k l)_{_{\mu}}$   denotes one of the  $n$  edges joining  $i$  to  $j,$ resp. $k$  to  $l,$ in the diagram. The part of  $u,$ $u^{(i,k)},$  corresponding to  $i$  and $k$   fixed is then an  $n\times{}n$ unitary, and the corresponding part of  $v$  is
\[\sqrt{\frac{\alpha_{_{i}}\delta_{_{k}}}{\beta_{_{j}}\gamma_{_{l}}}}\left(u^{(i,k)}\right)^{t}\]
and hence proportional to a unitary. The proportionality constant is the modulus of the corresponding entry in the case $n=1.$ 

\noindent{}If   $j$   and  $l$   do not label  the same  vertex of   $\Gamma$ a similar argument is valid.

\noindent{}If  $i$  and $k$  denote the same vertex of  $\Gamma,$ and $j$ and $l$  also  denote the  same vertex  of   $\Gamma,$ there  are two essentially different situations:
\begin{description}
\item[a)]$i$  is only  connected to  $j$  and  vice  versa. This  implies that $\Gamma =$
\setlength{\unitlength}{1cm}
\raisebox{-0.05cm}{\begin{picture}(1,0.3)
\put(0,0.15){\circle*{0.15}}
\put(0,0.15){\line(1,0){1}}
\put(1,0.15){\circle*{0.15}}
\end{picture}} ,  and hence  $\| \Gamma \|  = 2.$
\item[b)]Either  $i$  or $j$  is connected to  some other vertex. 
\end{description}
Assume that there exists  $j_{_{1}}\neq{}j$ s.t.  $G_{_{ij_{1}}}\neq{}0,$ and let ${\cal J}=\left\{j|G_{_{ij}}\neq{}0\right\}.$

\noindent{}The only paths joining  $i-j-k-l-i,$ which fix $i$ and  $k,$ are of the form
\[i-j_{_{1}}-k-j_{_{2}}-i,\;\mbox{ where } j_{_{1}},j_{_{2}}\in{\cal J}.\]If   $j_{_{1}}\neq{}j_{_{2}},$ we  know that  the part of $v,$ corresponding  to $j_{_{1}}$ and $j_{_{2}}$ fixed, is  an  $n\times{}n$  unitary  $v^{(j_{1},j_{2})},$  and hence the corresponding part of  $u$  is proportional to a unitary. Hence the part of  $u$  corresponding to $i$ and $k$  fixed, is of the form
\[u^{(i,k)}=\left(\begin{array}{cccc}
X_{_{j_{1}j_{1}}} &X_{_{j_{1}j_{2}}}&\cdots&X_{_{j_{1}j_{m}}}\\
X_{_{j_{2}j_{1}}}&\ddots{} & & \\
\vdots{} & & & \\
X_{_{j_{m}j_{1}}}& & &X_{_{j_{m}j_{m}}}
\end{array}\right)\]
with all the off-diagonal blocks  proportional to   $n\times{}n$   unitaries. This  implies that  the diagonal blocks are also  proportional to $n\times{}n$ unitaries, i.e.
\[X_{_{j_{q}j_{q}}}=k_{_{j_{q}j_{q}}}C_{_{j_{q}j_{q}}},\;\;C_{_{j_{q}j_{q}}}\in{}U(M_{_{n}}({\Bbb C})).\]
Since   $u^{(i,k)}$ is unitary,  the matrix formed  by all  moduli squared of the entries of $u^{(i,k)}$ is doubly stochastic. Summing over a row in this matrix yields
\[k_{_{j_{q}j_{q}}}^{2}+\sum_{j\neq{}j_{_{q}}}\frac{\beta_{_{j}}\beta_{_{j_{q}}}}{\alpha_{_{i}}^{2}}=1,\]
which is also the equation that occurs in the case $n=1.$ Hence $k_{_{j_{q}j_{q}}}$ equals the  modulus of the corresponding entry in the case $n=1.$ 
\end{proof}
\begin{cor}
\label{fire2cor23}Let    $\Gamma$  be  a connected  bi-partite graph  without multiple  edges, and consider a fixed path  $i-j-k-l-i$  in the diagram (\ref{fire2eqn21}). If one of  the edges in  $\Gamma$   labeled by   $i-j$  or  $j-k$ and one of the edges in  $\Gamma$  labeled by  $i-l$  or  $l-k$  are not edges in a cycle of length 4, then the statement of lemma \ref{fire2lem22} holds for the path  $i-j-k-l-i.$ 
\end{cor}
\begin{proof}Copy the proof of lemma \ref{fire2lem22}.\end{proof}
For the discussion to follow we let  $\Gamma$   be a bi-partite connected graph with  $\| \Gamma \|^{2}  \in  (4,5),$ such that there exists a commuting square of the form (\ref{fire2eqn21}) for some  $n.$
\begin{remark}{\rm 
The graph 
\setlength{\unitlength}{1cm}
\raisebox{-0.4cm}{\begin{picture}(2.3,1)
\put(0.15,0.5){\circle*{0.15}}
\put(1.15,0.5){\circle*{0.15}}
\put(2.15,0.5){\circle*{0.15}}
\put(0.15,0.45){\line(1,0){1}}
\put(0.15,0.55){\line(1,0){1}}
\put(1.15,0.5){\line(1,0){1}}
\end{picture}}
has Perron--Frobenius eigenvalue $\sqrt{5},$ so $\Gamma$ does not contain multiple edges.
}\end{remark}
\begin{lemma}
\label{fire2lem24}Assume  $\Gamma$ has an edge which is not  part of a 4-cycle. Denote this  edge by  \raisebox{-0.3cm}{\setlength{\unitlength}{1cm}\begin{picture}(1.7,0.4)
\put(0.5,0.4){\line(1,0){1}}
\multiput(0.5,0.4)(1,0){2}{\circle*{0.15}}
\put(0.4,0.1){$p$}
\put(1.4,0.1){$q$}
\end{picture}},  and the corresponding Perron--Frobenius coordinates by  
$\alpha_{_{p}}$ and $\alpha_{_{q}}.$ Then either 
\[\frac{\alpha_{_{p}}}{\alpha_{_{q}}}\geq{}e^{x} \mbox{ or }\frac{\alpha_{_{p}}}{\alpha_{_{q}}}\leq{}e^{-x},\] 
where  $\| \Gamma \|  = e^{x}  + e^{-x}.$
\end{lemma}
\begin{proof}If there is a solution for  $n = 1$ we have:  At  $p$  the graph will look  like
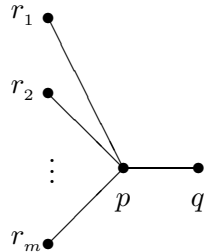
\begin{center}
\setlength{\unitlength}{1cm}
\begin{picture}(4,3)
\put(0.5,0){\circle*{0.15}}
\put(0.5,0.8){$\vdots$}
\put(0.5,2){\circle*{0.15}}
\put(0.5,3){\circle*{0.15}}
\put(1.5,1){\circle*{0.15}}
\put(2.5,1){\circle*{0.15}}
\put(1.5,1){\line(1,0){1}}
\put(0.5,2){\line(1,-1){1}}
\put(0.5,3){\line(1,-2){1}}
\put(0.5,0){\line(1,1){1}}
\put(0,0){$r_{_{m}}$}
\put(0,2){$r_{_{2}}$}
\put(0,3){$r_{_{1}}$}
\put(1.4,0.5){$p$}
\put(2.4,0.5){$q$}
\end{picture}
\end{center}
Let  $\alpha_{_{r_{i}}}$   denote the Perron--Frobenius coordinate corresponding to  $r_{_{i}},$ $ i = 1,...,m.$
\noindent{}Let  $u$  and  $v$  denote the two unitaries of the bi-unitarity condition as in  remark \ref{biunitcond}. In the labels of  cycles in  $A\; \mbox{\scriptsize \raisebox{-0.2cm}{$\stackrel{\rightarrow}{G}$}}\; B  \;\mbox{\scriptsize \raisebox{-0.25cm}{$\stackrel{\rightarrow}{G^{t}}$}}\;D\;  \mbox{\scriptsize\raisebox{-0.2cm}{$\stackrel{\rightarrow}{G}$}} \;C \; \mbox{\scriptsize\raisebox{-0.25cm}{$\stackrel{\rightarrow}{G^{t}}$}}\;A$  we may substitute   $B$  labels for   $C$   labels. Consider the  cycles of the form $p?p?$. These are
\[pqpq, \;pqpr_{_{1}},\ldots,pqpr_{_{m}}, \;pr_{_{1}}pq,\;pr_{_{1}}pr_{_{2}},\ldots,pr_{_{m}}pr_{_{m}}.\]
Since the only path of length  2    from  $q$ to  $r_{_{i}}$ passes through  $p,$ the  modulus of the corresponding entry  in say  $u$   (or  $v$  depending on which  bi-partition is chosen) is   1  for  all paths   $pr_{_{i}}pq.$ Hence the  modulus  squared  of  the corresponding  entry  of   $v$   is  $\frac{\alpha_{q}\alpha_{r_{i}}}{\alpha_{p}^{2}}.$ By double stochastic property  of the matrix formed by the moduli squared of the entries of $v,$ we get, that the modulus squared, of the entry corresponding to  $pqpq$  is  $1 - \frac{\alpha_{q}}{\alpha_{p}^{2}}\sum_{i}\alpha_{r_{i}}.$ Since  $\sum{}\alpha_{_{r_{i}}}+\alpha_{_{q}}   = \lambda{}\alpha_{_{p}},$ we have
\begin{equation}
\label{fire2eqn241}
0\leq{}1-\frac{\alpha_{q}}{\alpha_{p}^{2}}\sum_{i}\alpha_{r_{i}}\leq{}1\;\;\;\Leftrightarrow
\end{equation}
\[\left(\frac{\alpha_{q}}{\alpha_{p}}\right)^{2}-\lambda{}\left(\frac{\alpha_{q}}{\alpha_{p}}\right)+1\in[0,1]\Leftrightarrow{}
\left(\frac{\alpha_{q}}{\alpha_{p}}\right)\in{}\left[\mbox{$\frac{1}{\lambda}$},\mbox{$\frac{\lambda-\sqrt{\lambda^{2}-4}}{2}$}\right]\cup\left[\mbox{$\frac{\lambda+\sqrt{\lambda^{2}-4}}{2}$},\lambda\right]=\left[\mbox{$\frac{1}{\lambda}$},e^{-x}\right]\cup\left[e^{x},\lambda\right].\]
For general  $n,$ the statement  of corollary \ref{fire2cor23} will hold for the path $pqpq,$ and we  have that the  corresponding entries of   $u$  and   $v$  are proportional  to  $n\times{}n$  unitaries, with proportionality constants equal the  entries of the above determined matrices. When we sum the moduli-squared of a row  or column  of an  $n\times{}n$  unitary we get   1, and hence also obtain the inequality (\ref{fire2eqn241}) in this case.
\end{proof}
Assume  $\Gamma$ has no 4-cycles, and  look  at an edge   \raisebox{-0.3cm}{\setlength{\unitlength}{1cm}\begin{picture}(1.7,0.4)
\put(0.5,0.4){\line(1,0){1}}
\multiput(0.5,0.4)(1,0){2}{\circle*{0.15}}
\put(0.4,0.1){$p$}
\put(1.4,0.1){$q$}
\end{picture}}. We may assume    $\frac{\alpha_{_{p}}}{\alpha_{_{q}}}\geq{}e^{x}$ (by lemma \ref{fire2lem24}). If the number of vertices adjacent to  $q$ is  $m + 1,$ we have
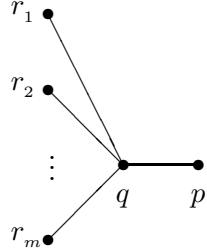
\begin{center}
\setlength{\unitlength}{1cm}
\begin{picture}(4,3)
\put(0.5,0){\circle*{0.15}}
\put(0.5,0.8){$\vdots$}
\put(0.5,2){\circle*{0.15}}
\put(0.5,3){\circle*{0.15}}
\put(1.5,1){\circle*{0.15}}
\put(2.5,1){\circle*{0.15}}
\put(1.5,1){\line(1,0){1}}
\put(0.5,2){\line(1,-1){1}}
\put(0.5,3){\line(1,-2){1}}
\put(0.5,0){\line(1,1){1}}
\put(0,0){$r_{_{m}}$}
\put(0,2){$r_{_{2}}$}
\put(0,3){$r_{_{1}}$}
\put(1.4,0.5){$q$}
\put(2.4,0.5){$p$}
\end{picture}
\end{center}
Since   $\lambda\alpha_{_{r_{i}}}\geq\alpha_{_{q}},$ $i=1,\ldots,m$ and $\alpha_{_{p}}+\sum_{i}\alpha_{_{r_{i}}}=\lambda\alpha_{_{q}}$ we have
\[\mbox{$\frac{m}{\lambda}$}+e^{x}\leq{}\lambda\Leftrightarrow{}m\leq{}\lambda{}e^{-x}=1+e^{-2x}\leq{}2.\]
Hence the valency of  $q$  is either 1 or 2.

\noindent{}If  $\Gamma$ has at least two vertices of valency  at least  3, $\Gamma$ looks like
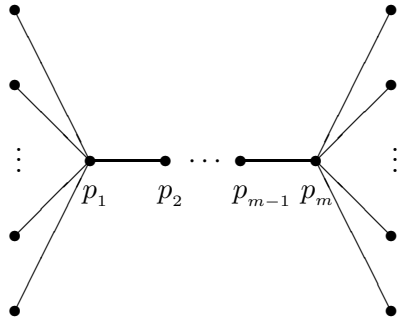
\begin{center}
\setlength{\unitlength}{1cm}
\begin{picture}(7,4)
\put(1,0){\circle*{0.15}}
\put(1,1){\circle*{0.15}}
\put(1,1.85){$\vdots$}
\put(1,3){\circle*{0.15}}
\put(1,4){\circle*{0.15}}
\put(1,0){\line(1,2){1}}
\put(1,1){\line(1,1){1}}
\put(1,3){\line(1,-1){1}}
\put(1,4){\line(1,-2){1}}

\put(2,2){\circle*{0.15}}
\put(3,2){\circle*{0.15}}
\put(2,2){\line(1,0){1}}
\put(3.3,1.9){$\cdots$}
\put(4,2){\line(1,0){1}}
\put(4,2){\circle*{0.15}}
\put(5,2){\circle*{0.15}}
\put(6,0){\circle*{0.15}}
\put(6,1){\circle*{0.15}}
\put(6,1.85){$\vdots$}
\put(6,3){\circle*{0.15}}
\put(6,4){\circle*{0.15}}
\put(6,0){\line(-1,2){1}}
\put(6,1){\line(-1,1){1}}
\put(6,3){\line(-1,-1){1}}
\put(6,4){\line(-1,-2){1}}
\put(1.9,1.5){$p_{_{1}}$}
\put(2.9,1.5){$p_{_{2}}$}
\put(3.9,1.5){$p_{_{m-1}}$}
\put(4.8,1.5){$p_{_{m}}$}
\end{picture}
\end{center}
Since the valency of $p_{_{1}}$ is three, we must have  $\frac{\alpha_{p_{2}}}{\alpha_{p_{1}}}\leq{}e^{-x},$ by the above argument. Since the $\alpha_{_{p}}'$s are  the  Perron--Frobenius coordinates, we have    $\lambda\alpha_{p_{k}}  =  \alpha_{p_{k-1}}+\alpha_{p_{k+1}}, $ $2\leq{}k\leq{}m-1.$

\noindent{}We have 
\[\frac{\alpha_{p_{3}}}{\alpha_{p_{2}}}=\frac{\lambda\alpha_{p_{2}}-\alpha_{p_{1}}}{\alpha_{p_{2}}}=\lambda-\frac{\alpha_{p_{1}}}{\alpha_{p_{2}}}=e^{x}+e^{-x}-\frac{\alpha_{p_{1}}}{\alpha_{p_{2}}}\leq{}e^{-x},\]
so $\frac{\alpha_{p_{3}}}{\alpha_{p_{2}}}\leq{}e^{-x}.$ Assume now that $\frac{\alpha_{p_{k+1}}}{\alpha_{p_{k}}}\leq{}e^{-x},$ for all $k\leq{}j\leq{}m-2.$ Then
\[\frac{\alpha_{p_{j+1}}}{\alpha_{p_{j}}}=\frac{\lambda\alpha_{p_{j}}-\alpha_{p_{j-1}}}{\alpha_{p_{j}}}=\lambda-\frac{\alpha_{p_{j-1}}}{\alpha_{p_{j}}}=e^{x}+e^{-x}-\frac{\alpha_{p_{j-1}}}{\alpha_{p_{j}}}\leq{}e^{-x},\]
and hence by induction: $\frac{\alpha_{p_{j+1}}}{\alpha_{p_{j}}}\leq{}e^{-x},$ for all $1\leq{}j\leq{}m-1.$ In particular we have $\frac{\alpha_{p_{m}}}{\alpha_{p_{m-1}}}\leq{}e^{-x},$ which contradicts that the valency of $p_{_{m}}$ is three. Hence we may conclude, that if $\Gamma$ has no cycles of length 4 then $\Gamma$ has at most one vertex of valency $\geq{}$ 3.

\noindent{}If $\Gamma$ has no cycles of length 4, but longer cycles, the discussion so far gives, that $\Gamma$ is of the form
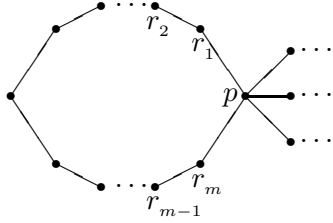
\begin{center}
\setlength{\unitlength}{0.6cm}
\begin{picture}(5,5)
\put(1,1){\circle*{0.15}}
\put(2,1.5){\circle*{0.15}}
\put(3,3){\circle*{0.15}}
\put(2,4.5){\circle*{0.15}}

\put(1,5){\circle*{0.15}}
\put(4,2){\circle*{0.15}}
\put(4,3){\circle*{0.15}}
\put(4,4){\circle*{0.15}}
\put(1,1){\line(2,1){1}}
\put(2,1.5){\line(2,3){1}}
\put(3,3){\line(-2,3){1}}
\put(2,4.5){\line(-2,1){1}}

\put(-0.2,1){\circle*{0.15}}
\put(-1.2,1.5){\circle*{0.15}}
\put(-2.2,3){\circle*{0.15}}
\put(-1.2,4.5){\circle*{0.15}}
\put(-0.2,5){\circle*{0.15}}
\put(-0.2,1){\line(-2,1){1}}
\put(-1.2,1.5){\line(-2,3){1}}
\put(-2.2,3){\line(2,3){1}}
\put(-1.2,4.5){\line(2,1){1}}

\put(0.1,5){$\ldots$}
\put(0.1,1){$\ldots$}
\put(3,3){\line(1,-1){1}}
\put(3,3){\line(1,0){1}}
\put(3,3){\line(1,1){1}}  
\put(4.2,2){$\ldots$}
\put(4.2,3){$\ldots$} 
\put(4.2,4){$\ldots$} 

\put(1.8,4){$r_{_{1}}$}
\put(0.8,4.5){$r_{_{2}}$}
\put(0.8,0.5){$r_{_{m-1}}$}
\put(1.8,1){$r_{_{m}}$}
\put(2.5,2.9){$p$}  
\end{picture}
\end{center}
We must have
\[\frac{\alpha_{_{r_{1}}}}{\alpha_{_{p}}}\leq{}e^{-x}\;\mbox{ and }\frac{\alpha_{_{r_{m}}}}{\alpha_{_{p}}}\leq{}e^{-x},\]
and, if  we use the previous argument starting in $r_{_{1}},$ we obtain
\[\frac{\alpha_{_{r_{k+1}}}}{\alpha_{_{r_{k}}}}\leq{}e^{-x},\;\;k=1,\ldots,m-1.\]
If we start in $r_{_{m}}$ we also get
\[\frac{\alpha_{_{r_{k}}}}{\alpha_{_{r_{k+1}}}}\leq{}e^{-x},\;\;k=1,\ldots,m-1.\]
Hence
\[e^{x}\leq{}\frac{\alpha_{_{r_{k+1}}}}{\alpha_{_{r_{k}}}}\leq{}e^{-x},\;\;k=1,\ldots,m-1,\]
and since $x\geq{}0$ we must have $x=0,$ and hence $\|\Gamma\|=2.$
\begin{defn}
\label{fire2defn25}
We  say that $\Gamma$ is a m-star, if  $\Gamma$ is connected, and has a ``central'' vertex  $p,$ of valency $ m,$  and   $m$  rays of the form  \raisebox{-0.3cm}{\setlength{\unitlength}{1cm}\begin{picture}(4.7,0.4)
\put(0.5,0.4){\line(1,0){2}}
\put(3.5,0.4){\line(1,0){1}}
\put(2.8,0.4){$\ldots$}
\multiput(0.5,0.4)(1,0){5}{\circle*{0.15}}
\put(0.4,0.1){$p$}
\end{picture}} with $k_{_{i}}$ vertices (not counting   $p$),  $i = 1,...,m.$   We will denote a m-star by  S$(k_{_{1}},\;k_{_{2}},\;\ldots,k_{_{m}}),$ $k_{_{1}}\leq{}k_{_{2}}\leq\cdots\leq{}k_{_{m}}.$
\end{defn}
Assume that $\Gamma$ has at least one cycle of length 4.

\noindent{}Since
\begin{enumerate}
\item{}
\[\left\|
\mbox{\setlength{\unitlength}{0.75cm}
\raisebox{-0.28cm}{\begin{picture}(1.8,1)
\put(0.5,0){\circle*{0.15}}
\put(0,0.5){\circle*{0.15}}
\put(0.5,1){\circle*{0.15}}
\put(1,0.5){\circle*{0.15}}
\put(1.5,0){\circle*{0.15}}
\put(1.5,1){\circle*{0.15}}
\put(0.5,0){\line(-1,1){0.5}}
\put(0.5,0){\line(1,1){0.5}}
\put(0.5,1){\line(-1,-1){0.5}}
\put(0.5,1){\line(1,-1){0.5}}
\put(1,0.5){\line(1,1){0.5}}
\put(1,0.5){\line(1,-1){0.5}}
\end{picture}}}
\right\|^{2}=3+\sqrt{5}>5\]
each vertex in a 4-cycle has valency at most 3.
\item{}
\[\left\|
\mbox{\setlength{\unitlength}{0.75cm}
\raisebox{-0.1cm}{\begin{picture}(1.2,0.7)
\put(0,0){\circle*{0.15}}
\put(0.5,0){\circle*{0.15}}
\put(1,0){\circle*{0.15}}
\put(0,0.5){\circle*{0.15}}
\put(0.5,0.5){\circle*{0.15}}
\put(1,0.5){\circle*{0.15}}
\put(0,0){\line(1,0){1}}
\put(0,0.5){\line(1,0){1}}
\put(0,0){\line(0,1){0.5}}
\put(0.5,0){\line(0,1){0.5}}
\end{picture}}}
\right\|^{2}=\frac{\sin^{2}\frac{3\pi}{7}}{\sin^{2}\frac{\pi}{7}}>5\]
\item{} 
\[\left\|
\mbox{\setlength{\unitlength}{0.75cm}
\raisebox{-0.28cm}{\begin{picture}(2.1,1)
\put(1,0){\circle*{0.15}}
\put(0.5,0.5){\circle*{0.15}}
\put(1,1){\circle*{0.15}}
\put(1.5,0.5){\circle*{0.15}}
\put(2,0.5){\circle*{0.15}}
\put(0,0.5){\circle*{0.15}}
\put(1,0){\line(-1,1){0.5}}
\put(1,0){\line(1,1){0.5}}
\put(1,1){\line(-1,-1){0.5}}
\put(1,1){\line(1,-1){0.5}}
\put(1.5,0.5){\line(1,0){0.5}}
\put(0,0.5){\line(1,0){0.5}}
\end{picture}}}
\right\|^{2}=5\]                         
\end{enumerate}
Hence at most one vertex in a 4-cycle has valency 3.

\noindent{}By  1, 2 and 3 an ``extra'' edge on  a vertex in a 4-cycle cannot be edge in a 4-cycle. The  argument used to show that there  is at most one vertex  of valency    3, in the case with no 4-cycles, only depends on the fact that we have an edge which not part of a 4-cycle. Hence $\Gamma$ must be of the form
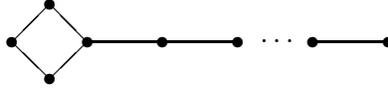
\begin{center}\setlength{\unitlength}{1cm}
\begin{picture}(5,1)
\put(0.5,0){\circle*{0.15}}
\put(0,0.5){\circle*{0.15}}
\put(0.5,1){\circle*{0.15}}
\put(1,0.5){\circle*{0.15}}
\put(2,0.5){\circle*{0.15}}
\put(3,0.5){\circle*{0.15}}
\put(4,0.5){\circle*{0.15}}
\put(5,0.5){\circle*{0.15}}
\put(3.3,0.5){$\ldots$}
\put(0.5,0){\line(-1,1){0.5}}
\put(0.5,0){\line(1,1){0.5}}
\put(0.5,1){\line(-1,-1){0.5}}
\put(0.5,1){\line(1,-1){0.5}}
\put(1,0.5){\line(1,0){2}}
\put(4,0.5){\line(1,0){1}}
\end{picture}
\end{center} 
$k + 4$  vertices, and $k\geq{}1$ since  
\[\left\|
\mbox{\setlength{\unitlength}{0.75cm}
\raisebox{-0.28cm}{\begin{picture}(1.3,1)
\put(0.5,0){\circle*{0.15}}
\put(0,0.5){\circle*{0.15}}
\put(0.5,1){\circle*{0.15}}
\put(1,0.5){\circle*{0.15}}
\put(0.5,0){\line(-1,1){0.5}}
\put(0.5,0){\line(1,1){0.5}}
\put(0.5,1){\line(-1,-1){0.5}}
\put(0.5,1){\line(1,-1){0.5}}
\end{picture}}}
\right\|^{2}=4.\]
We shall call the above graph  kite($k$ ).

\noindent{}Note that the Perron--Frobenius eigenvalue of an $m-$star with all rays of length 1 is  $\sqrt{m},$ so our discussion has excluded all graphs except 3-stars, 4-stars  and  kite($k$ ), $k\in{\Bbb N}.$ 

Since it turns out that the computation of the bi-unitarity condition for 3-stars  and 4-stars, defining the inclusions in (\ref{fire2eqn21}), is almost the same, we will use the  rest of this section to compute the condition for an $m-$star, $m  \geq{} 2.$

\begin{lemma}
\label{fire2lem26}If $\Gamma$ is a graph, with part  of $\Gamma$ a ray (in the sense of the definition of $m-$stars)
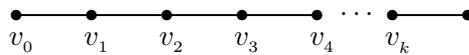
\begin{center}
\setlength{\unitlength}{1cm}
\begin{picture}(8,1)
\multiput(1,0.5)(1,0){7}{\circle*{0.15}}
\put(1,0.5){\line(1,0){4}}
\put(6,0.5){\line(1,0){1}}
\put(5.3,0.5){$\ldots$}
\put(0.9,0.1){$v_{_{0}}$}
\put(1.9,0.1){$v_{_{1}}$}                                                 \put(2.9,0.1){$v_{_{2}}$}
\put(3.9,0.1){$v_{_{3}}$}
\put(4.9,0.1){$v_{_{4}}$}
\put(5.9,0.1){$v_{_{k}}$} 
\end{picture}
\end{center}
then   the  coordinate  of  the  Perron--Frobenius vector   corresponding  to    $v_{_{i}}$   is proportional  to   $\Rnlam{i}{},$  $i = 0,\ldots,k,$  where   $\lambda$     denotes the Perron--Frobenius eigenvalue of  $\Gamma.$
\end{lemma}
\begin{proof}Let   $\alpha_{_{0}}$  denote the Perron--Frobenius coordinate at   $v_{_{0}}.$ Then   $\lambda\alpha_{_{0}}=\alpha_{_{1}},$ and we have  $\alpha_{_{1}} = R_{_{1}}(\lambda)\alpha_{_{0}}.$

\noindent{}Assume $\alpha_{_{j}} = R_{_{j}}(\lambda)\alpha_{_{0}},$ for $j\leq{}n<k.$ We then have
\[\alpha_{_{n+1}}=\lambda\alpha_{_{n}}-\alpha_{_{n-1}}=(\lambda\Rnlam{n}{}-\Rnlam{n-1}{})\alpha_{_{0}}=\Rnlam{n+1}{}\alpha_{_{0}}.\]
\end{proof}
If  we apply lemma \ref{fire2lem26} to each ray of   S($k_{_{1}},k_{_{2}},\ldots,k_{_{m}})$ we have, for the $i'$th ray,
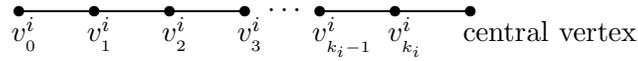
\begin{center}
\setlength{\unitlength}{1cm}
\begin{picture}(10,1)
\multiput(1,0.5)(1,0){7}{\circle*{0.15}}
\put(1,0.5){\line(1,0){3}}
\put(5,0.5){\line(1,0){2}}
\put(4.3,0.5){$\ldots$}
\put(0.9,0.1){$v_{_{0}}^{i}$}
\put(1.9,0.1){$v_{_{1}}^{i}$}                                           \put(2.9,0.1){$v_{_{2}}^{i}$}
\put(3.9,0.1){$v_{_{3}}^{i}$}
\put(4.9,0.1){$v_{_{k_{i}-1}}^{i}$}
\put(5.9,0.1){$v_{_{k_{i}}}^{i}$} 
\put(6.9,0.1){central vertex}
\end{picture}
\end{center}
and the Perron--Frobenius coordinate at  $\vij{j}{i}$    can be chosen to 
\begin{equation}
\label{PFdefnstar}\Rnlamkvot{j}{k_{i}}{},\;\;j=1,\ldots,k_{_{i}},\;i=1,\ldots,m,\end{equation}
and  the  central  vertex  is  assigned  the  value  1.  If  the  defined coordinates  shall define an eigenvector on all of   $\Gamma,$ they must match up at the central vertex to give
\[\sum_{i=1}^{m}\Rnlamkvot{k_{i}-1}{k_{i}}{}=\lambda,\]
that is, the Perron--Frobenius eigenvalue of S($k_{_{1}},k_{_{2}},\ldots,k_{_{m}})$  is the largest solution to
\begin{equation}
\label{stareigenvaleq}
\sum_{i=1}^{m}\frac{R_{_{k_{i}-1}}(t)}{R_{_{k_{i}}}(t)}=t.
\end{equation}
The above equation will be referred to as the {\em eigenvalue equation}.

To  determine the bi-unitarity condition for  the diagram (\ref{fire2eqn21}) with  $G$  the adjacency matrix of an m-star, lemma \ref{fire2lem22}    suggests that we first look   at  the case  $ n = 1,$ and  determine the  moduli of  the matrix  entries.  For  $n = 1$   cycles in the diagram  corresponds to cycles on $\Gamma$ of length 4. In this case a cycle in the diagram (\ref{fire2eqn21}) is completely determined by listing the edges 
\raisebox{-0.3cm}{\setlength{\unitlength}{1cm}\begin{picture}(1.5,0.4)
\put(0.2,0.4){\line(1,0){1}}
\multiput(0.2,0.4)(1,0){2}{\circle*{0.15}}
\put(0.1,0.1){$i$}
\put(1.1,0.1){$l$}
\end{picture}},\raisebox{-0.3cm}{\setlength{\unitlength}{1cm}\begin{picture}(1.5,0.4)
\put(0.2,0.4){\line(1,0){1}}
\multiput(0.2,0.4)(1,0){2}{\circle*{0.15}}
\put(0.1,0.1){$j$}
\put(1.1,0.1){$k$}
\end{picture}} of $\Gamma.$ Hence we can picture the blocks in a diagram with the edges \raisebox{-0.3cm}{\setlength{\unitlength}{1cm}\begin{picture}(1.5,0.4)
\put(0.2,0.4){\line(1,0){1}}
\multiput(0.2,0.4)(1,0){2}{\circle*{0.15}}
\put(0.1,0.1){$i$}
\put(1.1,0.1){$l$}
\end{picture}} on the vertical axis, and the edges \raisebox{-0.3cm}{\setlength{\unitlength}{1cm}\begin{picture}(1.7,0.4)
\put(0.5,0.4){\line(1,0){1}}
\multiput(0.5,0.4)(1,0){2}{\circle*{0.15}}
\put(0.4,0.1){$j$}
\put(1.4,0.1){$k$}
\end{picture}} on the horizontal axis. 

\noindent{}We will first look at what happens on a ray of the m-star, so far from the central vertex, that the cycles do not involve the central vertex. To indicate that a vertex is in the ``$i$'' or ``$k$'' corner of the diagram (\ref{fire2eqn21}), we will draw it as a $\star.$ The vertices of the ``$j$'' and ``$l$'' corners of the diagram are drawn as a $\bullet.$

\noindent{}We have chosen to put the central vertex in the ``$i$'' and ``$k$'' corners of the diagram. Had we chosen to put the central vertex in the other two corners, the only difference in the discussion to come would be, that all the $u$ matrices would be $v$ matrices and vice versa. Since the graphs do not contain multiple edges we do not need the $\sigma,\rho,\phi$ and $\psi$ labels in the bi--unitary condition.

\noindent{}The block structure of $u$ and $v$ is pictured in the diagrams by the thick lines.

\noindent{}The moduli in the boxes are determined using the following facts
\begin{enumerate}
\item{}If we have a $1\times{}1$ block, the modulus of this element is 1, since it has to be a $1\times{}1$ unitary.
\item{}In a $2\times{}2$ block the moduli have to be of the following form\[\left(\begin{array}{cc}a&\sqrt{1-a^{2}}\\\sqrt{1-a^{2}}&a\end{array}\right),\;\;\;\;0\leq{}a\leq{}1.\]
\item{}The transition formula between $u$ and $v$
\[\left|v_{_{ik}}^{jl}\right|=\sqrt{\frac{\alpha_{_{i}}\delta_{_{k}}}{\beta_{_{j}}\gamma_{_{l}}}}\left|u_{_{jl}}^{ik}\right|\]
\[\left|u_{_{jl}}^{ik}\right|=\sqrt{\frac{\beta_{_{j}}\gamma_{_{l}}}{\alpha_{_{i}}\delta_{_{k}}}}\left|v_{_{ik}}^{jl}\right|\]
\end{enumerate}
1, 2, and 3 allows us to determine the modulus of a entry in each $2\times{}2$ block, and hence by 2 of every entry in the block.

\noindent{}Let   $v_{_{c}}$  denote the central vertex of  S($k_{_{1}},k_{_{2}},\ldots,k_{_{m}})$ and  consider the $j'$th ray
\begin{center}
\setlength{\unitlength}{1cm}
\begin{picture}(10,1)
\multiput(2,0.5)(2,0){3}{\circle*{0.15}}
\multiput(0.95,0.45)(2,0){4}{$\star$}
\put(1,0.5){\line(1,0){3}}
\put(5,0.5){\line(1,0){2}}
\put(4.3,0.5){$\ldots$}
\put(0.9,0){$v_{_{0}}^{j}$}
\put(1.9,0){$v_{_{1}}^{j}$}                                           
\put(2.9,0){$v_{_{2}}^{j}$}
\put(3.9,0){$v_{_{3}}^{j}$}
\put(4.9,0){$v_{_{k_{j}-1}}^{j}$}
\put(5.9,0){$v_{_{k_{j}}}^{j}$}
\put(6.9,0){$c$}
\put(7.6,0.4){$k_{_{j}}$ even}
\end{picture}
\end{center}
\begin{center}
\setlength{\unitlength}{1cm}
\begin{picture}(10,1)
\multiput(1.95,0.45)(2,0){2}{$\star$}
\multiput(1,0.5)(2,0){2}{\circle*{0.15}}
\put(6,0.5){\circle*{0.15}}
\put(6.95,0.45){$\star$}
\put(4.95,0.45){$\star$}
\put(1,0.5){\line(1,0){3}}
\put(5,0.5){\line(1,0){2}}
\put(4.3,0.5){$\ldots$}
\put(0.9,0){$v_{_{0}}^{j}$}
\put(1.9,0){$v_{_{1}}^{j}$}                                           
\put(2.9,0){$v_{_{2}}^{j}$}
\put(3.9,0){$v_{_{3}}^{j}$}
\put(4.9,0){$v_{_{k_{j}-1}}^{j}$}
\put(5.9,0){$v_{_{k_{j}}}^{j}$}
\put(6.9,0){$c$}
\put(7.6,0.4){$k_{_{j}}$ odd}
\end{picture}\\[0.2cm]
\end{center} 
\begin{center}
\setlength{\unitlength}{1cm}
\begin{picture}(13,13)
\multiput(1,0)(2,0){7}{\line(0,1){12}}
\multiput(1,0)(0,2){7}{\line(1,0){12}}
\put(3.5,12.9){Ray of even length. The blocks of $u$.}
\linethickness{2pt}
\multiput(1,10.05)(0,0.2){10}{\line(0,1){0.1}}
\multiput(1.05,12)(0.2,0){10}{\line(1,0){0.1}}

\put(3,12){\line(1,0){2}}
\put(1,10){\line(1,0){6}}
\put(1,8){\line(1,0){2}}
\put(7,8){\line(1,0){2}}
\put(3,6){\line(1,0){8}}
\put(5,4){\line(1,0){2}}
\put(11,4){\line(1,0){2}}
\put(7,2){\line(1,0){6}}
\put(9,0){\line(1,0){4}}

\put(1,10){\line(0,-1){2}}
\put(3,12){\line(0,-1){6}}
\put(5,12){\line(0,-1){2}}
\put(5,6){\line(0,-1){2}}
\put(7,10){\line(0,-1){8}}
\put(9,8){\line(0,-1){2}}
\put(9,2){\line(0,-1){2}}
\put(11,6){\line(0,-1){6}}
\put(13,4){\line(0,-1){4}}

\put(0.5,12.3){$v_{_{6}}^{j}$}
\put(2.9,12.3){$v_{_{5}}^{j}$}
\put(4.9,12.3){$v_{_{4}}^{j}$}
\put(6.9,12.3){$v_{_{3}}^{j}$}
\put(8.9,12.3){$v_{_{2}}^{j}$}
\put(10.9,12.3){$v_{_{1}}^{j}$}
\put(12.9,12.3){$v_{_{0}}^{j}$}
\put(0.3,9.9){$v_{_{5}}^{j}$}
\put(0.3,7.9){$v_{_{4}}^{j}$}
\put(0.3,5.9){$v_{_{3}}^{j}$}
\put(0.3,3.9){$v_{_{2}}^{j}$}
\put(0.3,1.9){$v_{_{1}}^{j}$}
\put(0.3,-0.1){$v_{_{0}}^{j}$}
\multiput(0.925,-0.05)(0,4){4}{$\star$}
\multiput(1,2)(0,4){3}{\circle*{0.15}}
\multiput(4.925,11.95)(4,0){3}{$\star$}
\multiput(3,12)(4,0){3}{\circle*{0.15}}

\put(9.9,0.9){$1$}
\put(11.9,0.9){$1$}

\put(7.1,2.9){$\Rijklamsq{1}{3}{2}$}
\put(9.6,2.9){$\Rnlamkvot{0}{2}{}$}
\put(11.9,2.9){$1$}

\put(5.9,4.9){$1$}
\put(7.6,4.9){$\Rnlamkvot{0}{2}{}$}
\put(9.1,4.9){$\Rijklamsq{1}{3}{2}$}

\put(3.1,6.9){$\Rijklamsq{3}{5}{4}$}
\put(5.6,6.9){$\Rnlamkvot{0}{4}{}$}
\put(7.9,6.9){$1$}

\put(1.9,8.9){$1$}
\put(3.6,8.9){$\Rnlamkvot{0}{4}{}$}
\put(5.1,8.9){$\Rijklamsq{3}{5}{4}$}

\put(1.6,10.9){$\Rnlamkvot{0}{6}{}$}
\put(3.9,10.9){$1$}
\end{picture}
\end{center}
\begin{center}
\setlength{\unitlength}{1cm}
\begin{picture}(13,13)
\multiput(1,0)(2,0){7}{\line(0,1){12}}
\multiput(1,0)(0,2){7}{\line(1,0){12}}
\put(3.5,12.9){Ray of even length. The blocks of $v$.}
\linethickness{2pt}
\put(1,12){\line(1,0){4}}
\put(5,10){\line(1,0){2}}
\put(1,8){\line(1,0){8}}
\put(3,6){\line(1,0){2}}
\put(9,6){\line(1,0){2}}
\put(5,4){\line(1,0){8}}
\put(7,2){\line(1,0){2}}
\put(9,0){\line(1,0){4}}

\put(1,12){\line(0,-1){4}}
\put(3,8){\line(0,-1){2}}
\put(5,12){\line(0,-1){8}}
\put(7,10){\line(0,-1){2}}
\put(7,4){\line(0,-1){2}}
\put(9,8){\line(0,-1){8}}
\put(11,6){\line(0,-1){2}}
\put(13,4){\line(0,-1){4}}
\put(0.5,12.3){$v_{_{6}}^{j}$}
\put(2.9,12.3){$v_{_{5}}^{j}$}
\put(4.9,12.3){$v_{_{4}}^{j}$}
\put(6.9,12.3){$v_{_{3}}^{j}$}
\put(8.9,12.3){$v_{_{2}}^{j}$}
\put(10.9,12.3){$v_{_{1}}^{j}$}
\put(12.9,12.3){$v_{_{0}}^{j}$}
\put(0.3,9.9){$v_{_{5}}^{j}$}
\put(0.3,7.9){$v_{_{4}}^{j}$}
\put(0.3,5.9){$v_{_{3}}^{j}$}
\put(0.3,3.9){$v_{_{2}}^{j}$}
\put(0.3,1.9){$v_{_{1}}^{j}$}
\put(0.3,-0.1){$v_{_{0}}^{j}$}
\multiput(0.925,-0.05)(0,4){4}{$\star$}
\multiput(1,2)(0,4){3}{\circle*{0.15}}
\multiput(4.925,11.95)(4,0){3}{$\star$}
\multiput(3,12)(4,0){3}{\circle*{0.15}}
\put(9.1,0.9){$\Rijklamsq{0}{2}{1}$}
\put(11.6,0.9){$\Rnlamkvot{0}{1}{}$}

\put(7.9,2.9){$1$}
\put(9.6,2.9){$\Rnlamkvot{0}{1}{}$}
\put(11.1,2.9){$\Rijklamsq{0}{2}{1}$}

\put(5.1,4.9){$\Rijklamsq{2}{4}{3}$}
\put(7.6,4.9){$\Rnlamkvot{0}{3}{}$}
\put(9.9,4.9){$1$}

\put(3.9,6.9){$1$}
\put(5.6,6.9){$\Rnlamkvot{0}{3}{}$}
\put(7.1,6.9){$\Rijklamsq{2}{4}{3}$}

\put(1.1,8.9){$\Rijklamsq{4}{6}{5}$}
\put(3.6,8.9){$\Rnlamkvot{0}{5}{}$}
\put(5.9,8.9){$1$}

\put(1.6,10.9){$\Rnlamkvot{0}{5}{}$}
\put(3.1,10.9){$\Rijklamsq{4}{6}{5}$}
\end{picture}
\end{center}
\begin{center}
\setlength{\unitlength}{1cm}
\begin{picture}(13,13)
\multiput(1,0)(2,0){7}{\line(0,1){12}}
\multiput(1,0)(0,2){7}{\line(1,0){12}}
\put(3.5,12.9){Ray of odd length. The blocks of $u$.}
\linethickness{2pt}
\put(1,12){\line(1,0){4}}
\put(5,10){\line(1,0){2}}
\put(1,8){\line(1,0){8}}
\put(3,6){\line(1,0){2}}
\put(9,6){\line(1,0){2}}
\put(5,4){\line(1,0){8}}
\put(7,2){\line(1,0){2}}
\put(9,0){\line(1,0){4}}

\put(1,12){\line(0,-1){4}}
\put(3,8){\line(0,-1){2}}
\put(5,12){\line(0,-1){8}}
\put(7,10){\line(0,-1){2}}
\put(7,4){\line(0,-1){2}}
\put(9,8){\line(0,-1){8}}
\put(11,6){\line(0,-1){2}}
\put(13,4){\line(0,-1){4}}
\put(0.5,12.3){$v_{_{6}}^{j}$}
\put(2.9,12.3){$v_{_{5}}^{j}$}
\put(4.9,12.3){$v_{_{4}}^{j}$}
\put(6.9,12.3){$v_{_{3}}^{j}$}
\put(8.9,12.3){$v_{_{2}}^{j}$}
\put(10.9,12.3){$v_{_{1}}^{j}$}
\put(12.9,12.3){$v_{_{0}}^{j}$}
\put(0.3,9.9){$v_{_{5}}^{j}$}
\put(0.3,7.9){$v_{_{4}}^{j}$}
\put(0.3,5.9){$v_{_{3}}^{j}$}
\put(0.3,3.9){$v_{_{2}}^{j}$}
\put(0.3,1.9){$v_{_{1}}^{j}$}
\put(0.3,-0.1){$v_{_{0}}^{j}$}
\multiput(1,0)(0,4){4}{\circle*{0.15}}
\multiput(0.925,1.95)(0,4){3}{$\star$}
\multiput(5,12)(4,0){3}{\circle*{0.15}}
\multiput(2.925,11.95)(4,0){3}{$\star$}
\put(9.1,0.9){$\Rijklamsq{0}{2}{1}$}
\put(11.6,0.9){$\Rnlamkvot{0}{1}{}$}

\put(7.9,2.9){$1$}
\put(9.6,2.9){$\Rnlamkvot{0}{1}{}$}
\put(11.1,2.9){$\Rijklamsq{0}{2}{1}$}

\put(5.1,4.9){$\Rijklamsq{2}{4}{3}$}
\put(7.6,4.9){$\Rnlamkvot{0}{3}{}$}
\put(9.9,4.9){$1$}

\put(3.9,6.9){$1$}
\put(5.6,6.9){$\Rnlamkvot{0}{3}{}$}
\put(7.1,6.9){$\Rijklamsq{2}{4}{3}$}

\put(1.1,8.9){$\Rijklamsq{4}{6}{5}$}
\put(3.6,8.9){$\Rnlamkvot{0}{5}{}$}
\put(5.9,8.9){$1$}

\put(1.6,10.9){$\Rnlamkvot{0}{5}{}$}
\put(3.1,10.9){$\Rijklamsq{4}{6}{5}$}
\end{picture}
\end{center}
\begin{center}
\setlength{\unitlength}{1cm}
\begin{picture}(13,13)
\multiput(1,0)(2,0){7}{\line(0,1){12}}
\multiput(1,0)(0,2){7}{\line(1,0){12}}
\put(3.5,12.9){Ray of odd length. The blocks of $v$.}
\linethickness{2pt}
\multiput(1,10.05)(0,0.2){10}{\line(0,1){0.1}}
\multiput(1.05,12)(0.2,0){10}{\line(1,0){0.1}}

\put(3,12){\line(1,0){2}}
\put(1,10){\line(1,0){6}}
\put(1,8){\line(1,0){2}}
\put(7,8){\line(1,0){2}}
\put(3,6){\line(1,0){8}}
\put(5,4){\line(1,0){2}}
\put(11,4){\line(1,0){2}}
\put(7,2){\line(1,0){6}}
\put(9,0){\line(1,0){4}}

\put(1,10){\line(0,-1){2}}
\put(3,12){\line(0,-1){6}}
\put(5,12){\line(0,-1){2}}
\put(5,6){\line(0,-1){2}}
\put(7,10){\line(0,-1){8}}
\put(9,8){\line(0,-1){2}}
\put(9,2){\line(0,-1){2}}
\put(11,6){\line(0,-1){6}}
\put(13,4){\line(0,-1){4}}

\put(0.5,12.3){$v_{_{6}}^{j}$}
\put(2.9,12.3){$v_{_{5}}^{j}$}
\put(4.9,12.3){$v_{_{4}}^{j}$}
\put(6.9,12.3){$v_{_{3}}^{j}$}
\put(8.9,12.3){$v_{_{2}}^{j}$}
\put(10.9,12.3){$v_{_{1}}^{j}$}
\put(12.9,12.3){$v_{_{0}}^{j}$}
\put(0.3,9.9){$v_{_{5}}^{j}$}
\put(0.3,7.9){$v_{_{4}}^{j}$}
\put(0.3,5.9){$v_{_{3}}^{j}$}
\put(0.3,3.9){$v_{_{2}}^{j}$}
\put(0.3,1.9){$v_{_{1}}^{j}$}
\put(0.3,-0.1){$v_{_{0}}^{j}$}
\multiput(1,0)(0,4){4}{\circle*{0.15}}
\multiput(0.925,1.95)(0,4){3}{$\star$}
\multiput(5,12)(4,0){3}{\circle*{0.15}}
\multiput(2.925,11.95)(4,0){3}{$\star$}
\put(9.9,0.9){$1$}
\put(11.9,0.9){$1$}

\put(7.1,2.9){$\Rijklamsq{1}{3}{2}$}
\put(9.6,2.9){$\Rnlamkvot{0}{2}{}$}
\put(11.9,2.9){$1$}

\put(5.9,4.9){$1$}
\put(7.6,4.9){$\Rnlamkvot{0}{2}{}$}
\put(9.1,4.9){$\Rijklamsq{1}{3}{2}$}

\put(3.1,6.9){$\Rijklamsq{3}{5}{4}$}
\put(5.6,6.9){$\Rnlamkvot{0}{4}{}$}
\put(7.9,6.9){$1$}

\put(1.9,8.9){$1$}
\put(3.6,8.9){$\Rnlamkvot{0}{4}{}$}
\put(5.1,8.9){$\Rijklamsq{3}{5}{4}$}

\put(1.6,10.9){$\Rnlamkvot{0}{6}{}$}
\put(3.9,10.9){$1$}
\end{picture}
\end{center}\newpage{}
We will now look at a ray in the vicinity of the central vertex, but only deal with the cycles that do not ``cross over'' to another ray. For convenience the $\lambda'$s are omitted in the notation.
\begin{center}
\setlength{\unitlength}{1cm}
\begin{picture}(13,13)
\multiput(1,0)(2,0){7}{\line(0,1){12}}
\multiput(1,0)(0,2){7}{\line(1,0){12}}
\put(3.5,12.9){The blocks of $u$ close to the center.}
\linethickness{2pt}
\multiput(1,10.05)(0,0.2){10}{\line(0,1){0.1}}
\multiput(1.05,12)(0.2,0){10}{\line(1,0){0.1}}

\put(3,12){\line(1,0){2}}
\put(1,10){\line(1,0){6}}
\put(1,8){\line(1,0){2}}
\put(7,8){\line(1,0){2}}
\put(3,6){\line(1,0){8}}
\put(5,4){\line(1,0){2}}
\put(11,4){\line(1,0){2}}
\put(7,2){\line(1,0){6}}
\put(9,0){\line(1,0){2}}
\multiput(11.05,0)(0.2,0){10}{\line(1,0){0.1}}

\put(1,10){\line(0,-1){2}}
\put(3,12){\line(0,-1){6}}
\put(5,12){\line(0,-1){2}}
\put(5,6){\line(0,-1){2}}
\put(7,10){\line(0,-1){8}}
\put(9,8){\line(0,-1){2}}
\put(9,2){\line(0,-1){2}}
\put(11,6){\line(0,-1){6}}
\put(13,4){\line(0,-1){2}}
\multiput(13,0.05)(0,0.2){10}{\line(0,1){0.1}}

\put(0.5,12.3){$v_{_{c}}$}
\put(2.9,12.3){$v_{_{k_{j}}}^{j}$}
\put(4.9,12.3){$v_{_{k_{j}-1}}^{j}$}
\put(6.9,12.3){$v_{_{k_{j}-2}}^{j}$}
\put(8.9,12.3){$v_{_{k_{j}-3}}^{j}$}
\put(10.9,12.3){$v_{_{k_{j}-4}}^{j}$}
\put(12.9,12.3){$v_{_{k_{j}-5}}^{j}$}
\put(0,9.9){$v_{_{k_{j}}}^{j}$}
\put(0,7.9){$v_{_{k_{j}-1}}^{j}$}
\put(0,5.9){$v_{_{k_{j}-2}}^{j}$}
\put(0,3.9){$v_{_{k_{j}-3}}^{j}$}
\put(0,1.9){$v_{_{k_{j}-4}}^{j}$}
\put(0,-0.1){$v_{_{k_{j}-5}}^{j}$}
\multiput(0.925,-0.05)(0,4){4}{$\star$}
\multiput(1,2)(0,4){3}{\circle*{0.15}}
\multiput(4.925,11.95)(4,0){3}{$\star$}
\multiput(3,12)(4,0){3}{\circle*{0.15}}
\put(9.9,0.9){$1$}
\put(11.6,0.9){$\Rnkvot{0}{k_{j}\!-\!6}{}$}

\put(7.05,2.9){$\Rijksq{k_{j}\!-\!3}{k_{j}\!-\!5}{k_{j}\!-\!4}$}
\put(9.6,2.9){$\Rnkvot{0}{k_{j}\!-\!4}{}$}
\put(11.9,2.9){$1$}

\put(5.9,4.9){$1$}
\put(7.6,4.9){$\Rnkvot{0}{k_{j}\!-\!4}{}$}
\put(9.05,4.9){$\Rijksq{k_{j}\!-\!3}{k_{j}\!-\!5}{k_{j}\!-\!4}$}

\put(3.05,6.9){$\Rijksq{k_{j}\!-\!1}{k_{j}\!-\!3}{k_{j}\!-\!2}$}
\put(5.6,6.9){$\Rnkvot{0}{k_{j}\!-\!2}{}$}
\put(7.9,6.9){$1$}

\put(1.9,8.9){$1$}
\put(3.6,8.9){$\Rnkvot{0}{k_{j}\!-\!2}{}$}
\put(5.05,8.9){$\Rijksq{k_{j}\!-\!1}{k_{j}\!-\!3}{k_{j}\!-\!2}$}

\put(1.7,10.9){$\Rnkvot{0}{k}{}$}
\put(3.9,10.9){$1$}
\end{picture}
\end{center}
\begin{center}
\setlength{\unitlength}{1cm}
\begin{picture}(13,13)
\multiput(1,0)(2,0){7}{\line(0,1){12}}
\multiput(1,0)(0,2){7}{\line(1,0){12}}
\put(3.5,12.9){The blocks of $v$ close to the center.}
\linethickness{2pt}
\put(1,12){\line(1,0){4}}
\put(5,10){\line(1,0){2}}
\put(1,8){\line(1,0){8}}
\put(3,6){\line(1,0){2}}
\put(9,6){\line(1,0){2}}
\put(5,4){\line(1,0){8}}
\put(7,2){\line(1,0){2}}
\put(9,0){\line(1,0){4}}

\put(1,12){\line(0,-1){4}}
\put(3,8){\line(0,-1){2}}
\put(5,12){\line(0,-1){8}}
\put(7,10){\line(0,-1){2}}
\put(7,4){\line(0,-1){2}}
\put(9,8){\line(0,-1){8}}
\put(11,6){\line(0,-1){2}}
\put(13,4){\line(0,-1){4}}
\put(0.5,12.3){$v_{_{c}}$}
\put(2.9,12.3){$v_{_{k_{j}}}^{j}$}
\put(4.9,12.3){$v_{_{k_{j}-1}}^{j}$}
\put(6.9,12.3){$v_{_{k_{j}-2}}^{j}$}
\put(8.9,12.3){$v_{_{k_{j}-3}}^{j}$}
\put(10.9,12.3){$v_{_{k_{j}-4}}^{j}$}
\put(12.9,12.3){$v_{_{k_{j}-5}}^{j}$}
\put(0,9.9){$v_{_{k_{j}}}^{j}$}
\put(0,7.9){$v_{_{k_{j}-1}}^{j}$}
\put(0,5.9){$v_{_{k_{j}-2}}^{j}$}
\put(0,3.9){$v_{_{k_{j}-3}}^{j}$}
\put(0,1.9){$v_{_{k_{j}-4}}^{j}$}
\put(0,-0.1){$v_{_{k_{j}-5}}^{j}$}
\multiput(0.925,-0.05)(0,4){4}{$\star$}
\multiput(1,2)(0,4){3}{\circle*{0.15}}
\multiput(4.925,11.95)(4,0){3}{$\star$}
\multiput(3,12)(4,0){3}{\circle*{0.15}}

\put(9.05,0.9){$\Rijksq{k_{j}\!-\!4}{k_{j}\!-\!6}{k_{j}\!-\!5}$}
\put(11.6,0.9){$\Rnkvot{0}{k_{j}\!-\!5}{}$}

\put(7.9,2.9){$1$}
\put(9.6,2.9){$\Rnkvot{0}{k_{j}\!-\!5}{}$}
\put(11.05,2.9){$\Rijksq{k_{j}\!-\!4}{k_{j}\!-\!6}{k_{j}\!-\!5}$}

\put(5.05,4.9){$\Rijksq{k_{j}\!-\!2}{k_{j}\!-\!4}{k_{j}\!-\!3}$}
\put(7.6,4.9){$\Rnkvot{0}{k_{j}\!-\!3}{}$}
\put(9.9,4.9){$1$}

\put(3.9,6.9){$1$}
\put(5.6,6.9){$\Rnkvot{0}{k_{j}\!-\!3}{}$}
\put(7.05,6.9){$\Rijksq{k_{j}\!-\!2}{k_{j}\!-\!4}{k_{j}\!-\!3}$}

\put(1.05,8.9){$\Rijksq{k_{j}}{k_{j}\!-\!2}{k_{j}\!-\!1}$}
\put(3.6,8.9){$\Rnkvot{0}{k_{j}\!-\!1}{}$}
\put(5.9,8.9){$1$}

\put(1.6,10.9){$\Rnkvot{0}{k_{j}\!-\!1}{}$}
\put(3.05,10.9){$\Rijksq{k_{j}}{k_{j}\!-\!2}{k_{j}\!-\!1}$}
\end{picture}
\end{center}
We will now look at the cycles ``crossing'' the central vertex. We will list these cycles by the vertices involved. To indicate which vertices are considered fixed, we will underline the fixed ones. I.e the cycle $\underline{v}_{_{1}}v_{_{2}}\underline{v}_{_{3}}v_{_{4}}$ is different from the cycle $v_{_{1}}\underline{v}_{_{2}}v_{_{3}}\underline{v}_{_{4}}.$ The first corresponds to an entry of $u$ and the second to an entry of $v.$ 

\noindent{}For simplicity we will call the vertices $v_{_{k_{j}}}^{j}$ for $a_{_{j}},$ $j=1,\ldots,m,$ and the vertex $v_{_{c}}$ will be called $c.$

\noindent{}The only cycles, for which we have not yet determined the modulus of the corresponding elements in $u$ and $v,$  span an $m\times{}m$ matrix of $u,$ indexed by
\begin{equation}
\label{fire2cycles}
\left(\begin{array}{cccc}
\underline{c}a_{_{1}}\underline{c}a_{_{1}}&
\underline{c}a_{_{1}}\underline{c}a_{_{2}}&
\cdots &
\underline{c}a_{_{1}}\underline{c}a_{_{m}}\\
\underline{c}a_{_{2}}\underline{c}a_{_{1}}&
\underline{c}a_{_{2}}\underline{c}a_{_{2}}&
\cdots &
\underline{c}a_{_{2}}\underline{c}a_{_{m}}\\
\vdots{}&\vdots{}&\ddots{}&\vdots{}\\
\underline{c}a_{_{m}}\underline{c}a_{_{1}}&
\underline{c}a_{_{m}}\underline{c}a_{_{2}}&
\cdots &
\underline{c}a_{_{m}}\underline{c}a_{_{m}}
\end{array}\right)\end{equation}
In the above we saw, that the modulus of the entry of $u$ corresponding to the cycle $\underline{c}a_{_{i}}\underline{c}a_{_{i}}$ is $\frac{1}{R_{_{k_{i}}}(\lambda)}.$

\noindent{}The cycles $c\underline{a}_{_{i}}c\underline{a}_{_{j}},$ $i\neq{}j,$ are the only cycles fixing $a_{_{i}}$ and $a_{_{j}},$ hence they must span a $1\times{}1$ block of $v$ and the modulus of the entry of $u$ corresponding to the cycle $\underline{c}a_{_{i}}\underline{c}a_{_{j}},$ $i\neq{}j,$ is
\[\sqrt{\frac{R_{_{k_{i}-1}}(\lambda)R_{_{k_{j}-1}}(\lambda)}{R_{_{k_{i}}}(\lambda)R_{_{k_{j}}}(\lambda)}}.\]
Hence the moduli of the elements of $u$ corresponding to the cycles in (\ref{fire2cycles}) are given by
\[\left(\begin{array}{cccc}
\frac{1}{R_{_{k_{1}}}(\lambda)}&
\sqrt{\frac{R_{_{k_{1}-1}}(\lambda)R_{_{k_{2}-1}}(\lambda)}{R_{_{k_{1}}}(\lambda)R_{_{k_{2}}}(\lambda)}}&
\cdots &
\sqrt{\frac{R_{_{k_{1}-1}}(\lambda)R_{_{k_{m}-1}}(\lambda)}{R_{_{k_{1}}}(\lambda)R_{_{k_{m}}}(\lambda)}}\\[0.4cm]
\sqrt{\frac{R_{_{k_{1}-1}}(\lambda)R_{_{k_{2}-1}}(\lambda)}{R_{_{k_{1}}}(\lambda)R_{_{k_{2}}}(\lambda)}}&
\frac{1}{R_{_{k_{2}}}(\lambda)}&
\cdots &
\sqrt{\frac{R_{_{k_{2}-1}}(\lambda)R_{_{k_{m}-1}}(\lambda)}{R_{_{k_{2}}}(\lambda)R_{_{k_{m}}}(\lambda)}}\\
\vdots{}&\vdots{}&\ddots{}&\vdots{}\\[0.4cm]
\sqrt{\frac{R_{_{k_{1}-1}}(\lambda)R_{_{k_{m}-1}}(\lambda)}{R_{_{k_{1}}}(\lambda)R_{_{k_{m}}}(\lambda)}}&
\sqrt{\frac{R_{_{k_{2}-1}}(\lambda)R_{_{k_{m}-1}}(\lambda)}{R_{_{k_{2}}}(\lambda)R_{_{k_{m}}}(\lambda)}}&
\cdots &
\frac{1}{R_{_{k_{m}}}(\lambda)}
\end{array}\right).\]
To show that the matrix, formed by the squares of the entries of the matrix above, is doubly stochastic, we need to show
\[\frac{1}{R_{_{k_{i}}}(\lambda)^{2}}+\sum_{j\neq{}i}\frac{R_{_{k_{i}-1}}(\lambda)R_{_{k_{j}-1}}(\lambda)}{R_{_{k_{i}}}(\lambda)R_{_{k_{j}}}(\lambda)}=1.\]
Since $\sum_{j}\Rnlamkvot{k_{i}-1}{k_{i}}{}=\lambda$ (see \ref{stareigenvaleq}), we have
\[\begin{array}{cl}
&\frac{1}{R_{_{k_{i}}}(\lambda)^{2}}+\sum_{j\neq{}i}\frac{R_{_{k_{i}-1}}(\lambda)R_{_{k_{j}-1}}(\lambda)}{R_{_{k_{i}}}(\lambda)R_{_{k_{j}}}(\lambda)}\\[0.4cm]
=&
\frac{1}{R_{_{k_{i}}}(\lambda)^{2}}+\frac{R_{_{k_{i}-1}}(\lambda)}{R_{_{k_{i}}}(\lambda)}\left(\lambda-\frac{R_{_{k_{i}-1}}(\lambda)}{R_{_{k_{i}}}(\lambda)}\right)\\[0.4cm]
=&
\left.\left(1+\lambda{}R_{_{k_{i}-1}}(\lambda)R_{_{k_{i}}}(\lambda)-R_{_{k_{i}-1}}(\lambda)R_{_{k_{i}-1}}(\lambda)\right)\right/R_{_{k_{i}}}(\lambda)^{2}\\[0.4cm]
=&
\left.\left(1+R_{_{k_{i}-1}}(\lambda)\left(\lambda{}R_{_{k_{i}}}(\lambda)-R_{_{k_{i}-1}}(\lambda)\right)\right)\right/R_{_{k_{i}}}(\lambda)^{2}\\[0.4cm]
=&
\left.\left(1+R_{_{k_{i}-1}}(\lambda)R_{_{k_{i}+1}}(\lambda)\right)\right/R_{_{k_{i}}}(\lambda)^{2}.
\end{array}\]
Hence we need to show 
\[R_{_{k_{i}}}(\lambda)^{2}=1+R_{_{k_{i}-1}}(\lambda)R_{_{k_{i}+1}}(\lambda).\]
To show that the moduli squared of the  $2\times{}2$ blocks of $u$ and $v$ form doubly stochastic matrices, we must show 
\[R_{_{j}}(\lambda)^{2}=1+R_{_{j-1}}(\lambda)R_{_{j+1}}(\lambda),\]
which is the same identity as above. A proof of this identity if found as part of the proof of lemma \ref{firegraflem0}.

\begin{remark}{\rm
If $u$ is a $n\times{}n$ unitary  and $I$ denotes the $n\times{}n$ identity matrix, then the matrix
\[\left(\begin{array}{cc}tu& \sqrt{1-t^{2}}I\\\sqrt{1-t^{2}}I&-tu^{*}\end{array}\right),\;\;0\leq{}t\leq{}1\]
is unitary in $M_{_{2n}}({\Bbb C}).$
}\end{remark}

\noindent{}The only part of the bi-unitary condition which is  non-trivial to solve, is the  $M_{_{m}}(M_{_{n}}({\Bbb C}))$  part  indexed by
\[\left(\underline{c}a_{_{j_{1}}}\underline{c}a_{_{j_{2}}}\right)_{j_{1},j_{2}=1}^{m}.\]
If we have a solution to this part, it will determine some of the entries in the $2\times{}2$ blocks of $u$ resp. $v$. However, at most one entry of a $2\times{}2$ block is determined, and the three remaining entries can be determined by the above remark. Continuing the argument, as we move towards the end of a ray, we see that for each  $2\times{}2$ block only one entry is determined by the previous blocks, and hence a solution to the entire $2\times{}2$ block can be determined.

\noindent{}If we put
\[\alpha_{_{j}}=\Rnlamkvot{k_{j}-1}{k_{j}}{}\;\mbox{ and }\delta_{_{j}}=\frac{1}{R_{_{k_{j}}}(\lambda)}\]
we have
\begin{prop}
\label{fire2prop27}
If $\Gamma$= S$(k_{_{1}},\ldots,k_{_{m}})$  and  $\alpha_{_{j}},\delta_{_{j}}$  are defined as above, there exists a commuting square of the form (\ref{fire2eqn21}) if and only if there exists $n\times{}n$ unitaries $u_{_{ij}}$ such that
\[\left(\begin{array}{cccc}
\delta_{_{1}}u_{_{11}} & \sqalialj{1}{2}u_{_{12}} & \cdots & \sqalialj{1}{m}u_{_{1m}}\\
\sqalialj{1}{2}u_{_{21}}&\delta_{_{2}}u_{_{22}}& \cdots & \sqalialj{2}{m}u_{_{2m}}\\
\vdots & \vdots & \ddots & \vdots\\
\sqalialj{m}{1}u_{_{m1}}&\sqalialj{m}{2}u_{_{m2}}& & \delta_{_{m}}u_{_{mm}}
\end{array}\right)\]
is a unitary matrix.
\end{prop}For  $\Gamma$= S$(k_{_{1}},k_{_{2}},k_{_{3}},k_{_{4}})$ the  answer to this  problem is given  by the theorems \ref{fire4thm411} and  \ref{fire4thm412},   which state,  that if  we label the  $\delta'$s   such that $\delta_{_{1}}\geq{}\delta_{_{2}}\geq{}\delta_{_{3}}\geq{}\delta_{_{4}}> 0,$ then
\begin{enumerate}
\item{}We  need only look  at  $n = 2,$ and a solution exists if and only if $\delta_{_{1}}-\delta_{_{2}}-\delta_{_{3}}-\delta_{_{4}}\leq{}0.$
\item{}A  solution  in the  case   $n =  1$   exists if and only if  $\delta_{_{1}}-\delta_{_{2}}-\delta_{_{3}}-\delta_{_{4}}\leq{}0$ and $\delta_{_{1}}-\delta_{_{2}}-\delta_{_{3}}+\delta_{_{4}}\geq{}0.$
\end{enumerate}
Since the proof of  the above results is long, we  have devoted  section \ref{algnece} to the proof,  and will concentrate on  the 3-stars in  the remainder of this  section.
\begin{lemma}
\label{33existence}Let
\[D=\left(\begin{array}{ccc}
d_{_{11}}&d_{_{12}}&d_{_{13}}\\
d_{_{21}}&d_{_{22}}&d_{_{23}}\\
d_{_{31}}&d_{_{32}}&d_{_{33}}
\end{array}\right)\]
be  doubly stochastic, and put  $\alpha  = d_{_{11}}d_{_{21}},$ $\beta  = d_{_{12}}d_{_{22}}$ and $\gamma  = d_{_{13}}d_{_{23}},$ then there exists a unitary  $u = (u_{_{ij}})$  such that   $|u_{_{ij}}|^{2}=d_{_{ij}},$ $i,j=1,2,3,$ if  and only if $\sqrt{\alpha},$ $\sqrt{\beta}$ and $\sqrt{\gamma}$ satisfy the triangle inequality. The last condition is equivalent to
\[\alpha^{2}+\beta^{2}+\gamma^{2}-2\alpha\beta-2\alpha\gamma-2\beta\gamma\leq{}0.\]
\end{lemma}
\begin{proof}If such a unitary exists, we can choose it to be of the form
\[\left(\begin{array}{ccc}
\sqrt{d_{_{11}}}&\sqrt{d_{_{12}}}&\sqrt{d_{_{13}}}\\
\sqrt{d_{_{21}}}&\sqrt{d_{_{22}}}e^{i\phi}&\sqrt{d_{_{23}}}e^{i\theta}\\
\sqrt{d_{_{31}}}&\sqrt{d_{_{32}}}e^{i\psi}&\sqrt{d_{_{33}}}e^{i\eta}
\end{array}\right)\]
and orthogonality of the two first rows implies
\begin{equation}\label{fire2lem28eqn1}
\sqrt{d_{_{11}}d_{_{21}}}+\sqrt{d_{_{12}}d_{_{22}}}e^{i\phi}+\sqrt{d_{_{13}}d_{_{23}}}e^{i\theta}=0\Leftrightarrow\sqrt{\alpha}+\sqrt{\beta}e^{i\phi}+\sqrt{\gamma}e^{i\theta}=0.\end{equation}
Conversely, if we can find scalars $e^{i\psi}$ and  $e^{i\theta}$ such that (\ref{fire2lem28eqn1}) is satisfied, the double  stochastic property  of   D  assures,  that we can find  a unitary of the desired form, by putting the $3'$rd row equal to the conjugate vector product of the $1'$st and $2'$nd row. 

\noindent{}Hence a necessary and sufficient condition for the existence of a unitary with the stated properties is
\[\begin{array}{cl}
& \sqrt{\alpha},\;\;\sqrt{\beta} \mbox{ and } \sqrt{\gamma}\mbox{ satisfy the triangle inequality }\\
\Updownarrow{} & \\
&|\sqrt{\alpha}-\sqrt{\beta}|\leq{}\sqrt{\gamma}\leq\sqrt{\alpha}+\sqrt{\beta}\\
\Updownarrow &\\
&\alpha^{2}+\beta^{2}+\gamma^{2}-2\alpha\beta-2\alpha\gamma-2\beta\gamma\leq{}0
\end{array}\]
\end{proof}
\begin{lemma}
\label{fire2lem29}Let
\[D=\left(\begin{array}{ccc}
d_{_{11}}&d_{_{12}}&d_{_{13}}\\
d_{_{21}}&d_{_{22}}&d_{_{23}}\\
d_{_{31}}&d_{_{32}}&d_{_{33}}
\end{array}\right)\]
be doubly stochastic, and put  $\alpha{} = d_{_{11}}d_{_{21}},$   $\beta{}=d_{_{12}}d_{_{22}}$ and $\gamma{}=d_{_{13}}d_{_{23}}.$ If there exists  $n\times{}n$   unitaries  $u_{_{ij}},$   $i,j = 1,2,3,$  such that
\[\left(\begin{array}{ccc}
\sqrt{d_{_{11}}}u_{_{11}}&\sqrt{d_{_{12}}}u_{_{12}}&\sqrt{d_{_{13}}}u_{_{13}}\\
\sqrt{d_{_{21}}}u_{_{21}}&\sqrt{d_{_{22}}}u_{_{22}}&\sqrt{d_{_{23}}}u_{_{23}}\\
\sqrt{d_{_{31}}}u_{_{31}}&\sqrt{d_{_{32}}}u_{_{32}}&\sqrt{d_{_{33}}}u_{_{33}}
\end{array}\right)\]
is unitary, then  $\sqrt{\alpha},$ $\sqrt{\beta}$ and $\sqrt{\gamma}$ satisfy the triangle inequality.
\end{lemma}
\begin{proof}If  $u_{_{ij}}$  exists, we have
\[\sqrt{d_{_{11}}d_{_{21}}}u_{_{11}}^{*}u_{_{21}}+\sqrt{d_{_{12}}d_{_{22}}}u_{_{12}}^{*}u_{_{22}}+\sqrt{d_{_{13}}d_{_{23}}}u_{_{13}}^{*}u_{_{23}}=0\Leftrightarrow{}
\sqrt{\alpha}u_{_{11}}^{*}u_{_{21}}+\sqrt{\beta}u_{_{12}}^{*}u_{_{22}}+\sqrt{\gamma}u_{_{13}}^{*}u_{_{23}}=0\]
Let  $\|\cdot{}\|_{_{HS}}$ denote the Hilbert-Schmidt norm on  $M_{_{n}}({\Bbb C}),$ then $\|u\|_{_{HS}}=n$ for any unitary 

\noindent{}$u\in{}M_{_{n}}({\Bbb C}),$ and we have
\[\begin{array}{cl}
& \left\|\sqrt{\alpha}u_{_{11}}^{*}u_{_{21}}+\sqrt{\beta}u_{_{12}}^{*}u_{_{22}}+\sqrt{\gamma}u_{_{13}}^{*}u_{_{23}}\right\|_{_{HS}}=0\\
\Updownarrow &\\
& \left\|\sqrt{\alpha}u_{_{11}}^{*}u_{_{21}}\right\|_{_{HS}},\;\left\|\sqrt{\beta}u_{_{12}}^{*}u_{_{22}}\right\|_{_{HS}},\;\left\|\sqrt{\gamma}u_{_{13}}^{*}u_{_{23}}\right\|_{_{HS}} \mbox{ satisfy the triangle inequality}\\ \Updownarrow &\\
& \sqrt{\alpha},\;\;\sqrt{\beta} \mbox{ and } \sqrt{\gamma}\mbox{ satisfy the triangle inequality }
\end{array}\]
\end{proof}
Lemmas \ref{33existence} and \ref{fire2lem29} now tell us, that we need only look  at  $n = 1$  in the diagram (\ref{fire2eqn21}), when $\Gamma$ is a 3-star.

\noindent{}Now  let $\Gamma$ = S$(k_{_{1}},k_{_{2}},k_{_{3}}),$ the  critical part of the  bi-unitarity  condition  of (\ref{fire2eqn21})  is  the existence  of a  $3\times{}3$ unitary $u$ such that
\[|u_{_{ij}}|^{2}=\left\{\begin{array}{ll}
\delta_{_{i}}^{2},&i=j\\
\alpha_{_{i}}\alpha_{_{j}},&i\neq{}j\end{array}\right.\]
where $\delta_{_{i}}=\frac{1}{R_{_{k_{i}}}(\lambda)}$ and $\alpha_{_{i}}=\frac{R_{_{k_{i}-1}}(\lambda)}{R_{_{k_{i}}}(\lambda)}.$ Note that
\begin{equation}
\label{deltadef}\delta_{_{i}}=\sqrt{\alpha_{_{i}}^{2}-\lambda\alpha_{_{i}}+1},\;\;i=1,2,3,\end{equation}
and that the eigenvalue equation (\ref{stareigenvaleq}) is
\[\alpha_{_{1}}+\alpha_{_{2}}+\alpha_{_{3}}=\lambda.\]
Put   $\alpha  = \delta_{_{1}}^{2}\alpha_{_{1}}\alpha_{_{2}},$   $\beta  = \alpha_{_{1}}\alpha_{_{2}}\delta_{_{2}}^{2}$ and $  \gamma  = \alpha_{_{1}} \alpha_{_{2}}  \alpha_{_{3}}^{2},$ then by lemma \ref{33existence} $u$ exists if and only if
\[\begin{array}{cl}& \alpha^{2} + \beta^{2} + \gamma^{2} - 2 \alpha\beta- 2\alpha \gamma  -2\beta\gamma\leq{}   0\\
\Updownarrow & \\
& (\gamma-\alpha-\beta)^{2}\leq{}4\alpha\beta
\end{array},\]
and since
\[\delta_{_{1}}^{2}=\alpha_{_{1}}^{2}-\lambda\alpha_{_{1}}+1=\alpha_{_{1}}^{2}-(\alpha_{_{1}}+\alpha_{_{2}}+\alpha_{_{3}})\alpha_{_{1}}+1=1-\alpha_{_{1}}\alpha_{_{2}}-\alpha_{_{1}}\alpha_{_{3}},\]
and
\[\delta_{_{2}}^{2}=1-\alpha_{_{1}}\alpha_{_{2}}-\alpha_{_{2}}\alpha_{_{3}},\]
we get
\[\begin{array}{cl}
& (\gamma-\alpha-\beta)^{2}\leq{}4\alpha\beta\\
\Updownarrow&\\
&(\alpha_{_{3}}^{2}+2\alpha_{_{1}}\alpha_{_{2}}+\alpha_{_{1}}\alpha_{_{3}}+\alpha_{_{2}}\alpha_{_{3}}-2)^{2}\leq{}4(1-\alpha_{_{1}}\alpha_{_{2}}-\alpha_{_{1}}\alpha_{_{3}})(1-\alpha_{_{1}}\alpha_{_{2}}-\alpha_{_{2}}\alpha_{_{3}})\\
\Updownarrow&\\
&\alpha_{_{3}}^{4}+\alpha_{_{1}}^{2}\alpha_{_{3}}^{2}+\alpha_{_{2}}^{2}\alpha_{_{3}}^{2}+2\alpha_{_{1}}\alpha_{_{3}}^{3}+2\alpha_{_{2}}\alpha_{_{3}}^{3}-4\alpha_{_{3}}^{2}+2\alpha_{_{1}}\alpha_{_{2}}\alpha_{_{3}}^{2}\leq{}0\\
\Updownarrow&\\
&\alpha_{_{3}}^{2}((\alpha_{_{1}}+\alpha_{_{2}}+\alpha_{_{3}})^{2}-4)\leq{}0\\
\Updownarrow&\\
&\lambda=\alpha_{_{1}}+\alpha_{_{2}}+\alpha_{_{3}}\leq{}2.
\end{array}\]
Hence  a 3-star $\Gamma$  can only define  the inclusions of a symmetric  commuting square of the form (\ref{fire2eqn21}) if  $\| \Gamma \|\leq{}2.$

We  shall  now briefly  discuss the  remaining  type of  graph,  which may produce commuting squares of the form (\ref{fire2eqn21}), with index in the interval $(4,5).$

\noindent{}Consider the graph
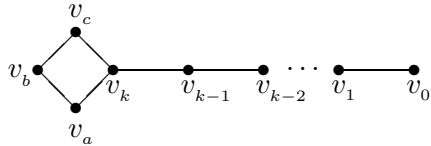
\begin{center}\setlength{\unitlength}{1cm}
\begin{picture}(6,2)
\put(1,0.5){\circle*{0.15}}
\put(0.5,1){\circle*{0.15}}
\put(1,1.5){\circle*{0.15}}
\put(1.5,1){\circle*{0.15}}
\put(2.5,1){\circle*{0.15}}
\put(3.5,1){\circle*{0.15}}
\put(4.5,1){\circle*{0.15}}
\put(5.5,1){\circle*{0.15}}
\put(0.9,0.1){$v_{_{a}}$}
\put(0.9,1.7){$v_{_{c}}$}
\put(0.1,0.9){$v_{_{b}}$}
\put(1.4,0.7){$v_{_{k}}$}
\put(2.4,0.7){$v_{_{k-1}}$}
\put(3.4,0.7){$v_{_{k-2}}$}
\put(4.4,0.7){$v_{_{1}}$}
\put(5.4,0.7){$v_{_{0}}$}
\put(3.8,1){$\ldots$}
\put(1,0.5){\line(-1,1){0.5}}
\put(1,0.5){\line(1,1){0.5}}
\put(1,1.5){\line(-1,-1){0.5}}
\put(1,1.5){\line(1,-1){0.5}}
\put(1.5,1){\line(1,0){2}}
\put(4.5,1){\line(1,0){1}}
\end{picture}
\end{center} 
If $\alpha_{_{j}}$ denotes the coordinate of the Perron-Frobenius vector at the vertex $v_{_{j}},$ the vector is given by:
\[\alpha_{_{j}}=2\Rnlamkvot{j}{k+1}{},\;\;j=0,\ldots,k,\;\;\alpha_{_{a}}=1,\;\;\alpha_{_{b}}=\mbox{$\frac{2}{\lambda}$},\;\;\alpha_{_{c}}=1.\]
The Perron-Frobenius eigenvalue, $\lambda,$ is the largest solution to
\[\mbox{$\frac{2}{\lambda}$}+2\Rnlamkvot{k}{k+1}{}=\lambda,\]
which is also the Perron-Frobenius eigenvalue of S$(1,1,k+1,k+1).$

\noindent{}As  we shall  see all these  4-stars give rise  to indices of irreducible subfactors  of the hyperfinite $II_{1}-$factor. We will therefore just mention, that  our computations for kite$(k)$ show that  all the graphs can  form  a commuting square of the form (\ref{fire2eqn21}) for   $n = 2,$  but none can define such a commuting square if $n = 1.$

\newpage{}

\setcounter{equation}{0}

\section{The 4-stars Satisfying the Conditions}
In this section we shall determine which 4-stars satisfy the conditions of Theorems \ref{fire4thm411} and \ref{fire4thm412}, i.e. if $ 1 \leq {} k_{_{1 }}\leq{}  k_{_{2}} \leq{}  k_{_{3}} \leq{}  k_{_{4 }},$ and  $\lambda$   is the Perron--Frobenius eigenvalue of \firestar{k_{_{1}}}{k_{_{2}}}{k_{_{3}}}{k_{_{4}}}, we put  $\deltasub{i}= \frac{1}{\Rnlam{k_{_{i}}}{}},$ and the conditions are
\begin{equation}
\label{firegrafcond1}(1) \;\; \deltasub{1}  -  \deltasub{2}- \deltasub{3}  - \deltasub{3}  \leq{}  0 \mbox{ see (\ref{fire4cond1}) }\end{equation}
\begin{equation}
\label{firegrafcond2}(2) \;\; \deltasub{1} - \deltasub{2}  - \deltasub{3}  + \deltasub{4} \geq{}0  \mbox{ see (\ref{fire4cond2}) }\end{equation}
The computations will use the following lemmas extensively: 
\begin{lemma}
\label{firegraflem0}
$\Rnlamkvot{n-1}{n}{}$ is increasing in $n$ and decreasing in $\lambda,$ when $\lambda\geq{}2.$
\end{lemma}
\begin{proof}
Consider
\[\begin{array}{cl}
 & \;\;\;\;\,\Rnlam{n}{}^{^{2}}-\lambda\Rnlam{n}{}\Rnlam{n-1}{}+\Rnlam{n-1}{}^{^{2}}\\[0.2cm] 
& -(\Rnlam{n}{}^{^{2}}-\lambda\Rnlam{n}{}\Rnlam{n+1}{}+\Rnlam{n+1}{}^{^{2}})\\[0.2cm]
=&\Rnlam{n-1}{}^{^{2}}-\lambda\Rnlam{n}{}\Rnlam{n-1}{}-\Rnlam{n+1}{}^{^{2}}+\lambda\Rnlam{n}{}\Rnlam{n+1}{}\\[0.2cm]
=&\Rnlam{n-1}{}(\Rnlam{n-1}{}-\lambda\Rnlam{n}{})+\Rnlam{n+1}{}(\lambda\Rnlam{n}{}-\Rnlam{n+1}{})\\[0.2cm]
=&-\Rnlam{n-1}{}\Rnlam{n+1}{}+\Rnlam{n+1}{}\Rnlam{n-1}{}\\[0.2cm]
=&0.\end{array}\]
Hence  $\Rnlam{n}{}^{^{2}}-\lambda\Rnlam{n}{}\Rnlam{n-1}{}+\Rnlam{n-1}{}^{^{2}}$ is independent of $n.$

\noindent{}For $n=1$ we have 
\[\Rnlam{1}{}^{^{2}}-\lambda\Rnlam{1}{}\Rnlam{0}{}+\Rnlam{0}{}^{^{2}}=\lambda^{^{2}}-\lambda^{^{2}}+1=1,\]
i.e. for all $n$ and for all $\lambda$
\begin{equation}
\label{firegrafeqn2}
\Rnlam{n}{}^{^{2}}-\lambda\Rnlam{n}{}\Rnlam{n-1}{}+\Rnlam{n-1}{}^{^{2}}=1.\end{equation}
We can rewrite (\ref{firegrafeqn2}) as follows
\[\begin{array}{lcl}
1 & = & \Rnlam{n}{}^{^{2}}-\lambda\Rnlam{n}{}\Rnlam{n-1}{}+\Rnlam{n-1}{}^{^{2}}\\[0.2cm]
& = & \Rnlam{n}{}^{^{2}}+\Rnlam{n-1}{}(-\lambda\Rnlam{n}{}+\Rnlam{n-1}{})\\[0.2cm]
&=&\Rnlam{n}{}^{^{2}}-\Rnlam{n-1}{}\Rnlam{n+1}{},
\end{array}\]
and conclude
\[\Rnlam{n}{}^{^{2}}=\Rnlam{n-1}{}\Rnlam{n+1}{}+1.\]
We now have
\[\Rnlamkvot{n-1}{n}{}\left/\Rnlamkvot{n}{n+1}{}\right.=
\frac{\Rnlam{n-1}{}\Rnlam{n+1}{}}{\Rnlam{n}{}^{^{2}}}<1,\]
and hence $\Rnlamkvot{n}{n+1}{}$ in increasing in $n.$ (Recall that $\Rnlam{n}{}>0$ for all $\lambda\geq{}2.$)

From the recursion formula, we have
\[\Rnlamkvot{n+1}{n}{}=\lambda-\Rnlamkvot{n-1}{n}{},\]
so if we can show that $\Rnlamkvot{n-1}{n}{}$ is decreasing in $\lambda,$ we can conclude that also $\Rnlamkvot{n}{n+1}{}$ is decreasing in $\lambda.$ For $n=0$ we have $\Rnlamkvot{0}{1}{}=\frac{1}{\lambda},$ which clearly is decreasing in $\lambda.$
\end{proof}
\begin{cor}
\label{firegrafcor0}For $m\geq{}1,$ $n\geq{}0,$ and $\lambda\geq{}2,$ $\Rnlamkvot{n}{n+m}{}$ has the following properties
\begin{enumerate}
\item{}$\Rnlamkvot{n}{n+m}{}$ is decreasing in $\lambda.$
\item{}$\Rnlamkvot{n}{n+m}{}$ is increasing in $n.$
\item{}$\Rnlamkvot{n}{n+m}{}$ is decreasing in $m.$
\end{enumerate}
\end{cor}
\begin{proof}

\noindent{}1 and 2.
\[\Rnlamkvot{n}{n+m}{}=\Rnlamkvot{n}{n+1}{}\Rnlamkvot{n+1}{n+2}{}\cdots{}\Rnlamkvot{n+m-1}{n+m}{},\]
and the right-hand side clearly has the stated properties.

\noindent{}3. For $\lambda\geq{}2$ we can write $\lambda = e^{x}+e^{-x}$ for some $x\geq{}0.$ We have
\[\Rnlamkvot{n}{n+1}{}=\eefrac{n}{(n+1)}\;\;\mbox{\raisebox{-1mm}{$\stackrel{\nearrow}{\mbox{\tiny $n\rightarrow\infty$}}$}}\;\;e^{-x}\leq{}1,\]
and since $\Rnlamkvot{n}{n+1}{}\leq\Rnlamkvot{n+1}{n+2}{}$ we have
\[\Rnlamkvot{n}{n+1}{}\leq{}1 \;\;\mbox{ for all }n.\]
We now have
\[\Rnlamkvot{n}{n+m+1}{}\left/\Rnlamkvot{n}{n+m}{}\right.= \Rnlamkvot{n+m}{n+m+1}{}\leq{}1,\]
and hence 
\[\Rnlamkvot{n}{n+m+1}{}\leq{}\Rnlamkvot{n}{n+m}{}.\]
\end{proof}
\begin{remark}{\rm 
\label{firegrafrem1}For $\lam{}\geq{}2,$ $\lam{}=e^{x}+e^{-x},\;\;x\geq{}0$ we have $\Rnlam{n}{}=\eefrac{n}{},$ and hence
\[\lim_{n\rightarrow\infty}\Rnlamkvot{n}{n+m}{}=e^{-mx}.\]
}\end{remark}
    
The following lemma reduces the number of 4--stars for which we have to check condition (\ref{firegrafcond1}).
\begin{lemma}
\label{firegraflem32}
Let  $\lambda$   be the Perron--Frobenius eigenvalue of  \firestar{j}{j+n_{1}}{j+n_{2}}{j+n_{3}}, $j \geq{}1,\;\;0\leq{}n_{_{1}}\leq{}n_{_{2}}\leq{}n_{_{3}},$ and let  $\lambda_{_{1}}$   be the Perron--Frobenius eigenvalue of  \firestar{j}{j+m_{_{1}}}{j+m_{_{2}}}{j+m_{_{3}}}, $j \geq{}1,\;\;0\leq{}m_{_{1}}\leq{}m_{_{2}}\leq{}m_{_{3}}.$ If $m_{_{1}}\leq{}n_{_{1}},$ $m_{_{2}}\leq{}n_{_{2}}$ and $m_{_{3}}\leq{}n_{_{3}}$ then
\[R_{_{j}}(\lambda)\left(\frac{1}{R_{j+n_{_{1}}}(\lambda)}+\frac{1}{R_{j+n_{_{2}}}(\lambda)}+\frac{1}{R_{j+n_{_{3}}}(\lambda)}\right)\leq{}
R_{_{j}}(\lambda_{_{1}})\left(\frac{1}{R_{j+m_{_{1}}}(\lambda_{_{1}})}+\frac{1}{R_{j+m_{_{2}}}(\lambda_{_{1}})}+\frac{1}{R_{j+m_{_{3}}}(\lambda_{_{1}})}\right).\]
\end{lemma}
\begin{proof}
Since $\lambda_{_{1}}\leq{}\lambda$ corollary \ref{firegrafcor0} 1. and 3. gives
\[\Rnlamkvot{j}{j+n_{i}}{}\leq{}\Rnlamkvot{j}{j+m_{i}}{}\leq{}\Rnlamkvot{j}{j+m_{i}}{1},\;\mbox{ for }i=1,2,3.\]
This proves the statement.
\end{proof}
The following corollary is just a restatement of lemma \ref{firegraflem32}.
\begin{cor}
\label{firegrafcor33}\hfill{}
\begin{description}
\item[(a)] If  \firestar{j}{j+n_{_{1}}}{j+n_{_{2}}}{j+n_{_{3}}}  satisfies condition (\ref{firegrafcond1}), then  \firestar{j}{j+m_{_{1}}}{j+m_{_{2}}}{j+m_{_{3}}}  also satisfies condition (\ref{firegrafcond1}), whenever  $m_{_{1}} \leq{}  n_{_{1}},$ $m_{_{2}} \leq{}  n_{_{2}}$ and  $m_{_{3}} \leq{}  n_{_{3}} .$
\item[(b)]If   \firestar{j}{j+n_{_{1}}}{j+n_{_{2}}}{j+n_{_{3}}} does not satisfy condition (\ref{firegrafcond1})), then  \firestar{j}{j+m_{_{1}}}{j+m_{_{2}}}{j+m_{_{3}}} does not satisfy condition (\ref{firegrafcond1}), whenever  $m_{_{1}} \geq{}  n_{_{1}},$ $m_{_{2}} \geq{}  n_{_{2}}$ and  $m_{_{3}} \geq{}  n_{_{3}} .$
\end{description}
\end{cor}
Recall that the Perron--Frobenius eigenvalue of \firestar{i}{j}{k}{l} satisfies the equation
\[\Rnlamkvot{i-1}{i}{}+\Rnlamkvot{j-1}{j}{}+\Rnlamkvot{k-1}{k}{}+\Rnlamkvot{l-1}{l}{}=\lambda\]
and that the polynomials $R_{_{n}}(\lambda)$ are defined recursively by 
\[R_{_{0}}(\lambda)=1,\;\;R_{_{1}}(\lambda)=\lambda,\;\;R_{_{n+1}}(\lambda)=\lambda{}R_{_{n}}(\lambda)-R_{_{n-1}}(\lambda)\]
\begin{lemma}
\label{firegraflem34}\hfill{}
\begin{enumerate}
\item{}If  $\lambda$  is the Perron--Frobenius eigenvalue of  \firestar{j}{j+1}{j+1}{j+1}, then  $\Rnlamkvot{j}{j+1}{}=\frac{1}{\sqrt{3}}.$
\item{}If  $\lambda$   is the Perron--Frobenius eigenvalue of  \firestar{j}{j+2}{j+2}{j+2}, then $\Rnlamkvot{j}{j+2}{}=\frac{1}{3}.$         
\end{enumerate}
\end{lemma}
\begin{proof}The recursion formula  can be rewritten as
\[\lambda  - \Rnlamkvot{j-1}{j}{}=\Rnlamkvot{j+1}{j}{}.\]
Using this we have

\noindent{}1. The eigenvalue equation is
\[\Rnlamkvot{j-1}{j}{}+3\Rnlamkvot{j}{j+1}{}=\lambda\Leftrightarrow{}
3\left(\Rnlamkvot{j}{j+1}{}\right)^{2}=1\Leftrightarrow{}\Rnlamkvot{j}{j+1}{}=\frac{1}{\sqrt{3}}.\]

\noindent{}2. The eigenvalue equation is
\[\Rnlamkvot{j-1}{j}{}+3\Rnlamkvot{j+1}{j+2}{}=\lambda\Leftrightarrow{}3\Rnlamkvot{j}{j+2}{}=1.\]
\end{proof}
Let  $\lambda_{_{\infty}}$    be the Perron--Frobenius eigenvalue of  \firestar{\mbox{``$\infty$''}}{\mbox{ ``$\infty$''}}{\mbox{ ``$\infty$''}}{\mbox{ ``$\infty$''}}, $\lambda_{_{\infty}}= \frac{4}{\sqrt{3}}.$ We shall first consider condition (\ref{firegrafcond1}) which, for  $\lambda$  = Perron--Frobenius eigenvalue of  \firestar{k_1}{k_2}{k_3}{k_4}, $k_1 \leq{}  k_2 \leq{}  k_3 \leq{}  k_4$  is equivalent to
\begin{equation}
\label{firegrafcond1prime}
R_{_{k_{1}}}(\lambda)\left(\frac{1}{R_{_{k_{2}}}(\lambda)}+\frac{1}{R_{_{k_{3}}}(\lambda)}+\frac{1}{R_{_{k_{4}}}(\lambda)}\right)\geq{}1.
\end{equation}
We divide the discussion into several steps:
\begin{description}
\item[(A)]\firestar{j}{j}{k}{l},  $j \leq{}  k \leq{}  l$  trivially satisfy condition (\ref{firegrafcond1prime}).
\item[
(B)]Consider  \firestar{j}{j+1}{j+1}{j+m}, $m \geq{}  1 $, with  Perron--Frobenius eigenvalue  $\lambda.$  By corollary \ref{firegrafcor0} 1. and 2. we obtain
\[2\Rnlamkvot{j}{j+1}{}+\Rnlamkvot{j}{j+m}{} \geq{}2\Rnlamkvot{1}{2}{}\geq{}2\Rnlamkvot{1}{2}{\infty}=\frac{8\sqrt{3}}{13}>1.\]
\item[(C)]Consider  \firestar{j}{j+1}{j+2}{j+4}  with Perron--Frobenius eigenvalue  $\lambda.$ 

\noindent{}For  $j \geq{} 3$  we have
\[\mbox{RHS}(\ref{firegrafcond1prime})\geq{} R_{_{3}}(\lambda_{_{\infty}})\left(\frac{1}{R_{_{4}}(\lambda_{_{\infty}})}+\frac{1}{R_{_{5}}(\lambda_{_{\infty}})}+\frac{1}{R_{_{7}}(\lambda_{_{\infty}})}\right)\approx{}1.01>1.\]

\noindent{}For  $j = 2$  we have $\lambda\approx{}2.2862<\sqrt{5.25}=\lambda_{_{0}},$ and hence
\[\mbox{RHS}(\ref{firegrafcond1prime})\geq{} R_{_{2}}(\lambda_{_{0}})\left(\frac{1}{R_{_{3}}(\lambda_{_{0}})}+\frac{1}{R_{_{4}}(\lambda_{_{0}})}+\frac{1}{R_{_{6}}(\lambda_{_{0}})}\right)\approx{}1.02>1.\]

\noindent{}For  $j = 1$  we have $\lambda\approx{}2.2291<\sqrt{5}=\lambda_{_{0}},$ and hence
\[\mbox{RHS}(\ref{firegrafcond1prime})\geq{} R_{_{1}}(\lambda_{_{0}})\left(\frac{1}{R_{_{2}}(\lambda_{_{0}})}+\frac{1}{R_{_{3}}(\lambda_{_{0}})}+\frac{1}{R_{_{5}}(\lambda_{_{0}})}\right)=\frac{6\sqrt{5}+11}{24}>1.\]
I.e. \firestar{j}{j+1}{j+2}{j+4} satisfies condition (\ref{firegrafcond1}) for all $j,$ and hence so does  \firestar{j}{j+1}{j+2}{j+2}  and  \firestar{j}{j+1}{j+2}{j+3} by corollary \ref{firegrafcor33}.
 \item[(D)]By lemma \ref{firegraflem34} 2. it is easily seen that  \firestar{j}{j+2}{j+2}{j+2} satisfies condition (\ref{firegrafcond1}) for all  $j.$
\end{description}                                                         To show that the 4-stars listed in (A)-(D) are the only ones that satisfy  condition (\ref{firegrafcond1}), corollary \ref{firegrafcor33} tells us that we only need to prove that
\begin{description}
\item[(E)]\firestar{j}{j+2}{j+2}{j+3},  $j \geq{}  1,$
\item[(F)]\firestar{j}{j+1}{j+3}{j+3},  $j \geq{}  1,$
\item[(G)]\firestar{j}{j+1}{j+2}{j+5)},  $j \geq{}  1$
\end{description}
do not satisfy condition (1).

\noindent{}(E)  Let  $\lambda$   be the Perron--Frobenius eigenvalue of  \firestar{j}{j+2}{j+2}{j+3}, and let  $\lambda_{_{1}}$  be the Perron--Frobenius eigenvalue of \firestar{j}{j+2}{j+2}{j+2}. Since $\lambda_{_{1}} < \lambda$ we have
\[\mbox{RHS}(\ref{firegrafcond1prime}) <2 \Rnlamkvot{j}{j+2}{1} + \Rnlamkvot{j}{j+3}{}<\mbox{$\frac{2}{3}$} + e^{-3x}.\]
The solution to $\frac{2}{3} + e^{-3x} < 1$  is  $x >  \frac{1}{3}\log{}3,$  corresponding to  $ \lambda>\lambda_{_{2}} = 2\mbox{cosh}(x) = 3^{\frac{1}{3}}    + 3^{-\frac{1}{3}} \approx{} 2.1356.$ Since the Perron--Frobenius eigenvalue of  $\firestar{1}{3}{3}{4} \approx{}2.2411,$ we obtain the statement of (E).

\noindent{}(F)  Let  $\lambda$   be the Perron--Frobenius eigenvalue of  \firestar{j}{j+1}{j+3}{j+3}  and  $\lambda_{_{1}}$   be the Perron--Frobenius eigenvalue of  \firestar{j}{j+1}{j+1}{j+1}, then
\[\mbox{RHS}(\ref{firegrafcond1prime}) < \Rnlamkvot{j}{j+1}{1}\left(1+2\Rnlamkvot{j+1}{j+3}{}\right)<\mbox{$\frac{1}{\sqrt{3}}$}(1+2e^{-2x}).\]
The solution to  $\frac{1}{\sqrt{3}}(1+2e^{-2x}) < 1$  corresponds to  $\lambda>\lambda_{_{2}} \approx{} 2.2579.$ For  $j = 2$  we have  $\lambda \approx{} 2.2870 > \lambda_{_{2}},$ so (F) is proved for  $j \geq  2.$
           
\noindent{}In the case  j = 1  we have  $\lambda  \approx{} 2.2323 > 2.22 =\lambda_{_{3}},$ hence
\[\mbox{RHS}(\ref{firegrafcond1prime}) < R_{_{1}}(\lambda_{_{3}})\left(\frac{1}{R_{_{2}}(\lambda_{_{3}})}+\frac{2}{R_{_{4}}(\lambda_{_{3}})}\right)\approx{}0.9878<1.\]

\noindent{}(G)  Let  $\lambda$   be the Perron--Frobenius eigenvalue of  \firestar{j}{j+1}{j+2}{j+5}  and $\lambda_{_{1}}$  be the Perron--Frobenius eigenvalue of  \firestar{j}{j+1}{j+1}{j+1}. Then   $\mbox{RHS}(\ref{firegrafcond1prime}) <\frac{1}{\sqrt{3}}(e^{-x}+e^{-4x}),$  and the solution of $\frac{1}{\sqrt{3}}(e^{-x}+e^{-4x})< 1 $  corresponds to  $\lambda>\lambda_{_{2}} \approx{} 2.0035.$ If $j = 1$  we have   $\lambda  \approx{} 2.2298 > \lambda_{_{2}},$ which proves (G).
       
 \noindent{}By corollary  \ref{firegrafcor33} neither of the 4-stars listed below satisfy condition (\ref{firegrafcond1})\\[0.2cm]

\begin{tabular}{clll}
(1)& \firestar{j}{j+2+n_{_{1}}}{j+2+n_{_{2}}}{j+3+n_{_{3}}},   &  $j  \geq{}  1,$  $0 \leq{}  n_{_{1}} \leq{}  n_{_{2}} \leq{} n_{_{3}}$ &  (implied by (E))\\[0.2cm]
(2)& \firestar{j}{j+1+n_{_{1}}}{j+3+n_{_{2}}}{j+3+n_{_{3}}},   &  $j  \geq{}  1,$ $0\leq{} n_{_{1}}  \leq{}  n_{_{2}} \leq{}  n_{_{3}}$ & (implied by (F))\\[0.2cm]
(3) & \firestar{j}{j+1+n_{_{1}}}{j+2+n_{_{2}}}{j+5+n_{_{3}}},   &  $j  \geq{}  1,$ $0\leq{}n_{_{1}} \leq n_{_{2}} \leq{}  n_{_{3}}$  &(implied  by (G))
\end{tabular}\\[0.2cm]
                        
\noindent{}Since we will only be concerned with condition (\ref{firegrafcond2}) when condition (\ref{firegrafcond1}) is satisfied, we will now determine which of the 4-stars listed in (A)-(D)  satisfy condition (\ref{firegrafcond2}).

\noindent{}If  $\lambda$  is the Perron--Frobenius eigenvalue of  \firestar{j}{j+n_{_{1}}}{j+n_{_{2}}}{j+n_{_{3}}},  $0 \leq{}  n_{_{1}}\leq{}  n_{_{2}}\leq{}  n_{_{3}},$ condition (\ref{firegrafcond2}) is
\begin{equation}
\label{firegrafcond2prime}
R_{_{j}}(\lambda)\left(\frac{1}{R_{_{j+n_{_{1}}}}(\lambda)}+\frac{1}{R_{_{j+n_{_{2}}}}(\lambda)}-\frac{1}{R_{_{j+n_{_{3}}}}(\lambda)}\right)\leq{}1.\end{equation}

\noindent{}(A') \firestar{j}{j}{k}{l}  $j \leq{}  k \leq{}  l$  is easily seen to satisfy (\ref{firegrafcond2prime}) if and only if k = l .

\noindent{}(B') Put  $n_{_{1}} = 1,$ $n_{_{2}} = 1$  and  $n_{_{3}} = m .$ 

\noindent{}If  $m = 2$  or  $m = 3,$ we have
\[\mbox{RHS}(\ref{firegrafcond2prime})<\mbox{$\frac{1}{\sqrt{3}}$}\left(2-\Rnlamkvot{2}{4}{\infty}\right)\approx{}0.9686 < 1 .\]
                  
\noindent{}Let  $m \geq  4.$  Then
\[\mbox{RHS}(\ref{firegrafcond2prime})\geq{}\Rnlamkvot{2}{3}{\infty}(2-e^{-3x})\] 
and the solution to $\Rnlamkvot{2}{3}{\infty}(2-e^{-3x})>1$ corresponds to  $\lambda>\lambda_{_{2}} \approx{} 2.2546.$ As the Perron--Frobenius eigenvalue of  \firestar{2}{3}{3}{6}  is $2.2823 $(approx.), we have shown that  \firestar{j}{j+1}{j+1}{j+m}, $j \geq  2,\; m \geq{} 4,$  does not satisfy condition (\ref{firegrafcond2}).

\noindent{}A similar argument shows that  \firestar{j}{j+1}{j+1}{j+m}, $j \geq{} 1,\; m \geq{}  6,$  does not satisfy condition (\ref{firegrafcond2}), hence we now only need to consider  \firestar{1}{2}{2}{5}. We will show that (\ref{firegrafcond2prime}) is satisfied with  equality for \firestar{1}{2}{2}{5}. 

\noindent{}We have  $R_{_{5}}(\lambda ) = R_{_{1}}(\lambda )R_{_{2}}(\lambda )(\lambda^{2}-3),$ and easy computations show
\[ 2\Rnlamkvot{1}{2}{}-\Rnlamkvot{1}{5}{}=1\Leftrightarrow{}R_{_{1}}(\lambda )R_{_{2}}(\lambda )(\lambda -2)(\lambda^{3} -4\lambda -2) = 0 ,\] 
and since  $R_{_{1}}(\lambda )R_{_{2}}(\lambda )(\lambda -2) \neq{}  0,$ we must show  $\lambda^{3} - 4\lambda  - 2 = 0.$

\noindent{}The eigenvalue equation is
\[\Rnlamkvot{0}{1}{} + 2\Rnlamkvot{1}{2}{}+\Rnlamkvot{4}{5}{}= \lambda \Leftrightarrow{}    \lambda (\lambda^{2} -1)(\lambda^{3}-4\lambda +2)(\lambda^{3}-4\lambda -2) = 0 .\] 
Since  $\lambda$   is the largest root of the above equation, we must have $\lambda^{3} - 4\lambda  - 2 = 0.$

\noindent{}(C') Put  $n_{_{1}} = 1,\; n_{_{2}} = 2,\; n_{_{3}} = m,$ then
\[\mbox{RHS}(\ref{firegrafcond2prime})<\Rnlamkvot{j}{j+1}{}\left(1+\Rnlamkvot{j+1}{j+2}{}\right)<\mbox{$\frac{1}{\sqrt{3}}$}(1+e^{-x}).\]

\noindent{}The solution of $  \frac{1}{\sqrt{3}}(1+e^{-x})\leq{}  1$  corresponds to $\lambda> \lambda_{_{2}} \approx{} 2.0981,$ and since the Perron--Frobenius eigenvalue of  \firestar{1}{2}{3}{3}  is 2.2216 (approx.),  \firestar{j}{j+1}{j+1}{j+m}  satisfies condition (\ref{firegrafcond2}) for $ j \geq{} 1$  and  $m \geq{}  2.$

\noindent{}(D') Using lemma \ref{firegraflem34} 2,  \firestar{j}{j+2}{j+2}{j+2} is easily seen to satisfy condition (\ref{firegrafcond2}).

\noindent{}Hence the following is a total list of 4-stars which satisfy\\[0.2cm]

\begin{tabular}{ll} 
Condition (\ref{firegrafcond1})  &  Condition (\ref{firegrafcond1}) \& Condition (\ref{firegrafcond2})\\[0.2cm]
 \firestar{j}{j}{k}{l}, $1\leq{} j\leq{} k\leq{} l ,$&\firestar{j}{j}{k}{k}, $1\leq{} j\leq{} k,$\\[0.2cm]
\firestar{j}{j+1}{j+1}{j+m}, $j\geq{} 1,\; m\geq{} 1,$ & \firestar{j}{j+1}{j+1}{j+m}, $j\geq{} 1,\; 1\leq{} m\leq{} 3,$\\[0.2cm]
\firestar{j}{j+1}{j+2}{j+m}, $j\geq{} 1, \;2\leq{} m\leq{} 4,$&\firestar{1}{2}{2}{5},\\[0.2cm]
 \firestar{j}{j+2}{j+2}{j+2}, $j\geq{} 1,$&\firestar{j}{j+1}{j+2}{j+m}, $j\geq{} 1,\; 2\leq{} m\leq{} 4,$\\[0.2cm]
 &   \firestar{j}{j+2}{j+2}{j+2}, $j\geq{}  1.$
\end{tabular}\\[0.3cm]

 We end this section with a list of indices of irreducible subfactors of the hyperfinite $II_{1}$-factor in the interval $(4,5)$ which are produced by our construction.
 Most  of the values are  obtained by numerical  methods, since it is not in general possible to solve the equation (\ref{stareigenvaleq}) for a 4--star analytically.

\begin{center}
\begin{tabular}{|l|l|l|l|l|l|}\hline
\multicolumn{6}{|c|}{S$(1,1,k,l)$,  $1 \leq{}  k \leq{}  l$}\\ \hline{}
l\ \ \ \vline{}\ k & \mbox{\hspace{0.5cm}} 1      & \mbox{\hspace{0.5cm}} 2      &\mbox{\hspace{0.5cm}}  3     & \mbox{\hspace{0.5cm}} 4      & \mbox{\hspace{0.5cm}} 5\\\hline
                
         1  &   4.00000*&             &            &            & \\
         2  &  4.30278   &  4.56155*&             &              &\\
         3  &  4.41421  &   4.65109* &   4.73205*&              &\\
         4  &   4.46050  &   4.68554 &    4.76251 &    4.79129*&\\
         5  &   4.48119  &   4.69963 &    4.77462 &    4.80262  &  4.81361*\\
         6  &   4.49086  &   4.70559 &    4.77959  &   4.80721 &   4.81804\\
         7  &   4.49551  &   4.70816 &    4.78165 &    4.80910 &   4.81986\\
         8  &   4.49778  &   4.70928 &    4.78252 &    4.80988 &   4.82060\\
         9  &   4.49889   &  4.70977 &    4.78289 &    4.81021 &   4.82092\\
        10  &   4.49945  &   4.70998  &   4.78304  &   4.81035 &   4.82104\\
 $\vdots$   & \multicolumn{1}{|c}{$\vdots$}  & \multicolumn{1}{|c}{ $\vdots$} &   \multicolumn{1}{|c}{$\vdots$} &   \multicolumn{1}{|c}{$\vdots$} &  \multicolumn{1}{|c|}{ $\vdots$} \\                          
      limit  &  4.50000   &  4.71015  &   4.78316  &   4.81044 &   4.82114\\ \hline
\end{tabular}
\end{center}

\begin{center}
\begin{tabular}{|l|l|l|l|l|l|}\hline
\multicolumn{6}{|c|}{S$(1,1,k,l)$,  $1 \leq{}  k \leq{}  l$}\\ \hline{}
l\ \ \ \vline{}\ k & \mbox{\hspace{0.5cm}} 6      & \mbox{\hspace{0.5cm}} 7      &\mbox{\hspace{0.5cm}}  8     & \mbox{\hspace{0.5cm}} 9      & \mbox{\hspace{0.5cm}} 10\\\hline

         6  &   4.82240*&           &           &           &          \\
         7  &   4.82419 &   4.82596*&           &           &          \\
         8  &   4.82492 &   4.82668 &   4.82741*&           &          \\
         9  &   4.82522 &   4.82698 &   4.82771 &   4.82801*&          \\
        10  &   4.82535 &   4.82711 &   4.82783 &   4.82813 &  4.82825*\\
$\vdots$   & \multicolumn{1}{|c}{$\vdots$}  & \multicolumn{1}{|c}{ $\vdots$} &   \multicolumn{1}{|c}{$\vdots$} &   \multicolumn{1}{|c}{$\vdots$} &  \multicolumn{1}{|c|}{ $\vdots$} \\ 
      limit &   4.82544 &   4.82720 &   4.82792 &   4.82822 &  4.82834\\
\hline\end{tabular}
\end{center}
The limit values converge to  $2 + 2\sqrt{2}\approx{}  4.82843.$

\begin{center}
\begin{tabular}{|c|l|c|l|}\hline
\multicolumn{2}{|c|}{S$(1,2,2,k)$}&\multicolumn{2}{c|}{S$(1,2,3,k)$ }\\
\hline{}
k& \multicolumn{1}{|c|}{$\lambda^{2}$ }& k & \multicolumn{1}{|c|}{$\lambda^{2}$ }\\\hline
2   & 4.79129*  &  3 &    4.93543*\\
3   & 4.86620*  &  4 &    4.95978*\\
4   & 4.89307*  &  5  &   4.96876*\\
5   & 4.90321*  & & \\
6   & 4.90715   & & \\            
7   & 4.90869   & & \\            
8   & 4.90930   & & \\ 
9   & 4.90955   & & \\
10  & 4.90964   & & \\
$\vdots$ &  \multicolumn{1}{|c|}{$\vdots$} & & \\   limit  &  4.90971 & & \\
\hline
\end{tabular}
\end{center}

\noindent{}From \firestar{j}{j+2}{j+2}{j+2} there are no indices in the interval $(4,5).$

By  the discussion  in the  beginning of  section \ref{firesection2},  these are  the only values  which can  arise from commuting squares of the form
\[\begin{array}{lcl}
  C & \subset_{nG^{^{t}}} & D \\
   \cup_{G} &\,&\cup_{G^{^{t}}}\\
  A & \subset_{nG} & A  \end{array},\;\;\;\; n\in{\Bbb N}.  \]      
The lowest value is $\frac{1+\sqrt{13}}{2}.$ The numbers  marked with a star, are those which come from a commuting square of the form
  \[\begin{array}{lcl}
  C & \subset_{G^{^{t}}} & D \\
   \cup_{G} &\,&\cup_{G^{^{t}}}\\
  A & \subset_{G} & A  \end{array},  \]  
the lowest of which is  $\frac{1+\sqrt{17}}{2}.$ 

We  have, of course, a lot of other values  of the index corresponding to the  other 4-stars which satisfy the two conditions, but it would take up too much space to list some of the index values obtained from these graphs.

Also, as we will see in chapter \ref{uendeligcsq}, all the limit values of the determined families are values of the index for an irreducible subfactor of the hyperfinite $II_{1}-$factor.

\newpage{}

\setcounter{equation}{0}

\section{Algebraic Necessities}
\label{algnece}
\setcounter{equation}{0}
\begin{lemma}
\label{fire4lem41}Let $\delta_{_{1}} \geq \delta_{_{2}} \geq \delta_{_{3}} \geq \delta_{_{4}}\geq{} 0.$ If $\delta_{_{1}} \pm \delta_{_{2}} \pm \delta_{_{3}} \pm \delta_{_{4}} = 0$ for some choice of signs, then either
\[\delta_{_{1}} - \delta_{_{2}} - \delta_{_{3}} - \delta_{_{4}} = 0 \]
or
\[\delta_{_{1}} - \delta_{_{2}} - \delta_{_{3}} + \delta_{_{4}} = 0 .\]
\end{lemma}
\begin{proof}Assume that $f = \delta_{_{1}} - \delta_{_{2}} - \delta_{_{3}} - \delta_{_{4}}\neq{} 0 $ and $g = \delta_{_{1}} - \delta_{_{2}} - \delta_{_{3}} + \delta_{_{4}}\neq{} 0.$ We have 
\begin{equation}
\label{fire4*} (\delta_{_{1}} - \delta_{_{2}}) + (\delta_{_{3}} - \delta_{_{4}}) \geq 0
\end{equation}
with equality if an only if $\delta_{_{1}}= \delta_{_{2}}$ and $\delta_{_{3}} = \delta_{_{4}}.$ Thus equality in (\ref{fire4*}) implies $g = 0 .$ I.e. $\delta_{_{1}} - \delta_{_{2}} + \delta_{_{3}} - \delta_{_{4}} > 0.$

\noindent{}We also have 
\begin{equation}
\label{fire4**} (\delta_{_{1}} - \delta_{_{3}}) + (\delta_{_{2}}-\delta_{_{4}}) \geq 0 
\end{equation}
with equality if and only if $\delta_{_{1}} = \delta_{_{2}} = \delta_{_{3}} = \delta_{_{4}}.$ Hence equality in (\ref{fire4**}) implies $ g = 0.$ I.e. $\delta_{_{1}} + \delta_{_{2}} - \delta_{_{3}} - \delta_{_{4}} > 0.$

\noindent{}If at least two of $\delta_{_{2}}, \delta_{_{3}}, \delta_{_{4}} $ must be chosen with positive sign, then either $\delta_{_{2}}$ or $\delta_{_{3}}$ is chosen positive. In this case we get: Sum of $\delta_{_{ i}}'$s with signs $\geq{}$ (\ref{fire4*}) resp. (\ref{fire4**}) $ > 0.$ The above contradicts the possible choice of signs as stated.
\end{proof}
       
\begin{prop}
\label{fire4prop42}Let $\delta_{_{1}} \geq \delta_{_{2}} \geq \delta_{_{3}} \geq \delta_{_{4}} \geq 0.$ Then the following two conditions are equivalent:
\begin{enumerate}
\item{}There exists $t \in{}[0,\delsqi{4}],$ and a choice of signs such that
\[\sqrt{\delsqi{1}-t}\pm\sqrt{\delsqi{2}-t}\pm\sqrt{\delsqi{3}-t}\pm\sqrt{\delsqi{4}-t}=0.\]
\item{}$\delta_{_{1}} - \delta_{_{2}} - \delta_{_{3}} - \delta_{_{4}} \leq{}0 $ and $\delta_{_{1}} - \delta_{_{2}} - \delta_{_{3}} + \delta_{_{4}} \geq{} 0.$
\end{enumerate}
\end{prop}
\begin{proof}Put $\delta_{_{i}}(t) = \sqrt{\delsqi{i}-t},\;i = 1,2,3,4,\;t \in{}[0,\delsqi{4}],$ then $\delta_{_{1}}(t) \geq \delta_{_{2}}(t) \geq \delta_{_{3}}(t) \geq{} \delta_{_{4}}(t) \geq{} 0.$ Put  
\[f(t) = \delta_{_{1}}(t) - \delta_{_{2}}(t) - \delta_{_{3}}(t) - \delta_{_{4}}(t) \]
and
\[g(t) = \delta_{_{1}}(t) - \delta_{_{2}}(t) - \delta_{_{3}}(t) + \delta_{_{4}}(t).\]
By lemma \ref{fire4lem41} 1 is equivalent to 1': $f(t) = 0$ for some $t \in{}[0,\delsqi{4}]$ or $g(t) = 0$ for some $t \in{}[0,\delsqi{4}].$ 
\noindent{}We will first prove that 2 $\Rightarrow$ 1'.

In the above notation the statement of 2 is $f(0) \leq{} 0$ and $g(0)\geq{} 0,$ and since $f(\delsqi{4} ) = g(\delsqi{4} )$ there exists a $t \in{}[0,\delsqi{4}],$ such that  $f(t) = 0$ or $g(t) = 0.$
     
\noindent{}Proof of 1' $\Rightarrow$ 2: Assume $1'.$ If $\delta_{_{4}}=0,$ the only possible value of $t$ is $t=0,$ so 2 follows immediately. Let $\delta_{_{4}}>0$ and assume that 2 is false, i.e.
\begin{description}
\item[(a)]$\delta_{_{1}} - \delta_{_{2}} - \delta_{_{3}} - \delta_{_{4}} > 0$ or
\item[(b)]$\delta_{_{1}} - \delta_{_{2}} - \delta_{_{3}} + \delta_{_{4}} < 0.$
\end{description} 

\noindent{}If (a) is valid, we get
\[f'(t) = \mbox{$\frac{1}{2}$}(-\delta_{_{1}}(t)^{-1} + \delta_{_{2}}(t)^{-1} + \delta_{_{3}}(t)^{-1} + \delta_{_{4}}(t)^{-1}),\;\; 0 \leq{}t < 
\delsqi{4},\]
and $\delta_{_{1}}(0) \geq{} \delta_{_{2}}(0) \geq{} \delta_{_{3}}(0) \geq{} \delta_{_{4}}(0) > 0$ implies $f'(0) > 0.$ Assume that there exists a $t \in{}(0,\delsqi{4} )$ such that  $f'(t) = 0,$ and let $t_{_{0}}$ be the smallest such $t.$ Since $f'(0) > 0$ we get $f'(t) > 0,\; 0 \leq{} t < t_{_{0}},$ and hence $f(t_{_{0}}) \geq{} f(0) > 0,$ i.e. $\delta_{_{1}}(t_{_{0}}) - \delta_{_{2}}(t_{_{0}}) - \delta_{_{3}}(t_{_{0}}) - \delta_{_{4}}(t_{_{0}}) > 0.$ The above argument (in the case $t = t_{_{0}}$ instead of $t=0$) yields $f'(t_{_{0}}) > 0,$ which is a contradiction. I.e. $f'(t) > 0 ,$ for all $t \in{} [ 0,\delsqi{4} )$ and thus $f(t) \geq{} f(0) > 0,$ for all $t \in{} [ 0,\delsqi{4}].$ As $g(t) = f(t) + 2\delta_{_{4}}(t)$ we also get $g(t) > 0,$ for all $t \in{} [ 0,\delsqi{4}].$ This proves the implication in case (a).

\noindent{}Assume (b). For all $t \in{} [ 0,\delsqi{4}]$
\[\begin{array}{lcl}
g'(t) &=&\frac{1}{2}(- \delta_{_{1}}(t)^{-1} + \delta_{_{2}}(t)^{-1} + 
      \delta_{_{3}}(t)^{-1} - \delta_{_{4}}(t)^{-1})\\[0.3cm]
 & = & 
\frac{1}{2}\left(\frac{\delta_{_{1}}(t)-\delta_{_{2}}(t)}{\delta_{_{1}}(t)\delta_{_{2}}(t)}-
\frac{\delta_{_{3}}(t)-\delta_{_{4}}(t)}{\delta_{_{3}}(t)\delta_{_{4}}(t)}\right).
\end{array}\]
Since $\delta_{_{1}}(t) \geq{} \delta_{_{3}}(t)$ and $\delta_{_{2}}(t) \geq{} \delta_{_{4}}(t)$ we have $\delta_{_{1}}(t)\delta_{_{2}}(t) \geq{} \delta_{_{3}}(t)\delta_{_{4}}(t),$ and we get
\[ g'(t)\leq{}\frac{1}{2\delta_{_{1}}(t)\delta_{_{2}}(t)}g(t),\;\;\mbox{for all } t\in{}[0,\delsqi{4}]\]
in particular $ g'(0) < 0.$
Assume there exists $t\in{}(0,\delsqi{4})$ such that  $g'(t) = 0,$ and let $t_{_{0}}$ be the smallest such $t.$ Then $g(t)$ is decreasing on $[0,t_{_{0}}],$ and hence $g(t_{_{0}}) \leq{} g(0) < 0,$ but then we get $g'(t_{_{0}}) < 0,$ which is a contradiction. I.e. $g(t) \leq{} g(0) < 0 ,$ for all $t\in{}[0,\delsqi{4}],$ and since $f(t) = g(t) - 2\delta_{_{4}}(t)$ we get $f(t) < 0 ,$ for all $t\in{}[0,\delsqi{4}].$ This proves the implication in case (b).
\end{proof}
For the rest of this section we let $\alpha_{_{1}},\alpha_{_{2}},\alpha_{_{3}}$ and $\alpha_{_{4}}$ denote positive reals, such that we for $\lambda = \alpha_{_{1}}+\alpha_{_{2}}+\alpha_{_{3}}+\alpha_{_{4}}$ have
\[\lambda{}>2\;\mbox{ and }\;0<\alpha_{_{i}}\leq\frac{\lambda-\sqrt{\lambda^{2}-4}}{2},\;\;i=1,2,3,4,\]
and we put
\[\delta_{_{i}} = \sqrt{\alpha_{_{i}}^{2}-\lambda\alpha_{_{i}}+1},\;\;i = 1,3,2,4.\]
This is well defined, because $\frac{1}{2}(\lambda-\sqrt{\lambda^{2}-4})$ is the smallest root of the polynomial $x^{2}-\lambda{}x+1.$
\begin{remark}{\rm Since for $\{i,j,k,l\}=\{1,2,3,4\}$ 
\[\delta_{_{i}}^{2} + \alpha_{_{i}} \alpha_{_{j}} + \alpha_{_{i}} \alpha_{_{k}} + \alpha_{_{i}} \alpha_{_{l}} = (\alpha_{_{i}}+\alpha_{_{j}}+\alpha_{_{k}}+\alpha_{_{l}})\alpha_{_{i}}-\lambda\alpha_{_{i}}+1=1\] 
the matrix
\begin{equation}
\label{fire4matrix}D=
\left(\begin{array}{cccc}
\delsqi{1} & \alialj{1}{2} & \alialj{1}{3}& \alialj{1}{4}\\
\alialj{1}{2} & \delsqi{2} &\alialj{2}{3} & \alialj{2}{4} \\
\alialj{1}{3} &\alialj{2}{3} &\delsqi{3} &\alialj{3}{4} \\  
\alialj{1}{4} & \alialj{2}{4} & \alialj{3}{4} &\delsqi{4}
\end{array}\right)
\end{equation}
is doubly stochastic, i.e.  rows and columns have sum 1.
}\end{remark}
\begin{lemma}
\label{fire4lem44}With $\alpha_{_{1}},\; \alpha_{_{2}},\; \alpha_{_{3}}$ and $\alpha_{_{4}}$ as above:
\begin{enumerate}
\item{}$\delta_{_{i}}<\frac{\lambda}{2}-\alpha_{_{i}},$ $i=1,2,3,4.$
\item{}$\delta_{_{i}} + \delta_{_{j}} <\alpha_{_{k}} + \alpha_{_{l}},$ when $\{i,j,k,l\}=\{1,2,3,4\}.$
\item{}$|\delta_{_{i}} - \delta_{_{j}}| \geq{} \alpha_{_{i}} - \alpha_{_{j}},$ with equality if and only if $\alpha_{_{i}} = \alpha_{_{j}}.$
\end{enumerate}
\end{lemma}
\begin{proof}\hfill{}

\noindent{}1. Since $\lambda>2,$ $\alpha_{_{i}}^{2}-\lambda\alpha_{_{i}}+1<(\alpha_{_{i}}-\frac{\lambda}{2})^{2}.$ Hence
\[\delta_{_{i}}<|\alpha_{_{i}}-\mbox{$\frac{\lambda}{2}$}|=\mbox{$\frac{\lambda}{2}$}-\alpha_{_{i}}.\]

\noindent{}2. Let $(i,j,k,l)$ be a permutation of $(1,2,3,4).$ By 1 \[\delta_{_{i}}+\delta_{_{j}}<\lambda-\alpha_{_{i}}-\alpha_{_{j}}=\alpha_{_{l}}+\alpha_{_{j}}.\]

\noindent{}3. The function
\[f(\alpha)=\sqrt{\alpha^{2}-\lambda\alpha+1},\;\;\;0<\alpha\leq\frac{\lambda-\sqrt{\lambda^{2}-4}}{2},\]
is strictly decreasing, and 
\[f'(\alpha)=-\frac{\frac{\lambda}{2}-\alpha}{\sqrt{\alpha^{2}-\lambda\alpha+1}}<-1,\;\;\;\left(\alpha\neq{}\frac{\lambda-\sqrt{\lambda^{2}-4}}{2}\right),\]
because $(\frac{\lambda}{2}-\alpha)^{2}>\alpha^{2}-\lambda\alpha+1,$ since $\lambda>2.$ Hence
\[|f(\alpha)-f(\beta)|>|\alpha-\beta|\]
for all $\alpha,\beta\in{}[0,\frac{1}{2}(\lambda-\sqrt{\lambda^{2}-4})],$ $\alpha\neq{}\beta.$ This proves 3.
\end{proof}
The rest of this section will be taken up by a study of the properties of some special matrices, which eventually will lead to the main results of this section.
\begin{lemma}
\label{fire4lem45}Let $u \in{} M_{_{4}}(M_{_{n}}({\Bbb C} )),\;u = \left(u_{_{ij}}\right)_{i,j=1}^{4}$  be a unitary matrix such that $\uij{i}{j}u_{_{ij}}^{*} = \alpha_{_{ij}}I_{_{n}},\; i,j=1,...,4,$ where $\left(\alpha_{_{ij}}\right)_{i,j=1}^{4}$  is doubly stochastic and symmetric, then
\[\sqrt{\alpha_{_{i_{1}i_{1}}}}\leq{}\sqrt{\alpha_{_{i_{2}i_{2}}}}+\sqrt{\alpha_{_{i_{3}i_{3}}}}+\sqrt{\alpha_{_{i_{4}i_{4}}}},\;\;\;\;\{i_{_{1}},i_{_{2}},i_{_{3}},i_{_{4}}\}=\{1,2,3,4\}.\]
\end{lemma}
\begin{proof}$u=\left(\begin{array}{cc}a&b\\c&d\end{array}\right),$ $a,b,c,d\in{}M_{_{2}}(M_{_{n}}({\Bbb C})).$ Since $u$ is unitary,  $a^{*}a=1-c^{*}c$ and  $dd^{*}=1-cc^{*}.$ Therefore $a^{*}a$ and $dd^{*}$ have the same list of eigenvalues, so
\[\|a\|^{2}_{_{2}}=\mbox{Tr}(a^{*}a)= \mbox{Tr}(dd^{*})=\|d\|^{2}_{_{2}}\]and
\[|\mbox{det}(a)|=\mbox{det}(a^{*}a)^{\mbox{\scriptsize$\frac{1}{2}$}}= \mbox{det}(dd^{*})^{\mbox{\scriptsize$\frac{1}{2}$}}=|\mbox{det}(d)|.\]
By the assumptions
\[u_{_{ij}}=\sqrt{\alpha_{_{ij}}}V_{_{ij}},\;\;\;\;i,j=1,2,3,4\]
where the $V_{_{ij}}'$s are unitary $n\times{}n$ matrices.

\noindent{}Set $V = V_{_{21}}^{*}V_{_{22}}V_{_{12}}^{*}V_{_{11}}.$ Then 
\[ a =
\left(\begin{array}{cc}\uij{1}{1} & \uij{1}{2}\\ 
                      \uij{2}{1} & \uij{2}{2}\end{array}\right)=\left(\begin{array}{cc}V_{_{11}} & 0  \\ 0 & V_{_{22}}\end{array}\right)
\left(\begin{array}{cc}\sqrt{\alpha_{_{11}}}&\sqrt{\alpha_{_{12}}}\\
                       \sqrt{\alpha_{_{21}}}&\sqrt{\alpha_{_{22}}}V
\end{array}\right)
\left(\begin{array}{cc}1 & 0 \\ 0 & V_{_{11}}^{*}V_{_{12}}
\end{array}\right).\]
Hence
\[|\mbox{det}(a)|=\left|\mbox{det}\left(\begin{array}{cc}\sqrt{\alpha_{_{11}}}&\sqrt{\alpha_{_{12}}}\\
                       \sqrt{\alpha_{_{21}}}&\sqrt{\alpha_{_{22}}}V
\end{array}\right)\right|.\]
Sine $V$ is unitary, it is unitary equivalent to a diagonal matrix, with diagonal elements $v_{_{1}},\ldots,v_{_{n}}$ of modulus 1. Hence
\[|\mbox{det}(a)|=\prod_{i=1}^{n}\left|\mbox{det}\left(\begin{array}{cc}\sqrt{\alpha_{_{11}}}&\sqrt{\alpha_{_{12}}}\\
                       \sqrt{\alpha_{_{21}}}&\sqrt{\alpha_{_{22}}}v_{_{i}}
\end{array}\right)\right|,\]
which shows that 
\[\left|\sqalialj{11}{22}-\alpha_{_{12}}\right|^{n}\leq{}|\mbox{det}(a)|\leq{}\left(\sqalialj{11}{22}+\alpha_{_{12}}\right)^{n}.\]
Similarly
\[\left|\sqalialj{33}{44}-\alpha_{_{34}}\right|^{n}\leq{}|\mbox{det}(d)|\leq{}\left(\sqalialj{33}{44}+\alpha_{_{34}}\right)^{n}.\]
Since $|\mbox{det}(a)|=|\mbox{det}(d)|$ it follows that the two intervals
\[I_{_{1}}^{0}=\left[\left|\sqalialj{11}{22}-\alpha_{_{12}}\right|,\sqalialj{11}{22}+\alpha_{_{12}}\right],\]
\[I_{_{2}}^{0}=\left[\left|\sqalialj{33}{44}-\alpha_{_{34}}\right|,\sqalialj{33}{44}+\alpha_{_{34}}\right]\]
have non-empty intersection. Hence also $I_{_{1}}\cap{}I_{_{2}}\neq\emptyset,$ where $I_{_{1}}$ and $I_{_{2}}$ are the two (possibly larger) intervals
\[I_{_{1}}=\left[\alpha_{_{12}}-\sqalialj{11}{22},\alpha_{_{12}}+\sqalialj{11}{22}\right],\]
\[I_{_{2}}=\left[\alpha_{_{34}}-\sqalialj{33}{44},\alpha_{_{34}}+\sqalialj{33}{44}\right].\]
Since $\alpha_{_{11}}+\alpha_{_{22}}+2\alpha_{_{12}}=\frac{1}{n}\|a\|_{_{2}}^{2}=\frac{1}{n}\|d\|_{_{2}}^{2}=\alpha_{_{33}}+\alpha_{_{44}}+2\alpha_{_{34}}$ the two intervals
\[J_{_{1}}=\alpha_{_{11}}+\alpha_{_{22}}+2\alpha_{_{12}}-2I_{_{1}}=
\left[(\sqrt{\alpha_{_{11}}}-\sqrt{\alpha_{_{22}}})^{2},(\sqrt{\alpha_{_{11}}}+\sqrt{\alpha_{_{22}}})^{2}\right]\]
\[J_{_{2}}=\alpha_{_{33}}+\alpha_{_{44}}+2\alpha_{_{34}}-2I_{_{2}}=
\left[(\sqrt{\alpha_{_{33}}}-\sqrt{\alpha_{_{44}}})^{2},(\sqrt{\alpha_{_{33}}}+\sqrt{\alpha_{_{44}}})^{2}\right]\]
also intersect. Hence
\[\left[|\sqrt{\alpha_{_{11}}}-\sqrt{\alpha_{_{22}}}|,\sqrt{\alpha_{_{11}}}+\sqrt{\alpha_{_{22}}}\right]\cap\left[|\sqrt{\alpha_{_{33}}}-\sqrt{\alpha_{_{44}}}|,\sqrt{\alpha_{_{33}}}+\sqrt{\alpha_{_{44}}}\right]\neq{}\emptyset,\]
which  is equivalent to the ``four--angle'' inequality for $\sqrt{\alpha_{_{11}}},$ $\sqrt{\alpha_{_{22}}},$ $\sqrt{\alpha_{_{33}}}$ and $\sqrt{\alpha_{_{44}}},$ stated in the lemma.
\end{proof}
\begin{cor}
\label{fire4cor46}
If $u\in{}M_{_{4}}({\Bbb C})$ is unitary and $|u_{_{ij}}|=|u_{_{ji}}|,$ $i,j=1,2,3,4,$ then
\[|u_{_{ii}}|\leq{}|u_{_{jj}}|+|u_{_{kk}}|+|u_{_{ll}}|,\;\;\;\{i,j,k,l\}=\{1,2,3,4\}.\]
\end{cor}
\begin{proof}Set $n=1$ in lemma \ref{fire4lem45}.\end{proof}
For the rest of this section $D$ denotes the double stochastic matrix with entries
\[D_{_{ij}}=\left\{\begin{array}{cl}\delta_{_{i}}&\mbox{ if } i=j\\
                                     \sqalialj{i}{j}&\mbox{ if } i\neq{}j
                   \end{array}\right.\]
we  then have the following
\begin{lemma}
\label{fire4lem47}If $u \in{}M_{_{4}}({\Bbb C})$ is a unitary, such that $|\uij{i}{j}| = D_{_{ij}},\; i,j = 1,2,3,4,$ then $u \neq{}u^{t}.$
\end{lemma}
\begin{proof}Suppose $u = u^{t}.$ By exchanging $u$ with $wuw$ for a suitably chosen diagonal unitary operator, $w,$ we can obtain $u_{_{12}},u_{_{13}},u_{_{14}}\geq{}0.$ So $u$ is of the form
\[\left(\begin{array}{cccc}
d_{_{1}}&\sqalialj{1}{2}&\sqalialj{1}{3}&\sqalialj{1}{4}\\
\sqalialj{1}{2}&d_{_{2}}&\sqalialj{2}{3}\sigma&\sqalialj{2}{4}\rho\\
\sqalialj{1}{3}&\sqalialj{2}{3}\sigma&d_{_{3}}&\sqalialj{3}{4}\tau\\
\sqalialj{1}{4}&\sqalialj{2}{4}\rho&\sqalialj{3}{4}\tau&d_{_{4}}
\end{array}\right)\]
where $\rho,\sigma,\tau\in{}{\Bbb C},$ $|\rho|=|\sigma|=|\tau|=1$ and $|d_{_{i}}|=\delta_{_{i}},\;i=1,2,3,4.$ Orthogonality implies
\[\begin{array}{clcl}
(1) & \overline{d}_{_{1}} + d_{_{2}} + \alpha_{_{3}}\sigma + \alpha_{_{4}}\rho{}& = & 0\\
(2) & d_{_{1}} +\overline{d}_{_{3}} + \alpha_{_{2}}\overline{\sigma}+ \alpha_{_{4}}\overline{\tau} & = & 0\\
(3) & \overline{d}_{_{2}}\sigma{} + d_{_{3}}\overline{\sigma} + \alpha_{_{1}} + \alpha_{_{4}}\overline{\rho}{\tau} & = & 0
\end{array}\]
which is equivalent to
\[\begin{array}{clcl}
(1') & \overline{d}_{_{1}}\overline{\sigma} + d_{_{2}}\overline{\sigma} + \alpha_{_{3}}+ \alpha_{_{4}}\rho{}\overline{\sigma}& = & 0\\
(2') & d_{_{1}}\sigma{} +\overline{d}_{_{3}}\sigma{} + \alpha_{_{2}}+ \alpha_{_{4}}\overline{\tau}\sigma{} & = & 0\\
(3') & \overline{d}_{_{2}}\sigma{} + d_{_{3}}\overline{\sigma} + \alpha_{_{1}} + \alpha_{_{4}}\overline{\rho}{\tau} & = & 0
\end{array}\]
Hence the sum of the left-hand sides of (1'), (2') and (3') is 0, which implies
\[2\mbox{Re}(d_{_{1}}\sigma + \overline{d}_{_{2}}\sigma + \overline{d}_{_{3}}\sigma ) + \alpha_{_{1}} + \alpha_{_{2}} +  \alpha_{_{3}} + \alpha_{_{4}}(\rho\overline{\sigma} + \overline{\tau}\sigma + \overline{\rho}\sigma  ) = 0\Rightarrow\mbox{Im}(\rho\overline{\sigma} + \overline{\tau}\sigma + \overline{\rho}\sigma)=0.\]
Let $T$ be the triangle with vertices at $\rho,\sigma,\tau.$ Then
\[\mbox{area}(T) = \mbox{$\frac{1}{2}$}|\mbox{Im}((\tau - \sigma )(\overline{\rho} - \overline{\sigma}))|= \mbox{$\frac{1}{2}$}|\mbox{Im}(\rho\overline{\sigma} + \overline{\tau}\sigma + \overline{\rho}\sigma)|=0.\] 
Hence $\tau, \sigma$  and $\rho$ lie on a straight line, and since $|\rho|=|\sigma|=|\tau|=1,$ at least two are equal.

\noindent{}Take for instance the case $\sigma = \rho.$ In this case (1) states $\overline{d}_{_{1}} + d_{_{2}}  + (\alpha_{_{3}} + \alpha_{_{4}})\sigma = 0 \Rightarrow{}$

\[\delta_{_{1}} + \delta_{_{2}} \geq |\overline{d}_{_{1}} + d_{_{2}}| = \alpha_{_{3}} + \alpha_{_{4}}.\]
Which contradicts lemma \ref{fire4lem44} 2.

\noindent{}The other cases are treated similarly.
\end{proof}
\begin{lemma}
\label{fire4lem10}
Let $a,b,c$ be $n\times{}n$ matrices, such that $aa^{*}+bb^{*}=1$ and $a^{*}a+c^{*}c=1.$ Then $|\det(b)|=|\det(c)|.$ If $b$ is invertible there is one and only one $n\times{}n$ matrix $d$ such that
\[\left(\begin{array}{cc}a&b\\c&d\end{array}\right)\]
is a unitary matrix, and $d$ is given by
\[d=-(c^{*})^{-1}a^{*}b=-ca^{*}(b^{*})^{-1}.\]
\end{lemma}
\begin{proof}
Since $a^{*}a$ and $aa^{*}$ have the same list of eigenvalues
\[|\det(b)|^{2}=\det(1-aa^{*})=\det(1-a^{*}a)=|\det(c)|^{2}.\]
Assume now $|\det(b)|=|\det(c)|\neq{}0.$ If 
\[u=\left(\begin{array}{cc}a&b\\c&d\end{array}\right)\]
is unitary, then $a^{*}b+c^{*}d=0$ and $ca^{*}+db^{*}=0,$ hence
\[d=-(c^{*})^{-1}a^{*}b\;\mbox{ and }d=-ca^{*}(b^{*})^{-1}.\]
This proves uniqueness of $d,$ and the stated formulas for $d.$ 

\noindent{}To prove existence, set $d=-(c^{*})^{-1}a^{*}b.$ Then $a^{*}b+c^{*}d=0.$ By the assumptions $a^{*}a+c^{*}c=1.$ Moreover, since $af(a^{*}a)=f(aa^{*})a$ for any function $f$ on $\mbox{sp}(a^{*}a)=\mbox{sp}(aa^{*}),$ we get
\[\begin{array}{lcl}
b^{*}b+d^{*}d & = &b^{*}(1+a(c^{*}c)^{-1}a^{*})b\\
&=&b^{*}(1+a(1-a^{*}a)^{-1}a^{*})b\\
&=&b^{*}(1+(1-aa^{*})^{-1}aa^{*})b\\
&=&b^{*}(1-aa^{*})^{-1}b\\
&=&b^{*}(bb^{*})^{-1}b\\
&=&1.
\end{array}\]
Hence $u^{*}u=1,$ i.e. $u$ is unitary.
\end{proof}
\begin{prop}
\label{fire4prop49}If there exists a choice of signs such that $\delta_{_{1}} \pm{}\delta_{_{2}} \pm{}\delta_{_{3}} \pm{}\delta_{_{4}} = 0,$ then there exists a selfadjoint unitary $4\times{}4-$matrix $u,$ with Tr($u$) = 0 and $|u_{_{ij}}|=D_{_{ij}}.$
\end{prop}
\begin{proof}If $\delta_{_{1}}=\delta_{_{2}}=\delta_{_{3}}=\delta_{_{4}}=0,$ then all the $\alpha'$s are equal, and since $\sum_{i}\alpha_{_{i}}=\lambda,$
\[\alpha_{_{i}}=\mbox{$\frac{\lambda}{4}$},\;\;\;i=1,2,3,4.\]
Since $0=\delta_{_{i}}^{2}=\alpha_{_{i}}^{2}-\lambda\alpha_{_{i}}+1,$ it follows that $\lambda=\frac{4}{\sqrt{3}},$ and thus
\[\alpha_{_{1}}=\alpha_{_{2}}=\alpha_{_{3}}=\alpha_{_{4}}=\mbox{$\frac{1}{\sqrt{3}}$}.\]
In this case
\[u=\frac{1}{\sqrt{3}}\left(
\begin{array}{rrrr}
0&1&1&1\\
1&0&i&-i\\
1&-i&0&i\\
1&i&-i&0\end{array}\right)\]
is a selfadjoint unitary matrix with $\mbox{Tr}(u)=0,$ for which $|u_{_{ij}}|^{2}=D_{_{ij}}.$

\noindent{}Assume now, that not all the $\delta_{_{i}}'$s are $0.$ By the assumption we can choose $\epsilon_{_{1}},\epsilon_{_{2}},\epsilon_{_{3}},\epsilon_{_{4}}\in{\Bbb R},$ such that $|\epsilon_{_{i}}|=\delta_{_{i}}$ and
\[\epsilon_{_{1}}+\epsilon_{_{2}}+\epsilon_{_{3}}+\epsilon_{_{4}}=0.\]
The three numbers $\epsilon_{_{1}}+\epsilon_{_{2}},$ $\epsilon_{_{1}}+\epsilon_{_{3}}$ and $\epsilon_{_{2}}+\epsilon_{_{3}}$ cannot all be zero, because this would imply $\epsilon_{_{1}}=\epsilon_{_{2}}=\epsilon_{_{3}}=0$ and $\epsilon_{_{4}}=-(\epsilon_{_{1}}+\epsilon_{_{2}}+\epsilon_{_{3}})=0,$ which contradicts that $\delta_{_{i}}\neq{}0$ for some $i.$ Hence, by permuting the indices, we can obtain $\epsilon_{_{1}}+\epsilon_{_{2}}\neq{}0.$ This implies that $\epsilon_{_{3}}+\epsilon_{_{4}}=-(\epsilon_{_{1}}+\epsilon_{_{2}})\neq{}0.$

\noindent{}We seek a solution of the form
\[u=\left(\begin{array}{cccc} 
\epsilon_{_{1}}& \sqalialj{1}{2}&\sqalialj{1}{3}& \sqalialj{1}{4}\\[0.3cm]\sqalialj{1}{2}& \epsilon_{_{2}}&\sqalialj{2}{3}\;\overline{\sigma}&\sqalialj{2}{4}\;\overline{\tau} \\[0.3cm]
\sqalialj{1}{3}& \sqalialj{2}{3}\sigma&*&*\\[0.3cm]
\sqalialj{1}{4}& \sqalialj{2}{4}\tau&*&*\end{array}\right)=
\left(\begin{array}{cc}a & b\\c& d\end{array}\right),\]
$a,b,c,d\in{}M_{_{2}}({\Bbb C}),$ and $|\sigma|=|\tau|=1.$ Orthogonality of the 1'st and 2'nd column is equivalent to
\begin{equation}
\label{fire4eq49}
\epsilon_{_{1}} + \epsilon_{_{2}} + \alpha_{_{3}}\sigma + \alpha_{_{4}}\tau = 0\end{equation}
By lemma \ref{fire4lem44} 2,
\[|\epsilon_{_{1}} + \epsilon_{_{2}} |\leq{}|\epsilon_{_{1}}| + |\epsilon_{_{2}} |\leq{}\delta_{_{1}}+\delta_{_{2}}<\alpha_{_{3}}+\alpha_{_{4}},\]
and since $\epsilon_{_{1}}+\epsilon_{_{2}}=-(\epsilon_{_{3}}+\epsilon_{_{4}})$ lemma \ref{fire4lem44} 3 gives
\[|\epsilon_{_{1}} + \epsilon_{_{2}} |\geq{}\left||\epsilon_{_{3}}| - |\epsilon_{_{4}} |\right|=\left|\delta_{_{3}}-\delta_{_{4}}\right|\geq{}\alpha_{_{3}}-\alpha_{_{4}}.\]
Hence $|\epsilon_{_{3}} + \epsilon_{_{4}}|,$ $\alpha_{_{3}}$ and $\alpha_{_{4}}$ satisfy  the triangle inequality, so we can choose $\sigma$ and $\tau\in{\Bbb C},$ $|\sigma|=|\tau|=1,$ such that (\ref{fire4eq49}) holds. Moreover $|\epsilon_{_{1}} + \epsilon_{_{2}} |<\alpha_{_{3}}+\alpha_{_{4}}$ implies that $\sigma\neq{}\tau.$ Therefore the matrix
\[c=\left(\begin{array}{cc}\sqalialj{1}{3} & \sqalialj{2}{3}\sigma\\
\sqalialj{1}{4}& \sqalialj{2}{4}\tau\end{array}\right)\]
is invertible. By construction $a^{*}a+c^{*}c=1,$ and since $a=a^{*}$ and $b=c^{*},$ also $aa^{*}+bb^{*}=1.$

\noindent{}Let 
\[d=-(c^{*})^{-1}a^{*}b=-b^{-1}ab\]
as in lemma \ref{fire4lem10}. Since $bb^{*}=1-a^{2},$
\[d=b^{-1}ab=-b^{*}(bb^{*})^{-1}ab=-b^{*}(1-a^{2})^{-1}ab.\]
Hence $d=d^{*},$ i.e.
\[d=\left(\begin{array}{cc}d_{_{1}}&z\\\overline{z}&d_{_{2}}\end{array}\right),\;\;\;d_{_{1}},d_{_{2}}\in{\Bbb R},\;\;z\in{\Bbb C}.\]
Moreover
\[\mbox{Tr}(d)=-\mbox{Tr}(b^{-1}ab)=-\mbox{Tr}(a),\]
i.e.
\begin{equation}
\label{fire4prop49eqn1}d_{_{1}}+d_{_{2}}=-(\epsilon_{_{1}} + \epsilon_{_{2}})=\epsilon_{_{3}} + \epsilon_{_{4}}.
\end{equation}
Since $u=\left(\begin{array}{cc}a&b\\c&d\end{array}\right)$ is unitary by lemma \ref{fire4lem10},
\[\alpha_{_{1}}\alpha_{_{3}}+\alpha_{_{2}}\alpha_{_{3}}+d_{_{1}}^{2}+|z|^{2}=1,\]
\[\alpha_{_{1}}\alpha_{_{4}}+\alpha_{_{2}}\alpha_{_{4}}+|z|^{2}+d_{_{2}}^{2}=1,\]
and since the matrix 
\[D_{_{ij}}=\left\{\begin{array}{cl}
\delta_{_{i}}^{2},&i=j\\[0.2cm]
\alpha_{_{i}}\alpha_{_{j}},&i\neq{}j\end{array}\right.\]
is doubly stochastic, it follows that
\begin{equation}
\label{fire4prop49eqn2}
\begin{array}{lcl}
d_{_{1}}^{2}+|z|^{2} & = & \delta_{_{3}}^{2}+\alpha_{_{3}}\alpha_{_{4}}\\[0.2cm]
|z|^{2}+d_{_{2}}^{2}& = & \alpha_{_{3}}\alpha_{_{4}}+\delta_{_{4}}^{2}.
\end{array}\end{equation}
Hence 
\[(d_{_{1}}+d_{_{2}})(d_{_{1}}-d_{_{2}})=d_{_{1}}^{2}-d_{_{2}}^{2}=\delta_{_{3}}^{2}-\delta_{_{4}}^{2}=\epsilon_{_{3}}^{2}-\epsilon_{_{4}}^{2}=(\epsilon_{_{3}}+\epsilon_{_{4}})(\epsilon_{_{3}}-\epsilon_{_{4}}).\]
But $d_{_{1}}+d_{_{2}}=\epsilon_{_{3}}+\epsilon_{_{4}}\neq{}0.$ Thus
\begin{equation}
\label{fire4prop49eqn3}
d_{_{1}}-d_{_{2}}=\epsilon_{_{3}}-\epsilon_{_{4}},
\end{equation}
so by (\ref{fire4prop49eqn1}) and (\ref{fire4prop49eqn3}), $d_{_{1}}=\epsilon_{_{3}}$ and $d_{_{2}}=\epsilon_{_{4}}.$ Finally (\ref{fire4prop49eqn2}) gives $|z|^{2}=\alpha_{_{3}}\alpha_{_{4}}.$ Hence
\[u=\left(\begin{array}{cc}a&b\\c&d\end{array}\right)\]
is a selfadjoint unitary with
\[\mbox{Tr}(u)=\mbox{Tr}(a)+\mbox{Tr}(d)=0,\]
and 
\[|u_{_{ij}}|^{2}=D_{_{ij}},\;\;\;i,j=1,2,3,4.\]
\end{proof}
\begin{prop}
\label{fire4prop410}If there exists a unitary $u \in{} M_{_{4}}({\Bbb C})$ such that  $|\uij{i}{j}| = D_{_{ij}},$ then $u$ can be chosen as $u = \sqrt{1-\gamma^{2}}v + i\gamma{}1,$ where $v$ is selfadjoint unitary with $\mbox{Tr}(v) = 0$ and $\gamma\in[-1,1].$
\end{prop}
\begin{proof}By proposition \ref{fire4prop49} we may assume that there is no choice of signs such that  $\delta_{_{1}} \pm{}\delta_{_{2}}\pm{} \delta_{_{3}}\pm{} \delta_{_{4}} = 0.$ Particularly not all $\delta'$s have the same value. If we relabel the $\delta'$s to get $\delta_{_{1}} > \delta_{_{2}},$  $u$ can be chosen to be
\[\left(\begin{array}{cccc}
\delta_{_{1}} & \sqalialj{1}{2} & \sqalialj{1}{3} & \sqalialj{1}{4}\\
\sqalialj{1}{2} & u_{_{22}} & * & * \\
\sqalialj{1}{3} & * & * & *\\ 
\sqalialj{1}{4} & * & * & *
\end{array}\right), \mbox{ where } |u_{_{22}}|=\delta_{_{2}}.\]
If $u = \left(\begin{array}{cc}a & b\\c& d\end{array}\right)$ is unitary, then so is $\left(\begin{array}{cc}ae^{i\theta} & b\\c& de^{-i\theta}\end{array}\right).$ Hence we may substitute $e^{i\theta}\delta_{_{1}}$ for $\delta_{_{1}}$ and $e^{-i\theta}u_{_{22}}$ for $u_{_{22}},$ where $\theta$ is chosen such that  $e^{i\theta}(\delta_{_{1}} + \overline{u}_{_{22}}) > 0.$ 

\noindent{}Put $u_{_{11}}' = e^{i\theta}\delta_{_{1}}$ and $u_{_{22}}' = e^{-i\theta}u_{_{22}},$ then Im($u_{_{11}}'$) = Im($u_{_{22}}'$)  and $u$ is now transformed to
\[\left(\begin{array}{cc} 
\begin{array}{cc} \epsilon_{_{1}}+i\gamma & \sqalialj{1}{2}\\[0.2cm]
                  \sqalialj{1}{2}& \epsilon_{_{2}}+i\gamma\end{array} & 
\begin{array}{cc} \sqalialj{1}{3}& \sqalialj{1}{4}\\[0.2cm]
                  \sqalialj{2}{3}\mu& \sqalialj{2}{4}\mu'\end{array}\\[0.5cm] 
\begin{array}{cc} \sqalialj{1}{3}&\sqalialj{2}{3}\sigma \\[0.2cm]
                  \sqalialj{1}{4}& \sqalialj{2}{4}\sigma'\end{array} & d
\end{array}\right),\]
where $\epsilon_{_{1}},\epsilon_{_{2}}\in{\Bbb R}.$

\noindent{}Orthogonality of the two first rows, respectively columns, gives\[\begin{array}{cl}
(1) &  \epsilon_{_{1}} + \epsilon_{_{2}} + \alpha_{_{3}}\sigma + \alpha_{_{4}}\sigma' = 0\\
(2) &  \epsilon_{_{1}} + \epsilon_{_{2}} + \alpha_{_{3}}\mu + \alpha_{_{4}}\mu' = 0
\end{array}\] 
Since $\epsilon_{_{1}}+\epsilon_{_{2}} > 0$ there are only two solutions to (1), hence either $\sigma=\mu$ (and $\sigma'=\mu'$) or $\sigma=\overline{\mu}$ (and $\sigma'=\overline{\mu}'$).

\noindent{}By lemma \ref{fire4lem44}  $\epsilon_{_{1}} + \epsilon_{_{2}}\leq{}\delta_{_{1}}+\delta_{_{2}}<\alpha_{_{3}}+\alpha_{_{4}},$ so the triangle \setlength{\unitlength}{1cm}
\raisebox{-0.8cm}{\begin{picture}(2,2)
\put(0.5,0.5){\line(1,2){0.5}}
\put(1.5,0.5){\line(-1,2){0.5}}
\put(0.5,0.5){\line(1,0){1}}
\put(0.5,0.2){$\epsilon_{_{1}}+\epsilon_{_{2}}$}
\put(0,1){$\alpha_{_{3}}\sigma$}
\put(1.3,1){$\alpha_{_{4}}\sigma'$}
\end{picture}}
 does not degenerate to a straight line. I.e. $\sigma$ and $\sigma'$ have non-trivial imaginary parts. 

\noindent{}We are now in one of the following situations
\[(a)\;\;\;\;u=\left(\begin{array}{cc}a+i\gamma{}1 & b \\
b^{t} & d \end{array}\right)\;\;\;\;\;\;\;
(b)\;\;\;\;u=\left(\begin{array}{cc}a+i\gamma{}1 & b \\
b^{*} & d \end{array}\right)\]
In case (a) $d$ is uniquely determined as
\[d = - (\overline{b})^{-1}(a - i\gamma{}1)b= -b^{t}(a - i\gamma{}1)(b^{*})^{-1},\]
hence $d = d^{t}\Rightarrow{}u=u^{t},$ which contradicts lemma \ref{fire4lem47}. I.e. we must be in case (b).

\noindent{}Here we get $d = -(b^{-1})(a - i\gamma{}1)b= - b^{*}(a - i\gamma{}1)(b^{*})^{-1},$ that is $ d = d_{_{sa}}  + i\gamma{}1,$  where $d_{_{sa}}$  is selfadjoint. Moreover 
\[\mbox{Tr}(\mbox{selfadjoint part of }u) = 0,\]
because
    \[\mbox{Tr}(d_{_{sa}}) = \mbox{Tr}(-b^{-1}ab) = - \mbox{Tr}(a).\]
Hence $u=s+i\gamma{}1,$ where $s$ is selfadjoint with $\mbox{Tr}(s)=0.$ But $u^{*}u=1$ implies $s^{*}s=(1-\gamma^{2})1.$ Thus $|\gamma|\leq{}1$ and $s=\sqrt{1-\gamma^{2}}v$ for a selfadjoint unitary $v$ with trace 0.
\end{proof}
\begin{theorem}
\label{fire4thm411}
Let $\lambda,\delta_{_{1}} \geq{} \delta_{_{2}} \geq{} \delta_{_{3}} \geq{} \delta_{_{4}} > 0$ be defined by $0 <\alpha_{_{1}} \leq{} \alpha_{_{2}}  \leq{}\alpha_{_{3}} \leq{} \alpha_{_{4}}$ as before. Then the following are equivalent
\begin{enumerate}
\item{}There exists a unitary $u\in{}M_{_{4}}({\Bbb C}),\; u = (\uij{i}{j})$  such that  $|\uij{i}{j}| = D_{_{ij}}.$
\item{}\begin{equation}
\label{fire4cond1}
\delta_{_{1}} - \delta_{_{2}} - \delta_{_{3}} - \delta_{_{4}} \leq{}0
\end{equation}
\begin{equation}
\label{fire4cond2}
\delta_{_{1}} - \delta_{_{2}} - \delta_{_{3}} + \delta_{_{4}} \geq{} 0.
\end{equation}
\end{enumerate}
\end{theorem}
\begin{proof}1 $\Rightarrow$ 2. If all the $\delta'$s are equal, 2 is trivially fulfilled. If the $\delta'$s are not all equal, proposition \ref{fire4prop410} states that $u$ can be chosen as $u = v + i\gamma{}1,$ where  $v$ is selfadjoint $\mbox{Tr}(v) = 0$ and $\gamma \in{}{\Bbb R},$ i.e. $u_{_{kk}}=\epsilon_{_{k}}+i\gamma$ and $\sum_{k}\epsilon_{_{k}}=0.$ Hence
\[\epsilon_{_{k}}=\pm\sqrt{|u_{_{kk}}|^{2}-\gamma^{2}}=\pm\sqrt{\delta_{_{k}}^{2}-\gamma^{2}},\]
where $\gamma\leq\min\{\delsqi{1},\delsqi{2},\delsqi{3},\delsqi{4}\}=\delsqi{4},$ and proposition \ref{fire4prop42} gives the implication.

\noindent{}2 $\Rightarrow$ 1. By proposition \ref{fire4prop42} we can choose $t\in{}[0,\delsqi{4}]$ and signs such that
\[\sqrt{\delsqi{1}-t}\pm{}\sqrt{\delsqi{2}-t}\pm{}\sqrt{\delsqi{3}-t}\pm{}\sqrt{\delsqi{4}-t}= 0.\]
Put $a_{_{i}}=\frac{1}{\sqrt{1-t}}\alpha_{_{1}},$ then $\lambda_{_{1}}=\sum_{i}a_{_{i}}=\frac{1}{\sqrt{1-t}}\lambda>2,$ and $a_{_{i}}^{2}-\lambda_{_{1}}a_{_{i}}+1=\frac{\delsqi{i}-t}{1-t}\geq{}0$\\[0.3cm]

\noindent{}Put $d_{_{i}}=\sqrt{a_{_{i}}^{2}-\lambda_{_{1}}\alpha_{_{i}}+1}=\frac{1}{\sqrt{1-t}}\sqrt{\delsqi{i}-t}.$ Since $\frac{1}{\lambda}\leq{}\frac{\alpha_{_{i}}}{1-t}$ we get $\frac{1}{\lambda_{_{1}}}\leq{}a_{_{i}}.$ 

\noindent{}Moreover, since $a_{_{i}}^{2}-\lambda_{_{1}}a_{_{i}}+1\geq{}0,$ either
\[a_{_{i}}\leq{}\frac{\lambda_{_{1}}-\sqrt{\lambda_{_{1}}^{2}-4}}{2}\;\mbox{ or }a_{_{i}}\geq{}\frac{\lambda_{_{1}}+\sqrt{\lambda_{_{1}}^{2}-4}}{2}.\]
However $\alpha_{_{i}}\leq{}\frac{\lambda-\sqrt{\lambda^{2}-4}}{2}\leq{}\frac{\lambda}{2}$ implies that $a_{_{i}}\leq\frac{\lambda_{_{1}}}{2}.$ Hence $a_{_{i}}\leq{}\frac{\lambda_{_{1}}-\sqrt{\lambda_{_{1}}^{2}-4}}{2}.$

\noindent{}Proposition \ref{fire4prop49} now produces a selfadjoint unitary $v,$ with  $\mbox{Tr}(v) = 0 $ such that
\[|v_{_{ij}}|^{2}=\left\{\begin{array}{cl}
d_{_{i}}^{2} & i=j\\
a_{_{i}}a_{_{j}}& i\neq{}j
\end{array}\right.\]
Put $u = \sqrt{1-t}v + i\sqrt{t}1,$ then $u$ is unitary and
\[\begin{array}{lclclcl}
|u_{_{ii}}|^{2} & = & (1-t)v_{_{ii}}+t & = & (1-t)d_{_{i}}^{2}+t & = & \delsqi{i}\\[0.3cm]
|u_{_{ij}}|^{2} & = & (1-t)v_{_{ij}} & = & (1-t)a_{_{i}}a_{_{j}} &= & \alpha_{_{i}}\alpha_{_{j}}\,\;\;\;\;i\neq{}j
\end{array}\]
\end{proof}
\begin{remark}{\rm 
\label{fire4rem411}Let $\delta_{_{1}} \geq{} \delta_{_{2}} \geq{} \delta_{_{3}} \geq{} \delta_{_{4}} \geq{}0.$ By trivial manipulations
\begin{enumerate}
\item{}
\[\begin{array}{cl}
& \delta_{_{1}} - \delta_{_{2}} - \delta_{_{3}} - \delta_{_{3}} \leq{}0\\
\Updownarrow & \\
 & (-\delta_{_{1}} + \delta_{_{2}}+ \delta_{_{3}}  +\delta_{_{4}})(\delta_{_{1}} - \delta_{_{2}} +\delta_{_{3}} + \delta_{_{4}})(\delta_{_{1}} + \delta_{_{2}}- \delta_{_{3}}  +\delta_{_{4}})(\delta_{_{1}}  +\delta_{_{2}} +\delta_{_{3}} - \delta_{_{4}}) \leq{} 0
\end{array}\]
\item{}
\[\begin{array}{cl}
& \delta_{_{1}} - \delta_{_{2}} - \delta_{_{3}} + \delta_{_{3}} \geq{}0\\
\Updownarrow & \\
 & (\delta_{_{1}}  +\delta_{_{2}} +\delta_{_{3}}  +\delta_{_{4}})(\delta_{_{1}} - \delta_{_{2}} -\delta_{_{3}}  +\delta_{_{4}})(\delta_{_{1}} - \delta_{_{2}}+ \delta_{_{3}} - \delta_{_{4}})(\delta_{_{1}} + \delta_{_{2}} -\delta_{_{3}} - \delta_{_{4}}) \geq{} 0.
\end{array}\] 
\end{enumerate}
}\end{remark}
Hence theorem \ref{fire4thm411} can be stated in the symmetric form:

\noindent{}{\em There exists a unitary $u \in{} M_{_{4}}({\Bbb C} ),$  $u = (\uij{i}{j})$ such that  $|\uij{i}{j}| = D_{_{ij}}$ if and only if
\[-\delta_{_{1}}\delta_{_{2}}\delta_{_{3}}\delta_{_{4}}\leq{}
\delta_{_{1}}^{4}+\delta_{_{2}}^{4}+\delta_{_{3}}^{4}+\delta_{_{4}}^{4}
-2\sum_{i<j}\delta_{_{i}}\delta_{_{j}}\leq{}\delta_{_{1}}\delta_{_{2}}\delta_{_{3}}\delta_{_{4}}.\]}

\begin{prop}
\label{fire4prop15}
If $\delta_{_{1}},\delta_{_{2}},\delta_{_{3}}$ $\delta_{_{4}}$ satisfy the ``four--angle'' inequality 
\[\delta_{_{i}}\leq{}\delta_{_{j}}+\delta_{_{k}}+\delta_{_{l}},\;\;\;\{i,j,k,l\}=\{1,2,3,4\}.\]
Then there is a unitary $4\times{}4$ matrix, $u,$ with entries in the quaternions ${\Bbb H},$ such that
\begin{equation}
\label{fire415eqn0}|u_{_{ij}}|^{2}=\left\{\begin{array}{cl}
\alpha_{_{i}}\alpha_{_{j}},&i\neq{}j\\[0.2cm]
\delta_{_{i}}^{2}&i=j.\end{array}\right.
\end{equation}
\end{prop}
\begin{proof}The quaternions ${\Bbb H}=\{a_{_{1}}+ia_{_{2}}+ja_{_{3}}+ka_{_{4}}\;|\;a_{_{1}},a_{_{2}},a_{_{3}},a_{_{4}}\in{\Bbb R}\}$ can be identified with the real subalgebra of $M_{_{2}}({\Bbb C})$ given by
\[\left\{\left.\left(
\begin{array}{cc}a_{_{1}}+ia_{_{2}}&a_{_{2}}+ia_{_{3}}\\
-a_{_{2}}+ia_{_{3}}&a_{_{1}}-ia_{_{2}}\end{array}\right)\right|\;a_{_{1}},a_{_{2}},a_{_{3}},a_{_{4}}\in{\Bbb R}\right\}.\]
Therefore any $v\in{}M_{_{n}}({\Bbb H})$ can be considered as an element in $M_{_{2n}}({\Bbb C}).$ Hence the second part of lemma \ref{fire4lem10} extends trivially to matrices with quaternionic entries.

\noindent{}By permuting the indices, we can assume that $\delta_{_{1}}+\delta_{_{2}}\leq{}\delta_{_{3}}+\delta_{_{4}},$ so by the assumptions on $\delta_{_{1}},\delta_{_{2}},\delta_{_{3}}$ and $\delta_{_{4}}$
\begin{equation}
\label{fire415eqn1}
|\delta_{_{3}}-\delta_{_{4}}|\leq{}\delta_{_{1}}+\delta_{_{2}}\leq{}\delta_{_{3}}+\delta_{_{4}}
\end{equation}
i.e. $\delta_{_{1}}+\delta_{_{2}},$ $\delta_{_{3}}$ and $\delta_{_{4}}$ satisfy the triangle inequality.

\noindent{}If $\delta_{_{2}}+\delta_{_{2}}=0,$ then $\delta_{_{1}}=\delta_{_{2}}=0$ and $\delta_{_{3}}=\delta_{_{4}}.$ In this case (\ref{fire415eqn0}) has a solution in $M_{_{4}}({\Bbb C})\subset{}M_{_{4}}({\Bbb H})$ by proposition \ref{fire4prop49}.

\noindent{}Hence we may assume that $\delta_{_{1}}+\delta_{_{2}}>0.$

\noindent{}We seek a solution of the form
\begin{equation}
\label{fire415eqn2}
u=\left(\begin{array}{cccc} 
\delta_{_{1}}& \sqalialj{1}{2}&\sqalialj{1}{3}& \sqalialj{1}{4}\\[0.3cm]\sqalialj{1}{2}& \delta_{_{2}}&\sqalialj{2}{3}\sigma'&\sqalialj{2}{4}\tau' \\[0.3cm]
\sqalialj{1}{3}& \sqalialj{2}{3}\sigma&*&*\\[0.3cm]
\sqalialj{1}{4}& \sqalialj{2}{4}\tau&*&*\end{array}\right)=
\left(\begin{array}{cc}a & b\\c& d\end{array}\right),\end{equation}
where $a,b,c,d\in{}M_{_{2}}({\Bbb H})$ and $\sigma,\sigma',\tau$ and $\tau'$ are quaternions of modulus 1. Orthogonality of the first two columns resp. rows is equivalent to
\begin{equation}
\label{fire415eqn3}
\delta_{_{1}}+\delta_{_{2}}+\alpha_{_{3}}\sigma+\alpha_{_{4}}\tau=0,
\end{equation}
resp.
\begin{equation}
\label{fire415eqn4}
\delta_{_{1}}+\delta_{_{2}}+\alpha_{_{3}}\sigma'+\alpha_{_{4}}\tau'=0.
\end{equation}
By lemma \ref{fire4lem44} 2, 3 and (\ref{fire415eqn1}) 
\[|\alpha_{_{3}}-\alpha_{_{4}}|\leq\delta_{_{1}}+\delta_{_{2}}<\alpha_{_{3}}+\alpha_{_{4}},\]
i.e. $\delta_{_{1}}+\delta_{_{2}},\alpha_{_{3}}$ and $\alpha_{_{4}}$ satisfy the triangle inequality.

\noindent{}If $\sigma,\tau$ are unit quaternions satisfying (\ref{fire415eqn3}), then
\[|\delta_{_{1}}+\delta_{_{2}}+\alpha_{_{3}}\sigma|^{2}=\alpha_{_{4}}^{2}.\]
Thus $\mbox{Re}\sigma=h,$ where 
\[h=\frac{\alpha_{_{4}}^{2}-\alpha_{_{3}}^{2}-(\delta_{_{1}}+\delta_{_{2}})^{2}}{2\alpha_{_{3}}(\delta_{_{1}}+\delta_{_{2}})}.\]
Since $\delta_{_{1}}+\delta_{_{2}},\alpha_{_{3}}$ and $\alpha_{_{4}}$ satisfy the triangle inequality, there are solutions to (\ref{fire415eqn4}), so in particular $|h|\leq{}1.$

\noindent{}Let $(i,j,k)$ be the standard basis for the imaginary part of ${\Bbb H}.$ Set 
\[\sigma=h+i\sqrt{1-h^{2}}\;\;\mbox{ and }\;\;\sigma'=h-(i\cos{}\theta+j\sin{}\theta)\sqrt{1-h^{2}},\;\;\theta\in{}[0,\pi]\]
Then $|\sigma|=|\sigma'|=1,$ $\mbox{Re}(\sigma)=\mbox{Re}(\sigma')=h$ and the angle between the imaginary parts of $\sigma$ and $\sigma'$ is $\pi-\theta.$ In particular
\[\sigma'=\overline{\sigma}\;\mbox{ iff } \;\theta=0,\]
\[\sigma'=\sigma \;\mbox{ iff }\; \theta=\pi.\]
Since $\mbox{Re}(\sigma)=\mbox{Re}(\sigma')=h,$
\[|\delta_{_{1}}+\delta_{_{2}}+\alpha_{_{3}}\sigma|^{2}=|\delta_{_{1}}+\delta_{_{2}}+\alpha_{_{3}}\sigma'|^{2}=\alpha_{_{4}}^{2}.\]
Hence
\[\tau=\mbox{$\frac{1}{\alpha_{_{4}}}$}(\delta_{_{1}}+\delta_{_{2}}+\alpha_{_{3}}\sigma)\]
and
\[\tau'=\mbox{$\frac{1}{\alpha_{_{4}}}$}(\delta_{_{1}}+\delta_{_{2}}+\alpha_{_{3}}\sigma')\]
are unit quaternions, and (\ref{fire415eqn3}) and (\ref{fire415eqn4}) hold, i.e.
\[a^{*}a+c^{*}c=1,\;\;\;aa^{*}+bb^{*}=1.\]
Moreover $b$ and $c$ are invertible, because the inequality $\delta_{_{1}}+\delta_{_{2}}<\alpha_{_{3}}+\alpha_{_{4}}$ implies that $\sigma\neq{}\tau$ and  $\sigma'\neq{}\tau'.$ Thus, by lemma \ref{fire4lem10} and the remarks in the beginning of this proof,
\[d=-(c^{*})^{-1}a^{*}b=-ca^{*}(b^{*})^{-1}\]
defines a matrix in $M_{_{2}}({\Bbb H}),$ such that 
\[u=\left(\begin{array}{cc}a&b\\c&d\end{array}\right)\]
is unitary.

\noindent{}Since $(|u_{_{ij}}|^{2})_{i,j}$ and $(D_{_{ij}})_{i,j}$ are doubly stochastic matrices, which coincide on the  two first rows and the two first columns
\[\begin{array}{lcl}
|d_{_{11}}|^{2}+|d_{_{12}}|^{2}&=&D_{_{33}}+D_{_{34}}\\[0.2cm]
|d_{_{21}}|^{2}+|d_{_{22}}|^{2}&=&D_{_{43}}+D_{_{44}}\\[0.2cm]
|d_{_{11}}|^{2}+|d_{_{21}}|^{2}&=&D_{_{33}}+D_{_{43}}\\[0.2cm]
|d_{_{12}}|^{2}+|d_{_{22}}|^{2}&=&D_{_{34}}+D_{_{44}}.
\end{array}\]
Since $D_{_{33}}=\delta_{_{3}}, D_{_{44}}=\delta_{_{4}}$ and $D_{_{34}}=\alpha_{_{3}}\alpha_{_{4}}$ it follows that
\begin{equation}
\label{fire415eqn5}
\left(\begin{array}{cc}
|d_{_{11}}|^{2}&|d_{_{12}}|^{2}\\
|d_{_{21}}|^{2}&|d_{_{22}}|^{2}\end{array}\right)=
\left(\begin{array}{cc}
\delta_{_{3}}^{2}-\kappa & \alpha_{_{3}}\alpha_{_{4}}+\kappa\\
\alpha_{_{3}}\alpha_{_{4}}+\kappa&\delta_{_{4}}^{2}-\kappa
\end{array}\right)\end{equation}
for some constant, $\kappa=\kappa(\theta),$ depending on $\theta.$ Hence to prove the proposition we have to show, that $\theta\in{}[0,\pi]$ can be chosen such that $\kappa(\theta)=0.$

\noindent{}If $\theta=0,$ then $b=c^{*},$ so as in the proof of proposition \ref{fire4prop49}. we have $d=d^{*}$ and
\[\mbox{Tr}(d)=-\mbox{Tr}(a)=-(\delta_{_{1}}+\delta_{_{2}}).\]
Thus $d_{_{11}},d_{_{22}}\in{\Bbb R},$
\begin{equation}
\label{fire415eqn6}
d_{_{11}}+d_{_{22}}=-(\delta_{_{1}}+\delta_{_{2}})
\end{equation}
and by (\ref{fire415eqn5})
\[d_{_{11}}^{2}-d_{_{22}}^{2}=\delta_{_{3}}^{2}-\delta_{_{4}}^{2}\]
Hence
\begin{equation}
\label{fire415eqn7}
d_{_{11}}-d_{_{22}}=-\frac{\delta_{_{3}}^{2}-\delta_{_{4}}^{2}}{\delta_{_{1}}+\delta_{_{2}}}.\end{equation}
By  (\ref{fire415eqn6}) and  (\ref{fire415eqn7})
\[d_{_{11}}=-\mbox{$\frac{1}{2}$}\left((\delta_{_{1}}+\delta_{_{2}})+
\frac{\delta_{_{3}}^{2}-\delta_{_{4}}^{2}}{\delta_{_{1}}+\delta_{_{2}}}\right),\]
\[d_{_{22}}=-\mbox{$\frac{1}{2}$}\left((\delta_{_{1}}+\delta_{_{2}})-
\frac{\delta_{_{3}}^{2}-\delta_{_{4}}^{2}}{\delta_{_{1}}+\delta_{_{2}}}\right).\]
Thus 
\[\begin{array}{lcl}\kappa(0)&=&\mbox{$\frac{1}{2}$}(\delta_{_{3}}^{2}+\delta_{_{4}}^{2}-d_{_{11}}^{2}-d_{_{22}}^{2})\\[0.3cm]
&=&-\mbox{$\frac{1}{4}$}\left((\delta_{_{1}}+\delta_{_{2}})^{2}-2(\delta_{_{3}}^{2}+\delta_{_{4}}^{2})+\frac{(\delta_{_{3}}^{2}-\delta_{_{4}}^{2})^{2}}{(\delta_{_{1}}+\delta_{_{2}})^{2}}\right)
\end{array}\]
Since the roots of the polynomial 
\[ t^{2} - 2(\delsqi{3}+\delsqi{4})t+(\delsqi{3} - \delsqi{4} )^{2}\]  
are $(\delta_{_{3}} + \delta_{_{4}})^{2}$ and $(\delta_{_{3}} - \delta_{_{4}})^{2}$ it follows that
\begin{equation}
\label{fire415eqn8}
\kappa(0)=\frac{1}{4(\delta_{_{1}}+\delta_{_{2}})^{2}}\left((\delta_{_{3}}+\delta_{_{4}})^{2}-(\delta_{_{1}}+\delta_{_{2}})^{2}\right)\left((\delta_{_{1}}+\delta_{_{2}})^{2}-(\delta_{_{3}}-\delta_{_{4}})^{2}\right).
\end{equation}
Note that $\kappa(0)\geq{}0,$ because $\delta_{_{1}}+\delta_{_{2}},\delta_{_{3}}$ and $\delta_{_{4}}$ satisfy the triangle inequality (\ref{fire415eqn1}).

\noindent{}Next we show that $\kappa(\pi)\leq{}0.$ Let $\theta\in[0,\pi].$ Since $\sigma,\tau\in\{a_{_{1}}+ia_{_{2}}\;|\;a_{_{1}},a_{_{2}}\in{\Bbb R}\}\cong{\Bbb C}$ one gets
\[(c^{*})^{-1}=\frac{1}{\sqrt{\alpha_{_{1}}\alpha_{_{2}}\alpha_{_{3}}\alpha_{_{4}}}(\overline{\tau}-\overline{\sigma})}\left(\begin{array}{cc}
\sqrt{\alpha_{_{2}}\alpha_{_{4}}}\;\overline{\tau}&-\sqrt{\alpha_{_{3}}\alpha_{_{4}}}\\[0.2cm]
-\sqrt{\alpha_{_{2}}\alpha_{_{3}}}\;\overline{\sigma}&\sqrt{\alpha_{_{1}}\alpha_{_{3}}}\end{array}\right).\]
Hence
\[d_{_{21}}=\left((c^{*})^{-1}ab\right)_{21}=
\frac{\sqrt{\alpha_{_{3}}}}{\sqrt{\alpha_{_{4}}}(\overline{\tau}-\overline{\sigma})}(\alpha_{_{1}}-\delta_{_{1}}\overline{\sigma}+\delta_{_{2}}\sigma'-\alpha_{_{2}}\overline{\sigma}\sigma').\]
Since $\mbox{Re}(\sigma) = \mbox{Re}(\sigma') = h$ and 
\[\mbox{Re}(\overline{\sigma}\sigma')=h^{2}-(1-h^{2})\cos{}\theta,\;\;\;\;\mbox{Re}(\sigma\sigma')=h^{2}+(1-h^{2})\cos{}\theta,\]
we have
\[|\alpha_{_{1}}-\alpha_{_{2}}\overline{\sigma}\sigma'-\delta_{_{1}}\overline{\sigma}+\delta_{_{2}}\sigma'|^{2}=\beta-\gamma\cos{}\theta,\]
where 
\[\beta= \alpha_{_{1}}^{2}+\alpha_{_{2}}^{2}+\delta_{_{1}}^{2}+\delta_{_{2}}^{2}-2h(\alpha_{_{1}}-\alpha_{_{2}})(\delta_{_{1}}-\delta_{_{2}})-2h^{2}(\alpha_{_{1}}\alpha_{_{2}}+\delta_{_{1}}\delta_{_{2}})\]
and 
\[\gamma=2(1-h^{2})(\alpha_{_{1}}\alpha_{_{2}}-\delta_{_{1}}\delta_{_{2}}).\]
Therefore 
\[\kappa(\theta)=|d_{_{21}}|^{2}-\alpha_{_{3}}\alpha_{_{4}}=\frac{\alpha_{_{3}}}{\alpha_{_{4}}|\overline{\tau}-\overline{\sigma}|^{2}}(\beta-\gamma\cos{}\theta)-\alpha_{_{3}}\alpha_{_{4}}.\]
In particular 
\[\kappa(0)-\kappa(\pi)=\frac{2\alpha_{_{3}}\gamma}{\alpha_{_{4}}|\overline{\tau}-\overline{\sigma}|^{2}}=\frac{4\alpha_{_{3}}(1-h^{2})}{\alpha_{_{4}}|\overline{\tau}-\overline{\sigma}|^{2}}(\alpha_{_{1}}\alpha_{_{2}}-\delta_{_{1}}\delta_{_{2}}).\]
By (\ref{fire415eqn3})
\[(\delta_{_{1}}+\delta_{_{2}})\overline{\sigma}+\alpha_{_{3}}+\alpha_{_{4}}\tau\overline{\sigma}=0.\]
Therefore 
\[(\delta_{_{1}}+\delta_{_{2}})|\mbox{Im}(\sigma)|=\alpha_{_{4}}|\mbox{Im}(\tau\overline{\sigma})|,\]
from which
\[1-h^{2}=|\mbox{Im}(\sigma)|^{2}=\frac{\alpha_{_{4}}^{2}}{(\delta_{_{1}}+\delta_{_{2}})^{2}}\left|\mbox{Im}(\tau\overline{\sigma})\right|^{2}=\frac{\alpha_{_{4}}^{2}}{(\delta_{_{1}}+\delta_{_{2}})^{2}}\left(1-\mbox{Re}(\tau\overline{\sigma})^{2}\right).\]
Moreover $|\overline{\sigma}-\overline{\tau}|^{2}=2(1-\mbox{Re}(\tau\overline{\sigma})),$ hence
\[\kappa(0)-\kappa(\pi)=\frac{2\alpha_{_{3}}\alpha_{_{4}}(1+\mbox{Re}(\tau\overline{\sigma}))}{(\delta_{_{1}}+\delta_{_{2}})^{2}}(\alpha_{_{1}}\alpha_{_{2}}-\delta_{_{1}}\delta_{_{2}}).\]
Again using (\ref{fire415eqn3})
\[|\alpha_{_{3}}\sigma+\alpha_{_{4}}\tau|^{2}=(\delta_{_{1}}+\delta_{_{2}})^{2}.\]
Thus
\[\mbox{Re}(\tau\overline{\sigma})=\frac{(\delta_{_{1}}+\delta_{_{2}})^{2}-\alpha_{_{3}}^{2}-\alpha_{_{4}}^{2}}{2\alpha_{_{3}}\alpha_{_{4}}},\]
and
\[1+\mbox{Re}(\tau\overline{\sigma})=\frac{(\delta_{_{1}}+\delta_{_{2}})^{2}-(\alpha_{_{3}}-\alpha_{_{4}})^{2}}{2\alpha_{_{3}}\alpha_{_{4}}}.\]
All together 
\begin{equation}
\label{fire415eqn9}
\kappa(0)-\kappa(\pi)=\frac{(\alpha_{_{1}}\alpha_{_{2}}-\delta_{_{1}}\delta_{_{2}})((\delta_{_{1}}+\delta_{_{2}})^{2}-(\alpha_{_{3}}-\alpha_{_{4}})^{2})}{(\delta_{_{1}}+\delta_{_{2}})^{2}}
\end{equation}
From (\ref{fire415eqn8}) and (\ref{fire415eqn9}) it follows that $\kappa(\pi)\leq{}0$ if and only if
\begin{equation}
\label{fire415eqn10}
(\alpha_{_{1}}\alpha_{_{2}}-\delta_{_{1}}\delta_{_{2}})((\delta_{_{1}}+\delta_{_{2}})^{2}-(\alpha_{_{3}}-\alpha_{_{4}})^{2})\geq{}
\mbox{$\frac{1}{4}$}((\delta_{_{3}}+\delta_{_{4}})^{2}-(\delta_{_{1}}+\delta_{_{2}})^{2})((\delta_{_{1}}+\delta_{_{2}})^{2}-(\delta_{_{3}}-\delta_{_{4}})^{2}).
\end{equation}
By lemma \ref{fire4lem44} 1,
\[\alpha_{_{1}}\alpha_{_{2}}-\delta_{_{1}}\delta_{_{2}}>\mbox{$\frac{\lambda}{2}$}(\alpha_{_{1}}+\alpha_{_{2}})-\mbox{$\frac{\lambda^{2}}{4}$}\]
\[\alpha_{_{3}}\alpha_{_{4}}-\delta_{_{3}}\delta_{_{4}}>\mbox{$\frac{\lambda}{2}$}(\alpha_{_{3}}+\alpha_{_{4}})-\mbox{$\frac{\lambda^{2}}{4}$},\]
hence
\begin{equation}
\label{fire415eqn11}
(\alpha_{_{1}}\alpha_{_{2}}-\delta_{_{1}}\delta_{_{2}})+(\alpha_{_{3}}\alpha_{_{4}}-\delta_{_{3}}\delta_{_{4}})>0.
\end{equation}
Since the matrix $D=(D_{_{ij}})$ is doubly stochastic
\[D_{_{11}}+D_{_{12}}+D_{_{21}}+D_{_{22}}=D_{_{33}}+D_{_{34}}+D_{_{43}}+D_{_{44}},\]
i.e.
\[\delta_{_{1}}^{2}+\delta_{_{2}}^{2}+2\alpha_{_{1}}\alpha_{_{2}}=\delta_{_{3}}^{2}+\delta_{_{4}}^{2}+2\alpha_{_{3}}\alpha_{_{4}}.\]
Equivalently
\begin{equation}
\label{fire415eqn12}
(\alpha_{_{1}}\alpha_{_{2}}-\delta_{_{1}}\delta_{_{2}})-(\alpha_{_{3}}\alpha_{_{4}}-\delta_{_{3}}\delta_{_{4}})=\mbox{$\frac{1}{2}$}\left((\delta_{_{3}}+\delta_{_{4}})^{2}-(\delta_{_{1}}+\delta_{_{2}})^{2}\right),
\end{equation}
so by adding (\ref{fire415eqn11}) and (\ref{fire415eqn12}),
\[(\alpha_{_{1}}\alpha_{_{2}}-\delta_{_{1}}\delta_{_{2}})>\mbox{$\frac{1}{4}$}\left((\delta_{_{3}}+\delta_{_{4}})^{2}-(\delta_{_{1}}+\delta_{_{2}})^{2}\right).\]
Moreover, by lemma \ref{fire4lem44} 3
\[(\delta_{_{1}}+\delta_{_{2}})^{2}-(\alpha_{_{3}}-\alpha_{_{4}})^{2}\geq{}(\delta_{_{1}}+\delta_{_{2}})^{2}-(\delta_{_{3}}-\delta_{_{4}})^{2}.\]
This proves (\ref{fire415eqn10}), because both factors on the right-hand side of (\ref{fire415eqn10}) are non-negative. Hence
\[\kappa(0)\geq{}0\;\mbox{ and }\;\kappa(\pi)\leq{}0,\]
so $\kappa(\theta)=0$ for some $\theta\in[0,\pi].$ This completes the proof.
\end{proof}
\begin{theorem}
\label{fire4thm412}Let $\delta_{_{1}} \geq{} \delta_{_{2}} \geq{} \delta_{_{3}} \geq{} \delta_{_{4}} > 0 $ and $\lambda$ be defined by $\alpha_{_{1}},\alpha_{_{2}},\alpha_{_{3}},\alpha_{_{4}}$ in the  usual way. Then the following are equivalent
\begin{enumerate}
\item{}There exists $n\in{\Bbb N}$ and unitaries $v_{_{ij}}\in{}M_{_{n}}({\Bbb C})$ such that  $\left(D_{_{ij}}v_{_{ij}}\right)_{i,j=1}^{4}$ is unitary in $M_{_{4n}}({\Bbb C}).$
\item{}There exists unitaries $v_{_{ij}} \in{}M_{_{2}}({\Bbb C})$ such that  $\left(D_{_{ij}}v_{_{ij}}\right)_{i,j=1}^{4}$  is unitary $M_{_{8}}({\Bbb C}).$
\item{}$\delta_{_{1}} - \delta_{_{2}} - \delta_{_{3}} - \delta_{_{4}} \leq{}0.$
\end{enumerate}
\end{theorem}
\begin{proof}2 $\Rightarrow{}$1 is trivial, and lemma \ref{fire4lem45} shows 1 $\Rightarrow{}$3. Moreover 3 $\Rightarrow{}$ 2 follows from proposition \ref{fire4prop15}, by using that the quaternions ${\Bbb H}=\{a_{_{1}}+ia_{_{2}}+ja_{_{3}}+ka_{_{4}}\;|\;a_{_{1}},a_{_{2}},a_{_{3}},a_{_{4}}\in{\Bbb R}\}$ can be identified with the real subalgebra of $M_{_{2}}({\Bbb C})$ given by
\[\left\{\left.\left(
\begin{array}{cc}a_{_{1}}+ia_{_{2}}&a_{_{2}}+ia_{_{3}}\\
-a_{_{2}}+ia_{_{3}}&a_{_{1}}-ia_{_{2}}\end{array}\right)\right|\;a_{_{1}},a_{_{2}},a_{_{3}},a_{_{4}}\in{\Bbb R}\right\},\]
and that the matrix representation of $a\in{\Bbb H}$ is a unitary matrix if and only if $|a|=1,$ where $|a|^{2}=a_{_{1}}^{2}+a_{_{2}}^{2}+a_{_{3}}^{2}+a_{_{4}}^{2}.$
\end{proof}\newpage{}

\chapter{Commuting Squares Based on 3--stars}
\label{trestarchap}\label{chapter2}

\setcounter{equation}{0}

\section{A Necessary and Sufficient Condition}
\setcounter{equation}{0}
Let $\Gamma$ be the 3--star with ray length $k,$ $l,$ and $m.$ We label the vertices of $\Gamma$ the following way

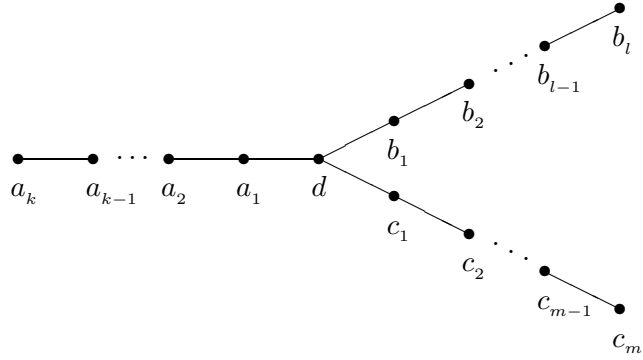
\begin{center}
\setlength{\unitlength}{1cm}
\begin{picture}(8,6)
\put(0,3){\circle*{0.15}}
\put(1,3){\circle*{0.15}}
\put(2,3){\circle*{0.15}}
\put(3,3){\circle*{0.15}}
\put(4,3){\circle*{0.15}}
\put(5,3.5){\circle*{0.15}}
\put(5,2.5){\circle*{0.15}}
\put(6,4){\circle*{0.15}}
\put(6,2){\circle*{0.15}}
\put(7,4.5){\circle*{0.15}}
\put(7,1.5){\circle*{0.15}}
\put(8,5){\circle*{0.15}}
\put(8,1){\circle*{0.15}}
\put(0,3){\line(1,0){1}}
\put(2,3){\line(1,0){2}}
\put(4,3){\line(2,1){2}}
\put(4,3){\line(2,-1){2}}
\put(7,4.5){\line(2,1){1}}
\put(7,1.5){\line(2,-1){1}}
\put(-0.1,2.5){$a_{_{k}}$}
\put(0.9,2.5){$a_{_{k-1}}$}
\put(1.9,2.5){$a_{_{2}}$}
\put(2.9,2.5){$a_{_{1}}$}
\put(3.9,2.5){$d$}
\put(4.9,3){$b_{_{1}}$}
\put(4.9,2){$c_{_{1}}$}
\put(5.9,3.5){$b_{_{2}}$}
\put(5.9,1.5){$c_{_{2}}$}
\put(6.9,4){$b_{_{l-1}}$}
\put(6.9,1){$c_{_{m-1}}$}
\put(7.9,4.5){$b_{_{l}}$}
\put(7.9,0.5){$c_{_{m}}$}
\put(1.3,3){$\ldots$}
\put(6.3,4.15){$.$}
\put(6.5,4.25){$.$}
\put(6.7,4.35){$.$}
\put(6.3,1.85){$.$}
\put(6.5,1.75){$.$}
\put(6.7,1.65){$.$}
\end{picture}
\end{center}
Let $\Gamma_{_{0}}$ denote the vertices of $\Gamma.$ For $p,q\in\Gamma_{_{0}},$ $\mbox{dist}(p,q)$ denotes the minimal number of edges in a path from $p$ to $q.$ $\Gamma$ is a bi--partite graph
\[\Gamma_{_{0}}=\Gamma_{_{\mbox{\tiny even}}}\cup\Gamma_{_{\mbox{\tiny odd}}}\;\;\;\;\;\mbox{ (disjoint)},\]
where  
\[\Gamma_{_{\mbox{\tiny even}}}=\left\{p\in{}\Gamma_{_{0}}|\mbox{dist}(p,d) \mbox{ is even }\right\}\]
\[\Gamma_{_{\mbox{\tiny odd }}}=\left\{p\in{}\Gamma_{_{0}}|\mbox{dist}(p,d) \mbox{ is odd }\right\},\]
then the adjacency matrix, $\Delta_{_{\Gamma}},$ of $\Gamma$ is of the form
\[\Delta_{_{\Gamma}}=\left(\begin{array}{cc}0&G\\G^{t}&0\end{array}\right),\]
where the rows (resp. the columns) of $G$ are labeled by $\Gamma_{_{\mbox{\tiny even}}}$ (resp. $\Gamma_{_{\mbox{\tiny odd }}}$). The entries of the matrix 
\[\Delta_{_{\Gamma}}^{2}-I=\left(\begin{array}{cc}GG^{t}-I&0\\0&G^{t}G-I\end{array}\right)\]
are easily found. The off--diagonal entries are
\[\left(\Delta_{_{\Gamma}}^{2}-I\right)_{pq}=\left\{
\begin{array}{cl}1&\mbox{if dist$(p,q)=2$}\\[0.2cm]
0&\mbox{otherwise}\end{array}\right.\;\;\;\;\;p\neq{}q,\]
and the diagonal entries are
\[\left(\Delta_{_{\Gamma}}^{2}-I\right)_{pp}=\left\{
\begin{array}{cl}2&\mbox{if $p=d$}\\[0.2cm]
0&\mbox{if $p=a_{_{k}},$ $p=b_{_{l}}$ or $p=c_{_{m}}.$}\\[0.2cm]
1&\mbox{otherwise}\end{array}\right.\]
In particular $\left(\Delta_{_{\Gamma}}^{2}-I\right)_{pq}\neq{}0$ implies that dist$(p,q)=0$ or dist$(p,q)=2.$

\noindent{}We will consider commuting squares of the form
\begin{equation}
\label{tre1eqn1}
\begin{array}{lcl}
      B & \subset_{G^{t}G-I} & D \\
       \cup_{G} &\,&\cup_{G}\\
      A & \subset_{GG^{t}-I} & C \end{array} 
\end{equation}
Note that such a commuting square is symmetric in the sense of \ref{symmetricdefn}, because
\[G^{t}(GG^{t}-I)=(G^{t}G-I)G^{t}.\]
The rest of this section will be used to prove:
\begin{theorem}
\label{tre1thm11}
Let $\xi:\Gamma_{_{0}}\rightarrow{\Bbb R}_{+}$ be the Perron--Frobenius eigenvector for $\Delta_{_{\Gamma}},$ with normalization $\xi(d)=1.$ Set 
\begin{equation}
\label{tre1*2}\alpha_{_{1}}=\xi(a_{_{1}}),\;\;\;\alpha_{_{2}}=\xi(b_{_{1}})\;\;\mbox{ and }\;\;  \alpha_{_{3}}=\xi(c_{_{1}}),
\end{equation} 
and let $\lambda=\alpha_{_{1}}+\alpha_{_{2}}+\alpha_{_{3}}$ be the Perron--Frobenius eigenvalue. Then $\Gamma$ admits a commuting square of the form (\ref{tre1eqn1}) if and only if there exists 9 vectors $(e_{_{ij}})_{_{i,j=1}}^{3}$ in ${\Bbb C}^{2},$ for which
\begin{description}
\item[(a)]\hspace{1.5cm}$\|e_{_{ii}}\|^{^{2}}=(\lambda-\alpha_{_{i}})(\lambda\alpha_{_{i}}-1).$
\item[(b)]\hspace{1.5cm}$\|e_{_{ij}}\|^{^{2}}=\alpha_{_{i}}+\alpha_{_{j}}-\lambda\alpha_{_{i}}\alpha_{_{j}},$ $i\neq{}j.$
\item[(c)]\hspace{1.5cm}$\sum_{j}\eijeij{i}{j}=\alpha_{_{i}}I_{_{2}},$ $i=1,2,3.$
\item[(d)]\hspace{1.5cm}$\sum_{i}\eijeij{i}{j}=\alpha_{_{j}}I_{_{2}},$ $j=1,2,3.$
\end{description}
where $I_{_{2}}$ is the unit $2\times{}2-$matrix.
\end{theorem}
{\bf{}Proof of the necessity of (a), (b), (c) and (d):}

\noindent{}First we consider the case $k,l,m\geq{}2.$ Using the eigenvector equation $\Delta_{_{\Gamma}}\xi=\lambda\xi$ at the vertices $a_{_{1}},$ $b_{_{1}}$ and $c_{_{1}}$ one gets
\begin{equation}
\label{tre1*3}
\xi(a_{_{2}})=\lambda\alpha_{_{1}}-1,\;\;\;\xi(b_{_{2}})=\lambda\alpha_{_{2}}-1,\;\;\;\xi(c_{_{2}})=\lambda\alpha_{_{3}}-1.
\end{equation}
Particularly $\lambda\alpha_{_{i}}-1>0,$ $i=1,2,3.$

Assume that there is a commuting square of the form (\ref{tre1eqn1}). The sets of minimal central projections $c(A)$ and $c(C)$ of $A$ and $C$ are labeled by the elements of $\Gamma_{_{\mbox{\tiny even}}},$ and the sets of minimal central projections $c(B)$ and $c(D)$ of $B$ and $D$ are labeled by the elements of $\Gamma_{_{\mbox{\tiny odd}}}.$ The bi--unitary condition (\ref{biunitcond}) gives the existence of\[u=\bigoplus_{(p,s)}u^{(p,s)},\;\;\;\;\;v=\bigoplus_{(q,r)}v^{(q,r)},\]
where $u^{(p,s)}$ and $v^{(q,r)}$ are block matrices
\[u^{(p,s)}=\left(u^{(p,s)}_{_{qr}}\right)_{q,r}\;\;\;\mbox{ and }\;\;\;v^{(q,r)}=\left(v^{(q,r)}_{_{ps}}\right)_{p,s}.\]
The only sets of indices $(p,q,r,s)$ which occur are those, that can be completed to a cycle of length 4, $p-r-s-q-p,$ via the given inclusion pattern. Moreover each block $u^{(p,s)}_{_{qr}}$ is a scalar unless $q=r=d,$ in which case $u^{(p,s)}_{_{qr}}$ is a $1\times{}2-$matrix, because the edge $dd$ is the only multiple edge coming from $GG^{t}-I$ or $G^{t}G-I,$ and the multiplicity is 2. Finally
\begin{equation}
\label{tre1*4}
v^{(q,r)}_{_{ps}}=w(p,q,r,s)\left(u^{(p,s)}_{_{qr}}\right)^{t},
\end{equation}
where
\begin{equation}
\label{tre1*5}
w(p,q,r,s)=\sqrt{\frac{\xi(p)\xi(s)}{\xi(q)\xi(r)}}.
\end{equation}
The possible 4--cycles $p-r-s-q-p$ are determined by the two ``vertical'' edges $pq$ and $rs$ from $\Gamma,$ but not all pairs $(pr,qs)$ will occur. We concentrate on the 6 edges
\[da_{_{1}},\;\;db_{_{1}},\;\;dc_{_{1}},\;\;a_{_{2}}a_{_{1}},\;\;b_{_{2}}b_{_{1}},\;\;c_{_{2}}c_{_{1}},\]
in $\Gamma,$ which connect vertices at a distance of at most 2 from the central vertex $d.$

\noindent{}Figure 1 shows which combinations $(pr,qs)$ occur, and the number of dots indicates the size of the corresponding block $u^{(p,s)}_{_{qr}}.$

\begin{center}
\setlength{\unitlength}{1cm}
\begin{picture}(7,7)
\multiput(1,0)(1,0){7}{\line(0,1){6}}
\multiput(1,0)(0,1){7}{\line(1,0){6}}
\put(0.3,0.4){$c_{_{2}}c_{_{1}}$}
\put(0.3,1.4){$b_{_{2}}b_{_{1}}$}
\put(0.3,2.4){$a_{_{2}}a_{_{1}}$}
\put(0.3,3.4){$dc_{_{1}}$}
\put(0.3,4.4){$db_{_{1}}$}
\put(0.3,5.4){$da_{_{1}}$}

\put(6.2,6.2){$c_{_{2}}c_{_{1}}$}
\put(5.2,6.2){$b_{_{2}}b_{_{1}}$}
\put(4.2,6.2){$a_{_{2}}a_{_{1}}$}
\put(3.3,6.2){$dc_{_{1}}$}
\put(2.3,6.2){$db_{_{1}}$}
\put(1.3,6.2){$da_{_{1}}$}
\put(1,6){\line(-1,1){1}}
\put(1,6){\line(-1,0){1}}
\put(1,6){\line(0,1){1}}
\put(0,6.2){$pq$}
\put(0.4,6.7){$rs$}
\put(1.3,5.5){\circle*{0.1}}\put(1.7,5.5){\circle*{0.1}}
\put(2.3,5.5){\circle*{0.1}}\put(2.7,5.5){\circle*{0.1}}
\put(3.3,5.5){\circle*{0.1}}\put(3.7,5.5){\circle*{0.1}}
\put(1.3,4.5){\circle*{0.1}}\put(1.7,4.5){\circle*{0.1}}
\put(2.3,4.5){\circle*{0.1}}\put(2.7,4.5){\circle*{0.1}}
\put(3.3,4.5){\circle*{0.1}}\put(3.7,4.5){\circle*{0.1}}
\put(1.3,3.5){\circle*{0.1}}\put(1.7,3.5){\circle*{0.1}}
\put(2.3,3.5){\circle*{0.1}}\put(2.7,3.5){\circle*{0.1}}
\put(3.3,3.5){\circle*{0.1}}\put(3.7,3.5){\circle*{0.1}}

\put(1.5,0.5){\circle*{0.1}}
\put(1.5,1.5){\circle*{0.1}}
\put(1.5,2.5){\circle*{0.1}}
\put(2.5,0.5){\circle*{0.1}}
\put(2.5,1.5){\circle*{0.1}}
\put(2.5,2.5){\circle*{0.1}}
\put(3.5,0.5){\circle*{0.1}}
\put(3.5,1.5){\circle*{0.1}}
\put(3.5,2.5){\circle*{0.1}}
\put(4.5,3.5){\circle*{0.1}}
\put(4.5,4.5){\circle*{0.1}}
\put(4.5,5.5){\circle*{0.1}}
\put(5.5,3.5){\circle*{0.1}}
\put(5.5,4.5){\circle*{0.1}}
\put(5.5,5.5){\circle*{0.1}}
\put(6.5,3.5){\circle*{0.1}}
\put(6.5,4.5){\circle*{0.1}}
\put(6.5,5.5){\circle*{0.1}}
\put(4.4,2.4){?}
\put(5.4,1.4){?}
\put(6.4,0.4){?}
\end{picture}

{\bf Figure 1.} Blocks of $u$ indexed by edges close to $d$\\[1cm]
\end{center}

The 3 question marks each represent a dot if $k,l,m\geq{}3,$ but if the $a-$ray, the $b-$ray or the $c-$ray has length 2, the corresponding question mark represents an empty box, i.e. the pair $(pq,rs)$ does not correspond to a 4--cycle.

\noindent{}The corresponding entries, $u^{(p,s)}_{_{qr}},$ of $u$ are given by
\[f_{_{ij}}\in{\Bbb C}^{2}\;\;\mbox{(row vectors) and }\;\;\;\sigma_{_{ij}},\tau_{_{ij}}\in{\Bbb C}\]
as in figure 2.

\begin{center}
\setlength{\unitlength}{1cm}
\begin{picture}(7,7)
\multiput(1,0)(1,0){7}{\line(0,1){6}}
\multiput(1,0)(0,1){7}{\line(1,0){6}}
\put(0.3,0.4){$c_{_{2}}c_{_{1}}$}
\put(0.3,1.4){$b_{_{2}}b_{_{1}}$}
\put(0.3,2.4){$a_{_{2}}a_{_{1}}$}
\put(0.3,3.4){$dc_{_{1}}$}
\put(0.3,4.4){$db_{_{1}}$}
\put(0.3,5.4){$da_{_{1}}$}
\put(6.2,6.2){$c_{_{2}}c_{_{1}}$}
\put(5.2,6.2){$b_{_{2}}b_{_{1}}$}
\put(4.2,6.2){$a_{_{2}}a_{_{1}}$}
\put(3.3,6.2){$dc_{_{1}}$}
\put(2.3,6.2){$db_{_{1}}$}
\put(1.3,6.2){$da_{_{1}}$}
\put(1,6){\line(-1,1){1}}
\put(1,6){\line(-1,0){1}}
\put(1,6){\line(0,1){1}}
\put(0,6.2){$pq$}
\put(0.4,6.7){$rs$}
\put(1.35,5.4){$f_{_{11}}$}
\put(2.35,5.4){$f_{_{12}}$}
\put(3.35,5.4){$f_{_{13}}$}
\put(1.35,4.4){$f_{_{21}}$}
\put(2.35,4.4){$f_{_{22}}$}
\put(3.35,4.4){$f_{_{23}}$}
\put(1.35,3.4){$f_{_{31}}$}
\put(2.35,3.4){$f_{_{32}}$}
\put(3.35,3.4){$f_{_{33}}$}
\put(1.35,0.4){$\tau_{_{31}}$}
\put(1.35,1.4){$\tau_{_{21}}$}
\put(1.35,2.4){$\tau_{_{11}}$}
\put(2.35,0.4){$\tau_{_{32}}$}
\put(2.35,1.4){$\tau_{_{22}}$}
\put(2.35,2.4){$\tau_{_{12}}$}
\put(3.35,0.4){$\tau_{_{33}}$}
\put(3.35,1.4){$\tau_{_{23}}$}
\put(3.35,2.4){$\tau_{_{13}}$}
\put(4.35,3.4){$\sigma_{_{31}}$}
\put(4.35,4.4){$\sigma_{_{21}}$}
\put(4.35,5.4){$\sigma_{_{11}}$}
\put(5.35,3.4){$\sigma_{_{32}}$}
\put(5.35,4.4){$\sigma_{_{22}}$}
\put(5.35,5.4){$\sigma_{_{12}}$}
\put(6.35,3.4){$\sigma_{_{33}}$}
\put(6.35,4.4){$\sigma_{_{23}}$}
\put(6.35,5.4){$\sigma_{_{13}}$}
\put(4.4,2.4){?}
\put(5.4,1.4){?}
\put(6.4,0.4){?}
\end{picture}

{\bf Figure 2.} Entries of $u$\\[1cm]
\end{center}
Let $f_{_{ij}}'\in{\Bbb C}^{2}$ (column vectors) and $\sigma_{_{ij}}',\tau_{_{ij}}'\in{\Bbb C}$ be the corresponding blocks, $v_{_{p,s}}^{(qr)},$ of $v.$ By the transformation formulas (\ref{tre1*4}) and (\ref{tre1*5}), together with $\xi(d)=1$ and  (\ref{tre1*2}), (\ref{tre1*3}) we get
\begin{equation}
\label{tre1*6}
f_{_{ij}}'=\sqrt{\frac{\alpha_{i}}{\alpha_{j}}}f_{_{ij}}^{t}
\end{equation}
\begin{equation}
\label{tre1*7}
\sigma_{_{ij}}'=\sqrt{\frac{\alpha_{j}}{\alpha_{i}(\lambda\alpha_{j}-1)}}\sigma_{_{ij}}
\end{equation}
\begin{equation}
\label{tre1*8}
\tau_{_{ij}}'=\sqrt{\frac{\alpha_{j}(\lambda\alpha_{i}-1)}{\alpha_{i}}}\tau_{_{ij}}
\end{equation}
The unitary summands $u^{(p,s)}$ of $u$ for $(p,s)=(d,a_{_{1}})$ (resp. $(d,b_{_{1}})$ and $(d,c_{_{1}})$) are
\begin{equation}
\label{tre1*9}
u^{(d,a_{1})}=
\left(\begin{array}{cc}
f_{_{11}}&\sigma_{_{11}}\\
f_{_{21}}&\sigma_{_{21}}\\
f_{_{31}}&\sigma_{_{31}}
\end{array}\right),\;\;
u^{(d,b_{1})}=
\left(\begin{array}{cc}
f_{_{12}}&\sigma_{_{12}}\\
f_{_{22}}&\sigma_{_{22}}\\
f_{_{32}}&\sigma_{_{32}}
\end{array}\right),\;\;
u^{(d,c_{1})}=
\left(\begin{array}{cc}
f_{_{13}}&\sigma_{_{13}}\\
f_{_{23}}&\sigma_{_{23}}\\
f_{_{33}}&\sigma_{_{33}}
\end{array}\right).
\end{equation}
Similarly the unitary summands, $v^{(q,r)},$ of $v$ for $(q,r)=(a_{_{1}},d)$ (resp. $(b_{_{1}},d)$ and $(c_{_{1}},d)$) are

\[v^{(a_{1},d)}=\left(\begin{array}{ccc}
f_{_{11}}' & f_{_{12}}' &f_{_{13}}'\\
\tau_{_{11}}'&\tau_{_{12}}'&\tau_{_{13}}'
\end{array}\right)\]
\begin{equation}
\label{tre1*10}v^{(b_{1},d)}=\left(\begin{array}{ccc}
f_{_{21}}' & f_{_{22}}' &f_{_{23}}'\\
\tau_{_{21}}'&\tau_{_{22}}'&\tau_{_{23}}'
\end{array}\right)\end{equation}
\[v^{(c_{1},d)}=\left(\begin{array}{ccc}
f_{_{31}}' & f_{_{32}}' &f_{_{33}}'\\
\tau_{_{31}}'&\tau_{_{32}}'&\tau_{_{33}}'
\end{array}\right).\]
Since dist$(a_{_{2}},b_{_{1}})=3,$ there is only one pair $(q,r),$ such that $a_{_{2}}-r-b_{_{1}}-q-a_{_{2}}$ is a 4--cycle with the given inclusion matrices, namely $r=d$ and $q=a_{_{1}}.$ Hence
\[u^{(a_{2},b_{1})}=u_{_{a_{1}d}}^{(a_{2},b_{1})}=\tau_{_{12}}\]
is a $1\times{}1-$summand of $u.$ The same argument shows that $\tau_{_{ij}}$ is a $1\times{}1$ unitary of $u$ whenever $i\neq{}j,$ and $\sigma_{_{ij}}'$ is a $1\times{}1$ unitary of $v$ when $i\neq{}j.$ In particular 
\begin{equation}
\label{tre1*11}
|\sigma_{_{ij}}'|=|\tau_{_{ij}}|=1,\;\;\;i\neq{}j.
\end{equation}
Hence by (\ref{tre1*7}) and (\ref{tre1*8}) 
\[|\sigma_{_{ij}}|=\sqrt{\frac{\alpha_{i}(\lambda\alpha_{j}-1)}{\alpha_{j}}},\;\;\;i\neq{}j\]
\[|\tau_{_{ij}}'|=\sqrt{\frac{\alpha_{j}(\lambda\alpha_{i}-1)}{\alpha_{i}}},\;\;\;i\neq{}j.\]
Using that the 3 matrices in (\ref{tre1*9}) are unitary, one has
\[\|f_{_{ij}}\|^{2}=1-|\sigma_{_{ij}}|^{2}=\frac{1}{\alpha_{j}}(\alpha_{_{i}}+\alpha_{_{j}}-\lambda\alpha_{_{i}}\alpha_{_{j}}),\;\;\;i\neq{}j,\]
and
\[\begin{array}{lclr}
\|f_{_{ii}}\|^{2} & = &
1-|\sigma_{_{ii}}|^{2}&\\[0.3cm]
& = &|\sigma_{_{ji}}|^{2}+|\sigma_{_{ki}}|^{2}&\;\;\;\;\;\;\;i\neq{}j\neq{}k\neq{}i\\[0.3cm]
&=&\frac{1}{\alpha_{i}}(\alpha_{_{j}}+\alpha_{_{k}})(\lambda\alpha_{_{i}}-1)&\\[0.3cm]
&=&\frac{1}{\alpha_{i}}(\lambda-\alpha_{_{i}})(\lambda\alpha_{_{i}}-1).&
\end{array}\]
Set $e_{_{ij}}=\sqrt{\alpha_{_{j}}}f_{_{ij}}\in{\Bbb C}^{2}.$ Then
\[\|e_{_{ij}}\|^{2}=\left\{
\begin{array}{cl}
(\lambda-\alpha_{_{i}})(\lambda\alpha_{_{i}}-1),& i=j\\[0.3cm]
\alpha_{_{i}}+\alpha_{_{j}}-\lambda\alpha_{_{i}}\alpha_{_{j}},&i\neq{}j
\end{array}\right.\]
Moreover, using that the rows of a unitary matrix are orthogonal vectors, one gets from (\ref{tre1*9})
\[\sum_{i}e_{_{ij}}\otimes{}\overline{e}_{_{ij}}=\alpha_{_{j}}\sum_{i}f_{_{ij}}\otimes{}\overline{f}_{_{ij}}=\alpha_{_{j}}\left(\begin{array}{cc}1&0\\0&1\end{array}\right),\;\;\;\;j=1,2,3.\]
Furthermore, since
\[e_{_{ij}}=\frac{1}{\sqrt{\alpha_{i}}}\left(f_{_{ij}}'\right)^{t},\]
and since the two first rows of each of the matrices (\ref{tre1*10}) are orthonormal, we have
\[\sum_{j}e_{_{ij}}\otimes{}\overline{e}_{_{ij}}=\alpha_{_{i}}\sum_{j}\left(f_{_{ij}}'\right)^{t}\otimes{}\left(\overline{f}_{_{ij}}'\right)^{t}=\alpha_{_{i}}\left(\begin{array}{cc}1&0\\0&1\end{array}\right),\;\;\;\;i=1,2,3.\]
This proves the necessity of (a), (b), (c) and (d) in the case $k,l,m\geq{}2.$

\noindent{}If one of the rays, say the $c-$ray, has length 1, but $k,l\geq{}2,$ then
\[\alpha_{_{3}}=\frac{\xi(d)}{\lambda}=\frac{1}{\lambda},\;\;\;\;\mbox{ i.e. } \lambda\alpha_{_{3}}-1=0.\]
Figure 2 reduces to
\begin{center}
\setlength{\unitlength}{1cm}
\begin{picture}(6,6)
\multiput(1,0)(1,0){6}{\line(0,1){5}}
\multiput(1,0)(0,1){6}{\line(1,0){5}}
\put(0.3,0.4){$b_{_{2}}b_{_{1}}$}
\put(0.3,1.4){$a_{_{2}}a_{_{1}}$}
\put(0.3,2.4){$dc_{_{1}}$}
\put(0.3,3.4){$db_{_{1}}$}
\put(0.3,4.4){$da_{_{1}}$}

\put(5.2,5.2){$b_{_{2}}b_{_{1}}$}
\put(4.2,5.2){$a_{_{2}}a_{_{1}}$}
\put(3.3,5.2){$dc_{_{1}}$}
\put(2.3,5.2){$db_{_{1}}$}
\put(1.3,5.2){$da_{_{1}}$}

\put(1,5){\line(-1,1){1}}
\put(1,5){\line(-1,0){1}}
\put(1,5){\line(0,1){1}}
\put(0,5.2){$pq$}
\put(0.4,5.7){$rs$}
\put(1.35,4.4){$f_{_{11}}$}
\put(2.35,4.4){$f_{_{12}}$}
\put(3.35,4.4){$f_{_{13}}$}
\put(1.35,3.4){$f_{_{21}}$}
\put(2.35,3.4){$f_{_{22}}$}
\put(3.35,3.4){$f_{_{23}}$}
\put(1.35,2.4){$f_{_{31}}$}
\put(2.35,2.4){$f_{_{32}}$}

\put(1.35,0.4){$\tau_{_{21}}$}
\put(1.35,1.4){$\tau_{_{11}}$}
\put(2.35,0.4){$\tau_{_{22}}$}
\put(2.35,1.4){$\tau_{_{12}}$}
\put(3.35,0.4){$\tau_{_{23}}$}
\put(3.35,1.4){$\tau_{_{13}}$}

\put(4.35,2.4){$\sigma_{_{31}}$}
\put(4.35,3.4){$\sigma_{_{21}}$}
\put(4.35,4.4){$\sigma_{_{11}}$}
\put(5.35,2.4){$\sigma_{_{32}}$}
\put(5.35,3.4){$\sigma_{_{22}}$}
\put(5.35,4.4){$\sigma_{_{12}}$}

\put(4.4,1.4){?}
\put(5.4,0.4){?}
\end{picture}

{\bf Figure 3.}\\[1cm]
\end{center}
Thus $u^{(d,c_{1})}$ and $v^{(c_{1},d)}$ in (\ref{tre1*9}) and  (\ref{tre1*10}) reduce to 
$2\times{}2$ matrices
\[u^{(d,c_{1})}=\left(\begin{array}{c}f_{_{13}}\\f_{_{23}}\end{array}\right),\;\;\;\;v^{(c_{1},d)}=\left(f_{_{31}}'\;,\;f_{_{32}}'\right).\]
However this case involves  the same computations as in the case $k,l,m\geq{}2,$ if we set $f_{_{33}}=0,$ $ \sigma_{_{13}}=\sigma_{_{23}}=0$ and $\tau_{_{31}}=\tau_{_{32}}=0,$ because $\lambda\alpha_{_{3}}-1=0.$ This proves the necessity part of theorem \ref{tre1thm11}.

\noindent{}{\bf{}Proof of sufficiency of (a), (b), (c) and (d):}

\noindent{}Assume that $(e_{_{ij}})_{i,j=1,2,3}$ are 9 vectors in ${\Bbb C}^{2}$ satisfying (a), (b), (c) and (d). Since any orthonormal set in ${\Bbb C}^{3}$ can be completed to an orthonormal basis, (d) implies that there exists $\sigma_{_{ij}}\in{\Bbb C},$ such that
\begin{equation}
\label{tre1*12}
u_{_{i}}=\left(\begin{array}{cc}
\frac{1}{\sqrt{\alpha_{i}}}e_{_{1i}} & \sigma_{_{1i}}\\
\frac{1}{\sqrt{\alpha_{i}}}e_{_{2i}} & \sigma_{_{2i}}\\
\frac{1}{\sqrt{\alpha_{i}}}e_{_{3i}} & \sigma_{_{3i}}
\end{array}\right)\;\;\;\;\;i=1,2,3,
\end{equation}
are 3 unitary $3\times{}3$ matrices, and by (c) there exists $\rho_{_{ij}}\in{\Bbb C},$ such that
\begin{equation}
\label{tre1*13}
v_{_{i}}=\left(\begin{array}{ccc}
\frac{1}{\sqrt{\alpha_{i}}}e_{_{i1}}^{t}&\frac{1}{\sqrt{\alpha_{i}}}e_{_{i2}}^{t}&\frac{1}{\sqrt{\alpha_{i}}}e_{_{i3}}^{t}\\
\rho_{_{i1}}&\rho_{_{i2}}&\rho_{_{i3}}
\end{array}\right)\;\;\;\;\;i=1,2,3,
\end{equation}
are 3 unitary $3\times{}3$ matrices. By multiplying the last column in each $u_{_{i}}$ and the last row in each $v_{_{i}}$ by suitable scalars of modulus 1, we can obtain
\[\sigma_{_{ii}}\geq{}0\;\;\mbox{ and }\;\;\rho_{_{ii}}\geq{}0,\;\;i=i,2,3.\]
From (a) and (b) and the unitarity of $u_{_{i}}$ and $v_{_{i}}$ we get
\begin{equation}
\label{tre1*14}
|\sigma_{_{ij}}|^{2}=1-\mbox{$\frac{1}{\alpha_{j}}$}\|e_{_{ij}}\|^{2}=\left\{
\begin{array}{ll}
\frac{\lambda}{\alpha_{j}}\left(\alpha_{j}^{2}-\lambda\alpha_{j}+1\right), & i=j\\[0.3cm]
\frac{\alpha_{i}}{\alpha_{j}}\left(\lambda\alpha_{j}-1\right),& i\neq{}j
\end{array}\right.
\end{equation}
and
\begin{equation}
\label{tre1*15}
|\rho_{_{ij}}|^{2}=1-\mbox{$\frac{1}{\alpha_{i}}$}\|e_{_{ij}}\|^{2}=\left\{
\begin{array}{ll}
\frac{\lambda}{\alpha_{i}}\left(\alpha_{i}^{2}-\lambda\alpha_{i}+1\right), & i=j\\[0.3cm]
\frac{\alpha_{j}}{\alpha_{i}}\left(\lambda\alpha_{i}-1\right),& i\neq{}j.
\end{array}\right.
\end{equation}
Let $\delta_{_{1}},$ $\delta_{_{2}}$ and $\delta_{_{3}}$ be the coordinates of the  Perron--Frobenius vector, $\xi,$ at the endpoints of the rays,
\[\delta_{_{1}}=\xi(a_{_{k}}),\;\;\;\delta_{_{2}}=\xi(b_{_{l}}),\;\;\;\delta_{_{3}}=\xi(c_{_{m}}).\]
By chapter \ref{chapter1} (\ref{deltadef}), $\delta_{_{i}}=\sqrt{\alpha_{_{i}}^{2}-\lambda\alpha_{_{i}}+1},$ $i=1,2,3,$ so for $i=j$ the above formulas reduce to
\[|\sigma_{_{ii}}|^{2}=|\rho_{_{ii}}|^{2}=\frac{\lambda\delta_{i}^{2}}{\alpha_{i}},\;\;\;i=1,2,3,\]
and since $\sigma_{_{ii}}\geq{}0$ and $\rho_{_{ii}}\geq{}0$ we have
\begin{equation}
\label{tre1*16}
\sigma_{_{ii}}=\rho_{_{ii}}=\sqrt{\frac{\lambda\delta_{i}^{2}}{\alpha_{i}}},\;\;\;i=1,2,3.
\end{equation}
By (\ref{tre1*14}) and (\ref{tre1*15}) there exists scalars $\sigma_{_{ij}}'$ and $\rho_{_{ij}}'$ with $|\sigma_{_{ij}}'|=|\rho_{_{ij}}'|=1,$ $i\neq{}j,$ such that
\begin{equation}
\label{tre1*17}
\sigma_{_{ij}}=\sqrt{\frac{\alpha_{i}}{\alpha_{j}}\left(\lambda\alpha_{j}-1\right)}\sigma_{_{ij}}',\;\;\;i\neq{}j,
\end{equation}
\begin{equation}
\label{tre1*18}
\rho_{_{ij}}=\sqrt{\frac{\alpha_{j}}{\alpha_{i}}\left(\lambda\alpha_{i}-1\right)}\rho_{_{ij}}',\;\;\;i\neq{}j.
\end{equation}
Let $R_{_{n}}(\lambda)$ be the polynomials defined by
\[R_{_{0}}(\lambda)=1,\;\;R_{_{1}}(\lambda)=\lambda,\;\;R_{_{n+1}}(\lambda)=\lambda{}R_{_{n}}(\lambda)-R_{_{n-1}}(\lambda),\;\;n\geq{}1,\]
as in chapter \ref{chapter1} \ref{firestar2def21}. Then by (\ref{PFdefnstar}) the Perron--Frobenius vector, $\xi,$ on $\Gamma$ is given by
\[\begin{array}{lccl}
\xi(d)&=&1&\\[0.2cm]
\xi(a_{_{j}})&=&\frac{R_{k-j}(\lambda)}{R_{k}(\lambda)},\;\;\;&j=1,\ldots,k\\[0.2cm]
\xi(b_{_{j}})&=&\frac{R_{l-j}(\lambda)}{R_{l}(\lambda)},&j=1,\ldots,l\\[0.2cm]
\xi(c_{_{j}})&=&\frac{R_{m-j}(\lambda)}{R_{m}(\lambda)},&j=1,\ldots,m.
\end{array}\]
Set 
\[\alpha_{_{1j}}=\xi(a_{_{j}}),\;\;\;\alpha_{_{2j}}=\xi(b_{_{j}}),\;\;\;\alpha_{_{3j}}=\xi(c_{_{j}}).\]
In particular
\[\begin{array}{lcllcllcl}
\alpha_{_{1}}&=&\alpha_{_{11}},\;\;\;\alpha_{_{2}}&=&\alpha_{_{21}},\;\;\;\alpha_{_{3}}&=&\alpha_{_{31}},\\[0.2cm]
\delta_{_{1}}&=&\alpha_{_{1k}},\;\;\;\delta_{_{2}}&=&\alpha_{_{2l}},\;\;\;\delta_{_{3}}&=&\alpha_{_{3m}}.
\end{array}\]
We proceed to construct $u=\bigoplus{}u^{(p,s)}$ and $v=\bigoplus{}v^{(q,r)}$ satisfying the bi--unitary condition, i.e. $u$ and $v$ are unitaries and
\[v^{(q,r)}_{_{ps}}=\sqrt{\frac{\xi(p)\xi(s)}{\xi(q)\xi(r)}}u^{(p,s)}_{_{qr}}\]
for all possible 4--cycles $p-r-s-q-p.$ Since the matrix elements 
\[\left(GG^{t}-I\right)_{p,r},\;\;\;p,r\in\Gamma_{_{\mbox{\tiny even}}}\]
and
\[\left(G^{t}G-I\right)_{q,s},\;\;\;q,s\in\Gamma_{_{\mbox{\tiny odd}}}\]
vanish unless dist$(p,r)\leq{}2$ and dist$(q,s)\leq{}2,$ the possible pairs of edges $(pr,qs)$ which define 4--cycles must either be on the same ray of $\Gamma,$ or they must connect vertices of $\Gamma$ with distance at most 2 from the central vertex $d.$ In the latter case we define the entries of $u$ and $v$ by figure 4 and figure 5 below. For the moment we will assume that $k,l,m\geq{}3.$\begin{center}
\setlength{\unitlength}{2cm}
\begin{picture}(7,7)
\multiput(1,0)(1,0){7}{\line(0,1){6}}
\multiput(1,0)(0,1){7}{\line(1,0){6}}
\put(0.6,0.45){$c_{_{2}}c_{_{1}}$}
\put(0.6,1.45){$b_{_{2}}b_{_{1}}$}
\put(0.6,2.45){$a_{_{2}}a_{_{1}}$}
\put(0.7,3.45){$dc_{_{1}}$}
\put(0.7,4.45){$db_{_{1}}$}
\put(0.7,5.45){$da_{_{1}}$}
\put(6.35,6.15){$c_{_{2}}c_{_{1}}$}
\put(5.35,6.15){$b_{_{2}}b_{_{1}}$}
\put(4.35,6.15){$a_{_{2}}a_{_{1}}$}
\put(3.45,6.15){$dc_{_{1}}$}
\put(2.45,6.15){$db_{_{1}}$}
\put(1.45,6.15){$da_{_{1}}$}
\put(1,6){\line(-1,1){0.5}}
\put(1,6){\line(-1,0){0.5}}
\put(1,6){\line(0,1){0.5}}
\put(0.5,6.1){$pq$}
\put(0.75,6.3){$rs$}
\put(1.2,5.45){$\frac{1}{\sqrt{\alpha_{1}}}e_{_{11}}$}
\put(2.2,5.45){$\frac{1}{\sqrt{\alpha_{2}}}e_{_{12}}$}
\put(3.2,5.45){$\frac{1}{\sqrt{\alpha_{3}}}e_{_{13}}$}
\put(1.2,4.45){$\frac{1}{\sqrt{\alpha_{1}}}e_{_{21}}$}
\put(2.2,4.45){$\frac{1}{\sqrt{\alpha_{2}}}e_{_{22}}$}
\put(3.2,4.45){$\frac{1}{\sqrt{\alpha_{3}}}e_{_{23}}$}
\put(1.2,3.45){$\frac{1}{\sqrt{\alpha_{1}}}e_{_{31}}$}
\put(2.2,3.45){$\frac{1}{\sqrt{\alpha_{2}}}e_{_{32}}$}
\put(3.2,3.45){$\frac{1}{\sqrt{\alpha_{3}}}e_{_{33}}$}
\put(1.45,0.45){$\rho_{_{31}}'$}
\put(1.45,1.45){$\rho_{_{21}}'$}
\put(1.2,2.45){$\sqrt{\frac{\lambda\delta_{1}^{2}}{\alpha_{1}\alpha_{12}}}$}
\put(2.45,0.45){$\rho_{_{32}}'$}
\put(2.2,1.45){$\sqrt{\frac{\lambda\delta_{2}^{2}}{\alpha_{2}\alpha_{22}}}$}
\put(2.45,2.45){$\rho_{_{12}}'$}
\put(3.2,0.45){$\sqrt{\frac{\lambda\delta_{3}^{2}}{\alpha_{3}\alpha_{32}}}$}
\put(3.45,1.45){$\rho_{_{23}}'$}
\put(3.45,2.45){$\rho_{_{13}}'$}
\put(4.45,3.45){$\sigma_{_{31}}$}
\put(4.45,4.45){$\sigma_{_{21}}$}
\put(4.25,5.45){$\sqrt{\frac{\lambda\delta_{1}^{2}}{\alpha_{1}}}$}
\put(5.45,3.45){$\sigma_{_{32}}$}
\put(5.25,4.45){$\sqrt{\frac{\lambda\delta_{2}^{2}}{\alpha_{2}}}$}
\put(5.45,5.45){$\sigma_{_{12}}$}
\put(6.25,3.45){$\sqrt{\frac{\lambda\delta_{3}^{2}}{\alpha_{3}}}$}
\put(6.45,4.45){$\sigma_{_{23}}$}
\put(6.45,5.45){$\sigma_{_{13}}$}
\put(4.1,2.45){$-\sqrt{\frac{\alpha_{13}}{\alpha_{1}\alpha_{12}}}$}
\put(5.1,1.45){$-\sqrt{\frac{\alpha_{23}}{\alpha_{2}\alpha_{22}}}$}
\put(6.1,0.45){$-\sqrt{\frac{\alpha_{33}}{\alpha_{3}\alpha_{32}}}$}
\end{picture}

{\bf Figure 4.} Entries of $u$ near the central vertex.
\end{center}

\begin{center}
\setlength{\unitlength}{2cm}
\begin{picture}(7,7)
\multiput(1,0)(1,0){7}{\line(0,1){6}}
\multiput(1,0)(0,1){7}{\line(1,0){6}}
\put(0.6,0.45){$c_{_{2}}c_{_{1}}$}
\put(0.6,1.45){$b_{_{2}}b_{_{1}}$}
\put(0.6,2.45){$a_{_{2}}a_{_{1}}$}
\put(0.7,3.45){$dc_{_{1}}$}
\put(0.7,4.45){$db_{_{1}}$}
\put(0.7,5.45){$da_{_{1}}$}
\put(6.35,6.15){$c_{_{2}}c_{_{1}}$}
\put(5.35,6.15){$b_{_{2}}b_{_{1}}$}
\put(4.35,6.15){$a_{_{2}}a_{_{1}}$}
\put(3.45,6.15){$dc_{_{1}}$}
\put(2.45,6.15){$db_{_{1}}$}
\put(1.45,6.15){$da_{_{1}}$}
\put(1,6){\line(-1,1){0.5}}
\put(1,6){\line(-1,0){0.5}}
\put(1,6){\line(0,1){0.5}}
\put(0.5,6.1){$pq$}
\put(0.75,6.3){$rs$}
\put(1.2,5.45){$\frac{1}{\sqrt{\alpha_{1}}}e_{_{11}}^{t}$}
\put(2.2,5.45){$\frac{1}{\sqrt{\alpha_{1}}}e_{_{12}}^{t}$}
\put(3.2,5.45){$\frac{1}{\sqrt{\alpha_{1}}}e_{_{13}}^{t}$}
\put(1.2,4.45){$\frac{1}{\sqrt{\alpha_{2}}}e_{_{21}}^{t}$}
\put(2.2,4.45){$\frac{1}{\sqrt{\alpha_{2}}}e_{_{22}}^{t}$}
\put(3.2,4.45){$\frac{1}{\sqrt{\alpha_{2}}}e_{_{23}}^{t}$}
\put(1.2,3.45){$\frac{1}{\sqrt{\alpha_{3}}}e_{_{31}}^{t}$}
\put(2.2,3.45){$\frac{1}{\sqrt{\alpha_{3}}}e_{_{32}}^{t}$}
\put(3.2,3.45){$\frac{1}{\sqrt{\alpha_{3}}}e_{_{33}}^{t}$}
\put(1.45,0.45){$\rho_{_{31}}$}
\put(1.45,1.45){$\rho_{_{21}}$}
\put(1.25,2.45){$\sqrt{\frac{\lambda\delta_{1}^{2}}{\alpha_{1}}}$}
\put(2.45,0.45){$\rho_{_{32}}$}
\put(2.25,1.45){$\sqrt{\frac{\lambda\delta_{2}^{2}}{\alpha_{2}}}$}
\put(2.45,2.45){$\rho_{_{12}}$}
\put(3.25,0.45){$\sqrt{\frac{\lambda\delta_{3}^{2}}{\alpha_{3}}}$}
\put(3.45,1.45){$\rho_{_{23}}$}
\put(3.45,2.45){$\rho_{_{13}}$}
\put(4.45,3.45){$\sigma_{_{31}}'$}
\put(4.45,4.45){$\sigma_{_{21}}'$}
\put(4.2,5.45){$\sqrt{\frac{\lambda\delta_{1}^{2}}{\alpha_{1}\alpha_{12}}}$}
\put(5.45,3.45){$\sigma_{_{32}}'$}
\put(5.2,4.45){$\sqrt{\frac{\lambda\delta_{2}^{2}}{\alpha_{2}\alpha_{22}}}$}
\put(5.45,5.45){$\sigma_{_{12}}'$}
\put(6.2,3.45){$\sqrt{\frac{\lambda\delta_{3}^{2}}{\alpha_{3}\alpha_{32}}}$}
\put(6.45,4.45){$\sigma_{_{23}}'$}
\put(6.45,5.45){$\sigma_{_{13}}'$}
\put(4.1,2.45){$-\sqrt{\frac{\alpha_{13}}{\alpha_{1}\alpha_{12}}}$}
\put(5.1,1.45){$-\sqrt{\frac{\alpha_{23}}{\alpha_{2}\alpha_{22}}}$}
\put(6.1,0.45){$-\sqrt{\frac{\alpha_{33}}{\alpha_{3}\alpha_{32}}}$}
\end{picture}

{\bf Figure 5.} Entries of $v$ near the central vertex.
\end{center}
The empty entries in $u$ resp. $v$ correspond to pairs $(pq,rs)$ which cannot be completed to a 4--cycle.

\noindent{}The entries of $u$ and $v$ for pairs of edges $(pq,rs)$ on the $a-$ray, are defined by the figures 6, 7, 8 and 9 below

\begin{center}
\setlength{\unitlength}{1cm}
\begin{picture}(13,13)
\multiput(1,0)(2,0){7}{\line(0,1){12}}
\multiput(1,0)(0,2){7}{\line(1,0){12}}
\put(1,12){\line(0,1){1}}
\put(1,12){\line(-1,0){1}}
\put(1,12){\line(-1,1){1}}
\put(0,12.2){$pq$}
\put(0.6,12.7){$rs$}

\put(1.7,12.2){$da_{_{1}}$}
\put(3.6,12.2){$a_{_{2}}a_{_{1}}$}
\put(5.6,12.2){$a_{_{2}}a_{_{3}}$}
\put(7.6,12.2){$a_{_{4}}a_{_{3}}$}
\put(9.6,12.2){$a_{_{4}}a_{_{5}}$}
\put(11.6,12.2){$a_{_{6}}a_{_{5}}$}

\put(0,10.9){$da_{_{1}}$}
\put(0,8.9){$a_{_{2}}a_{_{1}}$}
\put(0,6.9){$a_{_{2}}a_{_{3}}$}
\put(0,4.9){$a_{_{4}}a_{_{3}}$}
\put(0,2.9){$a_{_{4}}a_{_{5}}$}
\put(0,0.9){$a_{_{6}}a_{_{5}}$}

\put(1.5,10.9){$\frac{1}{\sqrt{\alpha_{_{1}}}}e_{_{11}}$}
\put(3.5,10.9){$\sqrt{\frac{\lambda\delta_{1}^{^{2}}}{\alpha_{_{1}}}}$}
\put(5.9,10.9){1}

\put(1.3,8.9){\sqdelal{1}{12}}
\put(3.1,8.9){$-$\sqalikl{13}{1}{12}}
\put(5.3,8.9){\sqdelal{12}{13}}
\put(7.3,8.9){\sqalijkl{11}{14}{12}{13}}

\put(1.3,6.9){\sqalikl{13}{1}{12}}
\put(3.3,6.9){\sqdelal{1}{12}}
\put(5.1,6.9){$-\!\sqalijkl{11}{14}{12}{13}$}
\put(7.3,6.9){\sqdelal{12}{13}}
\put(9.9,6.9){1}

\put(3.9,4.9){1}
\put(5.3,4.9){\sqdelal{13}{14}}
\put(7.1,4.9){$-\!\sqalijkl{12}{15}{13}{14}$}
\put(9.3,4.9){\sqdelal{14}{15}}
\put(11.3,4.9){\sqalijkl{13}{16}{14}{15}}

\put(5.3,2.9){\sqalijkl{12}{15}{13}{14}}
\put(7.3,2.9){\sqdelal{13}{14}}
\put(9.1,2.9){$-\!\sqalijkl{13}{16}{14}{15}$}
\put(11.3,2.9){\sqdelal{14}{15}}

\put(7.9,0.9){1}
\put(9.3,0.9){\sqdelal{15}{16}}
\put(11.1,0.9){$-\!\sqalijkl{14}{17}{15}{16}$}
\linethickness{2pt}
\put(1.05,12){\line(1,0){0.1}}
\put(1.25,12){\line(1,0){0.1}}
\put(1.45,12){\line(1,0){0.1}}
\put(1.65,12){\line(1,0){0.1}}
\put(1.85,12){\line(1,0){0.1}}
\put(2.05,12){\line(1,0){0.1}}
\put(2.25,12){\line(1,0){0.1}}
\put(2.45,12){\line(1,0){0.1}}
\put(2.65,12){\line(1,0){0.1}}
\put(2.85,12){\line(1,0){0.1}}
\put(3.05,12){\line(1,0){0.1}}
\put(3.25,12){\line(1,0){0.1}}
\put(3.45,12){\line(1,0){0.1}}
\put(3.65,12){\line(1,0){0.1}}
\put(3.85,12){\line(1,0){0.1}}
\put(4.05,12){\line(1,0){0.1}}
\put(4.25,12){\line(1,0){0.1}}
\put(4.45,12){\line(1,0){0.1}}
\put(4.65,12){\line(1,0){0.1}}
\put(4.85,12){\line(1,0){0.1}}

\put(9.05,0){\line(1,0){0.1}}
\put(9.25,0){\line(1,0){0.1}}
\put(9.45,0){\line(1,0){0.1}}
\put(9.65,0){\line(1,0){0.1}}
\put(9.85,0){\line(1,0){0.1}}
\put(10.05,0){\line(1,0){0.1}}
\put(10.25,0){\line(1,0){0.1}}
\put(10.45,0){\line(1,0){0.1}}
\put(10.65,0){\line(1,0){0.1}}
\put(10.85,0){\line(1,0){0.1}}
\put(11.05,0){\line(1,0){0.1}}
\put(11.25,0){\line(1,0){0.1}}
\put(11.45,0){\line(1,0){0.1}}
\put(11.65,0){\line(1,0){0.1}}
\put(11.85,0){\line(1,0){0.1}}
\put(12.05,0){\line(1,0){0.1}}
\put(12.25,0){\line(1,0){0.1}}
\put(12.45,0){\line(1,0){0.1}}
\put(12.65,0){\line(1,0){0.1}}
\put(12.85,0){\line(1,0){0.1}}

\put(5,12){\line(1,0){2}}
\put(1,10){\line(1,0){8}}
\put(9,8){\line(1,0){2}}
\put(1,6){\line(1,0){12}}
\put(3,4){\line(1,0){2}}
\put(5,2){\line(1,0){}}
\put(5,2){\line(1,0){8}}
\put(7,0){\line(1,0){2}}

\put(1,12){\line(0,-1){6}}
\put(3,6){\line(0,-1){2}}
\put(5,12){\line(0,-1){10}}
\put(7,12){\line(0,-1){2}}
\put(7,2){\line(0,-1){2}}
\put(9,10){\line(0,-1){10}}
\put(11,8){\line(0,-1){2}}
\put(13,6){\line(0,-1){6}}
\end{picture}

{\bf Figure 6.} Entries of $u;$ $a-$ray near the central vertex
\end{center}

\begin{center}
\setlength{\unitlength}{1cm}
\begin{picture}(13,13)
\multiput(1,0)(2,0){7}{\line(0,1){12}}
\multiput(1,0)(0,2){7}{\line(1,0){12}}

\put(1,12){\line(0,1){1}}
\put(1,12){\line(-1,0){1}}
\put(1,12){\line(-1,1){1}}
\put(0,12.2){$pq$}
\put(0.6,12.7){$rs$}

\put(1.7,12.2){$da_{_{1}}$}
\put(3.6,12.2){$a_{_{2}}a_{_{1}}$}
\put(5.6,12.2){$a_{_{2}}a_{_{3}}$}
\put(7.6,12.2){$a_{_{4}}a_{_{3}}$}
\put(9.6,12.2){$a_{_{4}}a_{_{5}}$}
\put(11.6,12.2){$a_{_{6}}a_{_{5}}$}

\put(0,10.9){$da_{_{1}}$}
\put(0,8.9){$a_{_{2}}a_{_{1}}$}
\put(0,6.9){$a_{_{2}}a_{_{3}}$}
\put(0,4.9){$a_{_{4}}a_{_{3}}$}
\put(0,2.9){$a_{_{4}}a_{_{5}}$}
\put(0,0.9){$a_{_{6}}a_{_{5}}$}

\put(1.5,10.9){$\frac{1}{\sqrt{\alpha_{_{1}}}}e_{_{11}}^{t}$}
\put(3.3,10.9){\sqdelal{1}{12}}
\put(5.3,10.9){\sqalikl{13}{1}{12}}

\put(1.5,8.9){$\sqrt{\frac{\lambda\delta_{1}^{^{2}}}{\alpha_{_{1}}}}$}
\put(3.1,8.9){$-\!\sqalikl{13}{1}{12}$}
\put(5.3,8.9){\sqdelal{1}{12}}
\put(7.9,8.9){1}

\put(1.9,6.9){1}
\put(3.3,6.9){\sqdelal{12}{13}}
\put(5.1,6.9){$-\!\sqalijkl{11}{14}{12}{13}$}
\put(7.3,6.9){\sqdelal{13}{14}}
\put(9.3,6.9){\sqalijkl{12}{15}{13}{14}}

\put(3.3,4.9){\sqalijkl{11}{14}{12}{13}}
\put(5.3,4.9){\sqdelal{12}{13}}
\put(7.1,4.9){$-\!\sqalijkl{12}{15}{13}{14}$}
\put(9.3,4.9){\sqdelal{13}{14}}
\put(11.9,4.9){1}

\put(5.9,2.9){1}
\put(7.3,2.9){\sqdelal{14}{15}}
\put(9.1,2.9){$-\!\sqalijkl{13}{16}{14}{15}$}
\put(11.3,2.9){\sqdelal{15}{16}}

\put(7.3,0.9){\sqalijkl{13}{16}{14}{15}}
\put(9.3,0.9){\sqdelal{14}{15}}
\put(11.1,0.9){$-\!\sqalijkl{14}{17}{15}{16}$}

\linethickness{2pt}
\put(1,11.95){\line(0,-1){0.1}}
\put(1,11.75){\line(0,-1){0.1}}
\put(1,11.55){\line(0,-1){0.1}}
\put(1,11.35){\line(0,-1){0.1}}
\put(1,11.15){\line(0,-1){0.1}}
\put(1,10.95){\line(0,-1){0.1}}
\put(1,10.75){\line(0,-1){0.1}}
\put(1,10.55){\line(0,-1){0.1}}
\put(1,10.35){\line(0,-1){0.1}}
\put(1,10.15){\line(0,-1){0.1}}
\put(1,9.95){\line(0,-1){0.1}}
\put(1,9.75){\line(0,-1){0.1}}
\put(1,9.55){\line(0,-1){0.1}}
\put(1,9.35){\line(0,-1){0.1}}
\put(1,9.15){\line(0,-1){0.1}}
\put(1,8.95){\line(0,-1){0.1}}
\put(1,8.75){\line(0,-1){0.1}}
\put(1,8.55){\line(0,-1){0.1}}
\put(1,8.35){\line(0,-1){0.1}}
\put(1,8.15){\line(0,-1){0.1}}

\put(13,3.95){\line(0,-1){0.1}}
\put(13,3.75){\line(0,-1){0.1}}
\put(13,3.55){\line(0,-1){0.1}}
\put(13,3.35){\line(0,-1){0.1}}
\put(13,3.15){\line(0,-1){0.1}}
\put(13,2.95){\line(0,-1){0.1}}
\put(13,2.75){\line(0,-1){0.1}}
\put(13,2.55){\line(0,-1){0.1}}
\put(13,2.35){\line(0,-1){0.1}}
\put(13,2.15){\line(0,-1){0.1}}
\put(13,1.95){\line(0,-1){0.1}}
\put(13,1.75){\line(0,-1){0.1}}
\put(13,1.55){\line(0,-1){0.1}}
\put(13,1.35){\line(0,-1){0.1}}
\put(13,1.15){\line(0,-1){0.1}}
\put(13,0.95){\line(0,-1){0.1}}
\put(13,0.75){\line(0,-1){0.1}}
\put(13,0.55){\line(0,-1){0.1}}
\put(13,0.35){\line(0,-1){0.1}}
\put(13,0.15){\line(0,-1){0.1}}

\put(1,12){\line(1,0){6}}
\put(7,10){\line(1,0){2}}
\put(1,8){\line(1,0){10}}
\put(1,6){\line(1,0){2}}
\put(11,6){\line(1,0){2}}
\put(3,4){\line(1,0){10}}
\put(5,2){\line(1,0){2}}
\put(7,0){\line(1,0){6}}

\put(1,8){\line(0,-1){2}}
\put(3,12){\line(0,-1){8}}
\put(5,4){\line(0,-1){2}}
\put(7,12){\line(0,-1){12}}
\put(9,10){\line(0,-1){2}}
\put(11,8){\line(0,-1){8}}
\put(13,6){\line(0,-1){2}}
\end{picture}

{\bf Figure 7.} Entries of $v;$ $a-$ray near the central vertex
\end{center}

\begin{center}
\setlength{\unitlength}{1cm}
\begin{picture}(11,11)
\multiput(1,0)(2,0){6}{\line(0,1){10}}
\multiput(1,0)(0,2){6}{\line(1,0){10}}

\put(1,10){\line(0,1){1}}
\put(1,10){\line(-1,0){1}}
\put(1,10){\line(-1,1){1}}
\put(0,10.2){$pq$}
\put(0.6,10.7){$rs$}

\put(1.35,10.4){$a_{_{k-4}}a_{_{k-5}}$}
\put(3.35,10.4){$a_{_{k-4}}a_{_{k-3}}$}
\put(5.35,10.4){$a_{_{k-2}}a_{_{k-3}}$}
\put(7.35,10.4){$a_{_{k-2}}a_{_{k-1}}$}
\put(9.4,10.4){$a_{_{k}}a_{_{k-1}}$}

\put(-0.5,8.9){$a_{_{k-4}}a_{_{k-5}}$}
\put(-0.5,6.9){$a_{_{k-4}}a_{_{k-3}}$}
\put(-0.5,4.9){$a_{_{k-2}}a_{_{k-3}}$}
\put(-0.5,2.9){$a_{_{k-2}}a_{_{k-1}}$}
\put(-0.4,0.9){$a_{_{k}}a_{_{k-1}}$}

\put(1.1,8.9){$-\!\sqdelijkl{16}{13}{15}{14}$}
\put(3.3,8.9){\sqlamdel{14}{13}}
\put(5.3,8.9){\sqdelijkl{15}{12}{14}{13}}

\put(1.3,6.9){\sqlamdel{15}{14}}
\put(3.1,6.9){$-\!\sqdelijkl{15}{12}{14}{13}$}
\put(5.3,6.9){\sqlamdel{14}{13}}
\put(7.9,6.9){1}

\put(1.9,4.9){1}
\put(3.3,4.9){\sqlamdel{13}{12}}
\put(5.1,4.9){$-\!\sqdelijkl{14}{11}{13}{12}$}
\put(7.3,4.9){\sqlamdel{12}{11}}
\put(9.3,4.9){\sqdelijkl{13}{10}{12}{11}}

\put(3.3,2.9){\sqdelijkl{14}{11}{13}{12}}
\put(5.3,2.9){\sqlamdel{13}{12}}
\put(7.1,2.9){$-\!\sqdelijkl{13}{10}{12}{11}$}
\put(9.3,2.9){\sqlamdel{12}{11}}

\put(5.9,0.9){1}
\put(7.9,0.9){1}
\linethickness{2pt}

\put(1,9.95){\line(0,-1){0.1}}
\put(1,9.75){\line(0,-1){0.1}}
\put(1,9.55){\line(0,-1){0.1}}
\put(1,9.35){\line(0,-1){0.1}}
\put(1,9.15){\line(0,-1){0.1}}
\put(1,8.95){\line(0,-1){0.1}}
\put(1,8.75){\line(0,-1){0.1}}
\put(1,8.55){\line(0,-1){0.1}}
\put(1,8.35){\line(0,-1){0.1}}
\put(1,8.15){\line(0,-1){0.1}}
\put(1,7.95){\line(0,-1){0.1}}
\put(1,7.75){\line(0,-1){0.1}}
\put(1,7.55){\line(0,-1){0.1}}
\put(1,7.35){\line(0,-1){0.1}}
\put(1,7.15){\line(0,-1){0.1}}
\put(1,6.95){\line(0,-1){0.1}}
\put(1,6.75){\line(0,-1){0.1}}
\put(1,6.55){\line(0,-1){0.1}}
\put(1,6.35){\line(0,-1){0.1}}
\put(1,6.15){\line(0,-1){0.1}}

\put(1,10){\line(1,0){6}}
\put(7,8){\line(1,0){2}}
\put(1,6){\line(1,0){10}}
\put(1,4){\line(1,0){2}}
\put(3,2){\line(1,0){8}}
\put(5,0){\line(1,0){4}}

\put(1,6){\line(0,-1){2}}
\put(3,10){\line(0,-1){8}}
\put(5,2){\line(0,-1){2}}
\put(7,10){\line(0,-1){10}}
\put(9,8){\line(0,-1){2}}
\put(9,2){\line(0,-1){2}}
\put(11,6){\line(0,-1){4}}
\end{picture}

{\bf Figure 8.} 
\end{center}
Entries of $u$: $a-$ray near the end vertex for k {\em even,} $\delta_{_{1i}}=\alpha_{_{1,k-i}},$ $i=0,1,\ldots,k-1.$ The corresponding entries of $v$ are given by the mirror image of figure 8 in the main diagonal.\newpage{}

\begin{center}
\setlength{\unitlength}{1cm}
\begin{picture}(11,11)
\multiput(1,0)(2,0){6}{\line(0,1){10}}
\multiput(1,0)(0,2){6}{\line(1,0){10}}

\put(1,10){\line(0,1){1}}
\put(1,10){\line(-1,0){1}}
\put(1,10){\line(-1,1){1}}
\put(0,10.2){$pq$}
\put(0.6,10.7){$rs$}

\put(1.35,10.4){$a_{_{k-5}}a_{_{k-4}}$}
\put(3.35,10.4){$a_{_{k-3}}a_{_{k-4}}$}
\put(5.35,10.4){$a_{_{k-3}}a_{_{k-2}}$}
\put(7.35,10.4){$a_{_{k-1}}a_{_{k-2}}$}
\put(9.4,10.4){$a_{_{k-1}}a_{_{k}}$}

\put(-0.5,8.9){$a_{_{k-5}}a_{_{k-4}}$}
\put(-0.5,6.9){$a_{_{k-3}}a_{_{k-4}}$}
\put(-0.5,4.9){$a_{_{k-3}}a_{_{k-2}}$}
\put(-0.5,2.9){$a_{_{k-1}}a_{_{k-2}}$}
\put(-0.4,0.9){$a_{_{k-1}}a_{_{k}}$}

\put(1.1,8.9){$-\!\sqdelijkl{16}{13}{15}{14}$}
\put(3.3,8.9){\sqlamdel{15}{14}}
\put(5.9,8.9){1}

\put(1.3,6.9){\sqlamdel{14}{13}}
\put(3.1,6.9){$-\!\sqdelijkl{15}{12}{14}{13}$}
\put(5.3,6.9){\sqlamdel{13}{12}}
\put(7.3,6.9){\sqdelijkl{14}{11}{13}{12}}

\put(1.3,4.9){\sqdelijkl{15}{12}{14}{13}}
\put(3.3,4.9){\sqlamdel{14}{13}}
\put(5.1,4.9){$-\!\sqdelijkl{14}{11}{13}{12}$}
\put(7.3,4.9){\sqlamdel{13}{12}}
\put(9.9,4.9){1}

\put(3.9,2.9){1}
\put(5.3,2.9){\sqlamdel{12}{11}}
\put(7.1,2.9){$-\!\sqdelijkl{13}{10}{12}{11}$}
\put(9.9,2.9){1}

\put(5.3,0.9){\sqdelijkl{13}{10}{12}{11}}
\put(7.3,0.9){\sqlamdel{12}{11}}

\linethickness{2pt}
\put(1.05,10){\line(1,0){0.1}}
\put(1.25,10){\line(1,0){0.1}}
\put(1.45,10){\line(1,0){0.1}}
\put(1.65,10){\line(1,0){0.1}}
\put(1.85,10){\line(1,0){0.1}}
\put(2.05,10){\line(1,0){0.1}}
\put(2.25,10){\line(1,0){0.1}}
\put(2.45,10){\line(1,0){0.1}}
\put(2.65,10){\line(1,0){0.1}}
\put(2.85,10){\line(1,0){0.1}}
\put(3.05,10){\line(1,0){0.1}}
\put(3.25,10){\line(1,0){0.1}}
\put(3.45,10){\line(1,0){0.1}}
\put(3.65,10){\line(1,0){0.1}}
\put(3.85,10){\line(1,0){0.1}}
\put(4.05,10){\line(1,0){0.1}}
\put(4.25,10){\line(1,0){0.1}}
\put(4.45,10){\line(1,0){0.1}}
\put(4.65,10){\line(1,0){0.1}}
\put(4.85,10){\line(1,0){0.1}}

\put(5,10){\line(1,0){2}}
\put(1,8){\line(1,0){8}}
\put(9,6){\line(1,0){2}}
\put(1,4){\line(1,0){10}}
\put(3,2){\line(1,0){2}}
\put(9,2){\line(1,0){2}}
\put(5,0){\line(1,0){4}}

\put(1,10){\line(0,-1){6}}
\put(3,4){\line(0,-1){2}}
\put(5,10){\line(0,-1){10}}
\put(7,10){\line(0,-1){2}}
\put(9,8){\line(0,-1){8}}
\put(11,6){\line(0,-1){4}}
\end{picture}

{\bf Figure 9.}
\end{center}
Entries of $u$: $a-$ray near the end vertex for k {\em odd,} $\delta_{_{1i}}=\alpha_{_{1,k-i}},$ $i=0,1,\ldots,k-1.$ The corresponding entries of $v$ are given by the mirror image of figure 9 in the main diagonal.\\[1cm]
The entries of $u$ and $v,$ for pairs of edges $(pq,rs)$ on the $b-$ray or on the $c-$ray, are defined the same way, with trivial changes of the indices.

\noindent{}It is easy to check that the entries of $u$ and $v$ satisfy the transformation formula (\ref{tre1*4}). The direct summands $u^{(p,s)}$ of $u$ are marked by the framing in  fig.\ 6, fig.\ 8 and fig.\ 9 for $(p,s)$ both on the $a-$ray (and similarly for the other rays). These are either $1\times{}1-$ or $2\times{}2-$matrices. To see that the $2\times{}2-$matrices are unitary, we have to check
\begin{equation}
\label{tre1*19}
\lambda\delta_{_{1}}^{2}=\alpha_{_{1,j+1}}\alpha_{_{1,j+2}}-\alpha_{_{1,j}}\alpha_{_{1,j+3}},\;\;\;j=1,2,\ldots,k-3.
\end{equation}
From  the proof of lemma \ref{firegraflem0} in chapter \ref{chapter1} we have the identity
\[R_{_{m+1}}(\lambda)^{2}-R_{_{m+2}}(\lambda)R_{_{m}}(\lambda)=1,\;\;\;m=0,1,2,\ldots\]
Using the recursion formula for $R_{_{m}}:$
\[R_{_{m+2}}(\lambda)R_{_{m+1}}(\lambda)-R_{_{m+3}}(\lambda)R_{_{m}}(\lambda)=
(\lambda{}R_{_{m+1}}(\lambda)-R_{_{m}}(\lambda))R_{_{m+1}}(\lambda)-(\lambda{}R_{_{m+2}}(\lambda)-R_{_{m+1}}(\lambda))R_{_{m}}(\lambda).\]
Hence 
\[R_{_{m+2}}(\lambda)R_{_{m+1}}(\lambda)-R_{_{m+3}}(\lambda)R_{_{m}}(\lambda)=\lambda,\;\;\;m=0,1,2,\ldots\]
If we use
\[\delta_{_{1}}=\frac{1}{R_{k}(\lambda)}\;\mbox{ and }\;\alpha_{_{1,j}}=\frac{R_{k-j}(\lambda)}{R_{k}(\lambda)},\]
(\ref{tre1*19}) follows by putting $m=k-j-3.$

\noindent{}The only summands, $u^{(p,s)},$ of $u,$ which are not visible in fig.\ 6, fig.\ 8, fig.\ 9 or the corresponding figures for the $b-$ and the $c-$rays, are three $3\times{}3-$matrices
\[u^{(d,a_{1})},\;\;\;\;\;u^{(d,b_{1})},\;\;\;\;\;u^{(d,c_{1})}\]
and six $1\times{}1-$matrices
\[u^{(a_{2},b_{1})},\;\;u^{(a_{2},c_{1})},\;\;u^{(b_{2},a_{1})},\;\;u^{(b_{2},c_{1})},\;\;u^{(c_{2},a_{1})},\;\;u^{(c_{2},b_{1})}.\]
By the previous analysis, the $3\times{}3-$matrices are the unitaries $u_{_{i}},$ $i=1,2,3$ in (\ref{tre1*12}), and the $1\times{}1-$matrices are given by the scalars $(\rho_{_{ij}})_{i\neq{}j}$ of modulus 1. Hence $u=\bigoplus{}u^{(p,s)}$ is unitary.

\noindent{}Similarly one gets, that the only summands, $v^{(q,r)},$ of $v,$ which are not visible in fig.\ 6 and the mirror images of fig.\ 8 and fig.\ 9 (or the corresponding diagrams for the $b-$ and $c-$ray) are the three $3\times{}3$ unitaries $v_{_{i}},$ $i=1,2,3$ in (\ref{tre1*13}) and six $1\times{}1-$matrices, given by the scalars $(\sigma_{_{ij}}')_{i\neq{}j}$ of modulus 1, so also $v=\bigoplus{}v^{(q,r)}$ is unitary.

\noindent{}This proves that (a), (b), (c) and (d) of theorem \ref{tre1thm11} implies the existence of a commuting square with the inclusions given by (\ref{tre1eqn1}), for $k,l,m\geq{}3.$

\noindent{}The cases where one or more of the numbers $k,l,m$ are less than 3 follows in the same way, by appropriate cancellations in fig.\ 4--9.

\noindent{}If $k=2,$ fig.\ 6 and fig.\ 7 degenerate to
\begin{center}
\setlength{\unitlength}{1cm}
\begin{picture}(12,5)
\multiput(1,0)(2,0){3}{\line(0,1){4}}
\multiput(1,0)(0,2){3}{\line(1,0){4}}
\put(1,4){\line(0,1){1}}
\put(1,4){\line(-1,0){1}}
\put(1,4){\line(-1,1){1}}
\put(0,4.2){$pq$}
\put(0.6,4.7){$rs$}
\put(1.7,4.2){$da_{_{1}}$}
\put(3.6,4.2){$a_{_{2}}a_{_{1}}$}
\put(0,0.9){$da_{_{1}}$}
\put(0,2.9){$a_{_{2}}a_{_{1}}$}
\put(1.5,2.9){$\frac{1}{\sqrt{\alpha_{_{1}}}}e_{_{11}}$}
\put(3.5,2.9){$\sqrt{\frac{\lambda\delta_{1}^{^{2}}}{\alpha_{_{1}}}}$}
\put(1.9,0.9){1}
\put(5.8,2){and}
\multiput(8,0)(2,0){3}{\line(0,1){4}}
\multiput(8,0)(0,2){3}{\line(1,0){4}}
\put(8,4){\line(0,1){1}}
\put(8,4){\line(-1,0){1}}
\put(8,4){\line(-1,1){1}}
\put(7,4.2){$pq$}
\put(7.6,4.7){$rs$}
\put(8.7,4.2){$da_{_{1}}$}
\put(10.6,4.2){$a_{_{2}}a_{_{1}}$}
\put(7,0.9){$da_{_{1}}$}
\put(7,2.9){$a_{_{2}}a_{_{1}}$}
\put(8.5,2.9){$\frac{1}{\sqrt{\alpha_{_{1}}}}e_{_{11}}$}
\put(8.5,0.9){$\sqrt{\frac{\lambda\delta_{1}^{^{2}}}{\alpha_{_{1}}}}$}
\put(10.9,2.9){1}
\end{picture}
\end{center}
and for $k=1$ all the elements of fig.\ 6 and fig.\ 7 disappear. The cases, where $k,l,m$ are not all greater than or equal to 3, also impose cancellations in fig.\ 4 and fig.\ 5, as described in the proof of the necessity of (a), (b), (c) and (d), but it is not hard to check, that the bi-unitary condition, for the pair $(u,v)$ is also satisfied in these cases. This completes the proof of theorem \ref{tre1thm11}.

\setcounter{equation}{0}

\section{Solution to the Vector Problem}
\setcounter{equation}{0}
In this section we will prove
\begin{theorem}\label{thmtre21}

\noindent{}\begin{itemize}
\item[(1)]\label{thmtre21cond}Let $\al{1}, \al{2},\al{3}>0$ and put $\lam{}= \al{1}+\al{2}+\al{3}.$ Then there exists 9 vectors, $(\eij{i}{j})_{i,j=1}^{3}$ in a two dimensional complex  Hilbert space, ${\cal H},$ such that
\[\begin{array}{clclc}
(a)&\|\eij{i}{i}\|^{^{2}} & = & (\lam{}-\al{i})(\lam{}\al{i}-1)& \\[0.3cm]
(b)&\|\eij{i}{j}\|^{^{2}} & = & \al{i}+\al{j}-\lam{}\al{i}\al{j},&i\neq{}j\\[0.3cm]
(c)&\sum_{j=1}^{3}\eijeij{i}{j}&=&\al{i}I_{_{{\cal H}}},& i=1,2,3\\[0.3cm]
(d)&\sum_{i=1}^{3}\eijeij{i}{j}&=&\al{j}I_{_{{\cal H}}},& j=1,2,3
\end{array}\]
if and only if
\item[(2)]\[\begin{array}{cl}
(i)&\lam{}\al{i}-1\geq0,\;\;\;\;i=1,2,3\\[0.2cm]
(ii)&\alpha_{_{i}}^{2}-\lam{}\al{i}+1\geq0,\;\;\;\;i=1,2,3\\[0.2cm]
(iii)&4\lam{}\al{1}\al{2}\al{3}-4(\al{1}\al{2}+\al{1}\al{3}+\al{2}\al{3})+3\geq0 
\end{array}\]
\end{itemize}
\end{theorem}
\begin{remark}
\label{remtre22}
{\rm Condition $(i)$ is clearly a necessary condition, because $\|\eij{i}{i}\|^{^{2}}=(\lam{}-\al{i})(\lam{}\al{i}-1)$ and $\lam{}>\al{i}.$

\noindent{}Condition $(ii)$ is also necessary, since $(b)$ implies that 
$\eijeij{i}{i}\leq\al{i}I_{_{\cal H}}.$ Therefore $\|\eij{i}{i}\|^{^{2}}\leq\al{i}.$ By $(a)$
\[\al{i}-\|\eij{i}{i}\|^{^{2}}=\lam{}(\alpha_{_{i}}^{2}-\lam{}\al{i}+1).\]
This shows $(ii)$.

The most difficult part of the proof, is to show that $(iii)$ is also a necessary condition.}
\end{remark}

\noindent{}{\bf Reformulation of condition (iii) in Theorem \ref{thmtre21}} 

\noindent{}Set
\[\gam{i}=\alpha_{_{i}}^{2}-\lam{}\al{i}+1,\;\;\;\;i=1,2,3\]
Note that\[\begin{array}{lcl}
\gam{1} & = & 1-\al{1}\al{2}-\al{1}\al{3}\\
\gam{2} & = & 1-\al{1}\al{2}-\al{2}\al{3}\\
\gam{3} & = & 1-\al{1}\al{3}-\al{2}\al{3}.
\end{array}\]
We compute
\[\begin{array}{cl}
 & 2\gam{1}\gam{2}+2\gam{1}\gam{3}+2\gam{2}\gam{3}-\gamma_{_{1}}^{2}-\gamma_{_{2}}^{2}-\gamma_{_{3}}^{2}\\[0.1cm]
= & \gam{1}(\gam{2}+\gam{3}-\gam{1}) + \gam{2}(\gam{1}+\gam{3}-\gam{2}) +
\gam{3}(\gam{1}+\gam{2}-\gam{3})\\[0.1cm]
=& (1-\al{1}\al{2}-\al{1}\al{3})(1-2\al{2}\al{3}) +
   (1-\al{1}\al{2}-\al{2}\al{3})(1-2\al{1}\al{3}) + \\[0.1cm]
 & (1-\al{1}\al{3}-\al{2}\al{3})(1-2\al{1}\al{2})\\[0.1cm]
= & 4(\al{1}+\al{2}+\al{3})\al{1}\al{2}\al{3} - 4(\al{1}\al{2}+\al{1}\al{3}+\al{2}\al{3})+3.
\end{array}\]
Hence condition $(iii)$ is equivalent to 
\[(iii')\;\;\;\;\;\;2\gam{1}\gam{2}+2\gam{1}\gam{3}+2\gam{2}\gam{3} -\gamma_{_{1}}^{2}-\gamma_{_{2}}^{2}-\gamma_{_{3}}^{2}\geq{}0.\]
If we assume $(ii)$, i.e. $\gam{i}\geq{}0,\;\;i=1,2,3,$ the inequality $(iii')$ can be further reduced. Indeed $(iii')$ is equivalent to
\[\gamma_{_{3}}^{2} -2(\gam{1}+\gam{2})\gam{3} +(\gam{1}-\gam{2})^{^{2}}\leq{}0.\]
The roots of this second order polynomial in \gam{3} are $\gam{1}+\gam{2}\pm{}2\sqrt{\gam{1}\gam{2}},$ so $(iii')$ is equivalent to
\[ \gam{1}+\gam{2}-2\sqrt{\gam{1}\gam{2}}\leq{}\gam{3}\leq{} \gam{1}+\gam{2}+2\sqrt{\gam{1}\gam{2}}\]
or
\[|\sqrt{\gam{1}}-\sqrt{\gam{2}}|\leq\sqrt{\gam{3}}\leq\sqrt{\gam{1}}+\sqrt{\gam{2}}.\]
Hence if we assume $(ii),$ and  define the numbers \del{1},\del{2} and \del{3} by
\[\del{i}=\sqrt{\alpha_{_{i}}^{2}-\lam{}\al{i}+1},\;\;\;\;i=1,2,3,\]
the inequality $(iii)$ is equivalent to the statement, that the numbers \del{1},\del{2} and \del{3} satisfy the triangle inequality,
\begin{equation}
\label{trekantulighed}
|\del{1}-\del{2}|\leq\del{3}\leq\del{1}+\del{2}.
\end{equation}
\begin{remark}
\label{remtre23}
{\rm If $\lam{}\leq{}2$ $(ii)$ is trivially true, because 
\[\alpha_{_{i}}^{2}-\lam{}\al{i}+1\geq{}\alpha_{_{i}}^{2}-2\al{i}+1 =(\al{i}-1)^{^{2}}.\]
Moreover for $\lam{}\leq{}2,$ $(i)\Rightarrow{}(iii).$ Indeed, if $(i)$ holds then
\begin{equation}
\label{lameq}
(\lam{}\al{1}-1)(\lam{}\al{2}-1)(\lam{}\al{3}-1)\geq{}0,
\end{equation}
but the left-hand side of this inequality is equal to
\[\lambda^{^{3}}\al{1}\al{2}\al{3}-\lambda^{^{2}}(\al{1}\al{2}+\al{1}\al{3}+\al{2}\al{3})+\lambda^{^{2}}-1,\]
so (\ref{lameq}) is equivalent to
\[\lambda\al{1}\al{2}\al{3}-(\al{1}\al{2}+\al{1}\al{3}+\al{2}\al{3})+1 -\mbox{$\frac{1}{\lambda^{^{2}}}$}\geq{}0.\]
If $\lam{}\leq{}2,$ then $1 -\frac{1}{\lambda^{^{2}}}\leq{}\frac{3}{4},$ so 
\[\lambda\al{1}\al{2}\al{3}-(\al{1}\al{2}+\al{1}\al{3}+\al{2}\al{3})+\mbox{ $\frac{3}{4}$}\geq{}0,\]
which proves $(iii)$.
}\end{remark}
The above remark shows, that for $\lam{}\leq{}2$ the conditions $(i),(ii)$ and $(iii)$ reduce to $(i).$ $(i)$ is clearly a necessary condition, because $\|\eij{i}{i}\|^{^{2}}=(\lam{}-\al{i})(\lam{}\al{i}-1),$ and $\lam{}-\al{i}>0.$

\noindent{}The following example shows, that for $\lam{}\leq{}2$ condition $(i)$ is also sufficient.
\begin{example}
\label{exatre24}{\rm Assume $\lam{}\leq{}2$ and write $\lam{}=2\cos\theta,$ where $0\leq\theta{}<\frac{\pi}{2}.$ Let ${\cal S}$ be the two dimensional subspace of ${\Bbb C}^{3}$ given by
\[{\cal S}=\left\{\left.(x_{_{1}},x_{_{2}},x_{_{3}})\in{\Bbb C}^{3}\right|\sqrt{\al{1}}x_{_{1}} +\sqrt{\al{2}}x_{_{2}} +\sqrt{\al{3}}x_{_{3}} =0\right\},\]
and consider the following 9 vectors in ${\Bbb C}^{3}$
\[\begin{array}{lll}
\eij{1}{1}=\sqrt{\frac{\lam{}\al{1}-1}{\lam{}}}\left(\begin{array}{c}\lam{}-\al{1}\\-\sqrt{\al{1}\al{2}}\\ -\sqrt{\al{1}\al{3}} \end{array}\right) &
\eij{1}{2}=\left(\begin{array}{c} \sqrt{\al{2}}(\al{1}-e^{i\theta})\\ \sqrt{\al{1}}(\al{2}-e^{-i\theta})\\ \sqrt{\al{1}\al{2}\al{3}} \end{array}\right) &
\eij{1}{3}=\left(\begin{array}{c}\sqrt{\al{3}}(\al{1}-e^{i\theta})\\ \sqrt{\al{1}\al{2}\al{3}}\\ \sqrt{\al{1}}(\al{3}-e^{-i\theta})\end{array}\right)\\[1cm]
\eij{2}{1}=\left(\begin{array}{c} \sqrt{\al{2}}(\al{1}-e^{-i\theta})\\ \sqrt{\al{1}}(\al{2}-e^{i\theta})\\ \sqrt{\al{1}\al{2}\al{3}} \end{array}\right)&
\eij{2}{2}=\sqrt{\frac{\lam{}\al{2}-1}{\lam{}}}\left(\begin{array}{c} -\sqrt{\al{1}\al{2}}\\ \lam{}-\al{2}\\ -\sqrt{\al{2}\al{3}} \end{array}\right) &
\eij{2}{3}=\left(\begin{array}{c} \sqrt{\al{1}\al{2}\al{3}}\\ \sqrt{\al{3}}(\al{2}-e^{i\theta})\\ \sqrt{\al{2}}(\al{3}-e^{-i\theta})\end{array}\right)\\[1cm]
\eij{3}{1}=\left(\begin{array}{c}\sqrt{\al{3}}(\al{1}-e^{-i\theta})\\ \sqrt{\al{1}\al{2}\al{3}}\\ \sqrt{\al{1}}(\al{3}-e^{i\theta})\end{array}\right) &
\eij{3}{2}=\left(\begin{array}{c} \sqrt{\al{1}\al{2}\al{3}}\\ \sqrt{\al{3}}(\al{2}-e^{-i\theta})\\ \sqrt{\al{2}}(\al{3}-e^{i\theta})\end{array}\right) &
\eij{3}{3}=\sqrt{\frac{\lam{}\al{3}-1}{\lam{}}}\left(\begin{array}{c} -\sqrt{\al{1}\al{3}}\\ -\sqrt{\al{2}\al{3}} \\ \lam{}-\al{3}\end{array}\right) 
\end{array}\]
One easily checks that $\eij{i}{j}\in{\cal S}$ and that $\|\eij{i}{j}\|^{^{2}}$ is given by the formulas $(a)$ and $(b).$ (Note that because of the symmetry of \eij{i}{j}, with respect to permutation of indices, it suffices to check $(a)$ and $(b)$ for \eij{1}{1} and \eij{1}{2}.)

\noindent{}To check $(c)$ and $(d)$ one has to identify $I_{_{\cal S}}$ with the orthogonal projection of ${\Bbb C}^{3}$ onto ${\cal S}.$ Since ${\cal S}^{\perp}$ is spanned by the unit vector 
\[\xi=\mbox{$\frac{1}{\sqrt{\lambda}}$}(\sqrt{\al{1}},\sqrt{\al{2}},\sqrt{\al{3}})\]
one has
\[I_{_{\cal S}}=I-\xi\otimes\overline{\xi}=\frac{1}{\lambda}\left(
\begin{array}{ccc}
\lam{}-\al{1}&-\sqrt{\al{1}\al{2}} & -\sqrt{\al{1}\al{3}}\\
-\sqrt{\al{1}\al{2}} & \lam{}-\al{2} &-\sqrt{\al{2}\al{3}}\\
-\sqrt{\al{1}\al{3}} &-\sqrt{\al{2}\al{3}} &\lam{}-\al{3}
\end{array}\right).\]
To check $(c),$ for $i=1,$ we show that
\[\eijeij{1}{2} +\eijeij{1}{3} =\al{1}I_{_{\cal S}}-\eijeij{1}{1}.\]
One easily gets that $\eijeij{1}{2}+\eijeij{1}{3}$ equals
\[\left(\begin{array}{ccc}
(\alpha_{_{1}}^{2}-\lam{}\al{1}+1)(\lam{}-\al{1})&
-(\alpha_{_{1}}^{2}-\lam{}\al{1}+1)\sqrt{\al{1}\al{2}}&
-(\alpha_{_{1}}^{2}-\lam{}\al{1}+1)\sqrt{\al{1}\al{3}}\\[0.3cm]
-(\alpha_{_{1}}^{2}-\lam{}\al{1}+1)\sqrt{\al{1}\al{2}}&
\al{1}(1-\al{1}\al{2})&
-\alpha_{_{1}}^{2}\sqrt{\al{2}\al{3}}\\[0.3cm]
-(\alpha_{_{1}}^{2}-\lam{}\al{1}+1)\sqrt{\al{1}\al{3}}&
-\alpha_{_{1}}^{2}\sqrt{\al{1}\al{3}}&
\al{1}(1-\al{1}\al{3})
\end{array}\right)\]
To compute $\al{1}I_{_{\cal S}}-\eijeij{1}{1},$ it is convenient to introduce the vector 
\[e_{_{11}}^{\perp}=\sqrt{\lam{}\al{1}-1}
\left(\begin{array}{c}0\\[0.2cm] \sqrt{\al{3}}\\[0.2cm] -\sqrt{\al{2}}\end{array}\right),\]
which is contained in ${\cal S}$ and is orthogonal to \eij{1}{1}. 

\noindent{}Moreover 
\[\|e_{_{11}}^{\perp}\|^{^{2}} = (\lam{}\al{1}-1)(\lam{}-\al{1}) =\|e_{_{11}}\|^{^{2}}.\]
Therefore 
\[\eijeij{1}{1}+e_{_{11}}^{\perp}\otimes\overline{e}_{_{11}}^{\perp} = (\lam{}\al{1}-1)(\lam{}-\al{1})I_{_{\cal S}}.\]
Thus 
\[\al{1}I_{_{\cal S}}-e_{_{11}}\otimes{}e_{_{11}}=e_{_{11}}^{\perp}\otimes\overline{e}_{_{11}}^{\perp} +\lam{}(\alpha_{_{1}}^{2}-\lam{}\al{1}+1)I_{_{\cal S}}\]
which, by straight forward computations, coincides with $\eijeij{1}{2}+\eijeij{1}{3}$ computed above. This proves $(c)$ for $i=1.$ By symmetry it also holds for $i=2$ and $i=3.$ Furthermore $(d)$ follows from $(c)$ because $\overline{e}_{_{ij}}=\eij{j}{i},\;\;\;i,j=1,2,3.$
}\end{example}
\begin{lemma}
\label{lemtre25}
Let $x_{_{1}},x_{_{2}}$ and $x_{_{3}}$ be 3 vectors in a two dimensional Hilbert space ${\cal H},$ and  $c\in{\Bbb R}_{+}$ be such that
\begin{equation}
\label{3xer}
\xx{1}+\xx{2}+\xx{3}=cI_{_{\cal H}}.
\end{equation}
Then 
\[\begin{array}{cl}
(a) & \|x_{_{1}}\|^{^{2}}+\|x_{_{2}}\|^{^{2}}+\|x_{_{3}}\|^{^{2}}=2c\\
(b) & \|x_{_{i}}\|^{^{2}}\leq{}c,\;\;\;\;i=1,2,3.\\
(c) & |(x_{_{i}},x_{_{j}})|^{^{2}}=(c-\|x_{_{i}}\|^{^{2}})(c-\|x_{_{j}}\|^{^{2}}),\;\;\;\;i\neq{}j
\end{array}\]
\end{lemma}
\begin{proof}
$(a)$ follows since $\mbox{Tr}(\xx{i})=\|x_{_{i}}\|^{^{2}}$ and 
$\mbox{Tr}(I_{_{\cal H}})=\mbox{dim}{\cal H} =2.$ $(b)$ is clear, because $\xx{i}\leq{}I_{_{\cal H}}.$

\noindent{}$(c):$ Let $x_{_{i}}=(x_{_{i1}},x_{_{i2}}),\;\;i=1,2,3$ be the coordinates of $x_{_{i}},$ with respect to an orthonormal basis. Define the vectors $\xi_{_{1}}=(x_{_{11}},x_{_{21}},x_{_{31}})$ and $\xi_{_{2}}=(x_{_{12}},x_{_{22}},x_{_{32}}),$ then condition \ref{3xer} is equivalent to $\xi_{_{1}},\xi_{_{2}}$ being orthogonal in ${\Bbb C}^{3},$ with $\|\xi_{_{1}}\|^{^{2}}=\|\xi_{_{2}}\|^{^{2}}=c.$ Choose a third vector $\xi_{_{3}}=(x_{_{13}},x_{_{23}},x_{_{33}})$ such that $\frac{1}{\sqrt{c}}\xi_{_{1}},\frac{1}{\sqrt{c}}\xi_{_{2}}$ and $\frac{1}{\sqrt{c}}\xi_{_{3}}$ form an orthonormal basis for ${\Bbb C}^{3}.$ Then
\[\frac{1}{\sqrt{c}}\left(\begin{array}{lll}
x_{_{11}}&x_{_{12}}&x_{_{13}}\\
x_{_{21}}&x_{_{22}}&x_{_{23}}\\
x_{_{31}}&x_{_{32}}&x_{_{33}}
\end{array}\right)\]
is a unitary matrix. In particular the rows of the matrix form an orthonormal basis. I.e.
\[\|x_{_{i}}\|^{^{2}}+|x_{_{3i}}|^{^{2}}=c,\;\;\;i=1,2,3\]
\[(x_{_{i}},x_{_{j}}) + x_{_{3i}}\overline{x}_{_{3j}}=0\;\;\;\;i\neq{}j\]
Thus\[|(x_{_{i}},x_{_{j}})|=|x_{_{3i}}||x_{_{3j}}|=\sqrt{c-\|x_{_{i}}\|^{^{2}}}\sqrt{c-\|x_{_{j}}\|^{^{2}}},\;\;\;\;i\neq{}j,\]
which proves $(c).$
\end{proof}
\begin{lemma}
\label{lemtre26}
If two vectors $x_{_{1}}$ and $x_{_{2}}$  in a two dimensional Hilbert space, ${\cal H},$ and $c\in{\Bbb R}_{+}$ are such that 
\[\begin{array}{cl}
(a)&\|x_{_{i}}\|^{^{2}}\leq{}c,\;\;\;\;i=1,2\\[0.2cm]
(b)&|(x_{_{1}},x_{_{2}})|^{^{2}} = (c-\|x_{_{i}}\|^{^{2}})(c-\|x_{_{j}}\|^{^{2}}).\end{array}\]
Then there exists a third vector $x_{_{3}}\in{\cal H}$ such that
\[\xx{1}+\xx{2}+\xx{3} =cI_{_{\cal H}}.\]
Moreover $\|x_{_{3}}\|^{^{2}}=2c-\|x_{_{1}}\|^{^{2}}-\|x_{_{2}}\|^{^{2}}.$\end{lemma}
\begin{proof}
Let $x_{_{1}} =(x_{_{11}},x_{_{12}})$ and $x_{_{2}} =(x_{_{21}},x_{_{22}})$ be the coordinates of $x_{_{1}}$ and $x_{_{2}}$ in a fixed orthonormal basis. Choose $\theta\in{\Bbb R},$ such that
\[(x_{_{1}},x_{_{2}})=e^{i\theta}|(x_{_{1}},x_{_{2}})|.\]
Using $(a)$ we can define $x_{_{13}},x_{_{23}}\in{\Bbb C}$ by
\[x_{_{13}}=e^{i\theta}\sqrt{c-\|x_{_{1}}\|^{^{2}}}\mbox{ and }
x_{_{23}}=-\sqrt{c-\|x_{_{2}}\|^{^{2}}}.\]
Then $\|x_{_{i}}\|^{^{2}}+|x_{_{i3}}|^{^{2}}=c,\;i=1,2$ and 
$(x_{_{1}},x_{_{2}})+x_{_{13}}\overline{x}_{_{23}}=0.$ \\
Therefore $\tilde{x}_{_{i}}=(x_{_{i1}},x_{_{i2}},x_{_{i3}}),\;i=1,2$ are two orthogonal vectors in ${\Bbb C}^{3},$ both of length$\sqrt{c}.$ Choose $\tilde{x}_{_{3}}=(x_{_{31}},x_{_{32}},x_{_{33}})$ such that 
\[\mbox{$\frac{1}{\sqrt{c}}$}\tilde{x}_{_{1}},
\mbox{$\frac{1}{\sqrt{c}}$}\tilde{x}_{_{2}},
\mbox{$\frac{1}{\sqrt{c}}$}\tilde{x}_{_{3}}\]
form an orthonormal basis of ${\Bbb C}^{3}.$ Put $x_{_{3}}=(x_{_{31}},x_{_{32}}).$
Since 
\[\left(\mbox{$\frac{1}{\sqrt{c}}$}x_{_{ij}}\right)_{i,j=1}^{3}\]
is a unitary matrix, the two first columns are orthogonal vectors. Hence
\[\xx{1}+\xx{2}+\xx{3}=cI_{_{\cal H}},\]
and by lemma \ref{lemtre25} $(a)$ we have
\[\|x_{_{3}}\|^{^{2}}=2c-\|x_{_{1}}\|^{^{2}}-\|x_{_{2}}\|^{^{2}}.\]
\end{proof}
\begin{prop}
\label{protre27}
Condition $(1)$ of Theorem \ref{thmtre21} is equivalent to \begin{description}
\item[$(1')$]There exists 6 vectors $(e_{_{ij}})_{i,j=1}^{3},\;i\neq{}j,$ in a two dimensional Hilbert space, such that
\[\begin{array}{cl}
(e)&\|e_{_{ij}}\|^{^{2}}=\al{i}+\al{j}-\lam{}\al{i}\al{j},\;\;\;i\neq{}j\\[0.2cm]
(f)&|(e_{_{ij}}e_{_{ik}})|=\sqrt{\al{j}\al{k}}(\lam{}\al{i}-1),\;\;\;i\neq{}j\neq{}k\neq{}i\\[0.2cm]
(g)&\eijeij{1}{2}-\eijeij{2}{1}=\eijeij{2}{3}-\eijeij{3}{2}=\eijeij{3}{1}-\eijeij{1}{3}.
\end{array}\]
\end{description}
\end{prop}
\begin{proof}

\noindent{}$(1)\Rightarrow{}(1')$: Let $(e_{_{ij}})_{i,j=1}^{3}$ be as in $(1)$ of Theorem \ref{thmtre21}, and consider the six vectors corresponding to $i\neq{}j.$ Then $(e)$ holds. From $(c)$ in Theorem \ref{thmtre21} we have
\[\eijeij{i}{i}+\eijeij{i}{j}+\eijeij{i}{k}=\al{i}I,\]
so by lemma \ref{lemtre25} $(c)$ we have
\[\begin{array}{lcl}|(e_{_{ij}},e_{_{ik}})|^{^{2}}&=&(\al{i}-\|e_{_{ij}}\|^{^{2}})(\al{i}-\|e_{_{ik}}\|^{^{2}})\\[0.2cm]
& = & \al{j}\al{k}(\lam{}\al{i} + 1)^{^{2}}.
\end{array}\]
This proves $(f),$ because $\lam{}\al{i}-1\geq{}0$ by $(a)$ of Theorem \ref{thmtre21}. 

\noindent{}By $(c)$ and $(d)$ of Theorem \ref{thmtre21} we have
\[\sum_{j=1}^{3}(\eijeij{i}{j}-\eijeij{j}{i}) = 0.\]
The term with $j=i$  vanishes in the sum, and by rearranging the remaining terms one gets $(g).$

\noindent{}$(1')\Rightarrow{}(1)$: Assume $(e_{_{ij}})_{i,j=1,i\neq{}j}^{3}$ satisfy $(1').$ Condition $(g)$ can be rewritten as
\begin{equation}
\label{eijeij}
\sum_{j\neq{}i}\eijeij{i}{j} = \sum_{j\neq{}i}\eijeij{j}{i},\;\;\;i=1,2,3.\end{equation}
Fix $i,$ and let $j,k$ be the two remaining numbers in $\{1,2,3\}.$ By $(f)$ $\lam{}\al{i}-1\geq{}0,$ so by $(e)$
\[\begin{array}{lcl}
\al{i} - \|e_{_{ij}}\|^{^{2}} & = & \al{j}(\lam{}\al{i}-1)\geq{}0\\[0.2cm]
\al{i} - \|e_{_{ik}}\|^{^{2}} & = & \al{k}(\lam{}\al{i}-1)\geq{}0.
\end{array}\]
Moreover, by $(e)$ and $(f)$
\[|(e_{_{ij}},e_{_{ik}})|^{^{2}}=(\al{i}-\|e_{_{ij}}\|^{^{2}})(\al{i}-\|e_{_{ik}}\|^{^{2}}).\]
Hence, by lemma \ref{lemtre26}, there exists $e_{_{ii}}\in{}{\cal H},$ such that
\[\sum_{l=1}^{3}\eijeij{i}{l} = \al{i}I_{_{\cal H}},\]
and
\[\begin{array}{lcl}
\|e_{_{ii}}\|^{^{2}} & = & 2\al{i} -\|e_{_{ij}}\|^{^{2}}-\|e_{_{ik}}\|^{^{2}}\\[0.2cm]
& = &(\al{j}+\al{k})(\lam{}\al{i}-1)\\[0.2cm]
& = & (\lam{}-\al{i})(\lam{}\al{i}-1).
\end{array}\]
By (\ref{eijeij}) also $\sum_{l=1}^{3}\eijeij{l}{i} = \al{i}I_{_{\cal H}}.$  Hence this construction, for $i=1,2,3,$ provides us with three new vectors $e_{_{11}},e_{_{22}},e_{_{33}},$ which, together with the given six vectors, satisfy $(a),(b),(c)$ and $(d)$ in Theorem \ref{thmtre21}.
\end{proof}
\begin{lemma}
\label{lemtre28}
Let $A=(a_{_{ij}})_{i,j=1}^{3}$ be a symmetric $3\times{}3$ matrix with non-negative entries, and set $A'= (a_{_{ij}}')_{i,j=1}^{3},$ where $a_{_{ii}}'=a_{_{ii}}$ and $a_{_{ij}}'=-a_{_{ij}},\;i\neq{}j.$ Then the following two conditions are equivalent
\begin{itemize}
\item[(1)]There exists 3 vectors $\xi_{_{1}},\xi_{_{2}},\xi_{_{3}}$ in ${\Bbb C}^{3}$ with
\[|(\xi_{_{i}},\xi_{_{j}})|=a_{_{ij}},\;\;\;i,j=1,2,3\]
\item[(2)]$\det{}(A)\geq{}0,\;\det{}(A')\leq{}0$ and $a_{_{ij}}^{2}\leq{}a_{_{ii}}a_{_{jj}},\;\;i\neq{}j.$
\end{itemize}
\end{lemma}
\begin{proof}

\noindent{}$(1)\Rightarrow{}(2)$: Assume $(1).$ Then $a_{_{ij}}^{2}\leq\aij{i}{i}\aij{j}{j}$ by the Cauchy-Schwartz inequality. Let $B$ be the matrix with entries $b_{_{ij}}=(\xi_{_{i}},\xi_{_{j}}).$ Since dim$({\Bbb C}^{2})<3,$ the vectors $\xi_{_{1}},\xi_{_{2}},\xi_{_{3}}$ are linearly dependent, i.e. there exists complex numbers $(c_{_{1}},c_{_{2}},c_{_{3}})\neq{}(0,0,0)$ such that $\sum_{j}c_{_{j}}\xi_{_{j}}=0.$ Thus
\[\sum_{j}b_{_{ij}}\overline{c}_{_{j}}=(\xi_{_{i}},\sum_{j}c_{_{j}}\xi_{_{j}})=0.\]
Hence $\det{}(B)=0.$ 

\noindent{}$B$ is of the form
\[
B=\left(\begin{array}{ccc}
\aij{1}{1} & \gam{3}\aij{1}{2}&\overline{\gamma}_{_{2}}\aij{1}{3}\\
\overline{\gamma}_{_{3}}\aij{1}{2} & \aij{2}{2} & \gam{1}\aij{2}{3}\\
\gam{2}\aij{1}{3} &\overline{\gamma}_{_{1}}\aij{2}{3}&\aij{3}{3}
\end{array}\right),\]
where $\gam{1},\gam{2},\gam{3}\in{\Bbb C},$ $|\gam{1}|=|\gam{2}|=|\gam{3}|.$ 

\noindent{}Hence
\[\det{}(B)= \aij{1}{1}\aij{2}{2}\aij{3}{3}+2\mbox{Re}(\gam{1}\gam{2}\gam{3})\aij{1}{2}\aij{2}{3}\aij{1}{3} - a_{_{11}}a_{_{23}}^{2}- a_{_{22}}a_{_{13}}^{2}- a_{_{33}}a_{_{12}}^{2}.\]
Since
\[\begin{array}{lcl}\det{}(A) &=& \aij{1}{1}\aij{2}{2}\aij{3}{3}+2\aij{1}{2}\aij{2}{3}\aij{1}{3} - a_{_{11}}a_{_{23}}^{2}- a_{_{22}}a_{_{13}}^{2}- a_{_{33}}a_{_{12}}^{2}\\[0.3cm]
\det{}(A')&=& \aij{1}{1}\aij{2}{2}\aij{3}{3}-2\aij{1}{2}\aij{2}{3}\aij{1}{3} - a_{_{11}}a_{_{23}}^{2}- a_{_{22}}a_{_{13}}^{2}- a_{_{33}}a_{_{12}}^{2}\end{array}\]
we have 
\[\det{}(A')\leq{}\det{}(B)\leq{}\det{}(A)\]
which proves $(2).$

\noindent{}$(2)\Rightarrow{}(1)$: Assume $(2).$ Since $\det{}(A')\leq{}0\leq{}\det{}(A),$ the above formulas for $\det{}(A)$ and $\det{}(A'),$ show that there exists $r\in[-1,1],$ such that 
\[\aij{1}{1}\aij{2}{2}\aij{3}{3}+2r\aij{1}{2}\aij{2}{3}\aij{1}{3} - a_{_{11}}a_{_{23}}^{2}- a_{_{22}}a_{_{13}}^{2}- a_{_{33}}a_{_{12}}^{2}=0.\]
Let $\gam{}=r+i\sqrt{1-r^{^{2}}}.$ Then $|\gam{}|=1,$ Re$(\gam{})=r,$ so the matrix
\[B=\left(\begin{array}{ccc}
\aij{1}{1} & \aij{1}{2}&\aij{1}{3}\\
\aij{1}{2} & \aij{2}{2} & \gam{}\aij{2}{3}\\
\aij{1}{3} &\overline{\gamma}\aij{2}{3}&\aij{3}{3}
\end{array}\right)\]
has $\det{}(B)=0.$ Since furthermore $\aij{1}{1},\aij{2}{2},\aij{3}{3}\geq{}0,$ and the minors
\[A_{_{11}}=\left(\begin{array}{cc}\aij{2}{2}&\gam{}\aij{2}{3}\\\overline{\gamma}\aij{2}{3}&\aij{3}{3}\end{array}\right),\;\;
A_{_{22}}=\left(\begin{array}{cc}\aij{1}{1}&\aij{1}{3}\\\aij{1}{3}&\aij{3}{3}\end{array}\right),\;\;
A_{_{33}}=\left(\begin{array}{cc}\aij{1}{1}&\aij{1}{2}\\\aij{1}{2}&\aij{2}{2}\end{array}\right),\]
have non-negative determinants by $(2),$ the matrix $B$ is positive semidefinite. (Indeed the characteristic polynomial
\[\det{}(\rho{}I-B)=\rho^{3}-\rho^{2}(\aij{1}{1}+\aij{2}{2}+\aij{3}{3}) +\rho(\sum_{i=1}^{3}\det{}(A_{_{ii}}))-\det{}(B)\]
is strictly negative when $\rho{}<0,$ therefore $B$ has only non-negative eigenvalues.) 

\noindent{}Let $\xi_{_{1}},\xi_{_{2}},\xi_{_{3}}$ be the column vectors of the matrix $B^{^{\frac{1}{2}}}.$ Then
\[|(\xi_{_{i}},\xi_{_{j}})|=|b_{_{ij}}|=\aij{i}{j}.\]
Moreover  $\mbox{span}\{\xi_{_{1}},\xi_{_{2}},\xi_{_{3}}\}=\mbox{range}(B^{^{\frac{1}{2}}})=\mbox{range}(B)$ has dimension at most 2, because $\det{}(B)=0.$ This proves $(1).$
\end{proof}
\begin{proof}{\bf of sufficiency of (i),(ii) and (iii) in Theorem \ref{thmtre21}}

\noindent{}The case $\lam{}\leq{}2$ was treated in example \ref{exatre24}, so we can assume $\lam{}>2.$ Assume $(i),(ii)$ and $(iii),$ and set 
\[A=(\aij{i}{j})_{i,j=1}^{3}=\left(\begin{array}{ccc}
\al{2}+\al{3}-\lam{}\al{2}\al{3}&\sqrt{\al{1}\al{2}}(\lam{}\al{3}-1)&\sqrt{\al{1}\al{3}}(\lam{}\al{2}-1)\\[0.2cm]
\sqrt{\al{1}\al{2}}(\lam{}\al{3}-1)&\al{1}+\al{3}-\lam{}\al{1}\al{3} &\sqrt{\al{2}\al{3}}(\lam{}\al{1}-1)\\[0.2cm]
\sqrt{\al{1}\al{3}}(\lam{}\al{2}-1)&\sqrt{\al{2}\al{3}}(\lam{}\al{1}-1)&\al{1}+\al{2}-\lam{}\al{1}\al{2} \end{array}\right)\]
Then $A$ has non-negative entries, because $\lam{}\al{i}-1\geq{}0$ by $(i),$ and for $i\neq{}j$ we have
\[\al{i}+\al{j}-\lam{}\al{i}\al{j}=(\lam{}-\al{i}-\al{j})(\lam{}\al{i}-1)+\lam{}(\alpha_{_{i}}^{2}-\lam{}\al{i}+1)\geq{}0\]
by $(i)$ and $(ii).$ We will show, that $A$ satisfies condition $(2)$ of lemma \ref{lemtre28}. Let $A'$ be as in lemma \ref{lemtre28}. A tedious, but straight forward computation shows that
\[\begin{array}{lcl}
\det{}(A)&=&\lambda^{2}\al{1}\al{2}\al{3}(4\lam{}\al{1}\al{2}\al{3} -4(\al{1}\al{2}+\al{1}\al{3}+\al{2}\al{3}) + 3)\\[0.2cm]
\det{}(A')&=&\lam{}(4-\lambda^{2})\al{1}\al{2}\al{3}.
\end{array}\]
Hence $\det{}(A)\geq{}0$ by $(iii)$ and $\det{}(A')<0,$ because $\lam{}>2.$

\noindent{}Moreover by $(ii)$
\[\begin{array}{lcl}
\aij{1}{1}\aij{2}{2}-a_{_{12}}^{2} & = & \aij{1}{1}\aij{2}{2} -(\al{3}-\aij{1}{1})(\al{3}-\aij{2}{2})\\[0.2cm]
& = & \al{3}(\aij{1}{1}+\aij{2}{2}-\al{3})\\[0.2cm]
& = &\lam{}\al{3}(\lambda^{2}-\lam{}\al{3}+1)\\[0.2cm]
&\geq &0,\end{array}\]
and similarly $\aij{1}{1}\aij{3}{3}-a_{_{13}}^{2}\geq{}0$ and 
$\aij{2}{2}\aij{3}{3}-a_{_{23}}^{2}\geq{}0.$ Hence by lemma \ref{lemtre28}, there exists 3 vectors $\xi_{_{1}},\xi_{_{2}},\xi_{_{3}}$ in ${\Bbb C}^{2}$ such that 
\[|(\xi_{_{i}},\xi_{_{j}})|=\aij{i}{j},\;\;\;i,j=1,2,3.\]
Now put $\eij{1}{2}=\eij{2}{1}=\xi_{_{3}},$ $\eij{1}{3}=\eij{3}{1}=\xi_{_{2}}$ and $\eij{2}{3}=\eij{3}{2}=\xi_{_{1}}.$ Then $(\eij{i}{j})_{i,j=1,i\neq{}j}^{3}$ clearly satisfy the conditions $(e),$ $(f)$ and $(g)$ in Proposition \ref{protre27}, so it can be extended to a set of 9 vectors $(\eij{i}{j})_{i,j=1}^{3}$ satisfying $(a),$ $(b),$ $(c)$ and $(d)$ in Theorem \ref{thmtre21}.
\end{proof}
\begin{remark}
\label{remtre29}
The above solution to $(a),$ $(b)$ and $(c)$ in Theorem \ref{thmtre21} for $\lam{}>2$ satisfies\\$\eij{i}{j}=\eij{j}{i},\;i,j=1,2,3.$ However if $\lam{}<2,$ $\det{}(A')<0,$ so by lemma \ref{lemtre28} there are no solutions to $(a),$ $(b),$ $(c)$ and $(d)$ which satisfy $\eij{i}{j}=\eij{j}{i},\;i,j=1,2,3.$
\end{remark}
To prove that $(i)$ $(ii)$ and $(iii)$ are necessary conditions, we need to look at the following map:
\begin{lemma}
\label{lemtre210}
The map 
\[q:{\Bbb C}^{2}\rightarrow\left\{a\in{}M_{_{2}}({\Bbb C})|a=a^{*},\mbox{Tr}(a)=0\right\}\]
given by
\[q(x)=\sqrt{2}(\xx{}-\mbox{$\frac{1}{2}$}\|x\|^{^{2}}I)\]
has the following properties
\begin{itemize}
\item[(1)]$q(x)=q(y)$ if and only if $y=\gam{}x$ for some $\gam{}\in{\Bbb C}, |\gam{}|=1.$
\item[(2)]$q$ maps ${\Bbb C}^{2}$ onto $\{a\in{}M_{_{2}}({\Bbb C})|a=a^{*},\mbox{Tr}(a)=0\}.$
\item[(3)]$\|q(x)\|=\|x\|^{^{2}},$ where we consider the Hilbert-Schmidt norm on $M_{_{2}}({\Bbb C}).$
\item[(4)]$(q(x),q(y))=2|(x,y)|^{^{2}}-\|x\|^{^{2}}\|y\|^{^{2}},$ where $(q(x),q(y))=\mbox{tr}(q(x)q(y)).$
\end{itemize}
\end{lemma}
\begin{remark}
\label{remtre211}{\rm 
By $(3)$ and $(4)$ it follows, that if $(x,y)=\|x\|\|y\|\cos{}\theta,\;0\leq\theta\leq\frac{\pi}{2},$ then $(q(x),q(y)) = \|q(x)\|\|q(y)\|\cos{}2\theta.$ Hence $q$ doubles the angles between vectors.}
\end{remark}
\begin{proof}{\bf of lemma \ref{lemtre210}} Since $\frac{1}{2}\|x\|^{^{2}}I$ is equal to the orthogonal projection of \xx{} onto ${\Bbb R}I$ in $\{a\in{}M_{_{2}}({\Bbb C})|a=a^{*}\},$ with respect to the inner product $(a,b)=\mbox{Tr}(ab),$ we have
\[\begin{array}{lcl}
|(x,y)|^{^{2}}& = & (\xx{},\yy{})\\[0.2cm]
& = & (\mbox{$\frac{1}{2}$}\|x\|^{^{2}}I,\mbox{$\frac{1}{2}$}\|y\|^{^{2}}I)+(\xx{}-\mbox{$\frac{1}{2}$}\|x\|^{^{2}}I,\yy{}-\mbox{$\frac{1}{2}$}\|y\|^{^{2}}I)\\[0.2cm]
&=&\mbox{$\frac{1}{4}$}\|x\|^{^{2}}\|y\|^{^{2}}\mbox{Tr}(I)+\mbox{$\frac{1}{2}$}(q(x),q(y)).\end{array}\]
Since $\mbox{Tr}(I)=2$ ($I$ is the identity on ${\Bbb C}^{2}$), $(4)$ follows. Moreover $(4)\Rightarrow{}(3)$ is trivial.

\noindent{}By $(3),$ $q(x)=q(y)$ implies $\|x\|=\|y\|,$ and then $\xx{}=\yy{}$ by the definition of $q.$ This shows that $y=\gam{}x$ for some $\gam{}\in{\Bbb C},$ $|\gam{}|=1.$ Hence $(1)$ holds.

\noindent{}Finally if $a\in{}M_{_{2}}({\Bbb C}),$ $a=a^{*},$ $\mbox{Tr}(a)=0,$ then the eigenvalues of $a$ are $\{\rho,-\rho\}$ for some $\rho\geq{}0,$ and we can choose an orthonormal basis $(e_{_{1}},e_{_{2}})$ for ${\Bbb C}^{2},$ such that \[\begin{array}{lcl}
a&=&\rho\ee{1}-\rho\ee{2}\\[0.2cm]
&=&2\rho\ee{1}-\rho{}I,\end{array}\]
since $\ee{1}+\ee{2}=I.$ Hence $a=q(x),$ where $x=(\sqrt{2}\rho)^{^{\frac{1}{2}}}e_{_{1}}.$ This shows $(2).$
\end{proof}
\begin{lemma}
\label{lemtre212}
Let $A=(\aij{i}{j})_{i,j=1}^{3}$ be a matrix with non-negative entries, and define a $3\times{}3$ matrix $Q(A)$ by
\[Q(A)_{_{ij}}=\left\{\begin{array}{l}
a_{_{ii}}^{2},\;\;\;i=j\\[0.32cm]
2a_{_{ij}}^{2}-\aij{i}{i}-\aij{j}{j},\;\;\;i\neq{}j\end{array}\right.\]
Then the conditions $(1)$ and $(2)$ of lemma \ref{lemtre28} are equivalent to

$(3)$ $Q(A)$ is positive semidefinite.
\end{lemma}
\begin{proof}
Since $(1)\Leftrightarrow{}(2),$ it suffices to prove $(3)\Rightarrow{}(1)\Rightarrow{}(3),$ where $(1)$ and $(2)$ refer to lemma \ref{lemtre28}.

Assume $\xi_{_{1}},\xi_{_{2}},\xi_{_{3}}\in{\Bbb C}^{2}$ are vectors, such that
\[|(\xi_{_{i}},\xi_{_{j}})|=\aij{i}{j},\;\;\;i,j=1,2,3.\]
By lemma \ref{lemtre210} $(3)$ and $(4)$
\[(q(\xi_{_{i}}),q(\xi_{_{j}}))=Q(A)_{_{ij}},\;\;\;i,j=1,2,3.\]
For $c_{_{1}},c_{_{2}},c_{_{3}}\in{\Bbb C},$
\[\sum_{i,j}(q(\xi_{_{i}}),q(\xi_{_{j}}))c_{_{i}}\overline{c}_{_{j}}=
\|\sum_{i}c_{_{i}}q(\xi_{_{i}})\|^{^{2}}\geq{}0,\]
which shows that $Q(A)$ is positive semidefinite. Hence $(1)\Rightarrow{}(3).$ Assume conversely that $Q(A)$ is positive semidefinite. Let $\eta_{_{1}},\eta_{_{2}},\eta_{_{3}}$ be the columns of $Q(A)^{^{\frac{1}{2}}}.$ Then $\eta_{_{1}},\eta_{_{2}},\eta_{_{3}}\in{\Bbb R}^{3}$ and 
\[(\eta_{_{i}},\eta_{_{j}})=Q(A)_{_{ij}},\;\;\;i,j=1,2,3.\]
We may identify ${\Bbb R}^{3}$ with the Hilbert space
\[\{a\in{}M_{_{2}}({\Bbb C})|a=a^{*},\mbox{Tr}(a)=0\},\]
with inner product $(a,b)=\mbox{Tr}(ab).$ Hence by lemma \ref{lemtre210} $(2)$ there exists $\xi_{_{1}},\xi_{_{2}},\xi_{_{3}}\in{\Bbb C}^{2}$ such that 
\[(q(\xi_{_{i}}),q(\xi_{_{j}}))=Q(A)_{_{ij}},\;\;\;i,j=1,2,3.\]
Thus \[\|\xi_{_{i}}\|^{^{4}}=a_{_{ii}}^{2}\mbox{ and }
2|(\xi_{_{i}},\xi_{_{j}})|^{^{2}}-\|\xi_{_{i}}\|^{^{2}}\|\xi_{_{j}}\|^{^{2}}=2a_{_{ij}}^{2}-\aij{i}{i}\aij{j}{j}\]
by lemma \ref{lemtre210} and the definition of $Q(A).$ 

\noindent{}Since $\aij{i}{j}\geq{}0$ for all $i,j,$ it follows that
\[\|\xi_{_{i}}\|^{^{2}}=\aij{i}{i} \mbox{ and }|(\xi_{_{i}},\xi_{_{j}})| = \aij{i}{j},\;\;i\neq{}j.\]
Hence $A$ satisfies $(1)$ of lemma \ref{lemtre28}, and $(3)\Rightarrow{}(1).$
\end{proof}
\begin{lemma}
\label{lemtre213}
Let $\al{1},\al{2},\al{3}>0$ and $\lam{}=\al{1}+\al{2}+\al{3}.$ Put
\[A=(\aij{i}{j})_{i,j=1}^{3}=\left(\begin{array}{ccc}
\al{2}+\al{3}-\lam{}\al{2}\al{3}&\sqrt{\al{1}\al{2}}(\lam{}\al{3}-1)&\sqrt{\al{1}\al{3}}(\lam{}\al{2}-1)\\[0.2cm]
\sqrt{\al{1}\al{2}}(\lam{}\al{3}-1)&\al{1}+\al{3}-\lam{}\al{1}\al{3} &\sqrt{\al{2}\al{3}}(\lam{}\al{1}-1)\\[0.2cm]
\sqrt{\al{1}\al{3}}(\lam{}\al{2}-1)&\sqrt{\al{2}\al{3}}(\lam{}\al{1}-1)&\al{1}+\al{2}-\lam{}\al{1}\al{2} \end{array}\right),\]
as in the proof of sufficiency of $(i),(ii)$ and $(iii)$ in Theorem \ref{thmtre21}, and let $Q(A)$ be the matrix with entries
\[Q(A)_{_{ij}}=\left\{\begin{array}{l}
a_{_{ii}}^{2},\;\;\;i=j\\[0.2cm]
2a_{_{ij}}^{2}-\aij{i}{i}\aij{j}{j},\;\;\;i\neq{}j.
\end{array}\right.\]
Moreover set 
\[D=\left(\begin{array}{rrr}1&-1&-1\\-1&1&-1\\-1&-1&1\end{array}\right).\]Then condition $(1)$ in Theorem \ref{thmtre21} implies that $\aij{i}{j}\geq{}0,\;i,j=1,2,3,$ and that
\[Q(A)-tD\]
is positive semidefinite for some $t\geq{}0.$
\end{lemma}
\begin{proof}
Clearly $\aij{i}{j}\geq{}0,$ by $(a)$ and $(b)$ of Theorem \ref{thmtre21}. Let $(\eij{i}{j})_{i,j=1}^{3}$ be a solution to $(a),(b),(c)$ and $(d)$ of Theorem \ref{thmtre21}. Set 
\[\qij{i}{j}=q(\eij{i}{j})=\sqrt{2}(\eijeij{i}{j}-\mbox{$\frac{1}{2}$}\|\eij{i}{j}\|^{^{2}}).\]
By Proposition \ref{protre27} and lemma \ref{lemtre210}
\begin{itemize}
\item[$(i)$]$\|\qij{i}{j}\|=\|\eij{i}{j}\|^{^{2}}.$
\item[$(ii)$]$(\qij{i}{j},\qij{i}{k})=2|(\eij{i}{j},\eij{i}{k})|^{^{2}} - \|\eij{i}{j}\|^{^{2}}\|\eij{i}{k}\|^{^{2}},\;\;i\neq{}j\neq{}k\neq{}i.$
\item[$(iii)$]$\qij{1}{2}-\qij{2}{1}=\qij{3}{1}-\qij{1}{3}=\qij{2}{3}-\qij{3}{2}.$
\end{itemize}
Put $\qijs{i}{j} =\frac{1}{2}(\qij{i}{j}+\qij{j}{i}),\;i\neq{}j$ and
\[\qa=\mbox{$\frac{1}{2}$}(\qij{1}{2}-\qij{2}{1})=
\mbox{$\frac{1}{2}$}(\qij{2}{3}-\qij{3}{2})=\mbox{$\frac{1}{2}$}(\qij{3}{1}-\qij{1}{3}).\]
Then
\[\begin{array}{ll}
\qij{1}{2}=\qijs{1}{2}+\qa,\;\;\;\;\;&\qij{2}{1}=\qijs{2}{1}-\qa\\
\qij{2}{3}=\qijs{2}{3}+\qa,\;\;\;\;\;&\qij{3}{2}=\qijs{3}{2}-\qa\\
\qij{3}{1}=\qijs{3}{1}+\qa,\;\;\;\;\;&\qij{1}{3}=\qijs{1}{3}-\qa.
\end{array}\]
Since $\|\qij{i}{j}\|^{^{2}}=\|\qij{j}{i}\|^{^{2}}$ it follows that \qa{} is orthogonal to \qijs{1}{2},\qijs{2}{3} and \qijs{3}{1}. (The range of $q$ is a real Hilbert space.) Set $t=\|\qa\|^{^{2}}.$ Then
\[\|\qij{i}{j}\|^{^{2}}=\|\qijs{i}{j}\|^{^{2}}+t,\;\;i\neq{}j.\]
Moreover
\[\begin{array}{l}
(\qij{1}{2},\qij{3}{1})=(\qijs{1}{2},\qijs{3}{1})-t\\
(\qij{2}{1},\qij{2}{3})=(\qijs{2}{1},\qijs{2}{3})-t\\
(\qij{3}{1},\qij{3}{2})=(\qijs{3}{1},\qijs{3}{2})-t.\end{array}\]
Hence by $(i)$ and $(ii)$
\begin{itemize}
\item[$(iv)$]$\|\qijs{i}{j}\|^{^{2}}=\|\eij{i}{j}\|^{^{4}}-t,\;\;i\neq{}j.$
\item[$(v)$]$(\qijs{1}{2},\qijs{3}{1})=2|(\eij{1}{2},\eij{3}{1})|^{^{2}}-\|\eij{1}{2}\|^{^{2}}\|\eij{3}{1}\|^{^{2}}+t.$
\item[$(vi)$]$(\qijs{2}{1},\qijs{2}{3})=2|(\eij{2}{1},\eij{2}{3})|^{^{2}}-\|\eij{2}{1}\|^{^{2}}\|\eij{2}{3}\|^{^{2}}+t.$
\item[$(vii)$]$(\qijs{3}{1},\qijs{3}{2})=2|(\eij{3}{1},\eij{3}{2})|^{^{2}}-\|\eij{3}{1}\|^{^{2}}\|\eij{3}{2}\|^{^{2}}+t.$
\end{itemize}
Put $\eta_{_{1}}=\qijs{2}{3},\eta_{_{2}}=\qijs{3}{1}$ and $\eta_{_{3}}=\qijs{1}{2}.$ By $(a)$ and $(b)$ of Theorem \ref{thmtre21} and the definition of $A$ we have
\[\begin{array}{ll}
\aij{1}{1} =\|\eij{2}{3}\|^{^{2}}=\|\eij{3}{2}\|^{^{2}},\;\;\;\; & 
\aij{1}{2}=\aij{2}{1}=|(\eij{3}{1},\eij{3}{2})|,\\
\aij{2}{2} =\|\eij{1}{3}\|^{^{2}}=\|\eij{3}{1}\|^{^{2}}, &
\aij{1}{3}=\aij{3}{1}=|(\eij{2}{1},\eij{2}{3})|,\\
\aij{3}{3} =\|\eij{1}{2}\|^{^{2}}=\|\eij{2}{1}\|^{^{2}}, &
\aij{2}{3}=\aij{3}{2}=|(\eij{1}{2},\eij{1}{3})|.\end{array}\]
Hence by the definition of $Q(A)$ and $(iv),(v),(vi)$ and $(vii),$
\[((\eta_{_{i}},\eta_{_{j}}))_{i,j=1}^{3}=Q(A)-tD.\]
For $(c_{_{1}},c_{_{2}},c_{_{3}})\in{\Bbb C}^{3},$
\[\sum_{i=1}^{3}(\eta_{_{i}},\eta_{_{j}})c_{_{i}}\overline{c}_{_{j}} = \|\sum_{i=1}^{3}c_{_{i}}\eta_{_{i}}\|^{^{2}}\geq{}0,\]
so the matrix is positive semidefinite. This proves lemma \ref{lemtre213}.\end{proof}
\begin{lemma}
\label{lemtre214}
Let $\al{1},\al{2},\al{3}>0$ and $\lam{}=\al{1}+\al{2}+\al{3}.$ Let $A,Q(A)$ and $D$ be as in lemma \ref{lemtre213}. Then
\[Q(A)-t_{_{0}}D\]
is positive semidefinite for $t_{_{0}}=\lam{}(4-\lambda^{^{2}})\al{1}\al{2}\al{3}.$
\end{lemma}
\begin{proof}
Assume first that $\lam{}\leq{}2$ and $\lam{}\al{i}-1\geq{}0,\;i=1,2,3.$ Let $(\eij{i}{j})_{i,j=1}^{3}$ be the explicit solution to condition $(1)$ in Theorem \ref{thmtre21}, described in example \ref{exatre24}. By the proof of lemma \ref{lemtre213}, $Q(A)+tD$ is positive definite for 
\[\begin{array}{lcl}t&=&\|\mbox{$\frac{1}{2}$}(q(\eij{1}{2})-q(\eij{2}{1}))\|^{^{2}}\\[0.2cm]
& = & \|\mbox{$\frac{1}{\sqrt{2}}$}(\eijeij{1}{2}-\eijeij{2}{1})\|^{^{2}}.
\end{array}\]
With the notation of example \ref{exatre24} we have
\[\eijeij{1}{2}-\eijeij{2}{1} =i2\sin(\theta)\sqrt{\al{1}\al{2}\al{3}}
\left(\begin{array}{ccc} 
0 & \sqrt{\al{3}} & -\sqrt{\al{2}}\\
-\sqrt{\al{3}}& 0 &\sqrt{\al{1}} \\
\sqrt{\al{2}}& -\sqrt{\al{1}}&0\end{array}\right).\]
Hence
\[t=4\lam{}\sin{}^{2}(\theta)\al{1}\al{2}\al{3} 
=\lam{}(4-\lambda^{^{2}})\al{1}\al{2}\al{3}=t_{_{0}}.\]
Furthermore by the proof of lemma \ref{lemtre212}
\[Q(A)-t_{_{0}}D =((\eta_{_{i}},\eta_{_{j}}))_{i,j=1}^{3},\]
where 
\[ \begin{array}{lcl}
\eta_{_{1}} & = & \mbox{$\frac{1}{2}$}(q(\eij{2}{3})+q(\eij{3}{2}))\\[0.2cm]
\eta_{_{2}} & = & \mbox{$\frac{1}{2}$}(q(\eij{1}{3})+q(\eij{3}{1}))\\[0.2cm]
\eta_{_{3}} & = & \mbox{$\frac{1}{2}$}(q(\eij{1}{2})+q(\eij{2}{1})).
\end{array}\]
By $(c)$ and $(d)$ of Theorem \ref{thmtre21}
\[\sum_{i=1}^{3}(q(\eij{i}{j})+q(\eij{j}{i}))\]
is a multiple of $I_{_{\cal S}},$ but since $\mbox{Tr}(q(x))=0$ for all $x\in{\cal S}$ actually
\[\sum_{i=1}^{3}(q(\eij{i}{j})+q(\eij{j}{i}))=0,\;\;\;j=1,2,3.\]
Hence $\eta_{_{2}}+\eta_{_{3}}=-q(\eij{1}{1}),$ $\eta_{_{1}}+\eta_{_{3}}=-q(\eij{2}{2})$ and $\eta_{_{1}}+\eta_{_{2}}=-q(\eij{3}{3})$ which gives
\[\begin{array}{lcl}
\eta_{_{1}} & = & q(\eij{1}{1})-q(\eij{2}{2})-q(\eij{3}{3})\\
\eta_{_{2}} & = & q(\eij{2}{2})-q(\eij{1}{1})-q(\eij{3}{3})\\
\eta_{_{3}} & = & q(\eij{3}{3})-q(\eij{1}{1})-q(\eij{2}{2}).
\end{array}\]
But $q(\eij{i}{i}),\;i=1,2,3$ are given by functions of \al{1}, \al{2}, \al{3} (and \lam{}=\al{1}+\al{2}+\al{3}) which are real analytic in the region $\al{1}>0,$ $\al{2}>0,$ and $\al{3}>0.$ For instance
\[\begin{array}{lccl}
q(\eij{1}{1}) & =  & & \sqrt{2}(\eijeij{1}{1}-\mbox{$\frac{1}{2}$}\|\eij{1}{1}\|^{^{2}}I_{_{\cal S}})\\[0.3cm]
& =  & & 
\mbox{$\frac{\sqrt{2}(\lambda\alpha_{_{1}}-1)}{\lambda}$}\left(\begin{array}{ccc} 
(\lam{}-\al{1})^{^{2}} & -\sqrt{\al{1}\al{2}}(\lam{}-\al{1})&-\sqrt{\al{1}\al{2}}(\lam{}-\al{1})\\
-\sqrt{\al{1}\al{2}}(\lam{}-\al{1})&\al{1}\al{2} &\al{1}\sqrt{\al{2}\al{3}}\\
-\sqrt{\al{1}\al{2}}(\lam{}-\al{1})&\al{1}\sqrt{\al{2}\al{3}}&\al{2}\al{3}\end{array}\right)\\[1cm]
& &-&\mbox{$\frac{(\lambda-\alpha_{_{1}})(\lambda\alpha_{_{1}}-1)}{\sqrt{2}\lambda}$}\left(
\begin{array}{ccc}
\lam{}-\al{1}&-\sqrt{\al{1}\al{2}} & -\sqrt{\al{1}\al{3}}\\
-\sqrt{\al{1}\al{2}} & \lam{}-\al{2} &-\sqrt{\al{2}\al{3}}\\
-\sqrt{\al{1}\al{3}} &-\sqrt{\al{2}\al{3}} &\lam{}-\al{3}
\end{array}\right).\end{array}\]
Hence $\eta_{_{1}},\eta_{_{2}},\eta_{_{3}}$ can also be extended to real analytic functions (with values in $B({\cal S)}$ identified with a subspace of $M_{_{3}}({\Bbb C})$). Since the entries of $Q(A)-t_{_{0}}D$ are polynomials in \al{1}, \al{2}, \al{3}, it follows by uniqueness of analytic continuation, that 
\[Q(A)-t_{_{0}}D=((\eta_{_{i}},\eta_{_{j}}))_{i,j=1}^{3},\]
for all $(\al{1}, \al{2}, \al{3}),$ $\al{i}>0,\;i=1,2,3,$ and therefore $Q(A)-t_{_{0}}D$ is positive semidefinite.
\end{proof}
\begin{proof}{\bf of necessity of (i), (ii) and (ii) in Theorem \ref{thmtre21}} By remark \ref{remtre22}, $(i)$ and $(ii)$ are necessary. If $\lam{}\leq{}2$ $(i)\Rightarrow(iii)$ by remark \ref{remtre23}, so in this case $(iii)$ is a necessary (but redundant!) condition. 

\noindent{}Assume $\lam{}>2.$ If $(\eij{i}{j})_{i,j=1}^{3}$ satisfy $(1)$ in Theorem \ref{thmtre21}, then by lemma \ref{lemtre213} and lemma \ref{lemtre214}
\[Q(A)-tD\]
is positive semidefinite for some $t\geq{}0,$ and also for $t=t_{_{0}}=\lam{}(4-\lambda^{^{2}})\al{1}\al{2}\al{3}<0,$ so by convexity of the positive cone in $M_{_{3}}({\Bbb R}),$ $Q(A)$ is also positive semidefinite. Hence, by lemma \ref{lemtre212} and lemma \ref{lemtre28}, $\det{}(A)\geq{}0.$ We have previously found that 
\[\det{}(A)=\lambda^{2}\al{1}\al{2}\al{3}(4\lam{}\al{1}\al{2}\al{3} -4(\al{1}\al{2}+\al{1}\al{3}+\al{2}\al{3}) + 3).\]
Thus $(1)$ implies $(iii)$ in Theorem \ref{thmtre21}. This completes the proof of Theorem \ref{thmtre21}.
\end{proof}

\setcounter{equation}{0}

\section{The Graphs Satisfying the Condition}
In this section we will determine which of the graphs
\[\Gamma=\trestar{i}{j}{k},\;\;i\leq{}j\leq{}k,\;\mbox{ with }\lambda(\Gamma)\geq{}2\]
satisfy the condition (\ref{trekantulighed}). The two other conditions of theorem \ref{thmtre21} are trivially satisfied by the $\alpha'$s coming from \trestar{i}{j}{k}. If $\lambda_{_{0}}$ denotes the Perron-Frobenius eigenvalue of the adjacency matrix of $\trestar{i}{j}{k},\;i\leq{}j\leq{}k$ then the condition is
\[\frac{1}{\Rnlam{i}{0}}\leq{}\frac{1}{\Rnlam{j}{0}}+\frac{1}{\Rnlam{k}{0}},\]
or, if we multiply by $\Rnlam{i}{0}$
\begin{equation}
\label{tregrafeqn1}
1\leq{}\Rnlamkvot{i}{j}{0}+\Rnlamkvot{i}{k}{0}.
\end{equation}
\begin{lemma}
\label{tregraflem1}For $m\geq{}1,$ $n\geq{}0,$ and $\lambda\geq{}2,$ $\Rnlamkvot{n}{n+m}{}$ has the following properties
\begin{enumerate}
\item{}$\Rnlamkvot{n}{n+m}{}$ is decreasing in $\lambda.$
\item{}$\Rnlamkvot{n}{n+m}{}$ is increasing in $n.$
\item{}$\Rnlamkvot{n}{n+m}{}$ is decreasing in $m.$
\item{}$\lim_{n\rightarrow\infty}\Rnlamkvot{n}{n+m}{}=e^{-mx}.$
\end{enumerate}
\end{lemma}
\begin{proof}We refer to the proof of corollary \ref{firegrafcor0} and remark \ref{firegrafrem1}.
\end{proof}

\begin{lemma}
\label{tregraflem2}
\hfill
\begin{enumerate}
\item{}For $\lam{0} = \lambda(\trestar{j}{j+1}{j+1})$ we have $\Rnlamkvot{j}{j+1}{0} =\frac{1}{\sqrt{2}}.$
\item{}For $\lam{0} = \lambda(\trestar{j}{j+2}{j+2})$ we have $\Rnlamkvot{j}{j+2}{0} =\frac{1}{2}.$
\end{enumerate}
\end{lemma}
\begin{proof}
1. \lam{0} satisfies the equation
\[\begin{array}{cccl}
 & \lam{0} & = & \Rnlamkvot{j-1}{j}{0} +2\Rnlamkvot{j}{j+1}{0}\\
\Downarrow & & &\\
 & 2\Rnlamkvot{j}{j+1}{0} &= &\frac{\lambda_{_{0}}\Rnlam{j}{0}-\Rnlam{j-1}{0}}{\Rnlam{j}{0}}\\
\Downarrow & & &\\
 & 2\Rnlamkvot{j}{j+1}{0} &= &\Rnlamkvot{j+1}{j}{0}\\
\Downarrow & & & \\
 & \Rnlamkvot{j}{j+1}{0} & = & \frac{1}{\sqrt{2}}
\end{array}\]

\noindent{}2. Here \lam{0} satisfies the equation
\[\begin{array}{cccl}
 & \lam{0} & = & \Rnlamkvot{j-1}{j}{0} +2\Rnlamkvot{j+1}{j+2}{0}\\
\Downarrow & & &\\
 & \Rnlamkvot{j+1}{j}{0} &= &2\Rnlamkvot{j+1}{j+2}{0}\\
\Downarrow & & & \\
 & \Rnlamkvot{j}{j+2}{0} & = & \frac{1}{2}.
\end{array}\]
\end{proof}

\begin{lemma}
\label{tregraflem3}
If \trestar{j}{k}{l}, $j\leq{}k\leq{}l$ does not satisfy (\ref{tregrafeqn1}), then \trestar{j}{k+k'}{l+l'}, \\$j\leq{}k+k'\leq{}l+l'$  does not satisfy (\ref{tregrafeqn1}).
\end{lemma}
\begin{proof}
Let $\lam{0}=\lambda(\trestar{j}{k}{l})$ and $\lam{1}=\lambda(\trestar{j}{k+k'}{l+l'}).$ Then $\lam{1}\geq{}\lam{0}$ and we have
\[\begin{array}{lcl}
\Rnlamkvot{j}{k+k'}{1} + \Rnlamkvot{j}{l+l'}{1} & \leq & \;\;\;\;\;\;\mbox{ (by lemma \ref{tregraflem1} (1)) }\\[0.3cm]
\Rnlamkvot{j}{k+k'}{0} + \Rnlamkvot{j}{l+l'}{0} & \leq & \;\;\;\;\;\;\mbox{ (by lemma \ref{tregraflem1} (2)) }\\[0.3cm]
\Rnlamkvot{j}{k}{0} + \Rnlamkvot{j}{l}{0} & < & 1.
\end{array}\]
Hence \trestar{j}{k+k'}{l+l'}, $j\leq{}k+k'\leq{}l+l'$  does not satisfy (\ref{tregrafeqn1}).
\end{proof}
\begin{prop}
\trestar{j}{k}{l}, $j\leq{}k\leq{}l$ satisfies (\ref{tregrafeqn1}) if and only if $(j,k,l)$ is of one of the following
\[\begin{array}{rcl}
1)&(j,j,j+n) & j\geq{}2,\;n\geq{}0\\
2)&(j,j+1,j+1)& j\geq{}2\\
3)&(j,j+1,j+2)&j\geq{}2\\
4)&(j,j+1,j+3)&j\geq{}2\\
5)&(j,j+2,j+2)&j\geq{}1
\end{array}\]
\end{prop}
\begin{proof} Let $\lam{\infty}$ denote the largest eigenvalue of ``\trestar{\infty}{\infty}{\infty}''. It is not difficult to show that $\lam{\infty}=\frac{3\sqrt{2}}{2}.$

\noindent{}{\bf 1)} If $\lam{0}=\lambda(\trestar{j}{j}{j+n}),$ then clearly $\Rnlamkvot{j}{j}{0}+\Rnlamkvot{j}{j+n}{0}\geq{}1.$\\[0.3cm]

\noindent{}{\bf 2)} Let $\lam{0}=\lambda(\trestar{j}{j+1}{j+1}).$ We must show that $2\Rnlamkvot{j}{j+1}{0}\geq{}1.$ By lemma \ref{tregraflem1} we have
\[2\Rnlamkvot{j}{j+1}{0}\geq{}2\Rnlamkvot{1}{2}{0}\geq{}2\Rnlamkvot{1}{2}{\infty} =\frac{6\sqrt{2}}{7}>1.\]
Hence \trestar{j}{j+1}{j+1} satisfies (\ref{tregrafeqn1}).\\[0.3cm]
\noindent{}{\bf 3)}  Let $\lam{0}=\lambda(\trestar{j}{j+1}{j+2}).$ We must show that $\Rnlamkvot{j}{j+1}{0}+\Rnlamkvot{j}{j+2}{0}\geq{}1.$ We have
\[\Rnlamkvot{j}{j+1}{0}\geq{}\Rnlamkvot{1}{2}{0}\geq{}\Rnlamkvot{1}{2}{\infty}= \frac{3\sqrt{2}}{7}\]
\[\Rnlamkvot{j}{j+2}{0}\geq{}\Rnlamkvot{1}{3}{0}\geq{}\Rnlamkvot{1}{3}{\infty}= \frac{2}{5}.\]
So  $\Rnlamkvot{j}{j+1}{0}+\Rnlamkvot{j}{j+2}{0}\geq{}\frac{3\sqrt{2}}{7}+\frac{2}{5}>1,$ and hence \trestar{j}{j+1}{j+2} satisfies (\ref{tregrafeqn1}).\\[0.3cm]
\noindent{}{\bf 4)}  Let $\lam{0}=\lambda(\trestar{j}{j+1}{j+3}).$ We must show that $\Rnlamkvot{j}{j+1}{0}+\Rnlamkvot{j}{j+3}{0}\geq{}1.$

\noindent{}For $j\geq{}3$ it is enough to show
\[\Rnlamkvot{3}{4}{\infty}+\Rnlamkvot{3}{6}{\infty}\geq{}1.\]
And since $\Rnlamkvot{3}{4}{\infty}\approx{}0.68$ and $\Rnlamkvot{3}{6}{\infty}\approx{}0.33$ this is satisfied.\\[0.2cm]
For $j=2$ we have $\lam{0}\approx{}2.0697.$ Put $\lam{1}=2.1,$ then all we need to show is
\[\Rnlamkvot{2}{3}{1}+\Rnlamkvot{2}{5}{1}\geq{}1.\]
And since $\Rnlamkvot{2}{3}{1}\approx{}0.67$ and $\Rnlamkvot{2}{5}{1}\approx{}0.33$ this is satisfied.\\[0.2cm]

\noindent{}{\bf 5)}  Let $\lam{0}=\lambda(\trestar{j}{j+2}{j+2}).$ We must show that $2\Rnlamkvot{j}{j+2}{0}\geq{}1.$ $\lam{0}$ satisfies the equation
\[\begin{array}{cccl}
& \lam{0} & = & \Rnlamkvot{j-1}{j}{0}+2\Rnlamkvot{j+1}{j+2}{0}\\
\Downarrow & & & \\
& \lam{0}\Rnlam{j}{0}\Rnlam{j+2}{0} & = & \Rnlam{j-1}{0}\Rnlam{j+2}{0} +2\Rnlam{j+1}{0}\Rnlam{j}{0} \\
\Downarrow & & & \\
& 2\Rnlam{j+1}{0}\Rnlam{j}{0}&=&\Rnlam{j+2}{0}(\lam{0}\Rnlam{j}{0}-\Rnlam{j-1}{0})\\
\Downarrow & & & \\
& 2\Rnlam{j+1}{0}\Rnlam{j}{0}&=&\Rnlam{j+2}{0}\Rnlam{j+1}{0})\\
\Downarrow & & & \\
&2\Rnlamkvot{j}{j+2}{0} & = & 1.
\end{array}\]
Hence \trestar{j}{j+2}{j+2} satisfies (\ref{tregrafeqn1}).

To show that these are the only values of $(j,k,l),$ for which \trestar{j}{k}{l} satisfies (\ref{tregrafeqn1}), we argue as follows.

By lemma \ref{tregraflem3} it suffices to show that \trestar{j}{j+1}{j+4} and \trestar{j}{j+2}{j+3} do not satisfy (\ref{tregrafeqn1}).

Let $\lam{0}=\lambda(\trestar{j}{j+1}{j+m}),\;m\geq{}1$ and $\lam{1}=\lambda(\trestar{j}{j+1}{j+1}),$ then $\lam{1}\leq{}\lam{0}$ and by lemmas \ref{tregraflem1},  \ref{tregraflem2} we have
\[\Rnlamkvot{j}{j+1}{0} + \Rnlamkvot{j}{j+m}{0}\leq{}\Rnlamkvot{j}{j+1}{1}+e^{-mx}=\mbox{$\frac{1}{\sqrt{2}}$}+e^{-mx}.\]

\noindent{}For $m=4,$ $\frac{1}{\sqrt{2}}+e^{-4x}<1$ corresponds to $x>\frac{\ln(\sqrt{2})-\ln(\sqrt{2}-1)}{4},$ which again corresponds to\\$\lambda_{_{0}}^{2}>4.3889\cdots.$ We have $\lambda^{2}(\trestar{3}{4}{7})\approx{}4.4107,$ hence we have excluded \trestar{j}{j+1}{j+m} for $j\geq{}3$ and $m\geq{}4.$

\noindent{}For $m=5,$ $\frac{1}{\sqrt{2}}+e^{-5x}<1$ corresponds to $\lambda_{_{0}}^{2}>4.2461\cdots.$ We have $\lambda^{2}(\trestar{2}{3}{7})\approx{}4.3027,$ hence we have excluded \trestar{2}{3}{2+m} for $j\geq{}3$ and $m\geq{}5.$ 

We are left with the case \trestar{2}{3}{6}. Let $\lam{0}=\lambda(\trestar{2}{3}{6})\approx{}2.0728,$ and let $\lam{1}=2.07.$ Then 
\[\Rnlamkvot{2}{3}{0} + \Rnlamkvot{2}{6}{0}\leq{}\Rnlamkvot{2}{3}{1}+\Rnlamkvot{2}{6}{1}\approx{} 0.69+0.28<1.\]

Let $\lam{0}=\lambda(\trestar{j}{j+2}{j+m})$ and $\lam{1}=\lambda(\trestar{j}{j+2}{j+2}).$ As before we have
\[\Rnlamkvot{j}{j+2}{0} + \Rnlamkvot{j}{j+m}{0}\leq{}\Rnlamkvot{j}{j+2}{1}+e^{-mx}=\mbox{$\frac{1}{2}$}+e^{-mx}.\]
\noindent{}For $m=3,$ $\frac{1}{2}+e^{-3x}<1$ corresponds to $x>\frac{\ln2}{3},$ which again corresponds to $\lambda_{_{0}}^{2}>4.2173\cdots.$ We have $\lambda^{2}(\trestar{2}{4}{5})\approx{}4.3235,$ hence we have excluded \trestar{j}{j+2}{j+m} for $j\geq{}2$ and $m\geq{}3.$

\noindent{}For \trestar{1}{3}{4} we have $\lam{0}\approx{}2.0153.$ Put $\lam{1}=2.01,$ then 
\[\Rnlamkvot{1}{3}{0} + \Rnlamkvot{1}{4}{0}\leq{}\Rnlamkvot{1}{3}{1}+\Rnlamkvot{1}{4}{1}\approx{} 0.49+0.38<1,\] and we conclude that \trestar{j}{j+2}{j+m} does not satisfy (\ref{tregrafeqn1}) for $j\geq{}1$ and $m\geq{}3.$
\end{proof}
The above determined graphs give rise to the following values of the index (all terms are approximate), of irreducible subfactors of the hyperfinite $II_{1}-$factor.

\begin{center}
{\bf Tables of the Index Values Corresponding to the Graphs:\\[5mm]}
\begin{tabular}{|r|c|c|c|c|}
\hline
j&S(j,j+1,j+1)&S(j,j+1,j+2)&
S(j,j+1,j+3)&S(j,j+2,j+2)\\
\hline

2 &    4.214320 & 
   4.260757 & 
   4.283998 & 
   4.302776\\  
   3 &    4.379878 & 
   4.397514 & 
   4.406262 & 
   4.414214\\  
   4 &    4.445787 & 
   4.453260 & 
   4.456966 & 
   4.460505\\  
   5 &    4.474491 & 
   4.477873 & 
   4.479553 & 
   4.481194\\  
   6 &    4.487695 & 
   4.489287 & 
   4.490080 & 
   4.490864\\  
   7 &    4.493975 & 
   4.494744 & 
   4.495127 & 
   4.495508\\  
   8 &    4.497024 & 
   4.497400 & 
   4.497588 & 
   4.497775\\  
   9 &    4.498522 & 
   4.498708 & 
   4.498801 & 
   4.498894\\  
  10 &    4.499264 & 
   4.499356 & 
   4.499402 & 
   4.499449\\  
  11 &    4.499633 & 
   4.499679 & 
   4.499702 & 
   4.499725\\  
  12 &    4.499817 & 
   4.499840 & 
   4.499851 & 
   4.499863\\  
  13 &    4.499908 & 
   4.499920 & 
   4.499926 & 
   4.499931\\  
  14 &    4.499954 & 
   4.499960 & 
   4.499963 & 
   4.499966\\  
  15 &    4.499977 & 
   4.499980 & 
   4.499981 & 
   4.499983\\  
  16 &    4.499989 & 
   4.499990 & 
   4.499991 & 
   4.499991\\  
  17 &    4.499994 & 
   4.499995 & 
   4.499995 & 
   4.499996\\  
  18 &    4.499997 & 
   4.499997 & 
   4.499998 & 
   4.499998\\  
  19 &    4.499999 & 
   4.499999 & 
   4.499999 & 
   4.499999\\  
  20 &    4.499999 & 
   4.499999 & 
   4.499999 & 
   4.499999\\  
  21 &    4.500000 & 
   4.500000 & 
   4.500000 & 
   4.500000\\  
\hline
\end{tabular}\hspace{4mm}
\begin{tabular}{|r|c|}
\hline
$j$ & $S(j,j,\infty)$ \\ \hline 
2  &   4.236068 \\ 
3   &   4.382976 \\ 
4 &   4.446352 \\ 
5  &  4.474609 \\ 
6&   4.487721 \\ 
7&   4.493981 \\ 
8&   4.497025 \\ 
9&   4.498522 \\ 
10&   4.499264\\  
11&   4.499633 \\ 
12 &   4.499817 \\ 
13 &   4.499908 \\ 
14 &   4.499954 \\ 
15 &   4.499977 \\ 
16 &   4.499989 \\ 
17 &   4.499994 \\ 
18 &   4.499997 \\ 
19 &   4.499999 \\ 
20 &   4.499999 \\ 
21 &   4.500000 \\ \hline
\end{tabular}\\
\end{center}
\begin{center}\begin{tabular}{|r|c|c|c|c|c|c|}\hline
$l $&  S$(2,2,2+l)$ &S$(3,3,3+l)$ &
     S$(4,4,4+l)$ &S$(5,5,5+l)$&S$(6,6,6+l)$&S$(7,7,7+l)$\\  
     \hline 
 0 &    4.000000 & 
   4.302776 & 
   4.414214 & 
   4.460505 & 
   4.481194 & 
   4.490864\\  
   1 &    4.114908 & 
   4.342923 & 
   4.430385 & 
   4.467599 & 
   4.484472 & 
   4.492427\\  
   2 &    4.170086 & 
   4.362340 & 
   4.438283 & 
   4.471092 & 
   4.486095 & 
   4.493204\\  
   3 &    4.198691 & 
   4.372130 & 
   4.442232 & 
   4.472834 & 
   4.486905 & 
   4.493592\\  
   4 &    4.214320 & 
   4.377203 & 
   4.444234 & 
   4.473710 & 
   4.487311 & 
   4.493786\\  
   5 &    4.223177 & 
   4.379878 & 
   4.445259 & 
   4.474153 & 
   4.487514 & 
   4.493883\\  
   6 &    4.228328 & 
   4.381305 & 
   4.445787 & 
   4.474377 & 
   4.487617 & 
   4.493932\\  
   7 &    4.231379 & 
   4.382072 & 
   4.446059 & 
   4.474491 & 
   4.487669 & 
   4.493957\\  
   8 &    4.233210 & 
   4.382486 & 
   4.446200 & 
   4.474549 & 
   4.487695 & 
   4.493969\\  
   9 &    4.234318 & 
   4.382710 & 
   4.446273 & 
   4.474578 & 
   4.487708 & 
   4.493975\\  
  10 &    4.234993 & 
   4.382831 & 
   4.446311 & 
   4.474593 & 
   4.487714 & 
   4.493978\\  
  11 &    4.235407 & 
   4.382897 & 
   4.446331 & 
   4.474601 & 
   4.487718 & 
   4.493980\\  
  12 &    4.235660 & 
   4.382933 & 
   4.446341 & 
   4.474605 & 
   4.487719 & 
   4.493980\\  
  13 &    4.235817 & 
   4.382953 & 
   4.446346 & 
   4.474607 & 
   4.487720 & 
   4.493981\\  
  14 &    4.235913 & 
   4.382963 & 
   4.446349 & 
   4.474608 & 
   4.487721 & 
   4.493981\\  
  15 &    4.235972 & 
   4.382969 & 
   4.446351 & 
   4.474608 & 
   4.487721 & 
   4.493981\\  
  16 &    4.236009 & 
   4.382972 & 
   4.446351 & 
   4.474608 & 
   4.487721 & 
   4.493981\\  
  17 &    4.236031 & 
   4.382974 & 
   4.446352 & 
   4.474608 & 
   4.487721 & 
   4.493981\\  
  18 &    4.236045 & 
   4.382975 & 
   4.446352 & 
   4.474609 & 
   4.487721 & 
   4.493981\\  
  19 &    4.236054 & 
   4.382975 & 
   4.446352 & 
   4.474609 & 
   4.487721 & 
   4.493981\\  
  20 &    4.236059 & 
   4.382975 & 
   4.446352 & 
   4.474609 & 
   4.487721 & 
   4.493981\\
\hline
\end{tabular}
\end{center}

\setcounter{equation}{0}
\setcounter{equation}{0}

\section{Ocneanu's example of a Commuting Square Based on the Graph $E_{10}=S(1,2,6).$}
The construction in this section is due to A. Ocneanu \cite{Ocn}.
As shown in chapter \ref{chapter1} and chapter  \ref{trestarchap} it is not possible to construct commuting squares of one of the simple forms
\[\begin{array}{lcl}
      {\cal B} & \subset_{G^{^{t}}} & {\cal D} \\
       \cup_{G} &\,&\cup_{G^{^{t}}}\\
      {\cal A} & \subset_{G} & {\cal C}  \end{array}\hspace{1cm}
\begin{array}{lcl}
       {\cal B}  & \subset_{G^{^{t}}G-I} & {\cal D} \\
       \cup_{G} &\,&\cup_{G}\\
       {\cal A} & \subset_{GG^{^{t}}-I} & {\cal C}  \end{array},\]
where $G$ is the inclusion matrix with Bratteli-diagram equal to $E_{10}.$ The following example was found  by a search on a computer for a polynomial $P,$ such that $P(GG^{^{t}})$ and $P(G^{^{t}}G)$ matrices with small non-negative integer entries. It turned out that
\[P(t)= t^{4}-8t^{3}+20t^{2}-16t+3\]
could be used to produce a commuting square with inclusions
\begin{equation}
\label{specialcsq}
\begin{array}{lcl}
       {\cal C}  & \subset_{P(G^{^{t}}G)} & {\cal D} \\
       \cup_{G} &\,&\cup_{G}\\
       {\cal A} & \subset_{P(GG^{^{t}})} & {\cal B}  \end{array}
\end{equation}
for $E_{10}.$ Note that $P(GG^{^{t}})G=GP(G^{^{t}}G)$ and $G^{^{t}}P(GG^{^{t}})=P(G^{^{t}}G)G^{^{t}},$ so any commuting square with these inclusion matrices will be symmetric in the sense of  \ref{symmetricdefn}.

Let us label the vertices of $\;\Gamma = E_{10}$ as follows
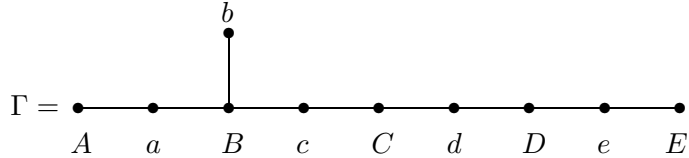
\begin{center}
\setlength{\unitlength}{1cm}
\begin{picture}(10,3)
\put(0.1,0.9){$\Gamma=$}
\put(1,1){\line(1,0){8}}
\put(3,1){\line(0,1){1}}
\multiput(1,1)(1,0){9}{\circle*{0.15}}
\put(3,2){\circle*{0.15}}
\put(0.9,0.4){$A$}
\put(1.9,0.4){$a$}
\put(2.9,0.4){$B$}
\put(3.9,0.4){$c$}
\put(4.9,0.4){$C$}
\put(5.9,0.4){$d$}
\put(6.9,0.4){$D$}
\put(7.9,0.4){$e$}
\put(8.8,0.4){$E$}
\put(2.9,2.15){$b$}
\end{picture}
\end{center}
where $(A,B,C,D,E)$ and $(a,b,c,d,e)$ correspond to the two layers in a bi-partition of $\Gamma.$

\noindent{}Writing the vertices in the order $(A,B,C,D,E,a,b,c,d,e)$ the adjacency matrix of $\Gamma$ is
\[\Delta_{_{\Gamma}}=
\left(\begin{array}{cc}
0 & G\\G^{^{t}} & 0\end{array}\right),\]
where
\[G=\left(\begin{array}{ccccc}
1&0&0&0&0\\
1&1&1&0&0\\
0&0&1&1&0\\
0&0&0&1&1\\
0&0&0&0&1\end{array}\right),\]
and with the above polynomial one gets
\[P(GG^{^{t}})=\left(\begin{array}{ccccc}
0&1&0&0&1\\
1&2&1&1&0\\
0&1&2&0&0\\
0&1&0&1&0\\
1&0&0&0&1\end{array}\right),\;\;\;P(G^{^{t}}G)=\left(\begin{array}{ccccc}
1&1&1&0&1\\
1&1&0&1&0\\
1&0&2&1&0\\
0&1&1&1&0\\
1&0&0&0&1\end{array}\right).\]
Note that $P(GG^{^{t}})$ mixes the elements of the ``upper case letters'' of $\Gamma$ strongly, in the sense that there is an edge joining the endpoints $A$ and $E$ of the two long legs of the graph.

Let $c({\cal A})$ (respective $c({\cal B}),$ $c({\cal C})$ and $c({\cal D})$) denote the set of minimal central projections in \ca{} (respectively \cc{}, \cc{} and \cd{}). Then, with the chosen bi-partition of $E_{10},$ the elements of $c(\ca)$ and $c(\cc)$ are labeled by $(A,B,C,D,E),$  and the elements of $c(\cb)$ and $c(\cd)$ are labeled by $(a,b,c,d,e).$

To prove the existence of a commuting square of the form (\ref{specialcsq}) is equivalent to constructing a unitary matrix, $u,$ satisfying the bi-unitarity condition (\ref{biunitcond}). $u$ is of the form
\[u=\bigoplus_{(p,s)}u^{(p,s)}\]
where $(p,s)$ runs over all $p\in{}c(\ca)$ and all $s\in{}c(\cd)$ which are connected by a path (either through $c(\cc)$ or through $c(\cb)$). Each direct summand of $u^{(p,s)}$ is a $n(p,s)\times{}n(p,s)-$matrix, where $n(p,s)$ is the number of paths from $p$ to $s$ through $c(\cc)$ (or $c(\cb)$), so each $u^{(p,s)}$ is a block matrix, indexed as follows
\[u^{(p,s)} = \left(u^{(p,s)}_{q,r}\right)_{q,r},\]
where $(q,r)$ runs over all possible $r\in{}c(\cc)$ and  $q\in{}c(\cb),$ that a path from $p$ to $s$ can go through. Since the vertical inclusions, $\ca\subseteq\cb$ and $\cc\subseteq\cd,$ do not have multiple edges, each $u^{(p,s)}_{q,r}$ is a $m\times{}n-$matrix, where $m$ is the multiplicity of the edge $qs$ in the inclusion $\cb\subseteq\cd$ and  $n$ is the multiplicity of the edge $pr$ in the inclusion $\ca\subseteq\cc.$ 

Let $\xi(\cdot)$ (resp. $\eta(\cdot)$) denote the Perron-Frobenius vector for the graph of $\ca\subseteq\cb$ (resp. $\cc\subseteq\cd$). Set
\[w(p,q,r,s)=\sqrt{\frac{\xi(p)\eta(s)}{\xi(q)\eta(r)}}.\]
Define a matrix $v$ by
\[v=\bigoplus_{(q,r)}v^{(q,r)},\]
where $v^{(q,r)}$ is a square matrix, which can be written as a block matrix $v^{(q,r)}=\left(v^{(q,r)}_{p,s}\right)_{p,s},$ with each block  given by
\begin{equation}
\label{e10vdef}
v^{(q,r)}_{p,s}=w(p,q,r,s)\left(u^{(p,s)}_{q,r}\right)^{t},
\end{equation}
where $(p,q,r,s)$ runs through all quadruples in $c(\ca)\times{}c(\cb)\times{}c(\cc)\times{}c(\cd),$ which can be completed to a cycle $p-r-s-q-p.$ 

The bi-unitary condition says, that one should be able to choose a unitary $u$ as above, such that $v$ is also unitary.

The possible quadruples $(p,r,q,s)$ are completely determined by the two vertical edges $pq$ and $rs.$ Since $E_{10}$ has $9$ edges, there are at most $9\times{}9$ blocks in $u$ and $v$. The diagrams for $u$ and $v$ on the next page  show which combinations of $(pq,rs)$ occur. The dots indicates the size of a block, and the frames tell which blocks in $u$ (resp. $v$) belong to the same direct summand of $u$ (resp. $v$). The two figures are easily deduced from the given inclusion matrices.

\noindent{}Note how one figure can be obtained from the other by reflecting in the main diagonal.

\label{e10page}
\begin{center}
\setlength{\unitlength}{0.6cm}
\begin{picture}(19,19)
\put(7.5,19){Diagram for $u$.}
\multiput(1,0)(2,0){10}{\line(0,1){18}}
\multiput(1,0)(0,2){10}{\line(1,0){18}}
\put(1,18){\line(-1,1){1}}
\put(1,18){\line(-1,0){1}}
\put(1,18){\line(0,1){1}}
\put(-0.2,18.2){$pq$}
\put(0.4,18.8){$rs$}
\put(1.7,18.2){$Aa$}
\put(3.7,18.2){$Ba$}
\put(5.7,18.2){$Bb$}
\put(7.7,18.2){$Bc$}
\put(9.7,18.2){$Cc$}
\put(11.7,18.2){$Cd$}
\put(13.7,18.2){$Dd$}
\put(15.7,18.2){$De$}
\put(17.7,18.2){$Ee$}
\put(0,16.9){$Aa$}
\put(0,14.9){$Ba$}
\put(0,12.9){$Bb$}
\put(0,10.9){$Bc$}
\put(0,8.9){$Cc$}
\put(0,6.9){$Cd$}
\put(0,4.9){$Dd$}
\put(0,2.9){$De$}
\put(0,0.9){$Ee$}

\put(3.9,16.9){$\bullet$}
\put(5.9,16.9){$\bullet$}
\put(7.9,16.9){$\bullet$}
\put(17.9,16.9){$\bullet$}

\put(1.9,14.9){$\bullet$}
\put(4.15,14.9){$\bullet$}
\put(3.65,14.9){$\bullet$}
\put(6.15,14.9){$\bullet$}
\put(5.65,14.9){$\bullet$}
\put(8.15,14.9){$\bullet$}
\put(7.65,14.9){$\bullet$}
\put(9.9,14.9){$\bullet$}
\put(16,14.9){$\bullet$}

\put(1.9,12.9){$\bullet$}
\put(4.15,12.9){$\bullet$}
\put(3.65,12.9){$\bullet$}
\put(6.15,12.9){$\bullet$}
\put(5.65,12.9){$\bullet$}
\put(11.9,12.9){$\bullet$}
\put(13.9,12.9){$\bullet$}

\put(1.9,10.9){$\bullet$}
\put(4.15,10.9){$\bullet$}
\put(3.65,10.9){$\bullet$}
\put(8.15,11.2){$\bullet$}
\put(7.65,11.2){$\bullet$}
\put(8.15,10.6){$\bullet$}
\put(7.65,10.6){$\bullet$}
\put(9.9,10.6){$\bullet$}
\put(9.9,11.2){$\bullet$}
\put(11.9,10.9){$\bullet$}
\put(13.9,10.9){$\bullet$}

\put(3.9,8.9){$\bullet$}
\put(7.9,8.6){$\bullet$}
\put(7.9,9.2){$\bullet$}
\put(10.15,9.2){$\bullet$}
\put(9.65,9.2){$\bullet$}
\put(10.15,8.6){$\bullet$}
\put(9.65,8.6){$\bullet$}
\put(12.15,8.9){$\bullet$}
\put(11.65,8.9){$\bullet$}

\put(5.9,6.9){$\bullet$}
\put(7.9,6.9){$\bullet$}
\put(9.65,6.9){$\bullet$}
\put(10.15,6.9){$\bullet$}
\put(11.65,6.9){$\bullet$}
\put(12.15,6.9){$\bullet$}

\put(5.9,4.9){$\bullet$}
\put(7.9,4.9){$\bullet$}
\put(13.9,4.9){$\bullet$}

\put(3.9,2.9){$\bullet$}
\put(15.9,2.9){$\bullet$}

\put(1.9,0.9){$\bullet$}
\put(17.9,0.9){$\bullet$}
\linethickness{2pt}
\put(3,18){\line(1,0){6}}
\put(17,18){\line(1,0){2}}
\put(1,16){\line(1,0){10}}
\put(15,16){\line(1,0){4}}
\put(11,14){\line(1,0){6}}
\put(5,12){\line(1,0){2}}
\put(1,10){\line(1,0){4}}
\put(7,10){\line(1,0){8}}
\put(3,8){\line(1,0){4}}
\put(5,6){\line(1,0){10}}
\put(3,4){\line(1,0){6}}
\put(13,4){\line(1,0){4}}
\put(1,2){\line(1,0){4}}
\put(15,2){\line(1,0){4}}
\put(1,0){\line(1,0){2}}
\put(17,0){\line(1,0){2}}
\put(1,16){\line(0,-1){6}}
\put(1,2){\line(0,-1){2}}
\put(3,18){\line(0,-1){2}}
\put(3,10){\line(0,-1){2}}
\put(3,4){\line(0,-1){4}}
\put(5,18){\line(0,-1){16}}
\put(7,18){\line(0,-1){14}}
\put(9,18){\line(0,-1){2}}
\put(9,6){\line(0,-1){2}}
\put(11,16){\line(0,-1){10}}
\put(13,10){\line(0,-1){6}}
\put(15,16){\line(0,-1){6}}
\put(15,6){\line(0,-1){4}}
\put(17,18){\line(0,-1){4}}
\put(17,4){\line(0,-1){4}}
\put(19,18){\line(0,-1){2}}
\put(19,2){\line(0,-1){2}}
\end{picture}
\end{center}\vspace{0.3cm}
\begin{center}
\setlength{\unitlength}{0.6cm}
\begin{picture}(19,19)
\put(7.5,19){Diagram for $v$.}
\multiput(1,0)(2,0){10}{\line(0,1){18}}
\multiput(1,0)(0,2){10}{\line(1,0){18}}
\put(1,18){\line(-1,1){1}}
\put(1,18){\line(-1,0){1}}
\put(1,18){\line(0,1){1}}
\put(-0.2,18.2){$pq$}
\put(0.4,18.8){$rs$}
\put(1.7,18.2){$Aa$}
\put(3.7,18.2){$Ba$}
\put(5.7,18.2){$Bb$}
\put(7.7,18.2){$Bc$}
\put(9.7,18.2){$Cc$}
\put(11.7,18.2){$Cd$}
\put(13.7,18.2){$Dd$}
\put(15.7,18.2){$De$}
\put(17.7,18.2){$Ee$}
\put(0,16.9){$Aa$}
\put(0,14.9){$Ba$}
\put(0,12.9){$Bb$}
\put(0,10.9){$Bc$}
\put(0,8.9){$Cc$}
\put(0,6.9){$Cd$}
\put(0,4.9){$Dd$}
\put(0,2.9){$De$}
\put(0,0.9){$Ee$}

\put(3.9,16.9){$\bullet$}
\put(5.9,16.9){$\bullet$}
\put(7.9,16.9){$\bullet$}
\put(17.9,16.9){$\bullet$}

\put(1.9,14.9){$\bullet$}
\put(3.9,14.6){$\bullet$}\put(3.9,15.2){$\bullet$}
\put(5.9,14.6){$\bullet$}\put(5.9,15.2){$\bullet$}
\put(7.9,14.6){$\bullet$}\put(7.9,15.2){$\bullet$}
\put(9.9,14.9){$\bullet$}
\put(16,14.9){$\bullet$}

\put(1.9,12.9){$\bullet$}
\put(3.9,12.6){$\bullet$}\put(3.9,13.2){$\bullet$}
\put(5.9,12.6){$\bullet$}\put(5.9,13.2){$\bullet$}
\put(11.9,12.9){$\bullet$}
\put(13.9,12.9){$\bullet$}

\put(1.9,10.9){$\bullet$}
\put(3.9,10.6){$\bullet$}\put(3.9,11.2){$\bullet$}
\put(8.15,11.2){$\bullet$}
\put(7.65,11.2){$\bullet$}
\put(8.15,10.6){$\bullet$}
\put(7.65,10.6){$\bullet$}
\put(9.65,10.9){$\bullet$}\put(10.15,10.9){$\bullet$}
\put(11.9,10.9){$\bullet$}
\put(13.9,10.9){$\bullet$}

\put(3.9,8.9){$\bullet$}
\put(7.65,8.9){$\bullet$}\put(8.15,8.9){$\bullet$}
\put(10.15,9.2){$\bullet$}
\put(9.65,9.2){$\bullet$}
\put(10.15,8.6){$\bullet$}
\put(9.65,8.6){$\bullet$}
\put(11.9,8.6){$\bullet$}\put(11.9,9.2){$\bullet$}

\put(5.9,6.9){$\bullet$}
\put(7.9,6.9){$\bullet$}
\put(9.9,6.6){$\bullet$}\put(9.9,7.2){$\bullet$}
\put(11.9,7.2){$\bullet$}\put(11.9,6.6){$\bullet$}

\put(5.9,4.9){$\bullet$}
\put(7.9,4.9){$\bullet$}
\put(13.9,4.9){$\bullet$}

\put(3.9,2.9){$\bullet$}
\put(15.9,2.9){$\bullet$}

\put(1.9,0.9){$\bullet$}
\put(17.9,0.9){$\bullet$}
\linethickness{2pt}

\put(3,18){\line(1,0){6}}
\put(17,18){\line(1,0){2}}
\put(1,16){\line(1,0){2}}
\put(9,16){\line(1,0){2}}
\put(15,16){\line(1,0){4}}
\put(1,14){\line(1,0){16}}
\put(1,12){\line(1,0){14}}
\put(1,10){\line(1,0){2}}
\put(3,8){\line(1,0){10}}
\put(9,6){\line(1,0){6}}
\put(3,4){\line(1,0){6}}
\put(13,4){\line(1,0){4}}
\put(1,2){\line(1,0){4}}
\put(15,2){\line(1,0){4}}
\put(1,0){\line(1,0){2}}
\put(17,0){\line(1,0){2}}

\put(1,16){\line(0,-1){6}}
\put(1,2){\line(0,-1){2}}
\put(3,18){\line(0,-1){10}}
\put(3,4){\line(0,-1){4}}
\put(5,8){\line(0,-1){6}}
\put(7,14){\line(0,-1){2}}
\put(9,18){\line(0,-1){4}}
\put(9,12){\line(0,-1){8}}
\put(11,16){\line(0,-1){4}}
\put(13,14){\line(0,-1){10}}
\put(15,16){\line(0,-1){4}}
\put(15,6){\line(0,-1){4}}
\put(17,18){\line(0,-1){4}}
\put(17,4){\line(0,-1){4}}
\put(19,18){\line(0,-1){2}}
\put(19,2){\line(0,-1){2}}
\end{picture}
\end{center}

\newpage{}
\noindent{}The two figures show that $u$ and $v$ both have 20 direct summands, namely
3 $3\times{}3$ matrices, 3 $2\times{}2$ matrices and 14 scalar matrices,
and  these are subdivided in 36 blocks $u_{q,r}^{(p,s)}$ respectively
$v_{q,r}^{(p,s)}.$

As in chapter \ref{chapter1} we let $R_{n}(\lambda)$ denote the $n'$th degree
polynomial, defined inductively by
\[R_{0}(\lambda)=1,\;\;R_{1}(\lambda)=\lambda,\;\;R_{n}(\lambda)=
\lambda{}R_{n-1}(\lambda)-R_{n-2}(\lambda).\]
For $\lambda\geq{}2,$  $R_{n}(\lambda)$ is positive for all $n$ and increasing
in $n.$ For simplicity we will denote $R_{n}(\lambda(E_{10}))$ by $R_{n},$
where $\lambda(E_{10})$ is the largest eigenvalue of the adjacency matrix  of $E_{10}.$\\$\lambda(E_{10})\approx{}2.006594.$ The corresponding
eigenvector (the Perron-Frobenius eigenvector) $\xi$ is given by (see chapter \ref{chapter1})

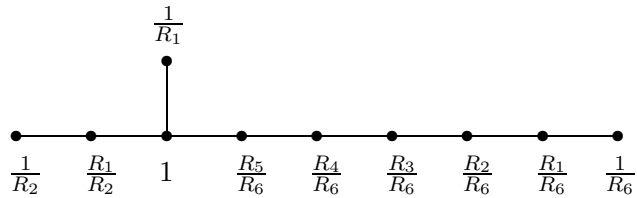
\begin{center}
\setlength{\unitlength}{1cm}
\begin{picture}(10,3)
\put(1,1){\line(1,0){8}}
\put(3,1){\line(0,1){1}}
\multiput(1,1)(1,0){9}{\circle*{0.15}}
\put(3,2){\circle*{0.15}}
\put(0.9,0.4){$\frac{1}{R_{2}}$}
\put(1.9,0.4){$\rik{1}{2}$}
\put(2.9,0.4){$1$}
\put(3.9,0.4){$\rik{5}{6}$}
\put(4.9,0.4){$\rik{4}{6}$}
\put(5.9,0.4){$\rik{3}{6}$}
\put(6.9,0.4){$\rik{2}{6}$}
\put(7.9,0.4){$\rik{1}{6}$}
\put(8.8,0.4){$\frac{1}{R_{6}}$}
\put(2.8,2.4){$\frac{1}{R_{1}}$}
\end{picture}
\end{center}
and the eigenvalue equation gives

\begin{equation}
\label{e10eqn}
\frac{R_{1}}{R_{2}}  + \frac{1}{R_{1}} + \frac{R_{5}}{R_{6}} =\lambda(\Gamma) = R_{1}.
\end{equation}

In the following figure we list the transformation factors 
\[w(p,q,r,s)=\sqrt{\frac{\xi(p)\xi(s)}{\xi(q)\xi(r)}}\]
($\xi=\eta$ because the two vertical graphs are equal). Note that $w(p,q,r,s)$ only depends on the two edges $pq$ and $rs.$

\begin{center}
\setlength{\unitlength}{0.6cm}
\begin{picture}(19,19)
\put(6.7,19){Table of $w(p,q,r,s)^{^{2}}$}

\multiput(1,0)(2,0){10}{\line(0,1){18}}
\multiput(1,0)(0,2){10}{\line(1,0){18}}
\put(1,18){\line(-1,1){1}}
\put(1,18){\line(-1,0){1}}
\put(1,18){\line(0,1){1}}
\put(-0.2,18.2){$pq$}
\put(0.4,18.8){$rs$}
\put(1.7,18.2){$Aa$}
\put(3.7,18.2){$Ba$}
\put(5.7,18.2){$Bb$}
\put(7.7,18.2){$Bc$}
\put(9.7,18.2){$Cc$}
\put(11.7,18.2){$Cd$}
\put(13.7,18.2){$Dd$}
\put(15.7,18.2){$De$}
\put(17.7,18.2){$Ee$}
\put(0,16.9){$Aa$}
\put(0,14.9){$Ba$}
\put(0,12.9){$Bb$}
\put(0,10.9){$Bc$}
\put(0,8.9){$Cc$}
\put(0,6.9){$Cd$}
\put(0,4.9){$Dd$}
\put(0,2.9){$De$}
\put(0,0.9){$Ee$}

\put(3.8,16.9){$\frac{1}{R_{_2}}$}
\put(5.7,16.9){$\frac{1}{R_{_1}^{2}}$}
\put(7.4,16.9){\rikl{5}{1}{6}}
\put(17.9,16.9){1}

\put(1.8,14.9){$R_{_2}$}
\put(3.9,14.9){1}
\put(5.7,14.9){$\frac{R_{2}}{R_{1}^{2}}$}
\put(7.4,14.9){\rijkl{2}{5}{1}{6}}
\put(9.4,14.9){\rijkl{2}{5}{1}{4}}
\put(15.9,14.9){1}

\put(1.7,12.9){$R_{_1}^{2}$}
\put(3.7,12.9){$\frac{R_{1}^{2}}{R_{2}}$}
\put(5.9,12.9){1}
\put(11.4,12.9){\rijl{1}{3}{4}}
\put(13.4,12.9){\rijl{1}{3}{2}}

\put(1.4,10.9){\rijl{1}{6}{5}}
\put(3.4,10.9){\rijkl{1}{6}{2}{5}}
\put(7.9,10.9){1}
\put(9.7,10.9){$\frac{R_{6}}{R_{4}}$}
\put(11.4,10.9){\rijkl{3}{6}{4}{5}}
\put(13.4,10.9){\rijkl{3}{6}{2}{5}}

\put(3.4,8.9){\rijkl{1}{4}{2}{5}}
\put(7.7,8.9){$\frac{R_{4}}{R_{6}}$}
\put(9.9,8.9){1}
\put(11.7,8.9){$\frac{R_{3}}{R_{5}}$}

\put(5.4,6.9){\rikl{4}{1}{3}}
\put(7.4,6.9){\rijkl{4}{5}{3}{6}}
\put(9.7,6.9){$\frac{R_{5}}{R_{3}}$}
\put(11.9,6.9){1}

\put(5.4,4.9){\rikl{2}{1}{3}}
\put(7.4,4.9){\rijkl{2}{5}{3}{6}}
\put(13.9,4.9){1}

\put(3.9,2.9){1}
\put(15.9,2.9){1}

\put(1.9,0.9){1}
\put(17.9,0.9){1}

\end{picture}
\end{center}

The squares $\left\|u_{q,r}^{(p,s)}\right\|^{2}$ of 2--Hilbert--Schmidt norms of the blocks of $u,$ can be determined as follows. Using that the scalar summands in $u$ and $v$ must have modulus 1, together with the transformation formula (\ref{e10vdef}), one immediately finds 11 of the 36 norms. The fact that rows and columns in a unitary matrix have 2--norm equal to 1, combined with the transformation formula (\ref{e10vdef}), shows that there is at most one possible value for each $(pq,rs).$ The values are listed in the following diagram

\begin{center}
\setlength{\unitlength}{0.7cm}
\begin{picture}(19,19)
\put(6.9,19){Table of $\left\|u_{q,r}^{(p,s)}\right\|^{2}$}

\multiput(1,0)(2,0){10}{\line(0,1){18}}
\multiput(1,0)(0,2){10}{\line(1,0){18}}
\put(1,18){\line(-1,1){1}}
\put(1,18){\line(-1,0){1}}
\put(1,18){\line(0,1){1}}
\put(-0.2,18.2){$pq$}
\put(0.4,18.8){$rs$}
\put(1.7,18.2){$Aa$}
\put(3.7,18.2){$Ba$}
\put(5.7,18.2){$Bb$}
\put(7.7,18.2){$Bc$}
\put(9.7,18.2){$Cc$}
\put(11.7,18.2){$Cd$}
\put(13.7,18.2){$Dd$}
\put(15.7,18.2){$De$}
\put(17.7,18.2){$Ee$}
\put(0,16.9){$Aa$}
\put(0,14.9){$Ba$}
\put(0,12.9){$Bb$}
\put(0,10.9){$Bc$}
\put(0,8.9){$Cc$}
\put(0,6.9){$Cd$}
\put(0,4.9){$Dd$}
\put(0,2.9){$De$}
\put(0,0.9){$Ee$}

\put(3.9,16.9){$1$}
\put(5.9,16.9){$1$}
\put(7.9,16.9){1}
\put(17.9,16.9){1}

\put(1.8,14.9){$\frac{1}{R_{2}}$}
\put(3.4,14.9){$\frac{R_{2}\!-\!1}{R_{2}}$}
\put(5.9,14.9){$1$}
\put(7.4,14.9){\rikl{7}{2}{5}}
\put(9.4,14.9){\rijkl{1}{4}{2}{5}}
\put(15.9,14.9){1}

\put(1.7,12.9){$\frac{1}{R_{1}^{2}}$}
\put(3.7,12.9){$\frac{R_{2}}{R_{1}^{2}}$}
\put(5.9,12.9){1}
\put(11.4,12.9){\rikl{4}{1}{3}}
\put(13.4,12.9){\rikl{2}{1}{3}}

\put(1.4,10.9){\rikl{5}{1}{6}}
\put(3.4,10.9){\rikl{7}{1}{6}}
\put(7.1,10.9){$1\!\!+\!\!\rijkl{1}{4}{2}{5}$}
\put(9.4,10.9){\rikl{7}{2}{5}}
\put(11.4,10.9){\rikl{2}{1}{3}}
\put(13.4,10.9){\rikl{4}{1}{3}}

\put(3.9,8.9){1}
\put(7.7,8.9){$\frac{R_{3}}{R_{5}}$}
\put(9.3,8.9){$2\!-\!\frac{R_{3}}{R_{5}}$}
\put(11.9,8.9){$1$}

\put(5.9,6.9){$1$}
\put(7.3,6.9){$1\!-\!\frac{R_{3}}{R_{5}}$}
\put(9.7,6.9){$\frac{R_{3}}{R_{5}}$}
\put(11.9,6.9){1}

\put(5.9,4.9){$1$}
\put(7.9,4.9){$1$}
\put(13.9,4.9){$1$}

\put(3.9,2.9){$1$}
\put(15.9,2.9){$1$}

\put(1.9,0.9){$1$}
\put(17.9,0.9){$1$}
\linethickness{2pt}
\put(3,18){\line(1,0){6}}
\put(17,18){\line(1,0){2}}
\put(1,16){\line(1,0){10}}
\put(15,16){\line(1,0){4}}
\put(11,14){\line(1,0){6}}
\put(5,12){\line(1,0){2}}
\put(1,10){\line(1,0){4}}
\put(7,10){\line(1,0){8}}
\put(3,8){\line(1,0){4}}
\put(5,6){\line(1,0){10}}
\put(3,4){\line(1,0){6}}
\put(13,4){\line(1,0){4}}
\put(1,2){\line(1,0){4}}
\put(15,2){\line(1,0){4}}
\put(1,0){\line(1,0){2}}
\put(17,0){\line(1,0){2}}
\put(1,16){\line(0,-1){6}}
\put(1,2){\line(0,-1){2}}
\put(3,18){\line(0,-1){2}}
\put(3,10){\line(0,-1){2}}
\put(3,4){\line(0,-1){4}}
\put(5,18){\line(0,-1){16}}
\put(7,18){\line(0,-1){14}}
\put(9,18){\line(0,-1){2}}
\put(9,6){\line(0,-1){2}}
\put(11,16){\line(0,-1){10}}
\put(13,10){\line(0,-1){6}}
\put(15,16){\line(0,-1){6}}
\put(15,6){\line(0,-1){4}}
\put(17,18){\line(0,-1){4}}
\put(17,4){\line(0,-1){4}}
\put(19,18){\line(0,-1){2}}
\put(19,2){\line(0,-1){2}}
\end{picture}
\end{center}
All the numbers listed are clearly positive. To check, that this table of $\left\|u_{q,r}^{(p,s)}\right\|^{2}$ is an admissible solution to the square-norm problem, we have to check:
\begin{enumerate}
\item{}Sums of rows (resp. columns) within each frame should be 1 or 2, according to the number of rows (resp. columns) in the block. (See the diagram of the blocks of $u$.)
\begin{description}
\item[(a)] $\frac{1}{R_{1}^{2}} + \frac{R_{2}}{R_{1}^{2}} = 1.$
\item[(b)] $\rikl{5}{1}{6} + \rikl{7}{1}{6} = 1.$
\item[(c)] $\frac{1}{R_{2}} + \frac{1}{R_{1}^{2}}+\rikl{5}{1}{6}=1.$
\item[(d)] $\frac{R_{2}-1}{R_{2}}+\frac{R_{2}}{R_{1}^{2}}+\rikl{7}{1}{6}=2.$
\item[(e)] $\rikl{7}{2}{5}+\rijkl{1}{4}{2}{5}=1.$
\item[(f)] $\rikl{4}{1}{3} + \rikl{2}{1}{3}=1.$
\end{description}
\item{}Moreover, by the transformation formula (\ref{e10vdef})
\begin{equation}
\label{e10vmodsq}
\left\|v^{(q,r)}_{p,s}\right\|^{2}=w(p,q,r,s)^{2}\left\|u^{(p,s)}_{q,r}\right\|^{2},
\end{equation}
so by multiplying the numbers in the table of $\left\|u^{(p,s)}_{q,r}\right\|^{2},$ with the numbers in the table of the transition factors, one should obtain an admissible solution to the square-norm problem for $v.$ Because of the properties of the tables on page \pageref{e10page}, the solution to the square-norm problem for $v,$  should just be a  reflection of the table of $\left\|u^{(p,s)}_{q,r}\right\|^{2}$ in the main diagonal. Hence we must show that
\begin{description}
\item[(g)] $\rikl{7}{2}{5}\rik{6}{4}=\rik{3}{5}$  (position $Bc-Cc$).
\item[(h)] $\rikl{2}{1}{3}\rijkl{3}{6}{4}{5}=1-\rik{3}{5}$ (position $Bc-Cd$).
\item[(i)] $\rikl{4}{1}{3}\rijkl{3}{6}{2}{5} = 1$ (position $Bc-Dd$).
\end{description}
All the other identities are trivially true.
\end{enumerate}
Since $R_{1}=\lambda,$ the recursion formula for the $R_{n}'$s can be rewritten as
\[R_{1}R_{n}=R_{n-1}+R_{n+1},\;\;\;n\geq{}1.\]
From this, and $R_{0}=1,$ it follows by induction on $m,$ that
\begin{equation}
\label{e10rprod}
\ri{m}\ri{n}=\ri{n-m}+\ri{n-m+2}+\cdots{}+\ri{m+n-2}+\ri{m+n},\;\;n\geq{}m.\end{equation}
This proves (a), (b), (e) and (f). Moreover (c) follows from the equation (\ref{e10eqn}) by dividing with \ri{1}. Also (d) follows by subtracting (c) from the sum of (a) and (b). The remaining equations can be rewritten as
\begin{description}
\item[(g')] $\ri{6}\ri{7}=\ri{2}\ri{3}\ri{4}.$
\item[(h')] $\ri{2}\ri{6}=\ri{1}\ri{4}(\ri{5}-\ri{3}).$
\item[(i')] $\ri{4}\ri{6}=\ri{1}\ri{2}\ri{5}.$
\end{description}
Multiplying the equation (\ref{e10eqn}) by $\ri{1}\ri{2}\ri{6}$ yields
\[R_{1}^{2}\ri{6}+\ri{2}\ri{6}+\ri{1}\ri{2}\ri{5} = R_{1}^{2}\ri{2}\ri{6},\]
which, by repeated use of (\ref{e10rprod}), transforms to
\[\ri{10}-\ri{6}-\ri{4}=0.\]
Using (\ref{e10rprod}) once more, we get
\[\begin{array}{l}
\ri{6}\ri{7}-\ri{2}\ri{3}\ri{4} = \ri{3}(\ri{10}-\ri{6}-\ri{4})=0,\\[0.3cm]\ri{1}\ri{4}(\ri{5}-\ri{3})-\ri{2}\ri{6}=\ri{10}-\ri{6}-\ri{4}=0,\\[0.3cm]
\ri{4}\ri{6}-\ri{1}\ri{2}\ri{5}=\ri{10}-\ri{6}-\ri{4}=0,
\end{array}\]
which proves (g), (h) and (i).

We can now write up an explicit solution to the bi-unitarity condition.

\noindent{}By (c) $(\sqrt{\frac{1}{R_{2}}},\frac{1}{R_{1}},\sqrt{\rikl{5}{1}{6}})$ is a unit vector in ${\Bbb R}^{3},$ so we can find $x_{_{i,j}}\in{\Bbb R},\;\;i=1,2,3,\;j=1,2,$ such that
\[Y=\left(\begin{array}{ccc}
\sqrt{\frac{1}{R_{2}}} & x_{_{11}} & x_{_{12}} \\
\frac{1}{R_{1}}& x_{_{21}} & x_{_{22}}\\[0.15cm]
\sqrt{\rikl{5}{1}{6}}& x_{_{31}} & x_{_{32}}\end{array}\right)\]
is a unitary matrix.

\noindent{}By (a) and (b)
\[\begin{array}{lcl}
x_{_{11}}^{2}+x_{_{12}}^{2}&=&\frac{R_{2}-1}{R_{2}},\\[0.3cm]
x_{_{21}}^{2}+x_{_{22}}^{2}&=&\frac{R_{2}}{R_{1}^{2}},\\[0.3cm]
x_{_{31}}^{2}+x_{_{32}}^{2}&=&\rikl{7}{1}{6}.
\end{array}\]
Hence there exists 3 unit vectors, $e=(e_{_{1}},e_{_{2}}),\;f=(f_{_{1}},f_{_{2}})$ and $g=(g_{_{1}},g_{_{2}})$ in ${\Bbb R}^{2},$ such that
\[Y=\left(\begin{array}{cc}
\sqrt{\frac{1}{R_{2}}} &  \sqrt{\frac{R_{2}-1}{R_{2}}}(e_{_{1}},e_{_{2}})\\[0.2cm]
\frac{1}{R_{1}}&  \sqrt{\frac{R_{2}}{R_{1}^{2}}}(f_{_{1}},f_{_{2}})\\[0.35cm]
\sqrt{\rikl{5}{1}{6}}& \sqrt{\rikl{7}{1}{6}}(g_{_{1}},g_{_{2}})\end{array}\right).\]
Moreover, by a change of basis, we may assume that $g=(1,0).$

\noindent{}Set 
\[\begin{array}{lcllcllcllcl}
\sigma_{_{1}} & = & \sqrt{\frac{1}{R_{2}}},\hspace{0.5cm} &
\sigma_{_{2}} & = & \sqrt{\frac{R_{2}-1}{R_{2}}},\hspace{0.5cm}&
\sigma_{_{3}} & = & \frac{1}{R_{1}},\hspace{0.5cm}&
\sigma_{_{4}} & = & \sqrt{\frac{R_{2}}{R_{1}}},\\[0.3cm]
\sigma_{_{5}} & = & \sqrikl{5}{1}{6}, &
\sigma_{_{6}} & = & \sqrikl{7}{1}{6},&
\tau_{_{1}}   & = & \sqrikl{7}{2}{5},&
\tau_{_{2}}   & = & \sqrijkl{1}{4}{2}{5},\\[0.3cm]
\rho_{_{1}}   & = & \sqrikl{4}{1}{3}, &
\rho_{_{2}}   & = & \sqrikl{2}{1}{3},&
\mu_{_{1}}    & = & \sqrt{\frac{R_{3}}{R_{5}}},&
\mu_{_{2}}    & = & \sqrt{1-\frac{R_{3}}{R_{5}}}.
\end{array}\]
Then by (a), (b), (c), (d), (e) and (f)
\[\sigma_{_{1}}^{2}+\sigma_{_{2}}^{2}=1,\;\;\sigma_{_{3}}^{2}+\sigma_{_{4}}^{2}=1,\;\;\sigma_{_{5}}^{2}+\sigma_{_{6}}^{2}=1,\]
\[\sigma_{_{1}}^{2}+\sigma_{_{3}}^{2}+\sigma_{_{5}}^{2}=1,\;\;\sigma_{_{2}}^{2}+\sigma_{_{4}}^{2}+\sigma_{_{6}}^{2}=2,\]
\[\tau_{_{1}}^{2}+\tau_{_{2}}^{2}=1,\;\;\rho_{_{1}}^{2}+\rho_{_{2}}^{2}=1,\mbox{ and }\mu_{_{1}}^{2}+\mu_{_{2}}^{2}=1.\]
Define $u=\sum_{p,s}^{\oplus}u^{(p,s)}$ by the table in figure \ref{e10figU}

\begin{figure}
\label{e10figU}
\begin{center}
\setlength{\unitlength}{0.85cm}
\begin{picture}(19,19)
\multiput(1,0)(2,0){10}{\line(0,1){18}}
\multiput(1,0)(0,2){10}{\line(1,0){18}}
\put(1,18){\line(-1,1){1}}
\put(1,18){\line(-1,0){1}}
\put(1,18){\line(0,1){1}}
\put(0,18.2){$pq$}
\put(0.5,18.7){$rs$}
\put(1.7,18.2){$Aa$}
\put(3.7,18.2){$Ba$}
\put(5.7,18.2){$Bb$}
\put(7.7,18.2){$Bc$}
\put(9.7,18.2){$Cc$}
\put(11.7,18.2){$Cd$}
\put(13.7,18.2){$Dd$}
\put(15.7,18.2){$De$}
\put(17.7,18.2){$Ee$}
\put(0.3,16.9){$Aa$}
\put(0.3,14.9){$Ba$}
\put(0.3,12.9){$Bb$}
\put(0.3,10.9){$Bc$}
\put(0.3,8.9){$Cc$}
\put(0.3,6.9){$Cd$}
\put(0.3,4.9){$Dd$}
\put(0.3,2.9){$De$}
\put(0.3,0.9){$Ee$}

\put(3.9,16.9){1}
\put(5.9,16.9){1}
\put(7.9,16.9){1}
\put(17.9,16.9){1}

\put(1.9,14.9){$\sigma_{_{1}}$}
\put(3.1,14.9){$\sigma_{_{2}}(e_{_{1}},e_{_{2}})$}
\put(5.4,14.9){$(f_{_{1}},f_{_{2}})$}
\put(7.4,14.9){$(\tau_{_{1}},0)$}
\put(9.9,14.9){$\tau_{_{2}}$}
\put(15.9,14.9){1}

\put(1.9,12.9){$\sigma_{_{3}}$}
\put(3.1,12.9){$\sigma_{_{4}}(f_{_{1}},f_{_{2}})$}
\put(5.2,12.9){$(f_{_{2}},-f_{_{1}})$}
\put(11.9,12.9){$\rho_{_{1}}$}
\put(13.9,12.9){$\rho_{_{2}}$}

\put(1.9,10.9){$\sigma_{_{5}}$}
\put(3.4,10.9){$(\sigma_{_{6}},0)$}
\put(7,10.9){$\left(\begin{array}{cc}\tau_{_{2}}\!\! & 0 \\0 & 1\end{array}\right)$}
\put(9.1,10.9){$\left(\begin{array}{c}-\tau_{_{1}}\\0\end{array}\right)$}
\put(11.9,10.9){$\rho_{_{2}}$}
\put(13.6,10.9){$-\rho_{_{1}}$}

\put(3.9,8.9){1}
\put(7.1,8.9){$\left(\begin{array}{c}-\mu_{_{1}}\\0\end{array}\right)$}
\put(9,8.9){$\left(\begin{array}{cc}\mu_{_{2}} \!\!& 0 \\0 & 1\end{array}\right)$}
\put(11.6,8.9){(1,0)}

\put(5.9,6.9){1}
\put(7.9,6.9){$\mu_{_{2}}$}
\put(9.4,6.9){$(\mu_{_{1}},0)$}
\put(11.6,6.9){$(0,1)$}

\put(5.9,4.9){1}
\put(7.7,4.9){-1}
\put(13.9,4.9){1}

\put(3.9,2.9){1}
\put(15.9,2.9){1}

\put(1.9,0.9){1}
\put(17.9,0.9){1}
\linethickness{2pt}
\put(3,18){\line(1,0){6}}
\put(17,18){\line(1,0){2}}
\put(1,16){\line(1,0){10}}
\put(15,16){\line(1,0){4}}
\put(11,14){\line(1,0){6}}
\put(5,12){\line(1,0){2}}
\put(1,10){\line(1,0){4}}
\put(7,10){\line(1,0){8}}
\put(3,8){\line(1,0){4}}
\put(5,6){\line(1,0){10}}
\put(3,4){\line(1,0){6}}
\put(13,4){\line(1,0){4}}
\put(1,2){\line(1,0){4}}
\put(15,2){\line(1,0){4}}
\put(1,0){\line(1,0){2}}
\put(17,0){\line(1,0){2}}
\put(1,16){\line(0,-1){6}}
\put(1,2){\line(0,-1){2}}
\put(3,18){\line(0,-1){2}}
\put(3,10){\line(0,-1){2}}
\put(3,4){\line(0,-1){4}}
\put(5,18){\line(0,-1){16}}
\put(7,18){\line(0,-1){14}}
\put(9,18){\line(0,-1){2}}
\put(9,6){\line(0,-1){2}}
\put(11,16){\line(0,-1){10}}
\put(13,10){\line(0,-1){6}}
\put(15,16){\line(0,-1){6}}
\put(15,6){\line(0,-1){4}}
\put(17,18){\line(0,-1){4}}
\put(17,4){\line(0,-1){4}}
\put(19,18){\line(0,-1){2}}
\put(19,2){\line(0,-1){2}}

\end{picture}
\end{center}
\caption{The summands of $u.$}
\end{figure}

Then each of the 20 directed summands in $u,$ indicated by the frames, are unitary by the construction of the unit vectors $e$ and $f.$ Moreover the transformation formula (\ref{e10vdef}), together with (g), (h) and (i), show that $v$ is simply the mirror image of figure \ref{e10figU} in the main diagonal (transposing all matrices), and hence $v$ is also unitary. The conclusion is, that there exists a commuting square with the inclusion matrices given in (\ref{specialcsq}).

\noindent{}It is elementary to check that 
\[\left(\begin{array}{cc}0 & P(GG^{^{t}})\\P(GG^{^{t}}) & 0\end{array}\right)\hspace{1cm}
\left(\begin{array}{cc}0 & P(G^{^{t}}G)\\P(G^{^{t}}G) & 0\end{array}\right)\]
are adjacency matrices for connected bi-partite graphs, so the construction gives a subfactor of the hyperfinite $II_{1}-$factor, with index $\lambda(E_{10})^{2}.$ Furthermore $G$  clearly satisfies Wenzl's criterion for irreducibility of the pair of hyperfinite $II_{1}-$factors (see \cite{Wen2}). Hence we have proved:\\[0.3cm]

{\it There is an irreducible subfactor of the hyperfinite $II_{1}-$ factor of index $\lambda(E_{10})^{2}\approx{}4.026418.$}\\[0.2cm]

By \cite{HGJ} chapter 4, this is the lowest value of $\lambda(\Gamma)^{2}$ above 4, which can be obtained from a finite graph $\Gamma.$

\chapter{Commuting Squares Based on Dynkin Diagrams of Type A}
\label{angraphs}

\setcounter{equation}{0}
\section{Preliminaries}
\setcounter{equation}{0}
We shall consider the graph $A_{_{m}}$ 

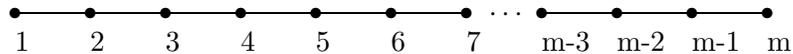
\begin{center}
\setlength{\unitlength}{1mm}
\begin{picture}(100,20)
\multiput(10,10)(10,0){11}{\circle*{1.5}}
\put(10,10){\line(1,0){60}}
\put(80,10){\line(1,0){30}}
\put(73,9){$\cdots$}
\put(10,5){1}
\put(20,5){2}
\put(30,5){3}
\put(40,5){4}
\put(50,5){5}
\put(60,5){6}
\put(70,5){7}
\put(80,5){m-3}
\put(90,5){m-2}
\put(100,5){m-1}
\put(110,5){m}
\end{picture}
\end{center}

\noindent{}The adjacency matrix for $A_{_{m}}$ is 
$H = (h_{_{i,j}})_{i,j=1}^{m}$ with
\[h_{i,j} = 
\left\{ \begin{array}{ll}
1 & \mbox{{\rm if }} |i-j| = 1 \\
0 & \mbox{{\rm otherwise}}
\end{array} \right.
\]
Define inductively polynomials $R_{n}(t)$ by
\[\begin{array}{lcl}
R_{_{0}}(t) & = & 1\\
R_{_{1}}(t) & = & t\\
R_{_{n+1}}(t) & = & tR_{_{n}}(t) - R_{_{n-1}}(t).
\end{array} \]
Note that if $0\leq{}t\leq{}2,$ $t = 2\cos{}x$ for some $x$, we have
\[R_{n}(t) = \frac{\sin{}((n+1)x)}{\sin{}(x)}\]
and if $t>2$ with  $t = 2\cosh{}x$ for some $x,$ we have
\[R_{n}(t) = \frac{\sinh{}((n+1)x)}{\sinh{}(x)}.\]
\begin{lemma}
\label{lemma1}
For $l<m$  we denote $R_{l}(H)$ by  $H^{(l)} = (h^{(l)}_{i,j})_{i,j=1}^{m}$, and
we have
\begin{enumerate}
\item{}For $l$ even, $l=2n$
\[
h^{(2n)}_{i,j} = \left\{
\begin{array}{lll}
1 & \mbox{{\rm if }} |i-j| = 0, & n+1\leq{}i,j\leq{}m-n \\
1 & \mbox{{\rm if }} |i-j| = 2, & n \leq{}i,j\leq{}m-n+1 \\
1 & \mbox{{\rm if }} |i-j| = 4, & n-1\leq{}i,j\leq{}m-n+2 \\
\vdots & \, & \, \\
1 & \mbox{{\rm if }} |i-j| = 2n-2, & 2\leq{}i,j\leq{}m-1\\
1 & \mbox{{\rm if }} |i-j| = 2n, & 1\leq{}i,j\leq{}m \\
1 & \mbox{{\rm otherwise}} & \, \end{array} \right. \]
or in a more compact notation
\begin{equation}
\label{heven}
h^{(2n)}_{i,j} = \left\{
\begin{array}{lll}
1 & \mbox{{\rm if }} |i-j| = 2k,  n+1-k\leq{}i,j\leq{}m-n+k \\
0 & \mbox{{\rm otherwise}} & \, \end{array} \right. 
\end{equation}
\item{}For $l$ odd, $l=2n + 1$
\[
h^{(2n+1)}_{i,j} = \left\{
\begin{array}{lll}
1 & \mbox{{\rm if }} |i-j| = 1, & n+1\leq{}i,j\leq{}m-n \\
1 & \mbox{{\rm if }} |i-j| = 3, & n \leq{}i,j\leq{}m-n+1 \\
1 & \mbox{{\rm if }} |i-j| = 5, & n-1\leq{}i,j\leq{}m-n+2 \\
\vdots & \, & \, \\
1 & \mbox{{\rm if }} |i-j| = 2n-1, & 2\leq{}i,j\leq{}m-1\\
1 & \mbox{{\rm if }} |i-j| = 2n+1, & 1\leq{}i,j\leq{}m \\
1 & \mbox{{\rm otherwise}} & \, \end{array} \right. \]
that is
\[
h^{(2n+1)}_{i,j} = \left\{
\begin{array}{lll}
1 & \mbox{{\rm if }} |i-j| = 2k+1, & n+1-k\leq{}i,j\leq{}m-n+k \\
0 & \mbox{{\rm otherwise}} & \, \end{array} \right. 
\]
\end{enumerate}
\end{lemma}
\begin{proof}
The assertion is obviously true for $l=0$ and $l=1$. Assume the statement is
true for all $l\leq{}p$, $p\geq{}1$.

\noindent{}By definition $H^{(p+1)}= HH^{(p)}-H^{(p-1)}$.

\noindent{}If $p$ is even, $p=2n,$ we have
\[\begin{array}{lcl}
(HH^{(2n)})_{i,j} & = & 
\sum_{k}h_{i,k}h^{(2n)}_{k,j} \\[0.5cm]
\, & = & h_{i,i-1}h^{(2n)}_{i-1,j} + h_{i,i+1}h^{(2n)}_{i+1,j} \\[0.5cm]
\, & = & h^{(2n)}_{i-1,j} + h^{(2n)}_{i+1,j} \end{array}\]
Hence we may view $HH^{(2n)}$ as $K + L$, where 
$(K)_{i,j} = h^{(2n)}_{i-1,j}$ and $(L)_{i,j} = h^{(2n)}_{i+1,j}.$ From the
hypothesis on $H^{(2n)}$ we then get 
\[(K)_{i,j} = \left\{
\begin{array}{llc}
1 & \mbox{{\rm if }} |i-j| = 2k+1, & n+1-k\leq{}i,j\leq{}m-n+k+1 \\
0 & \mbox{{\rm otherwise}} & \, \end{array} \right. \]
and
\[(L)_{i,j} = \left\{
\begin{array}{llc}
1 & \mbox{{\rm if }} |i-j| = 2k+1, & n-k\leq{}i,j\leq{}m-n+k\\
0 & \mbox{{\rm otherwise}} & \, \end{array} \right. \]
and we have
\[h^{(2n+1)}_{i,j} = \left\{
\begin{array}{llc}
1 & \mbox{{\rm if }} |i-j| = 2k+1, & n+1-k\leq{}i,j\leq{}m-n+k \\
0 & \mbox{{\rm otherwise}} & \, \end{array} \right. \]

\noindent{}If $p$ is odd, $p=2n-1,$ we have
\[(HH^{(2n-1)})_{i,j}  =  h^{(2n-1)}_{i-1,j} + h^{(2n-1)}_{i+1,j} \]
and again we can  view $HH^{(2n-1)}$ as $K+L.$

\noindent{}This time we have
\[(K)_{i,j} = \left\{
\begin{array}{llc}
1 & \mbox{{\rm if }} |i-j| = 2k, & n+1-k\leq{}i,j\leq{}m-n+k+1 \\
0 & \mbox{{\rm otherwise}} & \, \end{array} \right. \]
and
\[(L)_{i,j} = \left\{
\begin{array}{llc}
1 & \mbox{{\rm if }} |i-j| = 2k, & n-k\leq{}i,j\leq{}m-n+k\\
0 & \mbox{{\rm otherwise}} & \, \end{array} \right. \]
and since
\[h^{(2n-2)}_{i,j} = \left\{
\begin{array}{llc}
1 & \mbox{{\rm if }} |i-j| = 2k, & n-k\leq{}i,j\leq{}m-n+k+1 \\
0 & \mbox{{\rm otherwise}} & \, \end{array} \right. \]
we get the desired result.
\end{proof}
\section{The Blocks of the Bi-unitary}
\setcounter{equation}{0}
We shall determine commuting squares og the form
\begin{equation}
\label{csq}
\begin{array}{lcl}
 C & \subset_{G^{t}} & D\\
  \cup_{H} &\,&\cup_{F}\\
 A & \subset_{G} & B \end{array}
\end{equation}
where $G$ is the adjacency matrix of $A_{_m}$ viewed as a bi--partite graph, and
$H,F$ are defined by:

If $l$ is even $R_{l}
\left(
\mbox{\scriptsize{$\begin{array}{cc} 0 & G^{t} \\G & 0 \end{array}$}}\right) = 
\left(\mbox{\scriptsize{$\begin{array}{cc} F & 0 \\0 & H \end{array}$}}\right)$

If $l$ is odd $R_{l}
\left(\mbox{\scriptsize{$\begin{array}{cc} 0 & G^{t} \\G & 0 \end{array}$}}\right) = 
\left(\mbox{\scriptsize{$\begin{array}{cc} 0 & F \\H & 0 \end{array}$}}\right)$

\noindent{}If $u$ is the bi-unitary matrix associated with a commuting square
of the form (\ref{csq}), then an entry in $u$ is specified by loops of the form
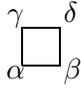
\begin{center}
\setlength{\unitlength}{0.25mm}
\begin{picture}(40,40)
\put(10,10){\line(1,0){20}} 
\put(10,10){\line(0,1){20}}
\put(30,10){\line(0,1){20}} 
\put(10,30){\line(1,0){20}}
\put(1.5,3){$\alpha$}
\put(32,3){$\beta$}
\put(32,33){$\delta$}
\put(2,33){$\gamma$}
\end{picture} 
\end{center}

\noindent{}This means that we can describe the entries of $u$ the following way
\begin{enumerate}
\item{}If $l$ is even. For any two edges of $A_{_m}$ (considered bi--partite)
$e,f$
 there corresponds an entry of $u,$ if there exists edges $\mu{},\nu{}$ in the
graphs corresponding to the matrices $F,H$ such that $\mu$ joins the odd
labeled vertices of $e$ and $f$, and $\nu$ joins the even
labeled vertices of $e$ and $f$.
\item{}If $l$ is odd. For any two edges of $A_{_m}$ (considered bi--partite)
$e,f$ there corresponds an entry of $u,$ if there exists edges $\mu{},\nu{}$ in the
graphs corresponding to the matrices $F,H$ such that $\mu$ joins the odd
labeled vertex of $e$ and the even labeled vertex of $f$, and $\nu$ joins the
two other  vertices of $e$ and $f$.
\end{enumerate}

\noindent{}Hence we can get a picture of the entries of $u$ in the form of a
diagram like

\begin{center}
\setlength{\unitlength}{1cm}
\begin{picture}(5,5)
\multiput(1,1)(0.5,0){8}{\line(0,1){3.5}}
\multiput(1,1)(0,0.5){8}{\line(1,0){3.5}}
\put(1,4.8){1}
\put(1.5,4.8){2}
\put(2,4.8){3}
\put(2.5,4.8){4}
\put(3,4.8){5}
\put(3.5,4.8){6}
\put(4,4.8){7}
\put(4.5,4.8){8}
\put(0.5,1){8}
\put(0.5,1.5){7}
\put(0.5,2){6}
\put(0.5,2.5){5}
\put(0.5,3){4}
\put(0.5,3.5){3}
\put(0.5,4){2}
\put(0.5,4.5){1}
\put(2.15,2.2){x}
\end{picture}
\end{center}

\noindent{}Where a box in the diagram will be filled if it corresponds to an entry
of $u$. Like the one with ``x'' in it, will correspond to an entry of $u,$
for $l$ even if the
loop
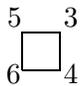
\begin{center}
\setlength{\unitlength}{0.25mm}
\begin{picture}(40,40)
\put(10,10){\line(1,0){20}} 
\put(10,10){\line(0,1){20}}
\put(30,10){\line(0,1){20}} 
\put(10,30){\line(1,0){20}}
\put(1.5,3){$6$}
\put(32,3){$4$}
\put(32,33){$3$}
\put(2,33){$5$}
\end{picture} 
\end{center}
exists, and in the case $l$ odd if the loop
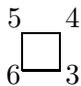
\begin{center}
\setlength{\unitlength}{0.25mm}
\begin{picture}(40,40)
\put(10,10){\line(1,0){20}} 
\put(10,10){\line(0,1){20}}
\put(30,10){\line(0,1){20}} 
\put(10,30){\line(1,0){20}}
\put(1.5,3){$6$}
\put(33,3){$3$}
\put(33,33){$4$}
\put(2,33){$5$}
\end{picture} 
\end{center}
exists.  

\noindent{}For the determination of which of the entries in the diagram
corresponds to entries of $u$, we introduce the following notation

The edge in $A_{_m}$ joining vertex $2j-1$ to vertex $2j$ is called $\{2j-1\}$

The edge in $A_{_m}$ joining vertex $2j+1$ to vertex $2j$ is called $\{2j\}$

\noindent{}We will also denote vertices of $A_{_m}$ by the respective number put
in square brackets, e.g. $[j]$, in order to distinguish a vertex from a number. 

\noindent{}Finally we denote by $\alpha_{_{j}}$ the coordinate of the
Perron--Frobenius eigenvector of $A_{_m},$ corresponding to the vertex $[j]$.

\noindent{}Hence the edges and vertices of $A_{_m}$ are labeled by
  
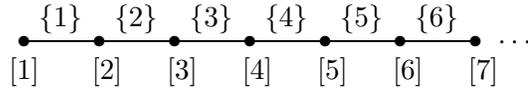
\begin{center}
\setlength{\unitlength}{1mm}
\begin{picture}(100,20)
\multiput(10,10)(10,0){7}{\circle*{1.5}}
\put(10,10){\line(1,0){60}}
\put(73,9){$\cdots$}
\put(8,5){[1]}
\put(19,5){[2]}
\put(29,5){[3]}
\put(39,5){[4]}
\put(49,5){[5]}
\put(59,5){[6]}
\put(69,5){[7]}
\put(12,12){\{1\}}
\put(22,12){\{2\}}
\put(32,12){\{3\}}
\put(42,12){\{4\}}
\put(52,12){\{5\}}
\put(62,12){\{6\}}

\end{picture}
\end{center}

\noindent{}and we will label a box in the diagram by the corresponding edges of
$A_{_m}$, e.g. the box with an $x$ in it, will be labeled $(\{3\},\{7\})$.
\begin{lemma}
\label{lemma2}
For $R_{_{l}}(A_{_{m}}),$ $l\leq{}m$ the boxes in the diagram corresponding to entries of $u$ are
given by

\[(\{1+t+s\},\{l-t+s\})\mbox{ {\rm where } }
t\in\{0,1,\ldots,l-1\},s\in\{0,1,\ldots,m-l-1\}\]
and
\[(\{1+t+s\},\{l+1+t-s\})\mbox{ {\rm where } }
t\in\{0,1,\ldots,m-l-2\},s\in\{0,1,\ldots,l\}\]

\end{lemma}
\begin{proof}
First we note the following symmetries. If \edg{j}{k} defines an entry of $u$, then so do \edg{k}{j},
\edg{m-j}{m-k} and \edg{m-k}{m-j}.

\noindent{}Consider $l$ even, $l=2n$.

We first look at \edg{1+t}{2n-t}, $t\in\{0,1,\ldots,2n-1\}$. 

\noindent{}If $t$ is even, $t=2s,$ we
must show that there is an edge between 
\begin{enumerate}
\item{}$[2s+2]$ and $[2n-2s]$
\item{}$[2s+1]$ and $[2n-2s+1]$
\end{enumerate}

Ad 1. We may assume $2s+2 \leq{}2n-2s$, since the other case will be settled by
the above symmetries. From the form of $R_{l}(A_{_m})$ we need to look at 
$|2n-2s-2s-2| = 2n-4s-2$ which corresponds to $k=n-2s-1$ in (\ref{heven}). The
desired edge exists if
\[\begin{array}{clcccl}
\, & n+1-(n-2s-1)&\leq & 2s+2,2n-2s &\leq{}& m-n+(n-2s-1)\\
\Updownarrow{} & \, & \, & \,  & \, & \, \\
\, & 2s+2&\leq & 2s+2,2n-2s &\leq{}& m-2s+1
\end{array} \]
We thus have that $[2s+2]$ is connected to $[2n-2s]$, and that $[2s+2]$ is not
connected to a vertex with lower index than $2n-2s$.

Ad 2. Here the corresponding value of $k$ in (\ref{heven}) is $n - 2s$, and we
must have $ 2s+1\leq{}2s+1,2n-2s+1\leq{}m-2s$. Hence $[2s+1]$ is connected to 
$[2n-2s+1]$ and not connected to any vertex with lower index than $2n-2s+1$.
 
\noindent{}For $t$ even we conclude that \edg{1+t}{2n-t} corresponds to to an
entry of $u$, and no box  to the left of \edg{1+t}{2n-t} defines an entry of $u$.

\noindent{}If $t$ is odd, $t=2s+1,$ we
must show that there is an edge between 
\begin{enumerate}
\item{}$[2s+3]$ and $[2n-2s-1]$
\item{}$[2s+2]$ and $[2n-2s]$
\end{enumerate}
Calculating as before we get

Ad 1. $ 2s+3\leq{}2s+3,2n-2s-1\leq{}m-2s-2$.

Ad 2. $ 2s+2\leq{}2s+2,2n-2s\leq{}m-2s-1$.

\noindent{}In either case we see that the edge exists, and as above we may
conclude that for $t$ odd  \edg{1+t}{2n-t} corresponds to to an
entry of $u$, and no box  to the left of \edg{1+t}{2n-t} defines an entry of $u$.

By the noted symmetries, we  can also conclude that  
\[ \edg{m-2n+t}{m-1-t},\;\;\;t\in\{0,1,\ldots,2n-1\} \]
defines an entries of $u$, and no box to the right of these defines an entry of
$u$.

We now look at \edg{2n+1+t}{1+t}, $t\in\{0,1,\ldots,m-2n-2\}$.

\noindent{}If $t$ is even, $t= 2s,$ we must show that we have edges joining 
\begin{enumerate}
\item{}$[2n+2s+1]$ and $[2s+1]$
\item{}$[2n+2s+2]$ and $[2s+2]$
\end{enumerate} 
Since the numerical difference of the indices is $2n$ in both cases, we see
from the condition in (\ref{heven}) that the box defines an entry of $u,$  and
that no box to the right of it does.

\noindent{}If $t$ is odd, $t= 2s+1,$ we must show that we have edges joining 
\begin{enumerate}
\item{}$[2n+2s+3]$ and $[2s+3]$
\item{}$[2n+2s+2]$ and $[2s+2]$
\end{enumerate} 
The argument from $t$ even also works here, and we conclude that
\[\edg{2n+1+t}{1+t}, t\in\{0,1,\ldots,m-2n-2\}\] 
defines entries of $u$, and no box to the left of these boxes does define an
entry of $u$. We also get 
\[\edg{1+t}{2n+1+t}, t\in\{0,1,\ldots,m-2n-2\}\] 
defines entries of $u$, and no box to the right of these boxes does define an
entry of $u$.

I.e. we have now established that the following boxes define entries of $u$:

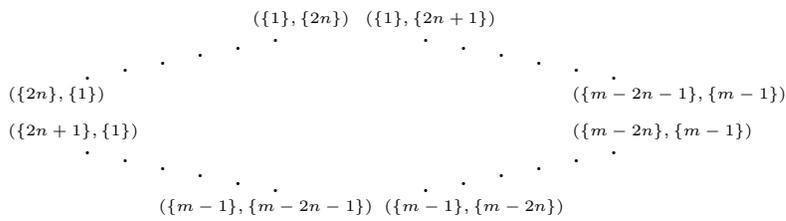
\setlength{\unitlength}{1cm}
\begin{picture}(10,4)
\put(4.25,3){\tinyedg{1}{2n}}
\put(5.75,3){\tinyedg{1}{2n+1}}
\put(1,2){\tinyedg{2n}{1}}
\put(8.5,2){\tinyedg{m-2n-1}{m-1}}
\put(1,1.5){\tinyedg{2n+1}{1}}
\put(8.5,1.5){\tinyedg{m-2n}{m-1}}
\put(3,0.5){\tinyedg{m-1}{m-2n-1}}
\put(6,0.5){\tinyedg{m-1}{m-2n}}
\multiput(4.5,2.75)(-0.5,-0.1){6}{.}
\multiput(6.5,2.75)(0.5,-0.1){6}{.}
\multiput(4.5,0.75)(-0.5,0.1){6}{.}
\multiput(6.5,0.75)(0.5,0.1){6}{.}
\end{picture}

\noindent{}The same type of argument, or using the convex nature of the criterion in
(\ref{heven}), now gives that all edge-pairs in the above sketched diamond
defines entries of $u$.

For odd $l$ we argue similarly.
\end{proof}

\noindent{}{\bf Example} $A_{_{8}},\; R_{_{3}}$
\begin{center}
\setlength{\unitlength}{1cm}
\begin{picture}(8,8)
\multiput(1,0)(1,0){8}{\line(0,1){7}}
\multiput(1,0)(0,1){8}{\line(1,0){7}}
\put(1.2,7.5){\{1\}}
\put(2.2,7.5){\{2\}}
\put(3.2,7.5){\{3\}}
\put(4.2,7.5){\{4\}}
\put(5.2,7.5){\{5\}}
\put(6.2,7.5){\{6\}}
\put(7.2,7.5){\{7\}}
\put(0.3,0.4){\{7\}}
\put(0.3,1.4){\{6\}}
\put(0.3,2.4){\{5\}}
\put(0.3,3.4){\{4\}}
\put(0.3,4.4){\{3\}}
\put(0.3,5.4){\{2\}}
\put(0.3,6.4){\{1\}}
\multiput(5.4,0.4)(1,1){3}{x}
\multiput(4.4,0.4)(1,1){4}{x}
\multiput(4.4,1.4)(1,1){3}{x}
\multiput(3.4,1.4)(1,1){4}{x}
\multiput(3.4,2.4)(1,1){3}{x}
\multiput(2.4,2.4)(1,1){4}{x}
\multiput(2.4,3.4)(1,1){3}{x}
\multiput(1.4,3.4)(1,1){4}{x}
\multiput(1.4,4.4)(1,1){3}{x}
\end{picture}
\end{center}

\section{The Bi-unitary Condition}
\setcounter{equation}{0}
We first consider  $l$ even. 

\noindent{}The boxes of the diagram that may define a block of
$u$ are of the form
\begin{equation}
\label{evenublock}
\begin{array}{cc}
\edg{2j-1}{2k} & \edg{2j-1}{2k+1} \\
\edg{2j}{2k} & \edg{2j}{2k+1}
\end{array}
\end{equation}
and the ones that may determine a block  of $v$ are of the form
\begin{equation}
\label{evenvblock}
\begin{array}{cc}
\edg{2k}{2j-1} & \edg{2k+1}{2j-1} \\
\edg{2k}{2j} & \edg{2k+1}{2j}
\end{array}
\end{equation}
since these are the labeling of edges which keep the required vertices fixed.

\noindent{}If we by \edgu{j}{k} and \edgv{j}{k} denote the value of the
entry  of $u$ resp. $v$, corresponding to the edge-pair \edg{j}{k}, then the
bi-unitary condition is 
\begin{equation}
\label{biu}
\begin{array}{lcl}
\edgv{2j-1}{2k} & = & 
\sqrt{
\frac{\alpha_{_{2k+1}}\alpha_{_{2j}}}{\alpha_{_{2k}}\alpha_{_{2j-1}}}}
\edgu{2j-1}{2k}\\[0.5cm]
\edgv{2j-1}{2k+1} & = & 
\sqrt{
\frac{\alpha_{_{2k+1}}\alpha_{_{2j}}}{\alpha_{_{2k+2}}\alpha_{_{2j-1}}}}
\edgu{2j-1}{2k+1}\\[0.5cm]
\edgv{2j}{2k} & = & 
\sqrt{
\frac{\alpha_{_{2k+1}}\alpha_{_{2j}}}{\alpha_{_{2k}}\alpha_{_{2j+1}}}}
\edgu{2j}{2k}\\[0.5cm]
\edgv{2j}{2k+1} & = & 
\sqrt{
\frac{\alpha_{_{2k+1}}\alpha_{_{2j}}}{\alpha_{_{2k+2}}\alpha_{_{2j+1}}}}
\edgu{2j}{2k+1}
\end{array} 
\end{equation}
If we can find a solution with
\[ \begin{array}{lcl}
\edgu{2j-1}{2k} & = & 
\sqrt{
\frac{\alpha_{_{2k+1}}\alpha_{_{2j}}}{\alpha_{_{2k}}\alpha_{_{2j-1}}}}
\edgu{2k}{2j-1}\\[0.5cm]
\edgu{2j-1}{2k+1} & = & 
\sqrt{
\frac{\alpha_{_{2k+1}}\alpha_{_{2j}}}{\alpha_{_{2k+2}}\alpha_{_{2j-1}}}}
\edgu{2k+1}{2j-1}\\[0.5cm]
\edgu{2j}{2k} & = & 
\sqrt{
\frac{\alpha_{_{2k+1}}\alpha_{_{2j}}}{\alpha_{_{2k}}\alpha_{_{2j+1}}}}
\edgu{2k}{2j}\\[0.5cm]
\edgu{2j}{2k+1} & = & 
\sqrt{
\frac{\alpha_{_{2k+1}}\alpha_{_{2j}}}{\alpha_{_{2k+2}}\alpha_{_{2j+1}}}}
\edgu{2k+1}{2j}
\end{array} \]
we will get 
\[\begin{array}{lclcl}
\edgu{2j-1}{2k} & = & \edgv{2k}{2j-1} & = & x_{_{1,1}}\\
\edgu{2j-1}{2k+1} & = & \edgv{2k+1}{2j-1}& = & x_{_{1,2}}\\
\edgu{2j}{2k} & = & \edgv{2k}{2j}& = & x_{_{2,1}}\\
\edgu{2j}{2k+1} & = & \edgv{2k+1}{2j}& = & x_{_{2,2}}
\end{array} \]
so if the block 
\mbox{\scriptsize{$\left(\begin{array}{cc} x_{_{1,1}} &  x_{_{1,2}} \\
 x_{_{2,1}} &  x_{_{2,2}} \end{array}\right)$}} 
of $u$ is unitary, then a corresponding block of $v$ also equals
\mbox{\scriptsize{$\left(\begin{array}{cc} x_{_{1,1}} &  x_{_{1,2}} \\
 x_{_{2,1}} &  x_{_{2,2}} \end{array}\right)$}}, and hence $v$  becomes a
unitary matrix.

\noindent{}For $l$ odd the situation is different. 

\noindent{}The boxes of the diagram that may define a block of
$u$ are of the form
\begin{equation}
\label{oddublock}
\begin{array}{cc}
\edg{2j-1}{2k-1} & \edg{2j-1}{2k} \\
\edg{2j}{2k-1} & \edg{2j}{2k}
\end{array}
\end{equation}
and those that may define a block of  $v$ are of the form
\begin{equation}
\label{oddvblock}
\begin{array}{cc}
\edg{2j}{2k} & \edg{2j}{2k+1} \\
\edg{2j+1}{2k} & \edg{2j+1}{2k+1}
\end{array}
\end{equation}
and the bi-unitarity condition is 
\[ \begin{array}{lcl}
\edgv{2j}{2k} & = & 
\sqrt{
\frac{\alpha_{_{2k}}\alpha_{_{2j}}}{\alpha_{_{2k+1}}\alpha_{_{2j+1}}}}
\edgu{2j}{2k}\\[0.5cm]
\edgv{2j}{2k+1} & = & 
\sqrt{
\frac{\alpha_{_{2k+2}}\alpha_{_{2j}}}{\alpha_{_{2k+1}}\alpha_{_{2j+1}}}}
\edgu{2j}{2k+1}\\[0.5cm]
\edgv{2j+1}{2k} & = & 
\sqrt{
\frac{\alpha_{_{2k}}\alpha_{_{2j+2}}}{\alpha_{_{2k+1}}\alpha_{_{2j+1}}}}
\edgu{2j+1}{2k}\\[0.5cm]
\edgv{2j+1}{2k+1} & = & 
\sqrt{
\frac{\alpha_{_{2k+2}}\alpha_{_{2j+2}}}{\alpha_{_{2k+1}}\alpha_{_{2j+1}}}}
\edgu{2j+1}{2k+1}
\end{array} \]

\section{A Solution}
\setcounter{equation}{0}
We will first give patterns of signs on the blocks of $u$ and $v$, which will
show that if there is a solution to the problem, then a real solution exists.
With this pattern of signs we will then proceed to show which modulus 
is to be put in each  of the boxes determined in lemma \ref{lemma2}.

Note that if $a,b\in{\Bbb R}$ with $a^{^{2}}+b^{^{2}}=1$ then any matrix of the
form 
\mbox{\scriptsize{$\left(\begin{array}{cc} \sigma_{_{1}}a & \sigma_{_{2}}b \\
\sigma_{_{3}}b & \sigma_{_{4}}a \end{array}\right)$}}, with $\sigma_{_{i}}$
denoting either $+$ or $-$ will be unitary if exactly three of the 
$\sigma_{_{i}}$'s are equal.

We will assign a sign to each box in the diagram, such that any of the $u$
and $v$ patterns described in lemma \ref{lemma2} will correspond to unitary
matrices, provided that the $a$'s and $b$'s we assign to the entries, are such
that  $a^{^{2}}+b^{^{2}}=1$.

For $l$ even the blocks of $u$ and $v$ coincide, and the extra condition we
have put on our desired solution, implies that any pattern of sign must be
symmetric with respect to the main diagonal. A pattern that will do the job is
found in figure \ref{fig1}.

For $l$ odd a pattern is found in figure \ref{fig2}.

\begin{theorem}
With the signs listed previously, the following is a solution of the bi-unitary
matrix $u$.
\begin{enumerate}
\item{}For $t\in\{0,1,\ldots,l-1\}$ and $s\in{}\{0,1,\ldots,m-l-1\}$ put
\[\edgu{1+s+t}{l-t+s} = \sqrt{
\frac{\alpha_{_{1+t}}\alpha_{_{l-t}}}{x_{_{1}}x_{_{2}}}}\]
where
\[x_{_{1}}=\left\{
\begin{array}{llll}
\alpha_{_{2+s+t}} & \mbox{if} & s+t & \mbox{is even} \\[0.3cm]
\alpha_{_{1+s+t}} & \mbox{if} & s+t & \mbox{is odd}
\end{array} \right. \]

\[x_{_{2}}=\left\{
\begin{array}{llll}
\alpha_{_{l-t+s+1}} & \mbox{if} & s - t & \mbox{is even} \\[0.3cm]
\alpha_{_{l-t+s}} & \mbox{if} & s - t & \mbox{is odd}
\end{array} \right. \]
\item{}For $t\in\{0,1,\ldots,m-l-2\}$ and $s\in{}\{0,1,\ldots,l\}$ put
\[\edgu{1+s+t}{l+t-s+1} = \sqrt{
\frac{\alpha_{_{1+t}}\alpha_{_{l+2+t}}}{y_{_{1}}y_{_{2}}}}\]
where
\[y_{_{1}}=\left\{
\begin{array}{llll}
\alpha_{_{2+s+t}} & \mbox{if} & s+t & \mbox{is even} \\[0.3cm]
\alpha_{_{1+s+t}} & \mbox{if} & s+t & \mbox{is odd}
\end{array} \right. \]

\[y_{_{2}}=\left\{
\begin{array}{llll}
\alpha_{_{l+t-s+1}} & \mbox{if} & s - t & \mbox{is even} \\[0.3cm]
\alpha_{_{l+t-s+2}} & \mbox{if} & s - t & \mbox{is odd}
\end{array} \right. \]
\end{enumerate}
\end{theorem}
\begin{proof}
We will first disregard the limitations on which boxes of the diagram that correspond 
to entries of $u$. We will do this by assuming that every box gets a number 
assigned to it
via the statement of the theorem, and by  putting $\alpha_{_{0}}=0,$ and 
$\alpha_{_{-n}}= -\alpha_{_{n}}$. This identification of negative labeled
$\alpha'$s is justified since $\alpha_{_{n}}=\frac{\sin(nx)}{\sin(x)}$ for some
$x.$ 

\noindent{}In the $2\times{}2$ blocks we must check that the sum of the
moduli squared in a row or a column equals 1. And for the $1\times{}1$ blocks
we must check that the assigned scalar is 1.

For $l=2n$ we have the following block indices from (\ref{evenublock})
\[\left(\begin{array}{cc}
\edgu{2j-1}{2k} & \edgu{2j-1}{2k+1} \\
\edgu{2j}{2k} & \edgu{2j}{2k+1}
\end{array}\right)\]
Obviously \edgu{2j-1}{2k} =  \edgu{2j}{2k+1} and \edgu{2j-1}{2k+1} = 
\edgu{2j}{2k}. We have that  \edg{2j-1}{2k} = \edg{1+t+s}{2n-t+s} for some
choice of $s$ and $t$. Hence $1+t+s-2n+t-s = 2j-1-2k$ and we must have 
$t=n+j-k-1$. This gives that the numerator of \edgu{2j-1}{2k} is
\[\sqrt{\alpha_{_{n+j-k}}\alpha_{_{n-j+k+1}}}\]
We also have that  \edg{2j-1}{2k+1} = \edg{1+t+s}{2n+1+t-s} for some
choice of $s$ and $t$. Hence $1+t+s+2n+1+t-s = 2j+2k$ and we must have 
$t=j+k-n-1$. This gives that the numerator of \edgu{2j-1}{2k+1} is
\[\sqrt{\alpha_{_{j+k-n}}\alpha_{_{n+j+k+1}}}\]
In either case the  denominator is 
\[\sqrt{\alpha_{_{2j}}\alpha_{_{2k+1}}}\]
and hence we must check that
\begin{equation}
\label{check1}
\alpha_{_{n+j-k}}\alpha_{_{n-j+k+1}}+
\alpha_{_{j+k-n}}\alpha_{_{n+j+k+1}} =
\alpha_{_{2j}}\alpha_{_{2k+1}}
\end{equation}
If we put $p=j-k$ and $q=j+k,$ (\ref{check1}) transforms to
\[\alpha_{_{n+p}}\alpha_{_{n-p+1}}+
\alpha_{_{q-n}}\alpha_{_{n+q+1}} =
\alpha_{_{p+q}}\alpha_{_{q-p+1}}\]
Since $\alpha_{_{j}}=\frac{\sin{}(jx)}{\sin{}(x)}$ for some $x$, it is enough
to verify
\[\begin{array}{r}
(\expp{n+p})(\expp{n-p+1}) \\
+(\expp{q-n})(\expp{q+n+1}) \\
=(\expp{p+q})(\expp{q-p+1})
\end{array}\]
which is a trivial calculation.

\noindent{}For $l$ odd, $l=2n+1,$ the $2\times{}2$ blocks of $u$ are of the form
\[\left(
\begin{array}{cc}
\edgu{2j-1}{2k-1} & \edgu{2j-1}{2k} \\
\edgu{2j}{2k-1} & \edgu{2j}{2k}
\end{array}
\right)\]
Also here we have $\edgu{2j-1}{2k-1}=\edgu{2j}{2k}$ and 
$ \edgu{2j-1}{2k}=\edgu{2j}{2k-1}$.

\noindent{}$\edg{2j-1}{2k-1} = \edg{1+s+t}{2n+1-t+s}$ for some $t$ and $s$. As
before we get that $t=n+k-j$, and that the numerator of $\edgu{2j-1}{2k-1}$ is
\[\sqrt{\alpha_{_{1-j+k+n}}\alpha_{_{n+1+j-k}}}.\]
Also \edg{2j-1}{2k} = \edg{1+s+t}{2n+1+1+t-s} for some $t$ and $s$. We get 
$t=j+k-n-2,$ and that the numerator of \edgu{2j-1}{2k} is
\[\sqrt{\alpha_{_{j+k-n-1}}\alpha_{_{n+j+k+1}}}.\]
In both cases the denominator is
\[\sqrt{\alpha_{_{2j}}\alpha_{_{2k}}},\]
so we need to verify
\[\alpha_{_{1+n-p}}\alpha_{_{n+p+1}}+
\alpha_{_{q-n-1}}\alpha_{_{n+q+1}} =
\alpha_{_{p+q}}\alpha_{_{q-p}}\]
where $p=j-k$ and $q=j+k$. 

\noindent{}Using exponentials this verification is also easy.

\noindent{}A $2\times{}2$ block of $v$ looks like (\ref{oddvblock}). By the
just determined 
\[ \begin{array}{lcl}
\edgu{2j}{2k} & = & 
\sqrt{
\frac{\alpha_{_{1-j+k+n}}\alpha_{_{n+1+j-k}}}
{\alpha_{_{2j}}\alpha_{_{2k}}}}\\[0.5cm]
\edgu{2j}{2k+1} & = & 
\sqrt{
\frac{\alpha_{_{j+k-n}}\alpha_{_{n+j+k+2}}}
{\alpha_{_{2j}}\alpha_{_{2k+2}}}}\\[0.5cm]
\edgu{2j+1}{2k} & = & 
\sqrt{
\frac{\alpha_{_{j+k-n}}\alpha_{_{n+k+j+2}}}
{\alpha_{_{2j+2}}\alpha_{_{2k}}}}\\[0.5cm]
\edgu{2j+1}{2k+1} & = & 
\sqrt{
\frac{\alpha_{_{1-j+k+n}}\alpha_{_{n+1+j-k}}}
{\alpha_{_{2j+2}}\alpha_{_{2k+2}}}}
\end{array} \]
and from the bi-unitary condition (\ref{biu}), we get

\[ \begin{array}{lclcl}
\edgv{2j}{2k} & = & \edgv{2j+1}{2k+1} & = &
\sqrt{
\frac{\alpha_{_{1-j+k+n}}\alpha_{_{n+1+j-k}}}
{\alpha_{_{2j+1}}\alpha_{_{2k+1}}}}\\[0.5cm]
\edgv{2j}{2k+1} & = & \edgv{2j+1}{2k} & = &
\sqrt{
\frac{\alpha_{_{j+k-n}}\alpha_{_{n+j+k+2}}}
{\alpha_{_{2j+1}}\alpha_{_{2k+1}}}}
\end{array} \]
We thus need to verify
\[\alpha_{_{1-j+k+n}}\alpha_{_{n+1+j-k}}+
\alpha_{_{j+k-n}}\alpha_{_{n+j+k+2}} =
\alpha_{_{2j+1}}\alpha_{_{2k+1}}\]
which again is easy.

We will now look at the blocks of $u$ and $v$ which are only
$1\times{}1-$blocks. The labels of these blocks must be found in the boundary
of the diamond determined in lemma \ref{lemma2}, i.e. among
\begin{enumerate}
\item{}\edg{1+t}{l-t}.
\item{}\edg{m+t-l}{m-t-1}.
\item{}\edg{1+t}{l+t+1}.
\item{}\edg{l+t+1}{1+t}.
\end{enumerate}
If a block is of the form 
$\left(
\mbox{\scriptsize{$\begin{array}{cc} (1,1) &(1,2) \\(2,1) & (2,2)
\end{array}$}}\right)$ with at least one index defining an entry of $u$ resp.
$v$, then the block does not define a $2\times{}2-$block of $u$ or $v$, if in
the above cases 
\begin{enumerate}
\item{}\edg{1+t}{l-t} defines index $(2,2)$.
\item{}\edg{m+t-l}{m-t-1} defines index $(1,1)$.
\item{}\edg{1+t}{l+t+1} defines index $(2,1)$.
\item{}\edg{l+t+1}{1+t} defines index $(1,2)$.
\end{enumerate}
\noindent{}For $l$ even, $l=2n$, the blocks of $u$ and of $v$ are determined
in (\ref{evenublock}). 

\noindent{}{\bf Case 1:} 
For $t$ even the index determined is of the form \edg{2j-1}{2k} 
and hence part of a $2\times{}2-$block.

\noindent{}For $t$ odd the index determined is of the form \edg{2j}{2k+1} for 
$j=\mbox{\scriptsize{$\frac{1+t}{2}$}}$ and
$k=\mbox{\scriptsize{$\frac{2n-t-1}{2}$}}.$ The previous calculations now
give that the modulus of the corresponding entry is 
\[\taa{1+t}{2n-t}{1+t}{2n-t} = 1.\]

\noindent{}{\bf Case 2:} For $m+t$ even the index determined is of the form \edg{2j}{2k+1} 
and hence part of a $2\times{}2-$block.

\noindent{}For $m+t$ odd the index determined is of the form 
\edg{2j-1}{2k}
for $j= \mbox{\scriptsize{$\frac{m+t-2n+1}{2}$}}$ 
and
$k=\mbox{\scriptsize{$\frac{m-t-1}{2}$}}.$ And we get that the modulus 
of the corresponding entry is  \[\taa{1+t}{2n-t}{m+1-2n+t}{m-t} = 1\]
since $\alpha_{_{m+1-t }}=\alpha_{_{t}}.$

\noindent{}{\bf Case 3:} For $t$ even the index determined is of the form \edg{2j-1}{2k+1} 
and hence part of a $2\times{}2-$block.

\noindent{}For $t$ odd the index determined is of the form 
\edg{2j}{2k}
for $j= \mbox{\scriptsize{$\frac{t+1}{2}$}}$ 
and
$k=\mbox{\scriptsize{$\frac{2n+t+1}{2}$}}.$ And we get that the modulus 
of the corresponding entry is  \[\taa{1+t}{2n+t+2}{1+t}{2n+t+2} = 1.\]

\noindent{}{\bf Case 4:} This is settled like case 3.\\[0.5cm]

For $l$ odd, $l=2n+1,$ the blocks of $u$ and $v$ are given by
 (\ref{oddublock}) and (\ref{oddvblock}).

\noindent{}The $1\times{}1$ blocks of $u$:

\noindent{}{\bf Case 1:} For $t$ even the index determined is of the form \edg{2j+1}{2k-1} 
and hence part of a $2\times{}2-$block.

\noindent{}For $t$ odd the index determined is of the form 
\edg{2j}{2k}
for $j= \mbox{\scriptsize{$\frac{t+1}{2}$}}$ 
and
$k=\mbox{\scriptsize{$\frac{2n+1-t}{2}$}}.$ And we get that the modulus 
of the corresponding entry is  \[\taa{1+t}{2n+1-t}{1+t}{2n+1-t} = 1.\]

\noindent{}{\bf Case 2:} For $m+t$ odd the index determined is of the form \edg{2j}{2k} 
and hence part of a $2\times{}2-$block.

\noindent{}For $m+t$ even the index determined is of the form 
\edg{2j-1}{2k-1}
for $j= \mbox{\scriptsize{$\frac{m+t-2n}{2}$}}$ 
and
$k=\mbox{\scriptsize{$\frac{m-t}{2}$}}.$ And we get that the modulus  
of the corresponding entry is  \[\taa{2n+1-t}{t+1}{m+t-2n}{m-t} = 1.\]

\noindent{}{\bf Case 3:} For $t$ even the index determined is of the form \edg{2-j}{2k} 
and hence part of a $2\times{}2-$block.

\noindent{}For $t$ odd the index determined is of the form 
\edg{2j}{2k-1}
for $j= \mbox{\scriptsize{$\frac{t+1}{2}$}}$ 
and
$k=\mbox{\scriptsize{$\frac{2n+t+3}{2}$}}.$ And we get that the modulus 
of the corresponding entry is  \[\taa{t+1}{2n+t+3}{t+1}{2n+t+3} = 1.\]

\noindent{}{\bf Case 4:}  Is settled like case 3.\\[0.5cm]

\noindent{}The $1\times{}1$ blocks of $v$:

\noindent{}{\bf Case 1:} For $t$ odd the index determined is of the form \edg{2j}{2k} 
and hence part of a $2\times{}2-$block.

\noindent{}For $t$ even the index determined is of the form 
\edg{2j+1}{2k+1}
for $j= \mbox{\scriptsize{$\frac{t}{2}$}}$ 
and
$k=\mbox{\scriptsize{$\frac{2n-t}{2}$}}.$ And we get that the modulus 
of the corresponding entry is  \[\taa{2n+1-t}{t+1}{2n+1-t}{t+1} = 1.\]

\noindent{}{\bf Case 2:} For $m+t$ even the index determined is of the form \edg{2j+1}{2k+1} 
and hence part of a $2\times{}2-$block.

\noindent{}For $m+t$ odd the index determined is of the form 
\edg{2j}{2k}
for $j= \mbox{\scriptsize{$\frac{m+t-2n-1}{2}$}}$ 
and
$k=\mbox{\scriptsize{$\frac{m-t-1}{2}$}}.$ And we get that the modulus 
of the corresponding entry is  \[\taa{2n+1-t}{t+1}{m+t-2n}{m-t} = 1.\]

\noindent{}{\bf Case 3:} For $t$ odd the index determined is of the form \edg{2j}{2k+1} 
and hence part of a $2\times{}2-$block.

\noindent{}For $t$ even the index determined is of the form 
\edg{2j+1}{2k}
for $j= \mbox{\scriptsize{$\frac{t}{2}$}}$ 
and
$k=\mbox{\scriptsize{$\frac{2n+2+t}{2}$}}.$ And we get that the modulus 
of the corresponding entry is  \[\taa{t+1}{2n+t+3}{t+1}{2n+t+3} = 1.\]

\noindent{}{\bf Case 4:} Is settled like case 3.
\end{proof}
The calculations in the above proof gives the following
\begin{cor}
\label{cor1}
For $l$ even the moduli of the entries of $u$ are given by 
\[\begin{array}{lclcl}
\edgu{2j-1}{2k}&=&\edgu{2j}{2k+1}&=&\taa{n+j-k}{n-j+k+1}{2j}{2k+1}\\
\edgu{2j}{2k}&=&\edgu{2j-1}{2k+1}&=&\taa{j+k-n}{n+j+k+1}{2j}{2k+1}
\end{array}\]

\noindent{}For $l$ odd the moduli of the entries of $u$ are given by

\[\begin{array}{lclcl}
\edgu{2j}{2k}&=&\edgu{2j-1}{2k-1}&=&\taa{1-j+k+n}{n+j-k+1}{2j}{2k}\\
\edgu{2j-1}{2k}&=&\edgu{2j}{2k-1}&=&\taa{j+k-n-1}{n+j+k+1}{2j}{2k}
\end{array}\]
\end{cor}

\noindent{}{\bf Example} $A_{_{13}}$, $R_{_{6}}.$
For simplicity we put $\beta_{_{j}}=\sqrt{\alpha_{_{j}}}$
\newlength{\leng}
\setlength{\leng}{\arraycolsep}
\setlength{\arraycolsep}{0mm}
\begin{center}
\[\begin{array}{cccccccccccc}
\cdot{}                
& \cdot{}                
&\cdot{}                
&\cdot{}                
&\cdot{}                
& \ta{6}{1}{2}{7}  
& -\ta{8}{1}{2}{7}  
&\cdot{}                
&\cdot{}                
&\cdot{}              
 & \cdot{}                   
&\cdot{}                \\[0.3cm]
\cdot{}               
 & \cdot{}               
 &\cdot{}                
&\cdot{}                
& 1                
& \ta{8}{1}{2}{7} 
 & \ta{6}{1}{2}{7} 
 & -1               
 &\cdot{}                
&\cdot{}               
& \cdot{}                   
&\cdot{}                \\[0.3cm]
\cdot{}                
& \cdot{}                
&\cdot{}                
& \ta{3}{4}{4}{5}  
& -\ta{8}{1}{4}{5}  
& \ta{2}{5}{4}{7}  
& \ta{2}{9}{4}{7}  
& \ta{6}{1}{4}{9}  
& -\ta{10}{3}{4}{9}
&\cdot{}               
& \cdot{}                   
&\cdot{}                \\[0.3cm]
\cdot{}                
& \cdot{}                
& 1                
& \ta{8}{1}{4}{5}  
& \ta{3}{4}{4}{5}  
& -\ta{2}{9}{4}{7}  
& \ta{2}{5}{4}{7}  
& \ta{10}{3}{4}{9} 
& \ta{6}{1}{4}{9}  
& -1               
& \cdot{}                   
&\cdot{}                \\[0.3cm]
\cdot{}                
& \ta{2}{5}{3}{6}   
& -\ta{8}{1}{3}{6}  
& \ta{3}{4}{5}{6}  
& \ta{2}{9}{5}{6}  
& \ta{3}{4}{7}{6}  
& -\ta{10}{3}{7}{6} 
& \ta{2}{5}{9}{6}  
& \ta{11}{4}{9}{6} 
& \ta{6}{1}{11}{6} 
&  -\ta{12}{5}{11}{6}   
 &\cdot{}                \\[0.3cm]
  1               
& \ta{8}{1}{3}{6}   
& \ta{2}{5}{3}{6}  
& -\ta{2}{9}{5}{6}  
& \ta{3}{4}{5}{6}  
& \ta{10}{3}{7}{6} 
& \ta{3}{4}{7}{6}  
&-\ta{11}{4}{9}{6}  
& \ta{2}{5}{9}{6} 
& \ta{12}{5}{11}{6}   
&  \ta{6}{1}{11}{6} 
& -1                \\[0.3cm]
  -1               
& \ta{6}{1}{3}{8}   
& \ta{2}{9}{3}{8}  
& \ta{2}{5}{5}{8}  
& -\ta{10}{3}{5}{8} 
& \ta{3}{4}{7}{8}  
& \ta{11}{4}{7}{8} 
& \ta{3}{4}{8}{9}  
& -\ta{12}{5}{8}{9} 
& \ta{2}{5}{11}{8} 
&  \ta{6}{13}{11}{8}  
& 1               \\[0.3cm]
\cdot{}                
& -\ta{2}{9}{3}{8}   
& \ta{6}{1}{3}{8}  
& \ta{10}{3}{5}{8} 
& \ta{2}{5}{5}{8}  
& -\ta{11}{4}{7}{8} 
& \ta{3}{4}{7}{8}  
& \ta{12}{5}{8}{9}  
&\ta{3}{4}{8}{9}  
& -\ta{6}{13}{11}{8} 
& \ta{2}{5}{11}{8}
&\cdot{}                \\[0.3cm]
\cdot{}                
& \cdot{}                
& -1                
& \ta{6}{1}{5}{10} 
& \ta{11}{4}{5}{10} 
& \ta{2}{5}{7}{10}   
& -\ta{12}{5}{7}{10} 
& \ta{3}{4}{9}{10}  
& \ta{6}{13}{9}{10} 
& 1               
& \cdot{}                   
&\cdot{}                \\[0.3cm]
\cdot{}                
& \cdot{}                
&\cdot{}                
& -\ta{11}{4}{5}{10}
&\ta{6}{1}{5}{10}  
&\ta{12}{5}{7}{10}  
&\ta{2}{5}{7}{10} 
&-\ta{6}{13}{9}{10} 
&\ta{3}{4}{9}{10}  
&\cdot{}               
& \cdot{}                   
&\cdot{}                \\[0.3cm]
\cdot{}                
& \cdot{}                
&\cdot{}                
&\cdot{}                
& -1               
&\ta{6}{1}{7}{12} 
&\ta{6}{13}{7}{12} 
& 1                
&\cdot{}                
&\cdot{}               
& \cdot{}                  
 &\cdot{}                \\[0.3cm]
\cdot{}                
& \cdot{}                
&\cdot{}                
&\cdot{}                
&\cdot{}                
&-\ta{6}{13}{7}{12} 
& \ta{6}{1}{7}{12} 
& \cdot{}               
&\cdot{}                
&\cdot{}               
& \cdot{}                   
&\cdot{}                
\end{array}\]
\end{center}
\setlength{\arraycolsep}{\leng}

\newpage{}

\begin{figure}
\caption{Sign pattern for $l$ even}
\label{fig1}
\setlength{\unitlength}{5mm}
\begin{center}
\begin{picture}(15,20)
\thinlines
\multiput(15.5,18.3)(0,-1){14}{$\cdots$}
\multiput(1.5,4)(1,0){14}{$\vdots$}
\put(15,4.3){$\ddots$}
\multiput(1,19)(0,-1){15}{\line(1,0){14}}
\multiput(1,19)(1,0){15}{\line(0,-1){14}}
\multiput(1.2,18.3)(1,-1){14}{+}
\multiput(2.2,18.3)(1,-1){13}{+}
\multiput(3.4,18.3)(1,-1){12}{-}
\multiput(4.2,18.3)(1,-1){11}{+}
\multiput(5.2,18.3)(1,-1){10}{+}
\multiput(6.2,18.3)(1,-1){9}{+}
\multiput(7.4,18.3)(1,-1){8}{-}
\multiput(8.2,18.3)(1,-1){7}{+}
\multiput(9.2,18.3)(1,-1){6}{+}
\multiput(10.2,18.3)(1,-1){5}{+}
\multiput(11.4,18.3)(1,-1){4}{-}
\multiput(12.2,18.3)(1,-1){3}{+}
\multiput(13.2,18.3)(1,-1){2}{+}
\multiput(14.2,18.3)(1,-1){1}{+}

\multiput(1.2,17.3)(1,-1){13}{+}
\multiput(1.4,16.3)(1,-1){12}{-}
\multiput(1.2,15.3)(1,-1){11}{+}
\multiput(1.2,14.3)(1,-1){10}{+}
\multiput(1.2,13.3)(1,-1){9}{+}
\multiput(1.4,12.3)(1,-1){8}{-}
\multiput(1.2,11.3)(1,-1){7}{+}
\multiput(1.2,10.3)(1,-1){6}{+}
\multiput(1.2,9.3)(1,-1){5}{+}
\multiput(1.4,8.3)(1,-1){4}{-}
\multiput(1.2,7.3)(1,-1){3}{+}
\multiput(1.2,6.3)(1,-1){2}{+}
\multiput(1.2,5.3)(1,-1){1}{+}
\multiput(3,18)(2,0){6}{\circle{0.5}}
\multiput(3,16)(2,0){6}{\circle{0.5}}
\multiput(3,14)(2,0){6}{\circle{0.5}}
\multiput(3,12)(2,0){6}{\circle{0.5}}
\multiput(3,10)(2,0){6}{\circle{0.5}}
\multiput(3,8)(2,0){6}{\circle{0.5}}
\multiput(3,6)(2,0){6}{\circle{0.5}}
\put(-2.7,2){Where }
\put(0,2.2){\circle{0.5}}
\put(0.4,2){ denotes that the 4 adjacent boxes may span a block of $u$ resp. $v$}
\end{picture}
\end{center}
\end{figure}

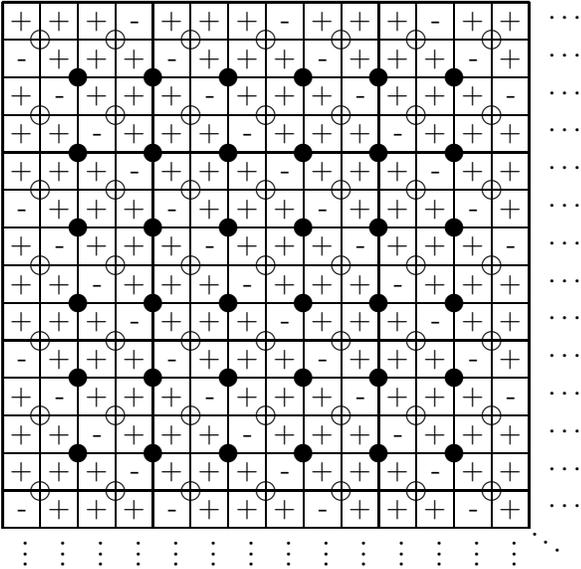
\begin{figure}
\caption{Sign pattern for $l$ odd}
\label{fig2}
\setlength{\unitlength}{5mm}
\begin{center}
\begin{picture}(15,20)
\thinlines
\multiput(15.5,18.4)(0,-1){14}{$\cdots$}
\multiput(1.5,4)(1,0){14}{$\vdots$}
\put(15,4.3){$\ddots$}
\multiput(1,19)(0,-1){15}{\line(1,0){14}}
\multiput(1,19)(1,0){15}{\line(0,-1){14}}
\multiput(1.2,18.3)(1,-1){14}{+}
\multiput(2.2,18.3)(1,-1){13}{+}
\multiput(3.2,18.3)(1,-1){12}{+}
\multiput(4.4,18.3)(1,-1){11}{-}
\multiput(5.2,18.3)(1,-1){10}{+}
\multiput(6.2,18.3)(1,-1){9}{+}
\multiput(7.2,18.3)(1,-1){8}{+}
\multiput(8.4,18.3)(1,-1){7}{-}
\multiput(9.2,18.3)(1,-1){6}{+}
\multiput(10.2,18.3)(1,-1){5}{+}
\multiput(11.2,18.3)(1,-1){4}{+}
\multiput(12.4,18.3)(1,-1){3}{-}
\multiput(13.2,18.3)(1,-1){2}{+}
\multiput(14.2,18.3)(1,-1){1}{+}
\multiput(1.4,17.3)(1,-1){13}{-}
\multiput(1.2,16.3)(1,-1){12}{+}
\multiput(1.2,15.3)(1,-1){11}{+}
\multiput(1.2,14.3)(1,-1){10}{+}
\multiput(1.4,13.3)(1,-1){9}{-}
\multiput(1.2,12.3)(1,-1){8}{+}
\multiput(1.2,11.3)(1,-1){7}{+}
\multiput(1.2,10.3)(1,-1){6}{+}
\multiput(1.4,9.3)(1,-1){5}{-}
\multiput(1.2,8.3)(1,-1){4}{+}
\multiput(1.2,7.3)(1,-1){3}{+}
\multiput(1.2,6.3)(1,-1){2}{+}
\multiput(1.4,5.3)(1,-1){1}{-}
\multiput(2,18)(2,0){7}{\circle{0.5}}
\multiput(2,16)(2,0){7}{\circle{0.5}}
\multiput(2,14)(2,0){7}{\circle{0.5}}
\multiput(2,12)(2,0){7}{\circle{0.5}}
\multiput(2,10)(2,0){7}{\circle{0.5}}
\multiput(2,8)(2,0){7}{\circle{0.5}}
\multiput(2,6)(2,0){7}{\circle{0.5}}
\multiput(3,17)(2,0){6}{\circle*{0.5}}
\multiput(3,15)(2,0){6}{\circle*{0.5}}
\multiput(3,13)(2,0){6}{\circle*{0.5}}
\multiput(3,11)(2,0){6}{\circle*{0.5}}
\multiput(3,9)(2,0){6}{\circle*{0.5}}
\multiput(3,7)(2,0){6}{\circle*{0.5}}
\put(-2.7,2){Where }
\put(0,2.2){\circle{0.5}}
\put(0.4,2){ denotes that the 4 adjacent boxes may span a block of $u$ and}
\put(-2.7,0.5){where }
\put(0,0.7){\circle*{0.5}}
\put(0.4,0.5){ denotes that the 4 adjacent boxes may span a block of $v$}
\end{picture}
\end{center}
\end{figure}

\part{Infinite Dimensional Commuting Squares}
\label{part2}

\setcounter{equation}{0}
\chapter{Hyperfinite $II_{_{1}}-$factors From Infinite Dimensional Multi-Matrix Algebras}
\label{uendeligcsq}

\setcounter{equation}{0}

\noindent{}In this chapter we will prove, that the construction of inclusions of Hyperfinite $II_{_{1}}-$factors, based on ladders of multi-matrix algebras, used previously, generalizes to constructions based on, what we have chosen to call infinite dimensional multi-matrix algebras.
\section{Preliminaries}
\setcounter{equation}{0}
An infinite  dimensional multi-matrix algebra is an infinite direct sum of full matrix 
algebras.  
\[ A = \bigoplus_{j=1}^{\infty}A_{j} , A_{j}
\cong{} M_{a_{j}}({\Bbb C})\] 
The {\em dimension vector} of $A$ is $(a_{j})_{j=1}^{\infty}$

\noindent{}A trace on $A$ is given by its action on each direct summand, i.e. 
a trace on 
$A$ is determined by the {\em trace vector} 
\[ \alpha = (\alpha_{j})_{j=1}^{\infty}\] 
where $ \alpha_{j}$ is the trace of a minimal projection in $A_{j}$

\noindent{}The trace $tr$ defined by $\alpha $ is {\em finite } if 
\begin{equation} tr({\bf1}) = \sum_{j=1}^{\infty}\alpha_{j}a_{j} < \infty 
\label{fintr} \end{equation} 
Note that \ref{fintr} implies that $\etnorm{\alpha} < \infty,$ since  $a_{j}
\geq 1$

\noindent{}If $A,B$ are infinite  dimensional multi-matrix algebras $A \subset{} B$ we define 
the  {\em
inclusion matrix} $G$ of $A \subset{} B$ by 
\[ G = (g_{ij})_{i,j=1}^{\infty}
,  g_{ij} = \mbox{ multiplicity of } A_{i} \mbox{ in } B_{j}\] 
and we write $A \subset_{G} B.$ If $\alpha,\beta$ denote trace vectors for $A,B$
defining finite traces which extend one another, and  $a,b$ denote the dimension
vectors then 
\[ \alpha = G\beta \mbox{ and } b = G^{t}a.\] 
All traces on multi-matrix algebras in the following, are assumed to be finite.

\noindent{}In \cite{S} Chapter 6 there is a discussion of countable non-negative 
matrices $T = (t_{ij})_{i,j=1}^{\infty}$ under the assumptions 
\begin{enumerate}
\item $T^{k} = (t^{k}_{ij})_{i,j=1}^{\infty}$ are all  element wise finite.
\item $T$ is irreducible, in the usual Perron-Frobenius-Theory sense.
\end{enumerate}
Theorem 6.4 of \cite{S} states

\noindent{}{\em{}If $x = (x_{i})_{i=1}^{\infty}$ is a positive right 
eigenvector of $T$ and $y = (y_{i})_{i=1}^{\infty}$ is a positive 
left eigenvector of $T$, both corresponding to the same eigenvalue then
(1) $\sum_{i=1}^{\infty}x_{i}y_{i} < \infty $ if and only if 
(2) $x$ resp. $y$ are multiples of unique right resp. left eigenvectors of $T$ 
corresponding to the largest eigenvalue of $T$.}

\noindent{}For our purposes we will be interested in infinite, locally finite, 
connected 
graphs. If $G$ is the adjacency matrix of a infinite, locally finite, connected 
graph, then $G$ is symmetric. Hence any left eigenvector of $G$ will also be a 
right eigenvector. And the above theorem allows us to conclude
\begin{cor}
\label{Scor}
If $x$ is a positive eigenvector of $G$ then

\noindent{}$\|x\|_{_{2}} < \infty\Leftrightarrow{}x$ is 
proportional to the unique positive eigenvector corresponding to the 
largest eigenvalue of $G$.
\end{cor}
For our constructions we will actually  be interested in positive eigenvectors 
$x$ satisfying $\|x\|_{_{1}} < \infty,$ but the above corollary also applies 
to such vectors.

All inclusion matrices in this chapter will be adjacency matrices for locally finite graphs, which may be finite or countably infinite.

\section{Towers of Infinite Multi-Matrix Algebras}

\setcounter{equation}{0}  
\begin{lemma}  \label{lemma21}	
 Let \mbox{\scriptsize{$\begin{array}{lcl}  B_{0} & \subset & B_{1} \\
  \cup &\,&\cup\\   A_{0} & \subset & A_{1} \end{array} $}}
	be a commuting square of infinite  dimensional multi-matrix algebras with respect to a trace  \trBen on $B_{1}.$

\noindent{}Put $B_{2} = \BenexBnul,$ 
 and  $A_{2} = \AenexBnul,$ then \begin{enumerate} \item
$\overline{A_{1}\eBnul{}A_{1}}^{weak} = wA_{2} ,$ 
 where
 $w = {\cal Z}_{A_{2}}(\eBnul) ,$ 
 the central support of  \eBnul in $A_{2}.$ \item If the representation of 
$A_{1}$ on \LAen{}  is denoted by  $\pi{},$ 
 we define   
\[ \phi : A_{2} \rightarrow \AenexAnul\]  
by  
\begin{equation}
\phi{}(x) = \pi{}( x\restrLAen), x \in{} A_{2}. \label{phidef} 
\end{equation} 
If \trBen is a Markov trace of modulus  $\beta$ for $B_{0} \subset{} B_{1},$
 and the restriction to  $A_{1},$ \trAen, is a Markov trace of modulus $\gamma$
for  $A_{0} \subset{} A_{1},$  then 	
\[A_{2} = zA_{2} \oplus ({\bf 1}-z)A_{2},\]  
where  $zA_{2} \cong \AenexAnul$  and $({\bf 1}-z)A_{2} \cong $
 a subalgebra of $A_{1}.$ \end{enumerate} Furthermore $\trAto(z) =
\frac{\gamma}{\beta}$ 

\end{lemma}

\begin{proof}
$A_{1}\eBnul{}A_{1}$  is a $*$-algebra,  and for $a \in{} A_{1}$ we have

\[ \eBnul{}a \eBnul{} = \EBnul(a)\eBnul =  \EBnul \EAen(a) \eBnul = \EAnul(a)
\eBnul, \] 
and 
\[ \eBnul{}a \eBnul{} = \eBnul\EAnul{}(a). \] 
Hence $A_{1}\eBnul{}A_{1}$  is a two-sided ideal in  alg$(A_{1},\eBnul),$
 and we have that  

\noindent{}$\overline{A_{1}\eBnul{}A_{1}}^{weak}$  is a two-sided ideal
in  $\AenexBnul= A_{2}.$
 In particular there exists a projection  $w \in{} A_{2}$ such that 
$\overline{A_{1}\eBnul{}A_{1}}^{weak}$  $= wA_{2},$  and since  $\eBnul \in{}
wA_{2}$  we must have  $w \geq {\cal Z}_{A_{2}}(\eBnul) ,$
 the central support of \eBnul{} in $A_{2}.$  Put $z = {\cal
Z}_{A_{2}}(\eBnul).$  For $x, y \in{} A_{1}$  we then have  $x \eBnul{}y \in{}
zA_{2},$  since  $z \in{} A_{2}' \subset{} A_{1}' $ i.e. 
$\overline{A_{1}\eBnul{}A_{1}}^{weak}\subset{} zA_{2} ,$
 and hence $w \leq{} z.$ This proves $z = w.$ 

\noindent{}For $a \in A_{1}$  we have  
\[ \phi{}(a) = \pi{}(a\restrLAen{} ) = \pi{}(a).\]  
Since \eBnul ``lives'' on \LBen{} the restriction to $\L^{2}(A_{1}) \subset{}
L^{2}(B_{1})$  is just composition with the orthogonal projection from 
$\L^{2}(B_{1})$ to $\L^{2}(A_{1}),$ and by the commuting square condition this 
equals \eAnul. I.e.  
\[ \phi{}(\eBnul) = \pi{}( \eBnul\restrLAen{}) = \eAnul. \] 
Since $\overline{A_{1}\eAnul{}A_{1}}^{weak} = \AenexAnul$  we have that 
\[ \phi{}(A_{1}\eBnul{}A_{1})  \mbox{ is dense in } \AenexAnul,\] 
and since $A_{1}\eBnul{}A_{1} \subset zA_{2},$   we get  $\phi{}(zA_{2})$  is
dense in \AenexAnul. Let \trBto{} be the uniquely defined Markov extension of
\trBen{} to $B_{2},$
 that is 
\[ \trBto(b\eBnul{}) = \beta{}^{-1}\trBto(b)\ \mbox{ for all } b \in B_{1},\] 
and let \trAto{} be the restriction of this trace to $A_{2}.$  Let also \trAen{}
be the trace on $A_{1}$ with Markov extension $tr'$ of modulus $\gamma$ to
\AenexAnul.

\noindent{}For $a , a'\in{} A_{1}$ we get 
\[	\trAto(a\eBnul{}a') = \trBto(a'a\eBnul) =\] 
\[	\beta{}^{-1}\trBen(a'a) = \beta{}^{-1}\trAen(a'a),\] 
and
	\[ tr'(a\eAnul{}a') = \gamma{}^{-1}\trAen(a'a) , \] 
and hence
	\[tr'\circ \phi{}(a\eBnul{}a') = \gamma{}^{-1}\trAen(a'a)
 = \mbox{$\frac{\beta}{\gamma}$}\trAto(a\eBnul{}a').\] 
Because  $\phi$ is normal we now have  
\[ tr'\circ \phi\restrLAto = \mbox{$\frac{\beta}{\gamma}$}\trAto , \] 
and the faithfulness of $tr'$ and \trAto{} implies  that $\phi$ is injective on
$zA_{1}.$

\noindent{}$z$ is the largest projection in  
$\overline{A_{1}\eBnul{}A_{1}}^{weak},$  and
hence
	$\phi{}(z) = {\bf 1},$ the largest projection in \AenexAnul, and
	$\phi{}({\bf 1 }- z) = 0.$

\noindent{}We have now established 
\begin{equation}  
\phi{}|_{zA_{2}} \mbox{ is an isomorphism of }\ zA_{2} \mbox{ onto }  \AenexAnul 
\label{21a} \end{equation}
\begin{equation} 
\phi{}|_{({\bf 1}-z)A_{2}} = 0 \label{21b} 
\end{equation}
(\ref{21b}) implies that $({\bf 1} - z)A_{2}$ is a subalgebra of $A_{1}$
 since $z = {\cal Z}_{A_{2}}(\eBnul).$  

\noindent{}Finally we have 
\[1=tr'(1)=tr'\circ{}\phi(z)=\mbox{$\frac{\beta}{\gamma}$}\trAto(z),\]
hence
\[ \trAto{}(z) = \mbox{$\frac{\gamma}{\beta}$} ,\] 
and we must have $\gamma \leq \beta$ with  equality if and only if $z = {\bf
1},$ 
 that is if and only if $A_{2}\cong \AenexAnul.$  

\end{proof}

\begin{lemma}  \label{lemma22}

Let $A \subset B \subset {\cal B}( \LB )$  be finite von Neumann-algebras, and
$p$ be a minimal central  projection of $A.$ Then $p'= J_{B}pJ_{B}$  is a
minimal central projection in \BexA{} and $\eA{}p' = \eA{}p.$ 
 
\end{lemma}

\begin{proof}
$p'= J_{B}pJ_{B}$ is a minimal central projection in \BexA,  since $\BexA$ $ =
J_{B}A'J_{B}.$  Let $\xi$ be the cyclic and separating trace vector for \LB, 
and let $x \in B.$ Then 
\renewcommand{\arraystretch}{1.5} 
\[ \begin{array}{ll}
 \eA{}J_{B}pJ_{B}(x\xi{}) = \eA{}xp\xi{} = &	 (\mbox{ since } \eA{} 
\mbox{ acts as } \EA{} \mbox{ on } B ) \\  
\EA(xp)\xi{} = p\EA(x)\xi{} =	& (\mbox{ since } p \in A' \cap{} A ) \\  
p\eA(x\xi{}) = \eA{}p(x\xi) 	&	(\mbox{ since }\ \eA \in{} A' ). 
\end{array}\]
\renewcommand{\arraystretch}{1} 
The density of $B\xi$ in \LB{} then implies $\eA{}p'$ = $\eA{}p.$
  
\end{proof}
 
\begin{lemma} \label{lemma23}

If ${\bf 1} \in{} A \subset B \subset {\cal B}(\LB ) $  are finite von
Neumann-algebras, and $p$ is a minimal  central projection of $A,$ and $f \in Ap$
is a projection,
 then $f\eA$ is a projection in $\BexA{}p',$  where $p' = J_{B}pJ_{B}.$ 
Furthermore, if $f$ is a minimal projection then $f\eA$
 is minimal in $\BexA{}p'.$   \end{lemma}

\begin{proof} 
Since $\eA \in{} A', \; f\eA$ is a projection. 

\noindent{}Let $\xi$ be the cyclic and separating trace vector for \LB{}
 and $x \in{} B$ then 
\begin{enumerate} \item	
$f\eA{}J_{B}pJ_{B}(x\xi) =$  $
f\eA{}px\xi{} =$	\hfill( by lemma \ref{lemma22} )

$f\EA(xp)\xi{} = f\EA(x)p\xi = $
	$fp\eA{}x\xi{} = f\eA(x\xi)	$\hfill(since $f \leq{} p $)

\item $J_{B}pJ_{B}f\eA(x\xi) = J_{B}pJ_{B}f\EA(x)\xi{} =$
		$f\EA(x)p\xi{}= fp\EA(x)\xi{} =$

		$f\EA(x)\xi = f\eA(x\xi)$		\hfill(since $f \leq p $) \end{enumerate}
I.e. $f\eA{}p' = p'f\eA = f\eA \Rightarrow{} f\eA \leq p',$
 hence $f\eA \in \BexA{}p'.$ 

\noindent{}Assume $f$ is minimal in $Ap$ and that $f_{0} \in \BexA$ is a 
projection  such
that $f_{0} \leq f\eA.$  Then  
\[ f_{0} \in \eA\BexA\eA, \mbox{ since } f_{0}
 \leq{} f\eA \in \eA \BexA \eA = A\eA ,\] 
hence we can find a projection  $g \in{} A,$ such that $f_{0} = g\eA \leq{}
f\eA,$ but
 $A \rightarrow{} A\eA$ is an isomorphism, so the minimality of $f$  implies $g
= f$ or $g = 0.$ Consequently $f_{0} = 0$
 or $f_{0} = f\eA,$ and $f\eA$ is minimal in \BexA.
 
\end{proof}
										
\begin{prop}  \label{prop24} 

Let \mbox{\scriptsize{$\begin{array}{lcl}
 B_{0} & \subset_{L} & B_{1} \\
  \cup_{K} &\,&\cup_{H}\\
 A_{0} & \subset_{G} & A_{1} \end{array} $}}	 be a commuting square of infinite  dimensional multi-matrix algebras with respect
to \trBen on $B_{1}$ and put $B_{2} =$ \BenexBnul{} and $A_{2} =$ \AenexBnul. 

\noindent{}If $\phi : A_{2} \rightarrow \AenexAnul$  is defined 
(as in \ref{phidef}) by  
\[\phi{}(x) = \pi{}( x\restrLAen), x \in A_{2},\] 
is  an isomorphism with inverse $\psi$ then 
\begin{enumerate}  
\item 
\[	A_{2} \subset_{K} B_{2} ,\]  
Or stated in other terms: For $p , q$  minimal central projections of $A_{0}$
resp. $B_{0},$  $q' = \psi(J_{A_{1}}qJ_{A_{1}})$ and  $p' =
J_{B_{1}}pJ_{B_{1}}$  are the corresponding minimal central projections of
$A_{2}$  resp. $B_{2},$ and  
\[\left[ (B_{0})_{pq} : (A_{0})_{pq} \right]
=  \left[ (B_{2})_{p'q'} : (A_{2})_{p'q'} \right].\]

\item{} Assume furthermore that \trBen{} is a Markov trace of modulus
 $\beta$ for $B_{0} \subset{} B_{1},$ and  the restriction to $A_{1},$ \trAen,
is a Markov trace of modulus
 $\gamma$ for $A_{0} \subset{} A_{1}$ and put 
\[B_{j} = \langle{} B_{j-1} ,e_{B_{j-2}} \rangle{} \mbox{ and } 
A_{j} = \{ A_{j-1} , e_{B_{j-2}} \}'', \mbox{ for }j \geq{} 2 ,\] 
then the inclusions are given by :

	$B_{j} \subset{} B_{j+1}$ 	is given by $L$ when $j$ 
 is even and $L^{t}$ when $j$ is odd.

	$A_{j} \subset{} A_{j+1}$	is given by $G$ when $j$  is even and $G^{t}$ when
$j$ is odd.

	$A_{j} \subset{} B_{j}$ 	is given by $K$ when $j$  is even and $H$ when $j$ is
odd. \end{enumerate} 

\end{prop}

\begin{proof} Let $f$ be a minimal projection in $A_{0}q,$  and let   $pf =
\sum_{i=1}^{n}g_{i}$
 be a decomposition into  minimal orthogonal projections in $B_{0}p.$  I.e.
\[ K_{pq} = n = \left[ (B_{0})_{pq} : (A_{0})_{pq} \right].\] 
By lemma \ref{lemma23} $f\eAnul$ is minimal in 
$\AenexAnul_{J_{A_{1}}qJ_{A_{1}}},$ hence $\psi(f\eAnul)$  $=f\eBnul{}$ is
minimal in $A_{2}q'.$  By lemma \ref{lemma21} we then have 
\renewcommand{\arraystretch}{1.5}
\[  \begin{array}{lcl} 
f\eBnul{}p' & = & \, \\
f\eBnul{}p & = & \, \\ 
fp\eBnul & = & (\mbox{ since } p \in B_{0} \mbox{ and } \eBnul \in B_{0}' ) \\
 pf\eBnul & = & \, \\ \sum_{i=1}^{n}g_{i}\eBnul. &\, \, \end{array} \] 
\renewcommand{\arraystretch}{1}
I.e. $f\eBnul{}p'$ is a sum of $n$ minimal orthogonal  projections in
$(B_{2})_{p'q'} ,$ and hence  $A_{2} \subset_{K} B_{2}.$ 

\noindent{}By assumption $A_{2} \cong \AenexAnul,$ so lemma \ref{lemma21}
 yields 
\[{\cal Z}_{A_{2}}(\eBnul) = {\bf 1} \Leftrightarrow \beta = \gamma.\]
Also by assumption and (1) we have the commuting squares 
\[\begin{array}{lclcl}
 B_{0} & \subset_{L} & B_{1} & \subset_{L^{t}} & B_{2}\\
  \cup_{K} &\,&\cup_{H}&\,&\cup_{K}\\
 A_{0} & \subset_{G} & A_{1} & \subset_{G^{t}} & A_{2} \end{array} \] 
and hence the extension \trBto{} of \trBen{} is a Markov trace for  $ B_{1}
\subset{}B_{2},$ and  $A_{2} \cong \AenexAnul$ also implies that the 
restriction of \trBto to $A_{2}$ is a Markov trace since:

\noindent{}For $a \in{} A_{1}$ we have 
\[tr_{\langle{}A_{1},e_{A_{0}}\rangle{}}(\psi{}(a\eAnul)) = 
tr_{A_{2}}(a\eBnul) =\]
\[tr_{B_{2}}(a\eBnul) = {\beta}^{-1}tr_{A_{1}}(a) = 
{\beta}^{-1}tr_{B_{1}}(a).\]
 Assume now that $\phi : A_{j+1} \rightarrow 
\langle{} A_{j} , e_{A_{j-1}} \rangle{}$ is an
isomorphism for some 

\noindent{}$j \geq 2,$
 where  $A_{j} = \{A_{j-1},e_{B_{j-2}}\}''$ is defined inductively.  We then
have the following picture 
\[\begin{array}{lclcl}
 B_{j} & \subset & B_{j+1} & \subset & B_{j+2}\\
  \cup &\,&\cup&\,&\cup\\
 A_{j} & \subset & A_{j+1} & \subset & A_{j+2} \end{array} , \] 
with $A_{j+1} \cong \langle{} A_{j} , e_{A_{j-1}} \rangle{}.$

\noindent{}According to lemma \ref{lemma21}  
\[ A_{j+2} = zA_{j+2}\oplus({\bf 1}-z)A_{j+2}, \mbox{ with }  z ={\cal
Z}_{A_{j+2}}(e_{B_{j}}).\] 
Since $z$ is central,  $zA_{j+2}$ is a two-sided ideal in $A_{j+2},$  and lemma
\ref{lemma21} also yields that $A_{j+1}e_{B_{j}}A_{j+1}$
 is a dense  $*$-subalgebra. In particular
\[\beta{}e_{B_{j-1}}e_{B_{j}}e_{B_{j-1}}= e_{B_{j}} \in{}zA_{j+2},\] 
that is
	\[zA_{j+2} \supset{} A_{j}e_{B_{j-1}}A_{j} ,\] 
and since $zA_{j+2}$ is weakly closed, we have  
\[zA_{j+2} \cong \overline{A_{j}e_{B_{j-1}}A_{j}}^{weak}
 = \{A_{j} , e_{B_{j-1}}\}'' \ni {\bf 1}.\] 
Hence $z = {\bf 1}$ and $A_{j+2} = $ $zA_{j+2} \cong$  
$\langle{} A_{j+1} , e_{A_{j}} \rangle{}
,$ and the statements concerning the  inclusion patterns follows from the first
part and induction.

\end{proof}
													 
\begin{cor} \label{cor25}

Let $\Gamma_{1}, \Gamma_{2},\Gamma_{3}$, and $\Gamma_{4}$ 
 be finite or infinite, locally finite bi-partite graphs with adjacency matrices  of a 
bi-partition $G, H, K $ resp. $L$ and Per\-ron-Fro\-be\-ni\-us vectors 
$\xi_{_{1}}, \xi_{_{2}},\xi_{_{3}}$ resp. $ \xi_{_{4}}.$  If 
\begin{equation}
\begin{array}{lcl}
 B_{0} & \subset_{L} & B_{1} \\
  \cup_{K} &\,&\cup_{H}\\
 A_{0} & \subset_{G} & A_{1} \end{array}  \label{sycsq} 
\end{equation} 
is a symmetric commuting square with respect to the finite trace, \trBen, on
$B_{1}$ given  by the corresponding partition of $\xi_{_{2}}$ resp.
$\xi_{_{4}},$ 
 we define inductively
\[ B_{j} = \langle{} B_{j-1} , e_{B_{j-2}}\rangle{}\mbox{ and  } A_{j} = 
\{ A_{j-1} , e_{B_{j-2}}\}'',\mbox{ for } j \geq 2.\] 
Then	
\[\begin{array}{lcl}
 B_{j} & \subset & B_{j+1} \\
  \cup &\,&\cup\\
 A_{j} & \subset & A_{j+1} \end{array} \] 
is a symmetric commuting square for each $j.$  And we obtain the ladder
	\[\begin{array}{lclclclc}
 B_{0} & \subset_{L} & B_{1} & \subset_{L^{t}} & B_{2}
        & \subset_{L} & B_{3} &\cdots\\
  \cup_{K} &\,&\cup_{H}&\,&\cup_{K}&\,&\cup_{H}&\,\\
 A_{0} & \subset_{G} & A_{1} & \subset_{G^{t}} & A_{2}
         & \subset_{G} & A_{3} &\cdots  \end{array} \] 
of multi-matrix algebras.

\end{cor} 

\begin{proof}
Since the square (\ref{sycsq}) is symmetric, \trBen{} is a Markov trace  for
$B_{0} \subset{} B_{1}$ of modulus \norm{ LL^{t}} and the restriction to $A_{1}$
is a Markov trace  of modulus \norm{GG^{t}}  for $A_{0} \subset{} A_{1},$ and
we  have \norm{ LL^{t}} = \norm{GG^{t}}. In the terminology of  lemma
\ref{lemma21} $\beta = \gamma$ and hence  $A_{2}\cong \AenexAnul.$ We then get
the ladder of multi-matrix  algebras by proposition \ref{prop24}.
 
\end{proof}

\section{The Limit of the Algebras}

\setcounter{equation}{0} 

\subsection{Extremality of the Trace}
 
Let $\Gamma$ be an infinite, locally finite bi-partite graph with 
Perron-Frobenius vector $\xi,$ 
with corresponding eigenvalue $\lambda.$ Let $G$ be the adjacency matrix of a
bi-partition of $\Gamma$
 and let $\xi_{_{1}}$ , $\xi_{_{2}}$ be the corresponding splitting of $\xi.$

\noindent{}Assume that  
\[A_{0} \subset_{G} A_{1} \subset_{G^{t}} A_{2} \subset_{G}
A_{3} \subset_{G^{t}} A_{4} \ldots\] 
is a tower of multi-matrix algebras, and that the obvious trace, $tr_{n},$   on
the $A_{n}$'s 
 defined by the vectors  \lampow{-(n-1)}$\xi_{_{1}}$ when $n$ is odd,
 and \lampow{-(n-1)}$\xi_{_{2}}$ when $n$ is even, is finite.

\noindent{}The induced trace on $ A_{\infty} = \cup_{n=1}^{\infty}A_{n}$
 is denoted by $tr.$

\noindent{}Assume now that $\omega_{n}$ is another trace on the $A_{n}$'s 
(extending one
another) with induced trace $\omega$  on $A_{\infty},$ with the property that $
0 < \omega_{n} \leq tr_{n}$ for all $n.$ Let $\omega_{n}$
 be given by the vector \etalow{n}. The assertion $ 0 < \omega_{n} \leq tr_{n}$
is equivalent to
 $ 0 < \etalow{n} \leq \lampow{-(n-1)}\xi_{_{i}}$ where $i = 1$ if $n$  is odd
and $i = 2\ $ if $n$ is even. I.e. we must have
	\[\etalow{2n} \leq \lampow{- (2n-1)} \xi_{_{2}}\] 
and  
\[\etalow{2n+1} \leq \lampow{-2n} \xi_{_{1}}\] 
for all n. In particular
\[ \tonorm{\etalow{2n}} \leq \lampow{-(2n-1)} \tonorm{\xi_{_{2}}} \leq 
\lampow{(-2n-1)} \tonorm{\xi} \]  
and 
\[ \tonorm{\etalow{2n+1}} \leq \lampow{-2n} \tonorm{\xi_{_{1}}} \leq  
\lampow{-2n} \tonorm{\xi} .\] 	 
I.e. 
\[\tonorm{\etalow{k}} \leq \lampow{-k+1}\tonorm{\xi}. \] 
The extension property of the $ \omega_{n}$'s is stated as
	\[ \etalow{1} = \GGt \etalow{3} = (\GGt )^{2} \etalow{5} =
 (\GGt )^{3} \etalow{7} \ldots \] 
and 
\[ \etalow{2} = \GtG \etalow{4} = (\GtG)^{2} \etalow{6} =
 (\GtG )^{3} \etalow{8} \ldots \] 
Consider $k = 2n + 1 , n \in{} {\Bbb N} .$ Let $\phi_{_{k}}$ equal the
projection of \etalow{k} on $\xi_{_{1}},$  and put
\[\psi_{_{k}} = \etalow{k} - \phi_{_{k}}.\] 
Since \GGt{} leaves ${\Bbb C}\xi_{_{1}}$ and its orthogonal complement
invariant,  we have
\[\GGt\phi_{_{k}} = \phi_{_{k-2}} \mbox{ and } \GGt \psi_{_{k}} =
\psi_{_{k-2}}.\] 
Consider the functions $ f_{n}(t) = t^{n} , t \in [0,1].$ 

\noindent{}Since
$\|\,f_{n}\,\|_{_{_{\infty}}} \leq{} 1$ and  $ f_{n} \stackrel{n \rightarrow
\infty}{\rightarrow} \mbox{{\large$\chi_{_{\{1\}}}$}}$ pointwise, and we have
\[ f_{n} ( \frac{1}{\lambda^{2}}GG^{t}) 
\stackrel{n \rightarrow \infty}{\rightarrow} 
\mbox{\Large{$\chi_{_{\{1\}}}$}}(\frac{1}{\lambda^{2}}GG^{t})  \mbox{
strongly. }\] 
I.e. 
\begin{equation} 
f_{n} (\frac{1}{\lambda^{2}}GG^{t}) \stackrel{n \rightarrow \infty}{\rightarrow} 
\mbox{ the projection onto }
\xi_{_{1}}  \label{xiproj} 
\end{equation} 
We now have  
\[\tonorm{\psi_{_{k}}}
= \tonorm{(GG^{t})^{n}\psi_{_{k+2n}}} \rightarrow{} 0 \mbox{ as } n
\rightarrow \infty \] 
hence $\tonorm{\psi_{_{k}}} = 0$ for all odd $k$

\noindent{}Similarly we get $\tonorm{\psi_{_{k}}} = 0$ for all even $k.$ 

\noindent{}I.e. $\etalow{k} = c_{k}\xi_{_{1}}$ for $k$ odd and $\etalow{k} =
c_{k}\xi_{_{2}}$ for $k$ even.

\noindent{}Let $k$ be odd, then  
\[ GG^{t}\etalow{k} = \etalow{k-2} =c_{k-2}\xi_{_{1}}\]
and 
\[ GG^{t}\etalow{k} = \lambda^{2}c_{k}\xi_{_{1}}\] 
hence $c_{k-2} = \lambda^{2}c_{k}$  which yields  $c_{2n + 1} =
\lambda^{-2n}c_{1}$

\noindent{}The same way we get $c_{2n} = \lambda^{-2n +1}c_{2}$

\noindent{}From $G\etalow{2} = \etalow{1} = c_{1}\xi_{_{1}}$ and  $Gc_{2}\xi_{_{2}} =
\lambda{}c_{2}\xi_{_{1}}$ we also have $c_{2} = \lambda^{-1}c_{1}$

\noindent{}Hence $tr' = c_{1}tr,$ implying that $tr$ is extremal.

\subsection{Construction of Subfactors}

Let $ \Gamma_{1} , \Gamma_{2} , \Gamma_{3} , \Gamma_{4}$ be bi-partite, 
locally finite graphs, 
and $G, H, K$ and $L$ the adjacency matrices of a bi-partition. Let
\mbox{\scriptsize{$\begin{array}{lcl}
 B_{0} & \subset_{L} & B_{1} \\
  \cup_{K} &\,&\cup_{H}\\
 A_{0} & \subset_{G} & A_{1} \end{array} $}} be a symmetric commuting square
with respect to the finite trace defined  by the Perron- Frobenius vector of
$L^{t}L$ and its restrictions to the other  multi-matrix algebras. Since the
square is a symmetric commuting square  $\norm{GG^{t}} = \norm{LL^{t}},$ so by
corollary \ref{cor25}, and the construction herein, we  get the infinite ladder
of multi-matrix algebras
	\[\begin{array}{lclclclc}
 B_{0} & \subset_{L} & B_{1} & \subset_{L^{t}} & B_{2}
        & \subset_{L} & B_{3} &\cdots\\
  \cup_{K} &\,&\cup_{H}&\,&\cup_{K}&\,&\cup_{H}&\,\\
 A_{0} & \subset_{G} & A_{1} & \subset_{G^{t}} & A_{2}
         & \subset_{G} & A_{3} &\cdots  \end{array} \] 
with traces $ \trA_{n}$ and $ \trBn$ extending each other. 

\noindent{}The induced trace on
the inductive limit  $B_{\infty} = \bigcup_{n=1}^{\infty}B_{n}$ is denoted by $
\trB_{\infty}.$ Put $A_{\infty} = \bigcup_{n=1}^{\infty}A_{n},$ the inductive
limit of the $A_{n}$'s, with limit of traces denoted by $ \trA_{\infty}.$  Then
$A_{\infty} \subset{} B_{\infty},$ and $ \trB_{\infty}$ extends $
\trA_{\infty}.$ Let $B$ equal the weak closure of the G-N-S-representation  of
$B_{\infty},$ and $A$ equal the weak closure of $A_{\infty}$ in $B.$ Then $A
\subset{} B.$

\noindent{}By Kaplansky's density theorem  $Unitball(B) =
\overline{Unitball(B_{\infty})}^{2-norm}.$  The unitball of the weak closure of
$A_{\infty}$ in $B$ is,
 according to Kaplansky, equal to  $\overline{Unitball(A_{\infty})}^{2-norm},$ 
 and since $ \trB_{\infty}$ extends $ \trA_{\infty}$  the weak closure of
$A_{\infty}$ in $B$ equals the  weak closure of $A_{\infty}$ in the
G-N-S-representation of  $A_{\infty}$ with respect to $ \trA_{\infty}.$  The
previous part of this section shoved that the traces $ \trA_{\infty}$ resp. 
 $ \trB_{\infty}$ are extremal, and hence $A \subset{} B$  are hyperfinite
$II_{1}$-factors with traces \trA{} resp. \trB.

\noindent{}From now on consider the $A_{n}$'s and the $B_{n}$'s  as algebras 
represented on
$L^{2}(B_{\infty},\trB_{\infty}) = $ \LB, we then have the following

\begin{lemma} \label{lemma32}

Let $e : \LB \rightarrow \LA$ be the orthogonal projection. For all $n$ the
restriction of $e$ to $L^{2}(B_{n},\trBn),$
 $e|_{L^{2}(B_{n},\trBn)},$ equals  $ e_{n} : L^{2}(B_{n},\trBn) \rightarrow{}
L^{2}(A_{n},\trA_{n}) .$  \end{lemma}

\begin{proof} 
The commuting square condition implies that $e_{n+1}$ extends $e_{n},$  hence we
can define a surjection  
\[f : \bigcup_{n = 1}^{\infty}L^{2}(B_{n},\trBn) \rightarrow{}
 \bigcup_{n = 1}^{\infty}L^{2}(A_{n},\trA_{n}) \] 
by
	\[f|_{L^{2}(B_{n},\trBn)} = e_{n}.\] 
Then $f$ is linear, $f^{2} = f,$ $ \norm{f} \leq{} 1$ and  $(fx,y) = (x,fy)$ for
all 

\noindent{}$x,y \in{} \bigcup_{n = 1}^{\infty}L^{2}(B_{n},\trBn) .$  Hence $f$  
has a unique extension to
\[g : \overline{\bigcup_{n = 1}^{\infty}L^{2}(B_{n},\trBn)}
\rightarrow
  \overline{\bigcup_{n = 1}^{\infty}L^{2}(A_{n},\trAn)}\] 
$g^{2} = g$ and $g^{*} = g ,$ hence $g$ is an orthogonal projection,   and its
image is closed and contains $\bigcup_{n = 1}^{\infty}L^{2}(A_{n},\trAn).$  This
now implies that $g$ is onto 
 $\overline{\bigcup_{n = 1}^{\infty}L^{2}(A_{n},\trAn)}.$ 

\noindent{}Since $\bigcup_{n = 1}^{\infty}B_{n}$ is dense in $B$ in the  
\tonorm{\cdot }-norm we have 
\[ \LB \supset \overline{\bigcup_{n = 1}^{\infty}L^{2}(B_{n},\trBn)} 
\supset \LB\]	 
hence 
\[ \LB = \overline{\bigcup_{n =1}^{\infty}L^{2}(B_{n},\trBn)}.\] 
Similarly we get 
\[ \LA = \overline{\bigcup_{n = 1}^{\infty}L^{2}(A_{n},\trAn)}.\] 
I.e. $g$ is the orthogonal projection of \LB{} onto \LA, that is $g = e ,$  and
hence  $e|_{L^{2}(B_{n},\trBn)} = e_{n}$ for all $n.$  

\end{proof}

\section{A Trace on $\langle{} B , e \rangle{}$ and the Index}

\setcounter{equation}{0} 
Assume the symmetric commuting squares of infinite  dimensional multi-matrix algebras
\[\begin{array}{lclclclc}
 B_{0} & \subset_{L} & B_{1} & \subset_{L^{t}} & B_{2}
        & \subset_{L} & B_{3} &\cdots\\
  \cup_{K} &\,&\cup_{H}&\,&\cup_{K}&\,&\cup_{H}&\,\\
 A_{0} & \subset_{G} & A_{1} & \subset_{G^{t}} & A_{2}
         & \subset_{G} & A_{3} &\cdots  \end{array} \] 
all algebras considered represented on \LB, and let $e$ denote the orthogonal
projection  $e : \LB \rightarrow \LA,$ $A$ and $B$ as previously.

\noindent{}By lemma \ref{lemma32}  $e|_{L^{2}(B_{n},\trBn)} = e_{n}$ is the fundamental
projection of

\noindent{}$A_{n}|_{L^{2}(B_{n},\trBn)} \subset B_{n}|_{L^{2}(B_{n},\trBn)}$ 
hence 
$\langle{}B_{n} , e \rangle{}\, \cong \,
\langle{} B_{n}|_{L^{2}(B_{n},\trBn)},e_{n} \rangle{}.$  Let $p_{n}$
denote the ortho\-gonal projection of \LB{} onto $L^{2}(B_{n},\trBn),$ then
$p_{n}$ is the projection  corresponding to the fundamental construction for
$B_{n} \subset B,$ 
 and $ep_{n} = p_{n}e$ since \mbox{\scriptsize{$\begin{array}{lcl}
 B_{n} & \subset & B \\
  \cup &\,&\cup\\
 A_{n} & \subset & A \end{array} $}}
 is a commuting square.

\noindent{}$\langle{} B_{n} ,e \rangle{}$ has a unique normal trace 
$tr_{\langle{}B_{n},e\rangle{}}$ such that 
	\[ tr_{\langle{}B_{n},e\rangle{}}(xe) = \beta^{-1}tr_{B_{n}}(x)  
\mbox{ for  all } x \in{}B_{n} ,\] 
where $\beta$ is the Perron-Frobenius eigenvalue of $HH^{t}.$ We will now show
that for $x \in \langle{} B , e \rangle{},$
	\[ p_{n}xp_{n}|_{L^{2}(B_{n},\trBn)} \in \langle{} B_{n} , e
\rangle{}|_{L^{2}(B_{n},\trBn)}.\] 
To prove this it is enough to consider $x \in{} B \cup{} eBe,$  which is a dense
subalgebra of $\langle{} B , e \rangle{}.$ 
\begin{enumerate} \item $x = b \in B.$ 
\renewcommand{\arraystretch}{2}
\[ \begin{array}{lcl}
 p_{n}bp_{n}|_{L^{2}(B_{n},\trBn)} & = & \EBn (b) p_{n}|_{L^{2}(B_{n},\trBn)}
=\\
	 \EBn (b) |_{L^{2}(B_{n},\trBn)} & \in & B_{n}|_{L^{2}(B_{n},\trBn)},
\end{array} \] 
\renewcommand{\arraystretch}{1} 
where \EBn{} is the trace preserving conditional expectation of $B$ onto
$B_{n}.$ \item $x = ebe,b \in{} B.$
\renewcommand{\arraystretch}{2}
	\[ \begin{array}{lcl}
 p_{n}ebep_{n}|_{L^{2}(B_{n},\trBn)} & = & \ \\
 ep_{n}bp_{n}e|_{L^{2}(B_{n},\trBn)} & = & \ \\
 e\EBn(b)ep_{n}|_{L^{2}(B_{n},\trBn)}& = & \ \\
	 e\EBn (b)|_{L^{2}(B_{n},\trBn)} & \in & \langle{} B_{n} , 
e \rangle{}|_{L^{2}(B_{n},\trBn)}.
\end{array}\] 
\renewcommand{\arraystretch}{1}
\end{enumerate} 
For $n \in {\Bbb N}$ define a positive normal state $ \tau_{n}$ on $\langle{} B , e \rangle{}$ by
	\[\tau_{n}(x) = 
tr_{\langle{}B_{n},e\rangle{}}(p_{n}xp_{n}|_{L^{2}(B_{n},\trBn)})  
\mbox{ for } x \in \langle{} B , e \rangle{}.\] 
We will show that $\tau_{n}$ is independent of $n.$

\noindent{}Again it is enough to consider $x \in{} B \cup{} eBe $ 
\begin{enumerate} 
\item $x = b \in B.$
	\[ \tau_{n}(b) = tr_{B_{n}}(\EBn(b)) = \trB(b) \] 
\item $x = ebe,b \in{} B.$
	\[ \tau_{n}(x) = tr_{\langle{}B_{n},e\rangle{}}(e\EBn(b)e) = \] 
\[\beta^{-1}tr_{B_{n}}(\EBn(b)) = \beta^{-1} \trB(b) \] 
\end{enumerate} 
I.e. $\tau_{n}$ is independent of $n.$

\noindent{}On $\bigcup_{n=1}^{\infty}\langle{} B_{n} , e \rangle{} ,$  
which is weakly dense in $\langle{} B , e
\rangle{},$ put $\tau = \tau_{n}$ ``for all $n$''.

For $b \in{} B_{n}$ we get, 
\begin{enumerate} 
\item  \[\tau(b) = \tau_{n}(b) =
tr_{B_{n}}(\EBn(b)) = tr_{B_{n}}(b) = tr_{\langle{}B_{n},e\rangle{}}(b)\] 
and 
\item  \[\tau(ebe) = \tau_{n}(ebe) = tr_{\langle{}B_{n},e\rangle{}}(e\EBn(b)e) = 
tr_{\langle{}B_{n},e\rangle{}}(ebe),\] 
\end{enumerate} 
 i.e. $\tau$ extends all the $tr_{\langle{}B_{n},e\rangle{}}$'s,
 and hence $\tau$ is a trace on 
$\langle{} B ,e \rangle{}$ extending the trace \trB{} on $B.$ 

\noindent{}For any $b \in{} B_{n}$ we have
	\[\tau(be) = tr_{\langle{}B_{n},e\rangle{}}(e\EBn(b)e) =  
\beta^{-1}tr_{B_{n}}(\EBn(b)) =
\beta^{-1}\trB(b),\] 
and since $B = \overline{\bigcup_{n=1}^{\infty} B_{n}}^{weak}$ we get  $\tau(be)
= \beta^{-1} \trB(b)$ for all $b \in{} B.$  I.e. $\tau$ is a Markov extension 
of \trB{} to $\langle{} B , e \rangle{}$ of modulus $\beta.$

\noindent{}The index $[B : A] = [\langle{} B , e \rangle{} : B]$ is 
now determined as  $\tau(e)^{-1} =
\beta.$

\section{The Dimension of the Relative Commutant of $A$ in $B$}

\setcounter{equation}{0} Let $\Gamma$ be an infinite, locally finite 
bi-partite graph  with
Perron-Frobenius vector $\xi,$ $G$ the  adjacency matrix of a bi-partition, and
$\ \xi_{_{1}} , \xi_{_{2}}$ the corresponding splitting of $\xi.$ 

\noindent{}Assume that  
\[ A_{0} \subset_{G} A_{1}  \subset_{G^{t}} A_{2} \subset_{G} A_{3} \cdots\] 
is a tower of multi-matrix algebras, with finite trace defined by $\xi$ that is,
if say $\xi_{_{1}}$ defines the trace on $A_{0}$ 
\[ tr({\bf 1}) = \sum_{j=1}^{\infty}a_{j}^{0}(\xi_{_{1}})_{j} < \infty \] 
where $a_{n} = (a_{i}^{n})_{i = 1}^{\infty}$ is the dimension vector of 
 $A_{n}.$  We then have the following proposition 

\begin{prop} \label{prop61}

In the above situation put $\xi_{_{1}}$ equal to the Perron-Frobeni\-us vector
of \GGt{} and $\xi_{_{2}}$ equal to the Perron-Frobenius vector of \GtG. 
 Then the dimension vectors converge pointwise to a multiple of  $\xi_{_{1}}$
when $n$ is even resp. a multiple of $\xi_{_{2}}$ when $n$ is odd. 

\end{prop}

\begin{proof} 
Since $a_{j}^{0} > 1$ we must have $\etnorm{\xi_{_{1}}} < \infty.$ 

\noindent{}Formally the dimension of the $i$'th summand in $A_{2n}$ is given as 
\[ \langle{}(\GGt)^{n}a_{0},\delta_{i}\rangle{} ,  
a_{0} = \left ( \begin{array}{c} a_{1}^{0}\\
a_{2}^{0}\\ \vdots \end{array} \right ) , 
(\delta_{i})_{j} = \left \{
\renewcommand{\arraystretch}{1.5}
\begin{array}{cl} 1 & \mbox{ if } j=i\\ 0 & \mbox{ if } j \neq{} i \end{array}
\renewcommand{\arraystretch}{1}
\right.\] 
Fubini's theorem for positive functions gives
\[\langle{}(\lGGt)^{n}a_{0},\delta_{i}\rangle{} = 
\langle{}a_{0},(\lGGt)^{n}\delta_{i}\rangle{}\] 
Choose $\xi_{_{1}}$ with $\tonorm{\xi_{_{1}}} = 1,$ and let 
$\xi_{_{1}},x_{1},x_{2},\ldots$ be an orthonormal basis for $l^{2}({\Bbb N})$
then 
\[(\lGGt)^{n}\delta_{i} = 
\langle{}(\lGGt)^{n}\delta_{i},\xi_{_{1}}\rangle{}\xi_{_{1}} +
\sum_{j=1}^{\infty}\langle{}(\lGGt)^{n}\delta_{i},x_{j}\rangle{}x_{j}\] 
and
\[\tonormsq{(\lGGt)^{n}\delta_{i} - (\xi_{_{1}})_{i}\xi_{_{1}}} =\] 
\[\tonormsq{\langle{}(\lGGt)^{n}\delta_{i},\xi_{_{1}}\rangle{}\xi_{_{1}} - 
(\xi_{_{1}})_{i}\xi_{_{1}}} + 
 \sum_{j=1}^{\infty}\tonormsq{\langle{}(\lGGt)^{n}\delta_{i},x_{j}\rangle{}x_{j}} \] 
Each $x_{j} \in ({\Bbb C}\xi_{_{1}})^{\perp}$ so by (\ref{xiproj}) we get
$(\lGGt)^{n}x_{j} \stackrel{n \rightarrow \infty}{\rightarrow} 0$ for any $j.$ 
Hence 
\[ \tonorm{\langle{}(\lGGt)^{n}\delta_{i},x_{j}\rangle{}x_{j}} = 
\tonorm{\langle{}\delta_{i},(\lGGt)^{n}x_{j}\rangle{}x_{j}} 
\stackrel{n \rightarrow\infty}{\rightarrow} 0 \] 
and 
\[\lim_{n \rightarrow\infty}\tonorm{(\lGGt)^{n}\delta_{i} - 
(\xi_{_{1}})_{i} \xi_{_{1}}} = 
\lim_{n \rightarrow \infty}
\tonorm{\langle{}(\lGGt)^{n}\delta_{i},\xi_{_{1}}\rangle{}\xi_{_{1}}-
(\xi_{_{1}})_{i}\xi_{_{1}}} \] 
For any $n$ we have 
\[ \langle{}(\lGGt)^{n}\delta_{i},
\xi_{_{1}}\rangle{} = \langle{}\delta_{i},(\lGGt)^{n}\xi_{_{1}}\rangle{} =
\langle{}\delta_{i},\xi_{_{1}}\rangle{} = (\xi_{_{1}})_{i}\] 
i.e. 
\[ (\lGGt)^{n}\delta_{i} \stackrel{\|\cdot\|_{_{_{2}}}}{\rightarrow}  
(\xi_{_{1}})_{i}\xi_{_{1}} \mbox{ as } n\rightarrow \infty.\] 
We have $\delta_{i} \leq \frac{1}{(\xi_{_{1}})_{i}}\xi_{_{1}},$ and since 
\etnorm{\xi_{_{1}}}$ < \infty \ ,\frac{1}{(\xi_{_{1}})_{i}}\xi_{_{1}}$ is
summable, and we get
\[\langle{}a_{0},(\lGGt)^{n}\delta_{i}\rangle{} = \sum_{j=1}^{\infty}a_{j}^{0}
((\lGGt)^{n}\delta_{i})_{j} \stackrel{n\rightarrow \infty} {\rightarrow}
\sum_{j=1}^{\infty}a_{j}^{0}(\xi_{_{1}})_{j}(\xi_{_{1}})_{i} = tr({\bf 1})
(\xi_{_{1}})_{i}\]

\noindent{}A similar argument holds for the odd labeled floors of the tower 

\end{proof}

\begin{remark} \label{rem1} 

{\rm If we are in the above situation then the trace vector $\alpha_{n}$ of
$A_{n}$ is given as 
\[\alpha_{2k} = \frac{1}{\lambda^{2k}}\xi_{_{1}} \;,\;
\alpha_{2k+1} = \frac{1}{\lambda^{2k+1}}\xi_{_{2}}\] 
and we have  
\[\tonorm{\xi_{_{2}}} = \langle{}\xi_{_{2}},\xi_{_{2}}\rangle{} = 
\frac{1}{\lambda^{2}}\langle{}G^{t}\xi_{_{1}},G^{t}\xi_{_{1}}\rangle{} = 
\tonorm{\xi_{_{1}}}\]

\noindent{}Let $z^{k}_{i}$ be the minimal central projection in $A_{k}$ 
corresponding to the $i$'th component. Then
\renewcommand{\arraystretch}{2} 
\[ \begin{array}{lcl} tr(z^{2l}_{i}) & = &
(\alpha_{2l})_{i}(a_{2l})_{i} \\ \ & = & \lambda^{2l}
(\alpha_{2l})_{i}\frac{1}{\lambda^{2l}}(a_{2l})_{i}\\ \ & = &
(\xi_{_{1}})_{i}( (\lGGt)^{l}1)_{i}\\ \ & \stackrel{n \rightarrow
\infty}{\rightarrow} & (\xi_{_{1}})_{i}(\xi_{_{1}})_{i}tr({\bf 1}) >
0. \end{array} \]
\renewcommand{\arraystretch}{1} 
In particular there exists a constant $c_{i}$ independent of $l$ s.t.  
\[ tr(z^{2l}_{i}) \geq{} c_{i} \mbox{ for all } l \] 
A similar argument holds for the odd labeled $A_{n}$'s.} 

\end{remark}

\noindent{}The following is essentially contained in \cite{Wen1} 

\begin{lemma} \label{lemma62}

Let $\{ \alpha_{1},\ldots ,\alpha_{m} \}$ be $m$ different real numbers, and
$\{ t_{1}, \ldots ,t_{m} \}$ be positive numbers with sum $1.$ Then there
exists  $\epsilon > 0$ such that:

\noindent{}For any $ II_{1}$-factor $A$ and two selfadjoint elements 
$ a,b \in{} A$ 
satisfying 
\begin{enumerate} 
\item $ a = \sum_{i=1}^{m} \alpha_{i}p_{i} , \{
p_{i} \}$ orthogonal projections  with $tr(p_{i}) = t_{i}$ 
\item $b$ has strictly less than $m$ spectral values, and  \norm{b} $\leq $
\norm{a} \end{enumerate} 
then \tonormsq{a-b} $\geq \epsilon$ 

\end{lemma} 

\begin{proof}
Let $k < m$ and put  
\[ {\cal K} = \left\{  (\beta,V)\left|\begin{array}{l}
 \beta =(\beta_{1},\ldots ,\beta_{k}) ,\\  
 0 \leq\beta_{j} \leq \max\{ |\alpha_{i}|\} ,\\ 
 V = (v_{ij}) \in M_{m\times{} k}([0,1]) , \sum_{j}^{}v_{ij} = t_{i} 
\end{array} \right.\right\} \] 
Then $\cal K$ is compact and 
\[F(\beta,V) = \sum_{i,j}^{}(\alpha_{i} - \beta_{j})^{2}v_{ij}\] 
has a minimum on $\cal K.$ Assume this minimum is
attained at $(\beta',V').$ Since $k < m$ there exists $i$ s.t. $\alpha_{i}
\not\in  (\beta_{1}',\ldots ,\beta_{k}') $ and $t_{i} > 0$ implies that at
least one $v_{ij}' > 0.$ Hence $F(\beta',V') > 0.$ 

\noindent{}Put $\epsilon = F(\beta',V')$, and let $A$ be a $II_{1}$-factor, 
and let $a,b\in{} A_{sa}$ satisfy $(1)$ and $(2).$ 

\noindent{}Let $ b =\sum_{j=1}^{k}\beta_{j}q_{j} $  
be the spectral decomposition of $b.$ Since
$\sum_{}^{}p_{j} = {\bf 1} =  \sum_{}^{}q_{i}$ we have 
\[ a = \sum_{i,j}^{}\alpha_{i}p_{i}q_{j} \mbox{ and } b =
\sum_{i,j}^{}\beta_{j}p_{i}q_{j}\]
 i.e. 
\[ (a-b) = \sum_{i,j}^{}(\alpha_{i} - \beta_{j})p_{i}q_{j} \mbox{ and } \] 
\renewcommand{\arraystretch}{2}
\[ \begin{array}{lc} 
(a-b)^{2} & = \\
(a-b)(a-b)^{*} & = \\ 
\sum_{i,j}^{}(\alpha_{i} - \beta_{j})^{2}p_{i}q_{j}p_{i} + & \ \\ 
\sum_{i,j,i',j',i \neq{} i',j \neq{}
j'} (\alpha_{i} - \beta_{j}) (\alpha_{i'} - \beta_{j'})p_{i}q_{j}q_{j'}p_{i'}
& .  \end{array} \] 
\renewcommand{\arraystretch}{1}
Hence 
\[ tr((a-b)^{2}) = \sum_{i,j}^{}(\alpha_{i} - 
\beta_{j})^{2}tr(p_{i}q_{j}p_{i}) \] 
since $tr(p_{i}q_{j}p_{i'}q_{j'}) = 0$ when $(i,j) \neq (i',j').$ 

\noindent{}Since 
\[ 0\leq{} tr(p_{i}q_{j}p_{i}) \leq
\sum_{j}^{}tr(p_{i}q_{j}p_{i}) =  tr(p_{i}) = t_{i}\] 
$W = (w_{ij}) , w_{ij}
= tr(p_{i}q_{j}p_{i})$ is a matrix of the type defining the second coordinate
of $\cal K,$ and we have obtained 
\[ \tonormsq{a-b} = tr((a-b)^{2}) =F{(\beta,W)} \geq{} F(\beta',V') = \epsilon\]

\end{proof} 

\begin{theorem} \label{thm64} 

If $ \Gamma_{G}, \Gamma_{H}, \Gamma_{K}, \Gamma_{L}$ are finite or infinite, 
locally finite bi-partite  
graphs with Per\-ron-Fro\-be\-ni\-us vectors  $\xi_{1},\xi_{2},\xi_{3},\xi_{4}$ 
defining finite traces on the ladder  
\[\begin{array}{lclclclc}
 B_{0} & \subset_{L} & B_{1} & \subset_{L^{t}} & B_{2}
        & \subset_{L} & B_{3} &\cdots\\
  \cup_{K} &\,&\cup_{H}&\,&\cup_{K}&\,&\cup_{H}&\,\\
 A_{0} & \subset_{G} & A_{1} & \subset_{G^{t}} & A_{2}
         & \subset_{G} & A_{3} &\cdots  \end{array} \] 
of multi-matrix algebras, and $A \subset B$ are $II_{1}$-factors constructed from
the tower in the usual way. Then 
\[ {\rm dim}\{A' \cap{} B\} \leq \left(\min\{\mbox{ 1-norm of rows of } K  
\mbox{ and } H\right\})^{^{2}}\] 

\end{theorem} 

\begin{proof}
Denote the steps of the ladder by $ \bigoplus_{i} A_{i}^{n}$ and  $ \bigoplus_{j}
B_{j}^{n}$ with dimension vectors  $a_{n} = (a_{i}^{n})$ and $b_{n} =
(b_{j}^{n}).$ 

\noindent{}Assume $n$ is even $n = 2l.$ 
 Put $ m_{0} = \min\{\mbox{ 1-norm of rows of } K \}$ and choose $i_{0}$ s.t.
the
 $i_{0}$'th row of $k$ has \etnorm{\cdot} $= m_{0}$

Let $z^{2l}_{i_{0}}$ be the corresponding minimal central projection in
$A_{2l}.$  
\[ z^{2l}_{i_{0}} \in {\cal Z}(A_{2l}) \subset A_{2l}' \cap{} B_{2l}\] 
Let $q_{j}^{2l}$ be the minimal central projection in $B_{2l}$
corresponding to $ B_{j}^{2l}$

\noindent{}$A_{2l}q_{j}^{2l}$ is of the form 
\[\underbrace{ \left (  \begin{array}{ccccc}
x_{1} &\, &\,&\,&\,\\ \, &\ddots &\,&\,&\,\\ \,&\,& x_{1} &\,&\,\\ \,&\,&\,&
x_{2} &\, \\ \,&\,&\,&\,&\ddots  \end{array} \right ) }_{b_{j}^{2l}} \] 
where $x_{i} \in{} A_{i}^{2l}$ is repeated $K_{ij}$-times. In this setting
$z^{2l}_{i_{0}}q_{j}^{2l}$ has
\renewcommand{\arraystretch}{1.5} 
\[ x_{i} = \left \{ \begin{array}{cc} 
{\bf 1}_{a^{2l}_{i_{0}} \times{} a^{2l}_{i_{0}}} & \mbox{ if } i= i_{0}\\ 
0 &\mbox{ otherwise.} \end{array} \right. \]
\renewcommand{\arraystretch}{1} 
We also have 
\[ A_{2l}' \cap{} B_{2l} \cong \bigoplus_{i,j}^{}M_{K_{ij}}({\Bbb C})\] 
and $z^{2l}_{i_{0}}$ corresponds to
the identity in 
\[\bigoplus_{j}^{}M_{K_{i_{0}j}}({\Bbb C})\] 
and hence $z^{2l}_{i_{0}}$ can be split in $m_{0}$ orthogonal projections in 
$A_{2l}' \cap{} B_{2l}$

\noindent{}Let $\{p_{j} \,|\,j=1,\ldots,m\}$ be a maximal splitting of the 
identity in
$A' \cap{} B$ into minimal non-trivial projections, and put 
\[x = \sum_{j=1}^{m} \frac{j}{m}p_{j} .\] 
Then $x$ is selfadjoint and \norm{x} $=
1.$  By Kaplansky's density theorem there exists  
\[ (x_{_{n}}) \subset ( \cup B_{n})_{sa}\, ,\, \norm{x_{_{n}}} \leq 1 , x_{_{n}} 
\stackrel{n \rightarrow \infty}{\rightarrow} x \mbox{ strongly }\] 
i.e. dist$_{2}(( \cup B_{n})_{1},x) = 0.$  Since $B_{0} \subset{} B_{1} \subset \cdots{}\ x_{n}$ can be chosen in
$B_{n}.$ 

\noindent{}In particular  
$x_{_{2l}} \stackrel{l \rightarrow \infty}{\rightarrow}
x $\ strongly. Put $y_{_{2l}} = E_{A_{2l}' \cap{} B_{2l}}(x_{_{2l}}) =  
E_{A_{2l}' \cap{} B}(x_{_{2l}})$ then 
\[\tonorm{y_{_{2l}} - x} \leq \tonorm{y_{_{2l}} - E_{A_{2l}'
\cap{} B}(x) } + \tonorm{E_{A_{2l}' \cap{} B}(x) - x}.\] 
Since $A_{2l}' \cap{} B \stackrel{l \rightarrow \infty}{\searrow} A' 
\cap B$ we get
\[\tonorm{E_{A_{2l}' \cap{} B}(x) - x} \stackrel{l \rightarrow\infty}
{\rightarrow} \tonorm{E_{A' \cap{} B}(x) - x} = 0\] 
and
\[\tonorm{y_{_{2l}} - E_{A_{2l}' \cap{} B}(x) } = \tonorm{E_{A_{2l}' \cap{}
B}(x_{_{2l}} -x)} \leq \tonorm{x_{_{2l}} -x} \stackrel{l \rightarrow
\infty}{\rightarrow} 0.\] 
This shows 
\[ y_{_{2l}} \stackrel{l \rightarrow \infty}{\rightarrow} x \mbox{ strongly }\]

\noindent{}$y_{_{2l}}z^{2l}_{i_{0}}$ has at most $m_{0}$ 
spectral projections, since:

\noindent{}$y_{_{2l}} \in{} A_{2l}' \cap{} B_{2l}$ and 
$z^{2l}_{i_{0}} \in{} A_{2l}$ hence
$\left [ y_{_{2l}},z^{2l}_{i_{0}}\right ] = 0.$ In particular 
\[y_{_{2l}}z^{2l}_{i_{0}} \in{} z^{2l}_{i_{0}}B_{2l}z^{2l}_{i_{0}}
 \cong \bigoplus_{j}^{}M_{K_{i_{0}j}}({\Bbb C})\] 
which only contains $m_{0}$
minimal projections.

\noindent{}Since $z^{2l}_{i_{0}} \in{} A$ and $p_{j} \in{} A' \cap{} B$ we have
$tr_{B}(p_{j}z^{2l}_{i_{0}}) = $ $tr_{B}(p_{j}) tr_{B}(z^{2l}_{i_{0}})$  $\neq
0.$ I.e. $p_{j}z^{2l}_{i_{0}} = z^{2l}_{i_{0}}p_{j}$ is a non-zero projection,
and we get 
\[xz^{2l}_{i_{0}} = \sum_{j=1}^{m} \frac{j}{m}p_{j}z^{2l}_{i_{0}}\]
has exactly $m_{0}$ spectral projections.

\noindent{}Assume $m > m_{0}.$ 

\noindent{}In lemma \ref{lemma62} put $\alpha_{j} = \frac{j}{m} , 
t_{j} = \trB{(p_{j})}$
and let the $II_{1}$-factor be $ z^{2l}_{i_{0}}Bz^{2l}_{i_{0}}.$ Then there
exists  $\epsilon > 0$ s.t. for all $y \in{} z^{2l}_{i_{0}}Bz^{2l}_{i_{0}}$
selfadjoint with
 
\noindent{}\norm{y} $\leq$ \norm{xz^{2l}_{i_{0}}} and $y$ less than $m$ 
spectral projections 
\[ \| y - xz^{2l}_{i_{0}}\|^{^{2}}_{_{_{2,z^{2l}_{i_{0}}Bz^{2l}_{i_{0}}}}}
 \geq \epsilon.\] 
The trace on $z^{2l}_{i_{0}}Bz^{2l}_{i_{0}}$ is given by
\[tr_{z^{2l}_{i_{0}}Bz^{2l}_{i_{0}}}( \cdot) =\frac{tr_{B}(\cdot)}
{tr_{B}(z^{2l}_{i_{0}})}\] 
hence 
\[ \| y - xz^{2l}_{i_{0}}
\|^{^{2}}_{_{_{2,B}}} \geq  \epsilon{}\, tr_{B}(z^{2l}_{i_{0}}) .\]
\norm{xz^{2l}_{i_{0}}} $ = 1$ and \norm{y_{_{2l}}z^{2l}_{i_{0}}} $ \leq $
\norm{x_{_{2l}}} $\leq 1$ and we get 
\[ \| y_{_{2l}}z^{2l}_{i_{0}} - xz^{2l}_{i_{0}}\|^{^{2}}_{_{_{2,B}}} \geq  
\epsilon{}\,tr_{B}(z^{2l}_{i_{0}}) \geq \epsilon{}\, c_{i_{0}} > 0 \] 
where $c_{i_{0}}$ is the constant discussed in remark \ref{rem1}.

\noindent{}On the other hand we have 
\[ \| y_{_{2l}}z^{2l}_{i_{0}} - xz^{2l}_{i_{0}}
\|_{_{_{2,B}}} \leq \ \| z^{2l}_{i_{0}} \|_{_{_{2,B}}} \ \| y_{_{2l}} - x
\|_{_{_{2,B}}} \stackrel{l \rightarrow \infty}{\rightarrow} 0.\] 
This is a contradiction. I.e. $m \leq{} m_{0}.$ 
Since 
\[\sum_{j}K_{_{i_{_{0}},j}}^{2} \leq \left(\sum_{j}K_{_{i_{_{0}},j}}\right)^{2}\]
we have 
\[ {\rm dim}\{A' \cap{} B\} \leq \left(\min\{\mbox{ 1-norm of rows of } K \}\right)^{^{2}}\] 

\noindent{}The same argument holds for odd $n$'s involving the matrix $H$ 
instead. And we get 
\[ {\rm dim}\{A' \cap{} B\} \leq \left(\min\{\mbox{ 1-norm of rows of } K 
\mbox{ and } H\}\right)^{^{2}},\] 
because 
\[{\rm dim}\{A' \cap{} B\}\leq{}({\rm dim}\{\mbox{ maximal Abelian subalgebra of }A' \cap{} B\})^{2}=m^{2}.\]
\end{proof}
\begin{cor}
\label{cor65}
If there exists a symmetric commuting square 
\mbox{\scriptsize{$\begin{array}{lcl}
 B_{0} & \subset_{L} & B_{1} \\
  \cup_{K} &\,&\cup_{H}\\
 A_{0} & \subset_{G} & A_{1} \end{array} $}} of infinite dimensional  multi-matrix algebras, then there exists a pair of hyperfinite $II_{1}-$factors
 $A\subset{}B$ with 
\[ {\rm dim}\{A' \cap{} B\} \leq \left(\min\{\mbox{ 1-norm of rows and columns of } K 
\mbox{ and } H\}\right)^{^{2}}\] 
\end{cor}
\begin{proof}
If the minimum is attained for a row of either $H$ or $K,$ the result follows from  Theorem \ref{thm64}. If the minimum is attained for a column of either $H$ or $K,$ we argue as follows.

Let $e_{_{A_{_{1}}}}$ denote the projection in the basic construction for $A_{_{1}}\subset{}B_{_{1}},$   $C_{_{1}}=\langle{}B_{_{1}},e_{_{A_{_{1}}}}\rangle$ and $C_{_{0}}=\langle{}B_{_{0}},e_{_{A_{_{1}}}}\rangle.$ Then
\[\begin{array}{lcl}
 C_{0} & \subset_{G} & C_{1} \\
  \cup_{K^{^{t}}} &\,&\cup_{H^{^{t}}}\\
 B_{0} & \subset_{L} & B_{1} \end{array} \]
is a symmetric commuting square of infinite dimensional multi-matrix algebras. By theorem \ref{thm64} we then get a pair of hyperfinite $II_{1}-$factors $A\subset{}B$ with
\[ {\rm dim}\{A' \cap{} B\} \leq \left(\min\{\mbox{ 1-norm of rows  of } K^{^{t}} \mbox{ and } H^{^{t}}\}\right)^{^{2}}.\] 
This proves the assertion.
\end{proof}
\section{Constructing Commuting Squares of Infinite Multi--matrix Algebras}
The proof of the bi--unitary condition (\ref{biunitcond}) for commuting squares of finite multi-matrix algebras, only uses local properties of the involved Bratteli diagrams. All the arguments on pp \pageref{fire1page1}--\pageref{fire1page2} proving theorem \ref{fire1thm17}, may be repeated for inclusions of infinite  dimensional multi-matrix algebras, provided that the inclusion matrices correspond to locally finite, countably infinite graphs. In particular the blocks, $u^{(i,k)}$ and $v^{(j,l)}$ of $u$ and $v$ are finite dimensional unitaries. Furthermore we may conclude 
\begin{theorem}
\label{uendeligversionfire1thm10}
Let $G$ $H$ $K$ and $L$ be  adjacency matrices for locally finite, countably infinite, bi-partite graphs, such that
\[GH=KL\;\;\mbox{ and }\;\;G^{t}K=HL^{t}.\]
Then the following conditions are equivalent
\begin{description}
\item[(a)]{}There exists a (symmetric) commuting square 
\[(A\subset{}B\subset{}D,\;\;A\subset{}C\subset{}D,\;\;\mbox{tr}_{_{D}})\]
of infinite  dimensional multi-matrix algebras, with inclusion matrices
\[\begin{array}{lcl}
      C & \subset_{L} & D \\
       \cup_{K} &\,&\cup_{H}\\
      A & \subset_{G} & B.  \end{array}  \]
\item[(b)]{}There exists a pair of matrices $(u,v)$ satisfying the bi--unitary condition, i.e.
\[u=\bigoplus_{(i,k)}u^{(i,k)},\;\;\;\;\;v=\bigoplus_{(j,l)}v^{(j,l)}\]
where the direct summands
\[u^{(i,k)}=\left(
u_{(j,\rho,\sigma)(l,\phi,\psi)}^{(i,k)}\right)_{\stackrel{\mbox{\tiny$(i,j,k,\rho,\sigma)\in{\cal S}$}}{\mbox{\tiny$(i,l,k,\phi,\psi)\in{\cal T}$}}},\]
\[v^{(j,l)}=\left(v_{_{(i,\rho,\phi)(k,\sigma,\psi)}}^{(j,l)}\right)_{\stackrel{\mbox{\tiny$(i,j,k,\rho,\sigma)\in{\cal S}$}}{\mbox{\tiny$(i,l,k,\phi,\psi)\in{\cal T}$}}}\]
are unitary matrices and
\[v_{_{(i,\rho,\phi)(k,\sigma,\psi)}}^{(j,l)}=
\sqrt{\mbox{$\frac{\alpha_{_{i}}\delta_{_{k}}}{\beta_{_{j}}\gamma_{_{l}}}$}}u_{_{(j,\rho,\sigma)(l,\phi,\psi)}}^{(i,k)}\]
Here $\alpha_{_{i}},$ $\beta_{_{j}},$ $\gamma_{_{l}}$ and $\delta_{_{k}}$ are the trace weights on $A,$ $B,$ $C$ resp. $D$ coming from $\mbox{tr}_{_{D}},$ and the indices $i,j,k,l,\rho,\sigma,\phi$ and $\psi$ are as in theorem \ref{fire1thm17}.\end{description}
\end{theorem}
We can also conclude the statements of  proposition \ref{fire1prop17} for infinite  dimensional multi-matrix algebras, since also this proof is  only concerned with local properties of the involved Bratteli diagrams.
\begin{prop}
\label{uendeligfire1prop17}
If
\[(A\subset_{G}B\subset_{H}D,\;\;A\subset_{K}C\subset_{L}D,\;\;\mbox{tr}_{_{D}})\]
is a symmetric commuting square of infinite  dimensional multi-matrix algebras, such that the Bratteli diagrams $\Gamma_{_{G}},$  $\Gamma_{_{H}},$  $\Gamma_{_{K}}$ and $\Gamma_{_{L}}$ are connected, locally finite, countably infinite graphs, then
\begin{description}
\item[(I)]{}$\|K\|=\|H\|.$ Moreover $\mbox{tr}_{_{D}}$ is the Markov trace of the embedding $C\subset{}D,$ and $\mbox{tr}_{_{D}}|_{_{B}}$ is  the Markov trace of the embedding $A\subset{}B.$
\item[(II)]{}$\|G\|=\|L\|.$ Moreover $\mbox{tr}_{_{D}}$ is the Markov trace of the embedding $B\subset{}D,$ and $\mbox{tr}_{_{D}}|_{_{C}}$ is  the Markov trace of the embedding $A\subset{}C.$
\end{description}
\end{prop}
Hence, if we, via the path model described in chapter \ref{chapter1}, construct  a square of infinite dimensional multi-matrix algebras
\[\begin{array}{lcl}
      C & \subset_{L} & D \\
       \cup_{K} &\,&\cup_{H}\\
      A & \subset_{G} & B  \end{array}  \]
with 
\[GH=KL\;\;\mbox{ and }\;\;G^{t}K=HL^{t},\] 
where the inclusions are defined by locally finite, countably infinite, connected Bratteli diagrams, and have a trace defined on $D$ by a summable vector, then, by proposition \ref{uendeligfire1prop17} and \cite{S} theorem 6.4,  this trace is the only trace for which there can exist a  symmetric commuting square of infinite  dimensional multi-matrix algebras, with the given  Bratteli diagrams. To show the existence of a symmetric commuting square, with the the above inclusions, we need to show the existence of $(u,v),$ from theorem \ref{uendeligversionfire1thm10}, satisfying the bi-unitary condition \ref{biunitcond}.

\newpage{}

\chapter{Remarks on Some Infinite Graphs}

\setcounter{equation}{0}
We will now be concerned with the properties of some well studied infinite graphs.

\section{The Graphs Determined by Shearer have Summable Perron-Frobenius Vectors}
\setcounter{equation}{0}
\label{l1ofshearer}
We will show that the eigenvector, corresponding to the largest eigenvalue of the infinite graphs constructed in \cite{shearer} is summable.

If $\Gamma$ is a graph we let $\lambda(\Gamma)$ denote the largest eigenvalue of the adjacency matrix, $\Delta_{_{\Gamma}},$ of $\Gamma.$ 
In \cite{shearer} it is proved that for any real number 
$\lambda>\sqrt{2+\sqrt{5}}$ there exists a sequence of graphs 
$\Gamma_{_{1}}^{\lambda},\Gamma_{_{2}}^{\lambda},\ldots$ such that
\[\lim_{n\rightarrow\infty}\lambda(\Gamma_{_{n}}^{\lambda})=\lambda.\]

Another way to view this construction is, that for $\lambda>\sqrt{2+\sqrt{5}}$ there is a countably infinite graph, $\Gamma_{_{\lambda}},$ with 
$\lambda(\Gamma_{_{\lambda}})=\lambda.$ The description of $\Gamma_{_{\lambda}}$ is as follows.

For $\lambda>\sqrt{2+\sqrt{5}}$ we choose $x>0$ such that $\lambda=e^{x}+e^{-x},$ and consider the graph consisting of the vertices $P_{_{0}},P_{_{1}},\ldots$ and edges $e_{_{0}},e_{_{1}},\ldots$ where $e_{_{i}}$ connects $P_{_{i}}$ and $P_{_{i+1}},$ i.e. the graph
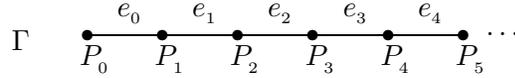
\setlength{\unitlength}{1cm}
\begin{center}
\begin{picture}(9,1.5)
\put(2,0,3){$\Gamma$}
\put(3,0.5){\line(1,0){5}}
\put(8.3,0.5){$\ldots$}
\multiput(3,0.5)(1,0){6}{\circle*{0.15}}
\put(2.9,0.1){$P_{_{0}}$}
\put(3.9,0.1){$P_{_{1}}$}
\put(4.9,0.1){$P_{_{2}}$}
\put(5.9,0.1){$P_{_{3}}$}
\put(6.9,0.1){$P_{_{4}}$}
\put(7.9,0.1){$P_{_{5}}$}
\put(3.4,0.75){$e_{_{0}}$}
\put(4.4,0.75){$e_{_{1}}$}
\put(5.4,0.75){$e_{_{2}}$}
\put(6.4,0.75){$e_{_{3}}$}
\put(7.4,0.75){$e_{_{4}}$}
\end{picture}
\end{center}

The numbers $n_{_{k}},r_{_{k}},a_{_{k}},k\in{\Bbb N}\cup\{0\}$ are defined inductively by $n_{_{0}}=0,r_{_{1}}=\lambda,a_{_{0}} =1$
and
\[
\begin{array}{clcll}
\mbox{(a)} & n_{_{k}} & = & 
\max\{j\in{\Bbb Z}|\lambda-\frac{1}{r_{_{k}}}-\frac{j}{\lambda}\geq{}e^{-x}\}, & k\geq{}1\\
\mbox{(b)} & r_{_{k+1}} & = &\lambda-\frac{1}{r_{_{k}}}-\frac{n_{_{k}}}{\lambda}, & k\geq{}1\\
\mbox{(c)} & a_{_{k}} & = & r_{_{k}}a_{_{k-1}}, & k\geq{}1
\end{array}
\]
\begin{prop}
\label{prop1s}
With the above notation the graph $\Gamma_{_{\lambda}}$ given by the graph $\Gamma$ with $n_{_{k}}$ leaves, $Q_{_{k,1}},\ldots,Q_{_{k,n_{_{k}}}},$ added at the vertex $P_{_{k}},$ has norm of its adjacency matrix equal to  $\lambda$. Furthermore the corresponding Perron-Frobenius vector $\xi$  is given by the coordinates 
\[\begin{array}{lcll}
\xi(P_{_{k}}) & = & a_{_{k}} & \\
\xi(Q_{_{k,j}}) & =  & \frac{a_{_{k}}}{\lambda} & \mbox{if } n_{_{k}}\neq{}0
\end{array} \] 
\end{prop}
\begin{proof}
From \cite{shearer} it follows that the vector $\xi,$ listed above, is an eigenvector for the adjacency matrix $\Delta_{_{\Gamma_{_{\lambda}}}}$ of $\Gamma_{_{\lambda}}$ corresponding to the eigenvalue $\lambda.$ That the vector is also  a Perron-Frobenius vector follows since 1) $\Delta_{_{\Gamma_{_{\lambda}}}}$ is symmetric and irreducible, 2) $\xi(v)>0$ for all vertices $v,$ 3) the result Thm 6.4 from \cite{S} and 4) our proof that the above vector is summable.
\end{proof} 
To sum up: For $\lambda>\sqrt{2+\sqrt{5}}=e^{x}+e^{-x}$ we are looking at the graph 

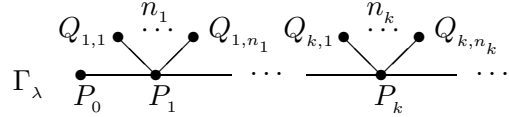
\setlength{\unitlength}{1cm}
\begin{center}
\begin{picture}(6,2)
\put(0.1,0,3){$\Gamma_{_{\lambda}}$}
\put(1,0.5){\line(1,0){2}}
\put(4,0.5){\line(1,0){2}}
\put(3.25,0.5){$\ldots$}
\put(6.25,0.5){$\ldots$}
\multiput(1,0.5)(1,0){2}{\circle*{0.15}}
\put(5,0.5){\circle*{0.15}}
\put(0.9,0.1){$P_{_{0}}$}
\put(1.9,0.1){$P_{_{1}}$}
\put(4.9,0.1){$P_{_{k}}$}
\put(2,0.5){\line(-1,1){0.5}}
\put(2,0.5){\line(1,1){0.5}}
\put(1.8,1){$\cdots$}
\put(1.8,1.3){$n_{_{1}}$}
\put(0.7,1){$Q_{_{1,1}}$}
\put(1.5,1){\circle*{0.15}}
\put(2.5,1){\circle*{0.15}}
\put(2.7,1){$Q_{_{1,n_{_{1}}}}$}
\put(5,0.5){\line(-1,1){0.5}}
\put(5,0.5){\line(1,1){0.5}}
\put(4.8,1){$\cdots$}
\put(4.8,1.3){$n_{_{k}}$}
\put(3.7,1){$Q_{_{k,1}}$}
\put(4.5,1){\circle*{0.15}}
\put(5.5,1){\circle*{0.15}}
\put(5.7,1){$Q_{_{k,n_{_{k}}}}$}
\end{picture}
\end{center}

\begin{lemma}
\label{lemma0s}
In the above notation $n_{_{k}}\leq{}[\lambda^{^{2}}]-2 = n_{_{\lambda}},$ where  $[\lambda^{^{2}}]$ denotes the integer part of $\lambda^{^{2}}.$
\end{lemma}
\begin{proof}
Let star($n$) denote the graph with vertices $0,1,\ldots,n$ and edges $e_{_{1}},\ldots,e_{_{n}},$ where $e_{_{i}}$ joins the vertices $0$ and $i$. Then
\[\lambda(\mbox{star}(n)) = \sqrt{n}.\]
Hence $\Gamma_{_{\lambda}}$ can only contain star($n$) as a subgraph if $[\lambda^{^{2}}]\geq{}n,$ and the maximal number of leaves that can be added to a $P_{_{k}}$ is $[\lambda^{^{2}}]-2.$
\end{proof}
\begin{lemma}\hfill{}
\label{lemma1s}
\begin{enumerate}
\item{}For $k\geq{}2$ we have 
$e^{-x}\leq{}r_{_{k}}\leq{}\frac{1}{\lambda} + e^{-x}.$
\item{}For $\lambda>\sqrt{2+\sqrt{5}}$ we have $\frac{1}{\lambda} + e^{-x}<e^{x},$ and $\lambda=\sqrt{2+\sqrt{5}}$ implies $\frac{1}{\lambda} + e^{-x}=e^{x}.$
\end{enumerate}
\end{lemma}
\begin{proof}

\noindent{}1. From the property (a) we have
\[\lambda-\frac{1}{r_{_{k}}}-\frac{n_{_{k}}}{\lambda}\geq{}e^{-x}
\mbox{ and }
\lambda-\frac{1}{r_{_{k}}}-\frac{n_{_{k}}+1}{\lambda}<e^{-x}.\]
So by property (b) we get
\[e^{-x}\leq{}r_{_{k+1}}\leq{}\frac{1}{\lambda} + e^{-x}.\]

\noindent{}2. The solution to $\frac{1}{\lambda}+e^{-x}=e^{x}$ is $e^{x}=\sqrt{\frac{1}{2}(1+\sqrt{5})},$ corresponding to
\[\lambda^{^{2}}=e^{2x}+2+e^{-2x}=2+\sqrt{5}.\]
This now implies the assertion.
\end{proof}
\begin{lemma}
\label{lemma3s}
On the interval $[e^{-x},e^{x}]$ the iteration $t_{_{k+1}}=\lambda-\frac{1}{t_{_{k}}}$ is non-decreasing with fix points  $e^{-x}$ and $e^{x}.$
\end{lemma}
\begin{proof}
If $\tilde{t}$ is a fix point for $t_{_{k}}\mapsto\lambda-\frac{1}{t_{_{k-1}}}$ we have 
\[\tilde{t}=\frac{\lambda\tilde{t}-1}{\tilde{t}}\Leftrightarrow
\tilde{t}=\frac{\lambda\pm\sqrt{\lambda^{2}-4}}{2}.\]
Since $\lambda^{^{2}}-4= (e^{x}-e^{-x})^{^{2}}$ we get that 
$\tilde{t}=e^{x}$ or $\tilde{t}=e^{-x}.$

Consider $x_{_{k+1}}-x_{_{k}} = \lambda-\frac{1}{x_{_{k}}}-x_{_{k}},$ then 
$x_{_{k+1}}-x_{_{k}}\geq{}0$ is equivalent to
\[x_{_{k}}^{2} -\lambda{}x_{_{k}} + 1 \leq{}0\Leftrightarrow\frac{1}{2}(\lambda-\sqrt{\lambda^{^{2}}-4})\leq{}x_{_{k}}\leq{}\frac{1}{2}(\lambda+\sqrt{\lambda^{^{2}}-4}),\]
and hence the iteration is non-decreasing for $e^{-x}\leq{}x_{_{k}}\leq{}e^{x}.$
\end{proof}
\begin{lemma}
\label{lemma35s}
Assume $\lambda>\sqrt{2+\sqrt{5}}.$ If $n_{_{k}}= 0$ for $k\geq{}k_{_{0}},$ then $r_{_{k}}=e^{-x}$ for $k\geq{}k_{_{0}}+1.$
\end{lemma}
\begin{proof}
By (b) we have that $r_{_{k+1}}=\lambda-\frac{1}{r_{_{k}}},\;k\geq{}k_{_{0}}.$ 
If $r_{_{k_{_{0}}+1}}>e^{-x}$ lemma \ref{lemma3s} gives
\[r_{_{k_{_{0}}+1}}\leq{}r_{_{k_{_{0}}+2}}\leq{}\cdots,\]
so $r_{_{k}}\nearrow{}r,$ with $r$ a fix point for the iteration in lemma \ref{lemma3s}. The fix points are $e^{-x}$ and $e^{x},$ and since $r>e^{-x}$  we must have $r=e^{x}.$

\noindent{}Using lemma \ref{lemma1s}  we get
\[r\leq{}\sup_{k\geq{}k_{_{0}}}r_{_{k}}\leq\lambda-e^{-x}<e^{x}.\]
This is a contradiction, so $r_{_{k_{_{0}}+1}}=e^{-x},$ and we must have 
\[r_{_{k}}=e^{-x} \mbox{ for } k\geq{}k_{_{0}}+1.\]
\end{proof}
\begin{lemma}
\label{lemma2s}
Assume $\lambda>\sqrt{2+\sqrt{5}}$ then, if $2\leq{}k<l$ are integers, such that $n_{_{k}}\neq{}0$ and $n_{_{l}}\neq{}0$ but $n_{_{j}}=0,k<j<l,$ then there is a constant $\epsilon(\lambda)$ (independent of $k,l$) such that
\[r_{_{k+1}}r_{_{k+2}}\cdots{}r_{_{l}}\leq(1-\epsilon(\lambda))^{^{\frac{l-k}{2}}}.\]
\end{lemma}
\begin{proof}
In the vertex $P_{_{k}},$   $\xi$ satisfies the equation
\[\begin{array}{clcl}
\ & a_{_{k-1}} + a_{_{k+1}} + \frac{n_{_{k}}}{\lambda}a_{_{k}} & = & \lambda{}a_{_{k}}\\
\Updownarrow &  & & \\
 & \frac{a_{_{k-1}}}{a_{_{k}}} +\frac{a_{_{k+1}}}{a_{_{k}}}+ \frac{n_{_{k}}}{\lambda} & = & \lambda\\
\Updownarrow & & & \\
 & \frac{1}{r_{_{k}}} +r_{_{k+1}}+ \frac{n_{_{k}}}{\lambda} & = & \lambda.
\end{array}\]
Using lemma \ref{lemma1s}, we then have
\[r_{_{k+1}}=\lambda-\frac{n_{_{k}}}{\lambda}-\frac{1}{r_{_{k}}}\leq{}\lambda-\frac{1}{\lambda}-\left(\frac{1}{\lambda}+e^{-x}\right)^{^{-1}}.\]
We also have $r_{_{l}}\leq{}\frac{1}{\lambda}+e^{-x},$ and so we get 
\[\begin{array}{lcl}
r_{_{k+1}}r_{_{l}} & \leq & \left(\lambda-\frac{1}{\lambda}-
\left(\frac{1}{\lambda}+e^{-x}\right)^{^{-1}}\right)
\left(\frac{1}{\lambda}+e^{-x}\right)\\[0.4cm]
 & = & (\lambda -\frac{1}{\lambda})(\frac{1}{\lambda}+e^{-x})-1\\[0.2cm] & = &
\frac{\lambda{}e^{-x}(\lambda^{^{2}}-1)-1}{\lambda^{^{2}}}.
\end{array}\]
The denominator can be rewritten as $e^{2x}+1-2e^{-2x}-e^{-4x},$ using 
$\lambda=e^{x}+e^{-x},$ and the numerator is $e^{2x}+2+e^{-2x}.$

We want to show that $r_{_{k+1}}r_{_{l}}<1,$ which is equivalent to 
\[e^{2x}+1-2e^{-2x}-e^{-4x}<e^{2x}+2+e^{-2x}\Leftrightarrow{}
\lambda>\sqrt{2+\sqrt{5}}.\]
Hence we can find $\epsilon(\lambda)=1-\frac{\lambda{}e^{-x}(\lambda^{^{2}}-1)-1}{\lambda^{^{2}}}>0$ such that $r_{_{k+1}}r_{_{l}}<1-\epsilon(\lambda).$

We will show that $r_{_{k+2}}r_{_{l-1}}\leq{}1-\epsilon(\lambda)$ for $l-k\geq{}3.$

Since $n_{_{k+1}},n_{_{k+2}}$ and $n_{_{l-1}}$ all equal $0,$ (b) yields that $r_{_{k+2}}=\lambda-\frac{1}{r_{_{k+1}}}$ and $r_{_{l-1}}=\frac{1}{\lambda-r_{_{l}}},$ so 
\[r_{_{k+2}}r_{_{l-1}}=\left(\lambda-\frac{1}{r_{_{k+1}}}\right)\left(\frac{1}{\lambda-r_{_{l}}}\right).\]
It suffices to show $r_{_{k+2}}r_{_{l-1}}\leq{}r_{_{k+1}}r_{_{l}},$ which is equivalent to showing
\[\frac{1}{r_{_{k+1}}}\left(\lambda-\frac{1}{r_{_{k+1}}}\right) \leq{}
r_{_{l}}(\lambda-r_{_{l}}).\]
For $j\in\{k+1,\ldots,l-1\}$ we have  $r_{_{j}}=\lambda-\frac{1}{r_{_{j-1}}},$ so by lemma \ref{lemma3s} 
$r_{_{k+1}},\ldots,r_{_{l}}$ is a non-decreasing sequence.

Since $\xi$ is eigenvector we have $\frac{1}{r_{_{k+1}}} + r_{_{k+2}}=\lambda,$ and hence
\[\frac{1}{r_{_{k+1}}}= \lambda-r_{_{k-2}}\geq{}\lambda-r_{_{l}}.\]
We also have $\frac{1}{r_{_{k+1}}}\geq{}r_{_{l}},$ since $r_{_{k+1}}r_{_{l}}<1,$ and we may conclude
\[\frac{1}{r_{_{k+1}}}\geq\max\{r_{_{l}},\lambda-r_{_{l}}\}.\]
By the properties of the function $t\mapsto{}t(\lambda-t), 0\leq{}t\leq\lambda$ we now conclude
\[\frac{1}{r_{_{k+1}}}\left(\lambda-\frac{1}{r_{_{k+1}}}\right) \leq{}
r_{_{l}}(\lambda-r_{_{l}})\]
and hence
\[r_{_{k+2}}r_{_{l-1}}\leq{}1-\epsilon(\lambda).\]
Continuing this argument we have that all the numbers $r_{_{k+1}}r_{_{l}},r_{_{k+2}}r_{_{l-1}},\ldots,r_{_{k+t}}r_{_{l-t+1}}$ are do\-minated by $1-\epsilon(\lambda)$ provided $k+t\leq{}l-t+1.$

For $l-k$ even we have $r_{_{k+1}}\cdots{}r_{_{l}}\leq{}(1-\epsilon(\lambda))^{^{\frac{l-k}{2}}}.$

For $l-k$ odd, $l-k=2n+1,$ we have
\[r_{_{k+1}}\cdots{}r_{_{k+n}}r_{_{k+n+2}}\cdots{}r_{_{l}}\leq{}(1-\epsilon(\lambda))^{^{\frac{l-k-1}{2}}}\]
and since the previous argument gives $r_{_{k+n+1}}=\sqrt{r_{_{k+n+1}}r_{_{k+n+1}}}\leq{}\sqrt{1-\epsilon(\lambda)}$ we have
\[r_{_{k+1}}\cdots{}r_{_{k+n}}r_{_{k+n+1}}r_{_{k+n+2}}\cdots{}r_{_{l}}\leq{}(1-\epsilon(\lambda))^{^{\frac{l-k}{2}}}.\]
\end{proof}
\begin{prop}
\label{prop2s}
For $\lambda>\sqrt{2+\sqrt{5}}$ the eigenvector $\xi$ of $\Gamma_{_{\lambda}}$ given by 
\[\begin{array}{lcll}
\xi(P_{_{k}}) & = & a_{_{k}} & \\
\xi(Q_{_{k,j}}) & =  & \frac{a_{_{k}}}{\lambda} & \mbox{if } n_{_{k}}\neq{}0
\end{array} \] 
is summable.
\end{prop}
\begin{proof}
If  $n_{_{k}}=0$ for $k\geq{}k_{_{0}},$ lemma \ref{lemma35s} implies that $r_{_{k}}=e^{-x}$ for $k\geq{}k_{_{0}}+1.$ By definition  $a_{_{k+1}} = r_{_{k+1}}a_{_{k}},$ so we have 
\[a_{_{k_{_{0}}+n}}=e^{-nx}a_{_{k_{_{0}}}},\]
and the eigenvector $\xi$ is seen to be summable.

If $n_{_{k}}$ does not eventually equal $0,$ we consider $k<l$  such that $n_{_{k}}\neq{}0$ and $n_{_{l}}\neq{}0$ but $n_{_{j}}=0$ for $k<j<l.$

By  lemma \ref{lemma2s} we have
\[a_{_{l}}\leq{}(1-\epsilon(\lambda))^{^{\frac{l-k}{2}}}a_{_{k}}.\]
If we consider $\log{}a_{_{k}},\log{}a_{_{k+1}},\ldots,\log{}a_{_{l}}$ this is a convex function of the index, since the differences 
$\log{}r_{_{k+1}},\log{}r_{_{k+2}},\ldots,\log{}r_{_{l}}$ satisfy
\[\log{}r_{_{k+1}}\leq{}\log{}r_{_{k+2}}\leq{}\ldots\leq{}\log{}r_{_{l}}\]
and it follows that
\[a_{_{n}}\leq{}(1-\epsilon(\lambda))^{^{\frac{n-k}{2}}}a_{_{k}},\,\,\,\,k\leq{}n\leq{}l.\]
If we let $k_{_{0}}\geq{}2$ denote the smallest integer such that $n_{_{k_{_{0}}}}\neq{}0$ and note that $\epsilon(\lambda)$ is independent of $k_{_{0}}$ we get
\[a_{_{n}}\leq{}(1-\epsilon(\lambda))^{^{\frac{n-k_{_{0}}}{2}}}a_{_{k_{_{0}}}},\,\,\,\,n\geq{}k_{_{0}}.\]
Using lemma \ref{lemma0s} we now get
\[\begin{array}{lcl}
\|\xi\|_{_{1}} & = & 
\sum_{j<k_{_{0}}}a_{_{j}} + \sum_{j\geq{}k_{_{0}}}(a_{_{j}}+n_{_{j}}\frac{a_{_{j}}}{\lambda})\\[0.2cm]
 & \leq & \sum_{j<k_{_{0}}}a_{_{j}} + \left(\frac{n_{_{\lambda}}}{\lambda} + 1\right)\sum_{j\geq{}k_{_{0}}}a_{_{j}}\\[0.3cm] 
& \leq & \sum_{j<k_{_{0}}}a_{_{j}} + \left(\frac{n_{_{\lambda}}}{\lambda} + 1\right)a_{_{k_{_{0}}}}\sum_{j\geq{}k_{_{0}}}(1-\epsilon(\lambda))^{^{\frac{j-n_{_{k_{_{0}}}}}{2}}}\\[0.3cm]
 & < & \infty.
\end{array}\]
\end{proof} 

\newpage{}

\setcounter{equation}{0}
\section{Not All Shearer-Graphs Can Define Commuting Squares}
\setcounter{equation}{0}
If we try to build a commuting square of infinite dimensional multi-matrix algebras, with one of the infinite graphs, $\Gamma_{_{\lambda}},$ defined by Shearer (see section \ref{l1ofshearer}) as the index defining side, there are not many obvious choices of the form of the inclusion matrices that define the other sides of the commuting square. The other inclusions have to have compatible Perron-Frobenius eigenvector, so a polynomial applied to the adjacency matrix of $\Gamma_{_{\lambda}}$ is a possibility (as described below), and it does not seem likely that one can define any other form of inclusions, that will work in general, to construct such commuting squares. We will show, that if $\Gamma_{_{\lambda}}$ is the index defining inclusion of a commuting square of infinite multi-matrix algebras of the above form, then $\Gamma_{_{\lambda}}$ has to be eventually periodic.

As a consequence Shearer's result, that the set of Perron-Frobenius eigenvalues of infinite graphs contains all of $ {\cal F}=\{x\in{\Bbb R}|x\geq\sqrt{2+\sqrt{5}}\},$ cannot be used to produce values of the index of irreducible Hyperfinite $II_{_{1}}-$factors which form a closed subset of ${\cal F}.$ 

The argument is as follows.

\noindent{}Let  $\lambda>\sqrt{2+\sqrt{5}}=e^{x}+e^{-x}$ and let $\Gamma_{_{\lambda}}$ be the graph discussed in (\ref{l1ofshearer}). Let $\Delta_{_{\lambda,bp}}$ be the adjacency matrix of a bi-partition of $\Gamma_{_{\lambda}}.$ I.e. 
$\Delta_{_{\lambda,bp}}= \left(\mbox{\tiny{$\begin{array}{cc}0&G\\G^{t}&0\end{array}$}}\right).$ Let $p$ be a polynomial such that 
$p(\Delta_{_{\lambda,bp}})$ corresponds to the adjacency matrix of a bi-partite graph. If the exponents of $p$ are all even, $p(\Delta_{_{\lambda,bp}})$ is of the form $\left(\mbox{\tiny{$\begin{array}{cc}H&0\\0&K\end{array}$}}\right).$ If the exponents are all odd, $p(\Delta_{_{\lambda,bp}})$ is of the form $\left(\mbox{\tiny{$\begin{array}{cc}0&H\\K&0\end{array}$}}\right).$ 

\noindent{}Let $n$ be the highest degree in $p(t)$ and $c_{_{n}}$ be the coefficient of $t^{^{n}}.$ Assume that there exists a commuting square of infinite dimensional multimatrix algebras of the form 
\[\begin{array}{lcl}
 C & \subset_{K} & D \\
  \cup_{G} &\,&\cup_{G}\\
 A & \subset_{H} & B \end{array} 
\mbox{ if $n$ is even, }\;\;\;\;\;\;
\begin{array}{lcl}
 C & \subset_{K} & D \\
  \cup_{G} &\,&\cup_{G^{^{t}}}\\
 A & \subset_{H} & B \end{array} 
\mbox{ if $n$ is odd.}
\]
If we look at the cycles involving the vertices $P_{_{k}},P_{_{k+1}},P_{_{n+k}}$ and $P_{_{n+k+1}}$

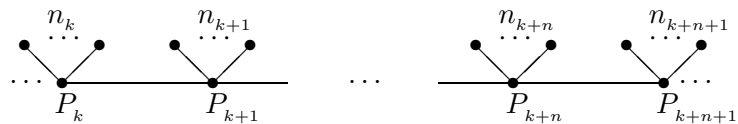
\begin{center}
\setlength{\unitlength}{1cm}
\begin{picture}(12,2)

\put(1.9,0.1){$P_{_{k}}$}
\put(1.8,1.3){$n_{_{k}}$}
\put(1.8,1){$\cdots$}
\put(1.5,1){\circle*{0.15}}
\put(2.5,1){\circle*{0.15}}
\put(2,0.5){\circle*{0.15}}
\put(2,0.5){\line(-1,1){0.5}}
\put(2,0.5){\line(1,1){0.5}}

\put(3.9,0.1){$P_{_{k+1}}$}
\put(3.8,1.3){$n_{_{k+1}}$}
\put(3.8,1){$\cdots$}
\put(3.5,1){\circle*{0.15}}
\put(4.5,1){\circle*{0.15}}
\put(4,0.5){\circle*{0.15}}
\put(4,0.5){\line(-1,1){0.5}}
\put(4,0.5){\line(1,1){0.5}}

\put(7.9,0.1){$P_{_{k+n}}$}
\put(7.8,1.3){$n_{_{k+n}}$}
\put(7.8,1){$\cdots$}
\put(7.5,1){\circle*{0.15}}
\put(8.5,1){\circle*{0.15}}
\put(8,0.5){\circle*{0.15}}
\put(8,0.5){\line(-1,1){0.5}}
\put(8,0.5){\line(1,1){0.5}}

\put(9.9,0.1){$P_{_{k+n+1}}$}
\put(9.8,1.3){$n_{_{k+n+1}}$}
\put(9.8,1){$\cdots$}
\put(9.5,1){\circle*{0.15}}
\put(10.5,1){\circle*{0.15}}
\put(10,0.5){\circle*{0.15}}
\put(10,0.5){\line(-1,1){0.5}}
\put(10,0.5){\line(1,1){0.5}}
\put(2,0.5){\line(1,0){3}}
\put(7,0.5){\line(1,0){3}}
\put(5.8,0.5){$\ldots$}
\put(10.2,0.5){$\ldots$}
\put(1.3,0.5){$\ldots$}
\end{picture}
\end{center}
Then either\\
\raisebox{0.4cm}{there is only one $c_{_{n}}\times{}c_{_{n}}$ unitary block, $A,$ in $u,$  corresponding to cycles of the form }
\setlength{\unitlength}{0.25mm}
\begin{picture}(50,45)
\put(20,10){\line(1,0){20}} 
\put(20,10){\line(0,1){20}}
\put(40,10){\line(0,1){20}} 
\put(20,30){\line(1,0){20}}
\put(10,3){\tiny\mbox{$P_{_{k}}$}}
\put(40,3){\tiny\mbox{$ P_{_{k+n}}$}}
\put(40,38){\tiny\mbox{$ P_{_{k+n+1}}$}}
\put(5,38){\tiny\mbox{$ P_{_{k+1}}$}}
\end{picture}\\[0.2cm]
\raisebox{0.4cm}{or there is only one $c_{_{n}}\times{}c_{_{n}}$ unitary block, $B,$ in $v,$  corresponding to cycles of the form }
\setlength{\unitlength}{0.25mm}
\begin{picture}(50,45)
\put(20,10){\line(1,0){20}} 
\put(20,10){\line(0,1){20}}
\put(40,10){\line(0,1){20}} 
\put(20,30){\line(1,0){20}}
\put(10,3){\tiny\mbox{$P_{_{k+1}}$}}
\put(40,3){\tiny\mbox{$ P_{_{k+n+1}}$}}
\put(40,38){\tiny\mbox{$ P_{_{k+n}}$}}
\put(7,38){\tiny\mbox{$ P_{_{k}}$}}
\end{picture}\\
In the first case, the scalar involved in the transition from $u$ to $v$ is $\sqrt{\frac{\xi(P_{_{k}})\xi(P_{_{k+n+1}})}{\xi(P_{_{k+1}})\xi(P_{_{k+n}})}},$ and in\\[0.2cm] the second case, the scalar involved in the transition from $v$ to $u$ is $\sqrt{\frac{\xi(P_{_{k}})\xi(P_{_{k+n+1}})}{\xi(P_{_{k+1}})\xi(P_{_{k+n}})}}.$ So in either case we must have 
\begin{equation}
\label{rcond}
\nu=\sqrt{\frac{\xi(P_{_{k}})\xi(P_{_{k+n+1}})}{\xi(P_{_{k+1}})\xi(P_{_{k+n}})}}\leq{}1\end{equation}
since either $\nu{}A$ must be part of a unitary in $v,$ or $\nu{}B$ must be part of a unitary in $u.$
In other terms (\ref{rcond}) can be stated as
\[r_{_{k+n}}\leq{}r_{_{k}},\]
and consequently we must have
\begin{equation}
\label{rrr}
r_{_{k}}\geq{}r_{_{k+n}}\geq{}r_{_{k+2n}}\geq\cdots\geq{}e^{-x},
\end{equation}
where the last inequality comes from lemma \ref{lemma1s}. Hence we have 
\[\lim_{j\rightarrow\infty}r_{_{k+nj}} = r \geq{}e^{-x}.\]
By definition $n_{_{k}}$ is the largest integer such that
\[\frac{1}{r_{_{k}}}+\frac{n_{_{k}}}{\lambda}+e^{-x}\leq{}\lambda,\]
so (\ref{rrr}) implies
\[n_{_{k}}\geq{}n_{_{k+n}}\geq{}n_{_{k+2n}}\geq\cdots\geq{}0,\]
and since they are all integers, there exists $j_{_{k}}$ such that 
$n_{_{k+jn}}$ is constant for $j\geq{}j_{_{k}}.$

Using this argument for $k=1,2,\ldots,n$ we find $J$ such that $n_{_{i}}=n_{_{i+n}},$ for all $i\geq{}J,$ i.e. $\Gamma_{_{\lambda}}$ is eventually periodic.

\newpage{}

\chapter{Some Index Values Which Do Not Occur From Finite Graphs}

\setcounter{equation}{0}

In this chapter we will show that the largest eigenvalue, $\lambda_{_{n}},$ of the graphs $T(1,n,\infty)$ cannot occur as eigenvalues of finite graphs. Furthermore we will construct commuting squares which will give $\lambda_{_{2}}^{2}, \lambda_{_{3}}^{2}$ and $\lambda_{_{4}}^{2}$ as index for a pair of irreducible Hyperfinite $II_{_{1}}-$factors.

\section{The Largest Eigenvalue of $T(1,n,\infty)$ Does Not Occur as Eigenvalue of a Finite Graph}
We will look at the graph $T(1,n,\infty)$ as defined by Hoffmann in \cite{Hof}.
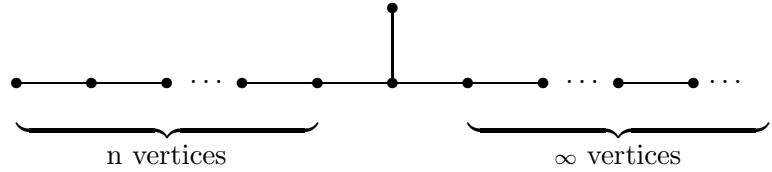
\begin{center}
\setlength{\unitlength}{1cm}
\begin{picture}(11,3)
\put(1,1){\line(1,0){2}}
\put(4,1){\line(1,0){4}}
\put(9,1){\line(1,0){1}}
\put(6,1){\line(0,1){1}}
\put(3.3,1){$\ldots$}
\put(8.3,1){$\ldots$}
\put(10.2,1){$\ldots$}
\multiput(1,1)(1,0){3}{\circle*{0.15}}
\multiput(4,1)(1,0){5}{\circle*{0.15}}
\multiput(9,1)(1,0){2}{\circle*{0.15}}
\put(6,2){\circle*{0.15}}
\put(1,0.5){$\underbrace{\mbox{\hspace{4cm}}}_{\mbox{n vertices}}$}
\put(7,0.5){$\underbrace{\mbox{\hspace{4cm}}}_{\infty{}\mbox{ vertices}}$}\end{picture}
\end{center}
\setcounter{equation}{0}
\noindent{}If we let \lam{} denote the largest eigenvalue of this graph, then $\lam{} =e^{x}+e^{-x}$ for some $x>0.$ If we put $\ro{}=e^{2x},$ $\rho$ satisfies the equation (see \cite{Hof})
\begin{equation}
\label{hofequation}
\ro{n+2}-\ro{n+1}-\ro{n} +1 =0,
\end{equation}
or, if we divide by $\ro{}-1$
\begin{equation}
\label{uffehofequation}
\ro{n+1}-\ro{n-1}-\ro{n-2}-\cdots{}-\ro{}-1.
\end{equation}
In this section we will show that $\lambda(T(1,n,\infty))$ is an algebraic integer, and that $\lambda(T(1,n,\infty))$  does not occur as eigenvalue for any finite graph.

\begin{remark}{\rm 
\label{hofrem1}
If we for $\lambda=e^{x}+e^{-x}$ define
\[P_{n}(\lambda)=\efrac{(n+1)}\]
then
\[P_{_{0}}(\lambda)=1\]
\[P_{_{1}}(\lambda)=\efrac{2}=\lambda.\]
Also 
\[(e^{x}+e^{-x})\efrac{(n+1)}-\efrac{n}=\efrac{(n+2)},\]
so $P_{n}(\lambda)=R_{n}(\lambda).$ If we expand $P_{n}(\lambda)$ as a finite sum of quotients, we have
\[R_{n}(\lambda)= e^{nx}+e^{(n-2)x}+e^{(n-4)x}+\cdots{}+e^{-nx}.\]

}\end{remark}
\begin{prop}
\label{hofpro1}
The largest eigenvalue, $\lam{n}=\lambda(T(1,n,\infty)),$ for $T(1,n,\infty)$ is a root of the $(2n+2)-$degree polynomial 
\[K_{n}(\lambda)= (R_{n+3}(\lambda) -R_{n+1}(\lambda)-R_{n-1}(\lambda))R_{n-1}(\lambda)-1.\]
\end{prop}
\begin{proof}
Let $\lam{0} = e^{x_{_{0}}}+e^{-x_{_{0}}}$ for some $x_{_{0}}>0$ and let $\rho_{_{0}}=e^{2x_{_{0}}}.$ Then
\[\phi_{_{n}}(\rho_{_{0}})=\rho^{n+1}_{_{0}}-\rho^{n-1}_{_{0}}-\rho^{n-2}_{_{0}}-\cdots{}-\rho^{}_{_{0}}-1=0.\]
Hence also 
\[-\phi_{_{n}}(\rho_{_{0}})\phi_{_{n}}(\rho^{-1}_{_{0}})=0.\]
Dividing the first factor by $e^{(n-1)x_{_{0}}}$ and multiplying the second by $e^{(n-1)x_{_{0}}}$ we get
\[\begin{array}{lcl}
0 &= & -\left(e^{(n+3)x_{_{0}}}-e^{(n-1)x_{_{0}}}-e^{(n-3)x_{_{0}}}-\cdots{}-e^{-(n-1)x_{_{0}}}\right)\\[0.2cm]
 & &\;\;\cdot{}\left(e^{-(n+3)x_{_{0}}}-e^{-(n-1)x_{_{0}}}-e^{-(n-3)x_{_{0}}}-\cdots{}-e^{(n-1)x_{_{0}}}\right)\\[0.2cm]
& = & -\left(e^{(n+3)x_{_{0}}}-R_{n-1}(\lam{0})\right)
       \left(e^{-(n+3)x_{_{0}}}-R_{n-1}(\lam{0})\right)\\[0.2cm]
& = & \left(e^{(n+3)x_{_{0}}}+e^{-(n+3)x_{_{0}}}\right)R_{n-1}(\lam{0})
    -R_{n-1}(\lam{0})^{2}-1\\[0.2cm]
&=&\left(R_{n+3}(\lam{0})-R_{n+1}(\lam{0})-R_{n-1}(\lam{0})\right)R_{n-1}(    \lam{0})-1\\[0.2cm]
&=&K_{n}(\lam{0}),
\end{array}\]
where the second equality follows by remark \ref{hofrem1}.
\end{proof}
\begin{prop}
\label{hofpro2}Consider the polynomial
\[K_{n}(\lam{})=\left(R_{n+3}(\lam{})-R_{n+1}(\lam{})-R_{n-1}(\lam{})\right)R_{n-1}(\lam{})-1,\;\;\; \lam{}\in{\Bbb C}.\]
Then
\begin{enumerate}
\item{}For $n$ even the only real roots of $K_{n}(\lam{})$ are $\pm{}\lambda(T(1,n,\infty)).$
\item{}For $n$ odd the only real roots of $K_{n}(\lam{})$ are $\pm{}\lambda(T(1,n,\infty))$ and 0.
\end{enumerate}
\end{prop}
\begin{proof}
Since $K_{n}$ is an even polynomial, it is enough to consider $\lam{}\geq{}0.$

\noindent{\bf{}I}. If $\lam{}>2,$ we may write $\lam{}=e^{x}+e^{-x},\;x\geq{}0$ and put $\rho=e^{2x}>1.$ Then 
\[K_{n}(\lam{})=-\phi_{_{n}}(\rho)\phi_{_{n}}(\rho^{-1})\]
where
\[\phi_{_{n}}(\rho)=\rho^{n+1}-\rho^{n-1}-\rho^{n-2}-\cdots{}-\rho^{}-1.\]Since  a) $\rho\mapsto\frac{\phi_{_{n}}(\rho)}{\rho^{n+1}}$ is strictly increasing on ${\Bbb R}_{+}$  and b) $\phi_{_{n}}(0)=-1, $ $\lim_{\rho\rightarrow\infty}\phi_{_{n}}(\rho)=\infty,$ the equation $\phi_{_{n}}(\rho)=0$ has precisely one solution, $\rho_{_{0}},$ in ${\Bbb R}_{+}.$ Moreover $\phi_{_{n}}(1)=1-n<0,$ so the solution $\rho_{_{0}}$ is in the interval $1<\rho_{_{0}}<\infty.$
 Hence the equation $K_{n}(\lam{})=0$ has exactly one solution, \lam{0}, with $2<\lam{0}<\infty.$ This value must then equal $\lam{}(T(1,n,\infty)).$

\noindent{\bf{}II}. If $0\leq{}\lam{}\leq{}2,$ we can write $\lam{}=2\cos\theta,\;0\leq\theta\leq\frac{\pi}{2}.$ Put $\rho=e^{i\theta}.$ Then
\[K_{n}(\lam{})=\left|\rho^{n+1}-\rho^{n-1}-\rho^{n-2}-\cdots{}-\rho^{}-1
\right|^{^{2}},\]
so $K_{n}(\lam{})=0$ implies
\begin{equation}
\label{zeroeqn}
\rho^{n+1}-\rho^{n-1}-\rho^{n-2}-\cdots{}-\rho^{}-1=0.
\end{equation}
This is in fact (\ref{uffehofequation}), which implies  (\ref{hofequation})
\[\ro{n+2}-\ro{n+1}-\ro{n} +1 =0,\]
or
\[\ro{2}-\ro{}-1 =-\ro{-n}.\]
Hence $|\ro{2}-\ro{}-1|=1.$ Since $\rho=e^{i\theta}$ we have
\[|\ro{2}-\ro{}-1| = 3-2\mbox{Re}(\rho) -2\mbox{Re}(\rho^{2})+2\mbox{Re}(\rho)=3-2\cos{}2\theta = 1 + 4\sin^{2}\theta,\]
i.e. $|\ro{2}-\ro{}-1|>1$ except for $\theta=p\pi,\;p\in{\Bbb Z},$ or equivalently: $|\ro{2}-\ro{}-1|>1$ except for $\rho=\pm{}1.$ The case $\rho=1$ is excluded by (\ref{zeroeqn}), and $\rho=-1$ is a solution to (\ref{zeroeqn}) if and only if $n$ is odd. Since $\rho=-1$ corresponds to $\lam{}=0,$ we have proved the assertions of the proposition.
\end{proof}
\begin{theorem}
\label{hofthm1}
The numbers $\lam{n}=\lambda(T(1,n,\infty)),\;\;n\geq{}2$ and $\lam{\infty}=\lambda(T(1,\infty,\infty))$ are algebraic integers, and none of these numbers can be obtained as an eigenvalue of a finite graph.
\end{theorem}
\begin{proof}
It is easy to check that $\lam{\infty}=\sqrt{2+\sqrt{5}}$ is a root in the polynomial $Q(\lam{})=\lambda^{4}-4\lambda^{2}-1.$ 

\noindent{}$Q$ is irreducible since: The roots of $Q$ are $\pm{}\sqrt{2+\sqrt{5}}$ and $\pm{}i\sqrt{\sqrt{5}-2}.$ None of these roots are integers, so any irreducible factor of $Q$ is of degree at least 2. Moreover the two complex conjugate roots must be roots of the same irreducible factor of $Q.$ Hence the only possible factorization of $Q$ into monic, irreducible polynomials, would be $Q=Q_{1}Q_{2}$ where
\[Q_{1}(\lam{})=(\lam{}-\sqrt{2+\sqrt{5}})(\lam{}+\sqrt{2+\sqrt{5}})=\lambda^{2}-2-\sqrt{5}\]
and
\[Q_{2}(\lam{})=(\lam{}-i\sqrt{\sqrt{5}-2})(\lam{}+i\sqrt{\sqrt{5}-2})=\lambda^{2}-2+\sqrt{5}.\]
However $Q_{1}$ and $Q_{2}$ do not have integer coefficients, and hence $Q$ is irreducible.

Assume that \lam{\infty} is an eigenvalue of the adjacency matrix $\Delta_{\Gamma}$ of the finite graph $\Gamma.$ The characteristic polynomial
\[f(\lam{})=\mbox{det}(\lam{}I-\Delta_{\Gamma})\] 
is  monic with integer coefficients. Moreover, since $\Delta_{\Gamma}$ is symmetric, all the roots of $f$ are real. Since $f(\lam{\infty})=0=Q(\lam{\infty}),$ the irreducibility of $Q$ implies that $Q$ divides $f$. This is impossible because $Q$ has non-real roots.

\noindent{}We will now turn to $\lam{n}=\lam{}(T(1,n,\infty)).$ By proposition \ref{hofpro1}, $\lam{n}$ is root in a monic polynomial, $K_{n},$ with integer coefficients. Let $Q_{n}$ be the minimal monic polynomial over ${\Bbb Q},$ which has $\lam{n}$ as a root. Since \lam{n} is an algebraic integer, $Q_{n}$ has integer coefficients (see \cite{ST} lemma 2.12). 

\noindent{}We claim that $Q_{n}$ must have non-real roots. Indeed, since $Q_{n}$ is a factor in $K_{n},$ the only possible real roots of $Q_{n}$ are $\pm\lam{n}$ and $0$ by proposition \ref{hofpro2}. However $0$ is not a root of $Q_{n},$ because $Q_{n}$ is irreducible. Hence, if $Q_{n}$ has only real roots, it must be of the form
\[Q_{n}(\lam{})=\lam{}-\lam{n}\]
or
\[Q_{n}(\lam{})=(\lam{}-\lam{n})(\lam{}+\lam{n})=\lambda^{2}-\lambda_{n}^{2}.\]
Here it is used that irreducible polynomials do not have multiple roots. (See \cite{ST} corollary 1.2)
However, $\lam{n}\not\in{\Bbb Z}$ and $\lambda_{n}^{2}\not\in{\Bbb Z},$ because 
$2<\lam{n}<\sqrt{2+\sqrt{5}}<\sqrt{5},$ which is a contradiction. Hence $Q_{n}$ has at least one non-real root, and, as in the case \lam{\infty}, it follows that \lam{n} is not an eigenvalue of a finite graph.
\end{proof}
\begin{remark}{\rm
The idea to the above proof is due to P. de la Harpe, \cite{Pierre}, who used the method to prove that $\lam{2},\lam{3}$ and $\lam{\infty}$ are not eigenvalues of any finite graph. 
}\end{remark}
\newpage{}

\setcounter{equation}{0}

\section{A Commuting Square Based on $T(1,2,\infty)$}
\setcounter{equation}{0}
We will look at the graph $T(1,2,\infty)$ as defined by Hoffmann in \cite{Hof}.
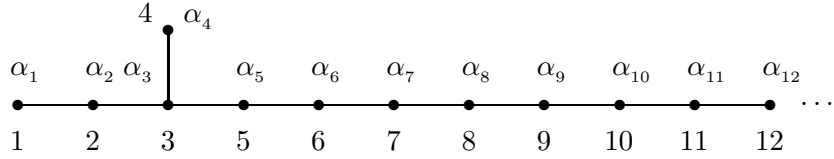
\begin{center}
\setlength{\unitlength}{1cm}
\begin{picture}(12,3)
\put(1,1){\line(1,0){10}}
\put(3,1){\line(0,1){1}}
\put(11.4,1){$\ldots$}
\multiput(1,1)(1,0){11}{\circle*{0.15}}
\put(3,2){\circle*{0.15}}
\put(0.9,0.4){1}
\put(1.9,0.4){2}
\put(2.9,0.4){3}
\put(3.9,0.4){5}
\put(4.9,0.4){6}
\put(5.9,0.4){7}
\put(6.9,0.4){8}
\put(7.9,0.4){9}
\put(8.8,0.4){10}
\put(9.8,0.4){11}
\put(10.8,0.4){12}
\put(2.6,2.1){4}

\put(0.9,1.4){$\alpha_{_{1}}$}
\put(1.9,1.4){$\alpha_{_{2}}$}
\put(2.4,1.4){$\alpha_{_{3}}$}
\put(3.9,1.4){$\alpha_{_{5}}$}
\put(4.9,1.4){$\alpha_{_{6}}$}
\put(5.9,1.4){$\alpha_{_{7}}$}
\put(6.9,1.4){$\alpha_{_{8}}$}
\put(7.9,1.4){$\alpha_{_{9}}$}
\put(8.9,1.4){$\alpha_{_{10}}$}
\put(9.9,1.4){$\alpha_{_{11}}$}
\put(10.9,1.4){$\alpha_{_{12}}$}
\put(3.2,2.1){$\alpha_{_{4}}$}
\end{picture}
\end{center}If $\lambda =e^{x}+e^{-x}$ is the Perron-Frobenius eigenvalue of $T(1,2,\infty),$  and $\rho$ denotes $e^{2x}$ then $\rho$ satisfies the equation (see \cite{Hof})
\begin{equation}
\label{egvaleq3}
\rho^{3} - \rho-1=0
\end{equation}
corresponding to $\rho\cong{}1.32472$ and $\lambda\cong{}2.01980.$

The corresponding coordinates of the Perron-Frobenius vector, $\alpha,$ are
\[\begin{array}{lclclcl}
\alpha_{_{1}} & = &e^{-8x}&\;\;\;\;&\alpha_{_{2}} & = &e^{-3x}\\[0.2cm]
\alpha_{_{3}} & = &1 & & \alpha_{_{4}} & = & e^{-5x}\end{array}\]
and for $i>5$ we have $\alpha_{_{i}}=e^{4-i}$.

\noindent{}These are determined as follows.  $\alpha_{_{4}}$ must equal $\frac{1}{\lambda},$ when we have scaled the vector to $1$ at the vertex $3.$ If we use the equation (\ref{egvaleq3}) this is easily seen to equal $e^{-5x}.$ If $z$ denotes $\alpha_{_{2}}$ then $z= \lambda{}-\frac{1}{\lambda}-e^{-x}=\frac{e^{2x}}{\lambda},$ and we can  use (\ref{egvaleq3}) to show that this equals $e^{-3x}.$ Finally $\alpha_{_{1}}=\frac{\alpha_{_{2}}}{\lambda}=e^{-8x}.$

We will construct a commuting square with the adjacency matrix, $\Delta,$ of  $T(1,2,\infty)$ as the index defining inclusion.\label{inftyconstruct} More precisely we let 
$\Delta_{_{bp}}=$ {\tiny\mbox{$ \left(
\begin{array}{cc}0 & H\\ K&0\end{array}\right)$}} be the adjacency matrix of a bi-partition of $T(1,2,\infty).$  We will construct a commuting square of the form
 \[\begin{array}{lcl}
 C & \subset_{L} & D \\
  \cup_{K} &\,&\cup_{H}\\
 A & \subset_{G} & B \end{array}\]
where $G$ and $L$ are defined by evaluating some polynomial $P$ in $\Delta,$ and then apply the  bi-partition used to obtain $\Delta_{_{bp}}$ to  the graph corresponding to $P(\Delta).$  I.e. the exponents in $P$ must all be even or all be odd. If all the exponents are odd, $H$ and $L$ are found as  
{\tiny\mbox{$\left(\begin{array}{cc}0 & G\\ L&0\end{array}\right),$}} and if all the exponents are even $G$ and $L$ can be found as 
{\tiny\mbox{$\left(\begin{array}{cc}G&0\\ 0&L\end{array}\right).$}}

In the case of the $A_{_{n}}-$graphs (see chapter \ref{angraphs}) we used the polynomials $R_{_{k}},$ defined inductively by 
\[\begin{array}{lcl}
R_{_{0}}(t) & = & 1\\
R_{_{1}}(t) & = & t\\
R_{_{k+1}}(t) & = & tR_{_{k}}(t) - R_{_{k-1}}(t).
\end{array} \]
These polynomials occurred naturally as coordinates of the Perron-Frobenius vector of $A_{_{n}}.$ In the present case the $R_{_{n}}$'s evaluated in $\Delta$ will also be positive (see \cite{wenzlharpe}), but because of the infinite ``ray'' we can find a polynomial, $P,$ with ``smaller'' entries in $P(\Delta)$ which will work. Define the polynomials $S_{_{k}}$ by 

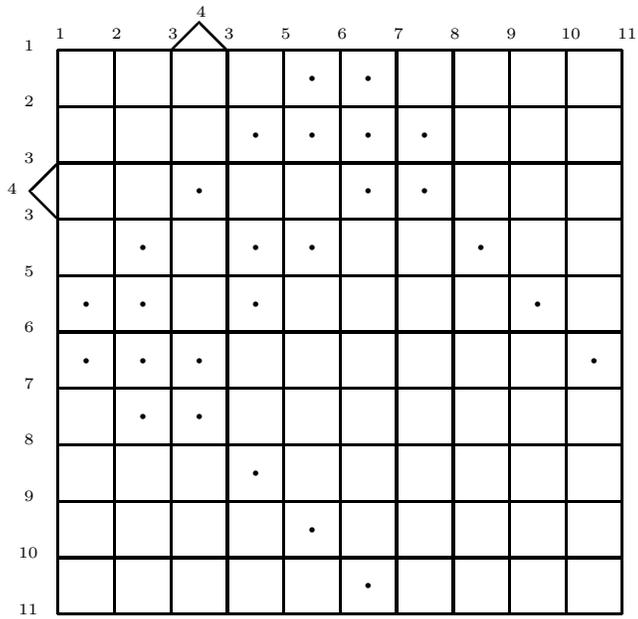
\begin{equation}
\label{Sdef}
\begin{array}{lcl}
S_{_{1}}(t) & = & t\\
S_{_{2}}(t) & = & t^{2}-2\\
S_{_{k+1}}(t) & = & tS_{_{k}}(t) - S_{_{k-1}}(t).
\end{array} \end{equation}

Then $S_{_{4}}(\Delta)$ is positive, and given by

\[\mbox{\tiny{$\left(\begin{array}{cccccccccccccccccccc}
   0 &    0 &    0 &    0 &    0 &    1 &    0 &    0 &    0 &    0 &    0 &    0 &    0 &    0 &    0 &    0 &    0 &    0 &    0 &     \\
   0 &    0 &    0 &    0 &    1 &    0 &    1 &    0 &    0 &    0 &    0 &    0 &    0 &    0 &    0 &    0 &    0 &    0 &    0 &     \\
   0 &    0 &    1 &    0 &    0 &    1 &    0 &    1 &    0 &    0 &    0 &    0 &    0 &    0 &    0 &    0 &    0 &    0 &    0 &     \\
   0 &    0 &    0 &    1 &    0 &    0 &    1 &    0 &    0 &    0 &    0 &    0 &    0 &    0 &    0 &    0 &    0 &    0 &    0 &     \\
   0 &    1 &    0 &    0 &    1 &    0 &    0 &    0 &    1 &    0 &    0 &    0 &    0 &    0 &    0 &    0 &    0 &    0 &    0 &     \\
   1 &    0 &    1 &    0 &    0 &    0 &    0 &    0 &    0 &    1 &    0 &    0 &    0 &    0 &    0 &    0 &    0 &    0 &    0 &     \\
   0 &    1 &    0 &    1 &    0 &    0 &    0 &    0 &    0 &    0 &    1 &    0 &    0 &    0 &    0 &    0 &    0 &    0 &    0 &     \\
   0 &    0 &    1 &    0 &    0 &    0 &    0 &    0 &    0 &    0 &    0 &    1 &    0 &    0 &    0 &    0 &    0 &    0 &    0 &     \\
\end{array}\right)$}} \]

Just like in the case of the $A_{_{n}}-$graphs we can picture the matrices $u$ and $v$ in a diagram with the edges of $T(1,2
,\infty)$ defining the ``boxes'' in the diagram. Since the graph is not a  \raisebox{0.5cm}{straight line, we introduce the notation}
\setlength{\unitlength}{1cm}
%
\raisebox{0.4cm}{to signify the edge }
%

\noindent{}We then have the following ``boxes'' which correspond to entries in $u$ and $v.$
\begin{center}
\setlength{\unitlength}{0.75cm}
\thicklines
%
\end{center}

\newpage{}\noindent{}\raisebox{0.5cm}{The south-east sloping rows of boxes, starting in the boxes corresponding to the cycles }  
\setlength{\unitlength}{0.25mm}
%
 \raisebox{0.5cm}{, all correspond to $1\times{}1$ blocks of both $u$ and $v.$}\\ 
The scalars defining the transition from $u$ to $v$ are all equal to $1$ for these cycles, so a solution of $u$ and $v$ in these boxes is given by putting all the scalars equal to $1.$

Using the coordinates of the Perron-Frobenius vector and reducing with the polynomial (\ref{egvaleq3}), we get the following table of the non-trivial values of the scalars defining the transition from $u$ to $v.$
\begin{center}
\setlength{\unitlength}{1cm}
%
\end{center}
The moduli and of the elements in $u$ and $v$ can all be determined using the block structure of $u$ and $v,$ and the scalars in the above table. I.e. using that the moduli squared of the entries in a $2\times{}2$ unitary must be of the form {\tiny\mbox{$ \left(\begin{array}{cc}\alpha & 1-\alpha\\ 1-\alpha&\alpha\end{array}\right),$}} and that a $1\times{}1$ unitary is a complex scalar of length $1.$ The block structure of $u$ and $v$ is indicated by the thick lines in the following diagrams. We have solutions to $u$ and $v$ looking like
\begin{center}
\setlength{\unitlength}{1cm}
%
\end{center}

All that is left to show, is that the determined $2\times{}2$ matrices correspond to doubly stochastic matrices, i.e. that 
\[\begin{array}{cl}
(1)&\left(e^{-3x}\right)^{2}+\left(e^{-2x}\right)^{2}=1\\[0.2cm]
(2)&\left(e^{-5x}\right)^{2}+\left(e^{-1x}\right)^{2}=1 \end{array}\]
$(1)$ follows from (\ref{egvaleq3}) since 
\[\left(e^{-3x}\right)^{2}+\left(e^{-2x}\right)^{2} =\rho^{-3}+\rho^{-1}.\]
$(2)$ follows since 
\[\left(e^{-5x}\right)^{2}+\left(e^{-1x}\right)^{2}=1\Leftrightarrow\rho^{5}-\rho^{4}-1=0,\]
and 
\[\rho^{5}-\rho^{4}-1=(\rho^{2}-\rho+1)(\rho^{3}-\rho-1)=0.\]
\section{A Commuting Square Based on $T(1,3,\infty)$}
\setcounter{equation}{0}
We will look at the graph $T(1,3,\infty)$ as defined by Hoffmann in \cite{Hof}.
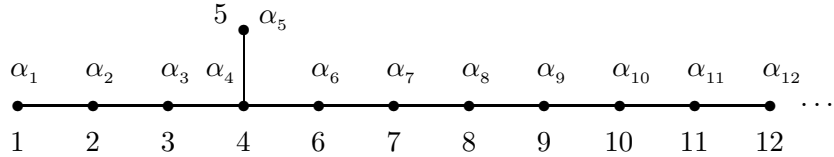
\begin{center}
\setlength{\unitlength}{1cm}
%
\end{center}If $\lambda =e^{x}+e^{-x}$ is the Perron-Frobenius eigenvalue of $T(1,3,\infty),$  and $\rho$ denotes $e^{2x}$ then $\rho$ satisfies the equation (see \cite{Hof})
\begin{equation}
\label{egvaleq2}
\rho^{3} - \rho^{2}-1=0
\end{equation}
corresponding to $\rho\cong{}1.46557$ and $\lambda\cong{}2.03664.$

The corresponding coordinates of the Perron-Frobenius vector, $\alpha,$ are
\[\begin{array}{lclclcl}
\alpha_{_{1}} & = &\frac{e^{-4x}}{\lambda}&\;\;\;\;&\alpha_{_{2}} & = &e^{-4x}\\[0.2cm]
\alpha_{_{3}} & = &\frac{e^{2x}}{\lambda} & & \alpha_{_{4}} & = & 1\\[0.2cm]
\alpha_{_{5}} & = & \frac{1}{\lambda} & & & & \end{array}\]
and for $i>6$ we have $\alpha_{_{i}}=e^{5-i}$.

We will construct a commuting square where the adjacency matrix, $\Delta,$ of  $T(1,3,\infty)$ gives the index defining inclusion, as described on page \pageref{inftyconstruct}. As in the example $T(1,2,\infty)$ we will look at the polynomials $S_{_{n}}(t)$ defined in (\ref{Sdef}). In this case the polynomial  $S_{_{5}}(\Delta)$ is positive, with $S_{_{5}}(\Delta)$  given by

\[\mbox{\tiny{$\left(\begin{array}{cccccccccccccccccccc}
   0 &    0 &    0 &    0 &    0 &    0 &    1 &    0 &    0 &    0 &    0 &    0 &    0 &    0 &    0 &    0 &    0 &    0 &    0 &     \\
   0 &    0 &    0 &    0 &    0 &    1 &    0 &    1 &    0 &    0 &    0 &    0 &    0 &    0 &    0 &    0 &    0 &    0 &    0 &     \\
   0 &    0 &    0 &    1 &    0 &    0 &    1 &    0 &    1 &    0 &    0 &    0 &    0 &    0 &    0 &    0 &    0 &    0 &    0 &     \\
   0 &    0 &    1 &    0 &    1 &    1 &    0 &    1 &    0 &    1 &    0 &    0 &    0 &    0 &    0 &    0 &    0 &    0 &    0 &     \\
   0 &    0 &    0 &    1 &    0 &    0 &    0 &    0 &    1 &    0 &    0 &    0 &    0 &    0 &    0 &    0 &    0 &    0 &    0 &     \\
   0 &    1 &    0 &    1 &    0 &    0 &    1 &    0 &    0 &    0 &    1 &    0 &    0 &    0 &    0 &    0 &    0 &    0 &    0 &     \\
   1 &    0 &    1 &    0 &    0 &    1 &    0 &    0 &    0 &    0 &    0 &    1 &    0 &    0 &    0 &    0 &    0 &    0 &    0 &     \\
   0 &    1 &    0 &    1 &    0 &    0 &    0 &    0 &    0 &    0 &    0 &    0 &    1 &    0 &    0 &    0 &    0 &    0 &    0 &     \\
\end{array}\right)$}} \]

\noindent{}We then have the following ``boxes'' which correspond to entries in $u$ and $v.$
\begin{center}
\setlength{\unitlength}{0.75cm}
\thicklines
%
\end{center}

\noindent{}\raisebox{0.5cm}{The south-east sloping rows of boxes, starting in the boxes corresponding to the cycles }  
\setlength{\unitlength}{0.25mm}
%
 \raisebox{0.5cm}{, all correspond to $1\times{}1$ blocks of both $u$ and $v.$}\\ 
The scalars defining the transition from $u$ to $v$ are all equal to $1$ for these cycles, so a solution of $u$ and $v$ in these boxes is given by putting all the scalars equal to $1.$

Using the coordinates of the Perron-Frobenius vector and reducing with the polynomial (\ref{egvaleq2}), we get the following table of the non-trivial values of the scalars defining the transition from $u$ to $v.$
\begin{center}
\setlength{\unitlength}{1.5cm}
%
\end{center}
The moduli and of the elements in $u$ and $v$ can all be determined using the block structure of $u$ and $v,$ and the scalars in the above table. (Just as in the example with $T(1,2,\infty).$)  The block structure of $u$ and $v$ is indicated by the thick lines in the following diagrams. We have solutions to $u$ and $v$ looking like
\begin{center}
\setlength{\unitlength}{1.5cm}
%
\end{center}
The only non-trivial computation to show the existence of a solution, is to show the existence of the $3\times{}3-$block of  $v.$ To show this, we apply the following proposition (see \ref{33existence})
\begin{prop}
Let 
{\small\mbox{$ \left(
\begin{array}{ccc}
d_{_{1,1}}^{2} & d_{_{1,2}}^{2} & d_{_{1,3}}^{2}\\[0.1cm]
d_{_{2,1}}^{2} & d_{_{2,2}}^{2} & d_{_{2,3}}^{2}\\[0.1cm]
d_{_{3,1}}^{2} & d_{_{3,2}}^{2} & d_{_{3,3}}^{2}\end{array}\right)$}} is a doubly stochastic matrix, and put \\[0.2cm]
$\alpha = d_{_{1,1}}d_{_{2,1}},$ $\beta = d_{_{1,2}}d_{_{2,2}}$ and $\gamma = d_{_{1,3}}d_{_{2,3}}.$ Then there exists a unitary $u =(u_{_{i,j}})$ with\\ $|u_{_{i,j}}|^{^{2}}=d_{_{i,j}},\;\;i,j=1,2,3$ if and only if 
\begin{equation}
\label{33existcond}\alpha^{^{2}}+\beta^{^{2}}+\gamma^{^{2}}-2\alpha\gamma-2\beta\gamma-2\alpha\beta\leq{}0
\end{equation}
\end{prop}
In our situation we have 
\[\alpha=\lambda^{^{-2}}e^{3x},\;\;\beta=\lambda^{^{-2}}e^{x},\;\;\gamma=\lambda^{^{-1}}e^{-4x}\]
for which we may substitute
\[\alpha'=\lambda^{^{-1}}e^{2x},\;\;\beta'=\lambda^{^{-1}},\;\;\gamma'=e^{-5x}\]
If we plug these values into the condition (\ref{33existcond}), and reduce as much as possible by the identity (\ref{egvaleq2}), we find that a solution exists if and only if
\[-4\rho^{2}-\rho-5\leq{}0,\]
which is clearly satisfied for any positive value of $\rho.$ 

By rescaling by complex numbers of modulus 1, we can obtain a solution to the $3\times{}3-$block of $v$ as listed in the table, where $\xi,\eta,\nu$ and $\mu$ are complex numbers with modulus 1.
\section{A Commuting Square Based on $T(1,4,\infty)$}
\setcounter{equation}{0}
\subsection{A Not So Successful Attempt}
\begin{center}
\setlength{\unitlength}{1cm}

\end{center}

\noindent{}If $\lambda = e^{x}+e^{-x}$ denotes the largest eigenvalue of $T(1,4,\infty),$ and we by $\rho$ denote $e^{2x},$ then $\rho$ satisfies the equation (see \cite{Hof}) 
\begin{equation}
\label{evaleq}
\rho^{5}-\rho^{3}-\rho^{2}-\rho-1=0,
\end{equation}
corresponding to $\lambda\cong{}2.04597.$

\noindent{}The coordinates of the corresponding eigenvector are given by
\[\begin{array}{lclllcl}
\alpha_{_{1}} & = & \frac{\rho^{3}-\rho-1}{\rho\lambda^{^{2}}} &\;\;\;\;&
\alpha_{_{2}} & = & \frac{\rho^{3}-\rho-1}{\rho\lambda} \\[0.5cm]
\alpha_{_{3}} & = & \rho-1 &  &\alpha_{_{4}} & = & \frac{\rho}{\lambda} \\[0.5cm]
\alpha_{_{5}} & = & 1 &  &\alpha_{_{6}} & = & \frac{1}{\lambda} 
\end{array}\]
and for $k\geq{}7$ we have $\alpha_{_{k}} = e^{6-k}.$

In analogy with the examples $T(1,2,\infty)$ and $T(1,3,\infty)$ we will try to construct a commuting square where the non-index defining inclusions are given by a polynomial in the adjacency matrix for $T(1,4,\infty).$ The candidate for the polynomial is, in analogy with the $T(1,2,\infty)-$ and $T(1,3,\infty)-$cases, given by $S_{_{6}}(t) = t^{^{6}}-6t^{^{4}}+9t^{^{2}}-2,$ and the matrix for $S_{_{6}}(T(1,4,\infty))$ is given by

 \[\left(\mbox{\tiny$\begin{array}{cccccccccccccccccccc}
   0 &    0 &    0 &    0 &    0 &    0 &    0 &    1 &    0 &    0 &    0 &    0 &    0 &    0 &    0 &    0 &    0 &    0 &    0 &    \\
   0 &    0 &    0 &    0 &    0 &    0 &    1 &    0 &    1 &    0 &    0 &    0 &    0 &    0 &    0 &    0 &    0 &    0 &    0 &    \\
   0 &    0 &    0 &    0 &    1 &    0 &    0 &    1 &    0 &    1 &    0 &    0 &    0 &    0 &    0 &    0 &    0 &    0 &    0 &    \\
   0 &    0 &    0 &    1 &    0 &    1 &    1 &    0 &    1 &    0 &    1 &    0 &    0 &    0 &    0 &    0 &    0 &    0 &    0 &    \\
   0 &    0 &    1 &    0 &    2 &    0 &    0 &    1 &    0 &    1 &    0 &    1 &    0 &    0 &    0 &    0 &    0 &    0 &    0 &    \\
   0 &    0 &    0 &    1 &    0 &    0 &    1 &    0 &    0 &    0 &    1 &    0 &    0 &    0 &    0 &    0 &    0 &    0 &    0 &    \\
   0 &    1 &    0 &    1 &    0 &    1 &    1 &    0 &    1 &    0 &    0 &    0 &    1 &    0 &    0 &    0 &    0 &    0 &    0 &    \\
   1 &    0 &    1 &    0 &    1 &    0 &    0 &    1 &    0 &    0 &    0 &    0 &    0 &    1 &    0 &    0 &    0 &    0 &    0 &    \\
\end{array}$}\right) \]    
The entries of the $u$ and $v$ matrices can again be pictured in a diagram, and we have the following ``boxes'' which define elements of the bi-unitary.

\begin{center}
\setlength{\unitlength}{0.75cm}
\thicklines
%
\end{center}

\noindent{}Where the number of dots in a box denotes the dimension of the respective element.

\noindent{}The south-east sloping rows of boxes, starting in the boxes corresponding to the cycles \\[0.2cm]
\setlength{\unitlength}{0.25mm}
%
 \raisebox{0.5cm}{, all correspond to $1\times{}1$ blocks of both $u$ and $v.$}\\ 
The scalars defining the transition from $u$ to $v$ are all equal to $1$ for these cycles, so a solution of $u$ and $v$ in these boxes is given by putting all the scalars equal to $1.$ 

\noindent{}Using the defined the coordinates of the eigenvector and reducing all expressions involving $\rho$ according to the equation (\ref{evaleq}), we get the following table of the scalars defining the transition from $u$ to $v$.
\begin{center}
\setlength{\unitlength}{1.45cm}
%
\end{center}

The moduli and 2--norms of the elements in $u$ and $v$ can all be determined using the block structure of $u$ and $v$ and the scalars in the previous table. The block structure of $u$ and $v$ is indicated by the thick lines in the diagrams.

\begin{center}
\setlength{\unitlength}{1.45cm}
\thicklines
%
\end{center}
That the squares of the numbers in the blocks form doubly stochastic matrices is easily seen using the equation (\ref{evaleq}). So we need to check whether we can find ``phases'' which will make the matrices unitary.
\newpage{}
\noindent{}The critical part is the existence of  vectors with the ``right'' inner products in the boxes
\setlength{\unitlength}{1.5cm}
\begin{center}
%
\end{center}

If a solution to these vectors has been found, all the other elements can be determined. \label{argument}This can be seen as follows. The ``phases'' in this $3\times{}3$ part of $u$ and $v,$ influences the other blocks of $u$ and $v$ via the transition from $u$ to $v$ and vice versa. However this influence is restricted to at most one entry of a $2\times{}2-$ or $1\times{}1-$block, and so phases for the entire $2\times{}2-$ resp. $1\times{}1-$block can be determined to make it unitary. The phases determined in a  $2\times{}2-$block may influence other blocks, but again at most one entry in a block is determined this way, and we can continue the argument as above.

If we rescale the first column, of the part coming from $u,$ by $\sqrt{\lambda\alpha_{_{4}}}=\sqrt{\rho},$ the second by $\sqrt{\lambda\alpha_{_{6}}}=1$ and the third by $\sqrt{\lambda\alpha_{_{7}}}=\sqrt{\lambda{}e^{-x}},$ and if we also rescale the first row, of the part coming from $v,$ by $\sqrt{\rho},$ the second by $1$ and the third by $\sqrt{\lambda{}e^{-x}},$ the moduli in both blocks becomes 
\setlength{\unitlength}{1.5cm}
\begin{center}
\begin{picture}(4,3)
\multiput(1,0)(1,0){4}{\line(0,1){3}}
\multiput(1,0)(0,1){4}{\line(1,0){3}}
\put(1.1,2.4){\tiny\mbox{$\sqrt{\rho(2-\rho)}$}}
\put(2.5,2.4){\tiny\mbox{$1$}}
\put(3.2,2.4){\tiny\mbox{$\sqrt{\rho^{2} - 1}$}}
\put(1.5,1.4){\tiny\mbox{$1$}}
\put(3.5,1.4){\tiny\mbox{$1$}}
\put(1.2,0.4){\tiny\mbox{$\sqrt{\rho^{2}  - 1}$}}
\put(2.5,0.4){\tiny\mbox{$1$}}
\put(3,0.4){\tiny\mbox{$\sqrt{\frac{2+2\rho-\rho^{3}}{\rho}}$}}
\end{picture}
\end{center}
And hence all we need to find are 9 vectors $(e_{_{i,j}})_{_{i,j+1}}^{^{3}}$ with 
\[\left(\begin{array}{ccc}\|e_{_{1,1}}\|_{_{2}}^{^{2}} &\|e_{_{1,2}}\|_{_{2}}^{^{2}} &\|e_{_{1,3}}\|_{_{2}}^{^{2}} \\[0.2cm]
\|e_{_{2,1}}\|_{_{2}}^{^{2}} &\|e_{_{2,2}}\|_{_{2}}^{^{2}} &\|e_{_{2,3}}\|_{_{2}}^{^{2}} \\[0.2cm]
\|e_{_{3,1}}\|_{_{2}}^{^{2}} &\|e_{_{3,2}}\|_{_{2}}^{^{2}} &\|e_{_{3,3}}\|_{_{2}}^{^{2}}\end{array}\right)=
\left(\begin{array}{ccc} \rho( 2-\rho) & 1 &\rho^{2} - 1\\ 
1 & 0 & 1\\ 
\rho^{2}-1 & 1 & \frac{2+2\rho-\rho^{3}}{\rho}\end{array}\right)\]
and the ``right'' inner products.

\noindent{}More precisely, having the right inner products can be stated as
\[\sum_{i}e_{_{i,j}}\otimes{}\overline{e}_{_{i,j}} = c_{_{j}}I \mbox{ and } \sum_{j}e_{_{i,j}}\otimes{}\overline{e}_{_{i,j}} = c_{_{i}}I\]
where
\[\begin{array}{lclcl}
c_{_{1}} & = & \frac{1}{2}(\rho( 2-\rho) +1 +\rho^{2} -1 ) & = & \rho \\[0.3cm]
c_{_{2}} & = & 1 & &\\[0.3cm]
c_{_{3}} & = & \frac{1}{2}(\rho^{2} + \frac{2+2\rho-\rho^{3}}{\rho})& = & \frac{\rho+1}{\rho}
\end{array}\]
\begin{prop}
\label{prop1t14}
There exists 9 vectors $(e_{_{i,j}})_{_{i,j+1}}^{^{3}}\in {\Bbb C}^{2}$ with 
\[\left(\begin{array}{ccc}\|e_{_{1,1}}\|_{_{2}}^{^{2}} &\|e_{_{1,2}}\|_{_{2}}^{^{2}} &\|e_{_{1,3}}\|_{_{2}}^{^{2}} \\[0.2cm]
\|e_{_{2,1}}\|_{_{2}}^{^{2}} &\|e_{_{2,2}}\|_{_{2}}^{^{2}} &\|e_{_{2,3}}\|_{_{2}}^{^{2}} \\[0.2cm]
\|e_{_{3,1}}\|_{_{2}}^{^{2}} &\|e_{_{3,2}}\|_{_{2}}^{^{2}} &\|e_{_{3,3}}\|_{_{2}}^{^{2}}\end{array}\right)=
\left(\begin{array}{ccc} a & d &b\\ 
d & 0 & d\\ 
b & d & c\end{array}\right)\]
and
\[\sum_{i}e_{_{i,j}}\otimes{}\overline{e}_{_{i,j}} = c_{_{j}}I \mbox{ and } \sum_{j}e_{_{i,j}}\otimes{}\overline{e}_{_{i,j}} = c_{_{i}}I\]
where $c_{_{1}}=\frac{1}{2}(a+b+d),$ $c_{_{2}}=d$ and $c_{_{3}}=\frac{1}{2}(b+c+d)$ if and only if
\[b^{^{2}}+d^{^{2}}\geq{}\frac{1}{2}(a^{^{2}}+c^{^{2}}) \mbox{ and }
 |b^{^{2}}-d^{^{2}}|\leq ac.\]
\end{prop}
\begin{proof}
Again we consider the map (see lemma \ref{lemtre210})
\[q:{\Bbb C}^{2}\rightarrow\left\{\mbox{ self-adjoint $2\times{}2$ matrices with trace 0}\right\}\cong{}{\Bbb R}^{3}\]
given by
\[q(x)=\sqrt{2}(x\otimes{}\overline{x}-\frac{1}{2}\|x\|_{_{2}}^{^{2}}I).\]
Then $\|q(x)\|_{_{HS}}=\|x\|_{_{2}}^{^{2}}$ since
\[\|q(x)\|_{_{HS}}^{^{2}}=2(\|x\otimes\overline{x}\|_{_{HS}}^{^{2}}-
\frac{1}{2}\|x\|_{_{2}}^{^{4}}) = \|x\|_{_{2}}^{^{4}}.\]
Also $(q(x),q(y))=2|(x,y)|-\|x\|_{_{2}}^{^{2}}\|y\|_{_{2}}^{^{2}},$ so if $|(x,y)|=\|x\|_{_{2}}\|y\|_{_{2}}\cos{}\theta,\;0\leq\theta\leq2\pi$ then
\[(q(x),q(y))=\|x\|_{_{2}}^{^{2}}|y\|_{_{2}}^{^{2}}(2\cos{}\theta -1)= \|q(x)\|\|q(y)\|\cos{}2\theta.\]
Put $q_{_{i,j}}=q(e_{_{i,j}})$ for $i,j=1,2,3.$ Then 
\[\sum_{i}q_{_{i,j}}=0 \mbox{ and } \sum_{j}q_{_{i,j}}=0\] 
since the sums are scalars with trace 0.

\noindent{}Hence we have 
\[\left(\begin{array}{ccc}
q_{_{1,1}} & q_{_{1,2}} & q_{_{1,3}} \\
q_{_{2,1}} & q_{_{2,2}} & q_{_{2,3}} \\
q_{_{3,1}} & q_{_{3,2}} & q_{_{3,3}} \end{array}\right)=
\left(\begin{array}{ccc}
q_{_{1,1}} & q_{_{1,2}} & q_{_{1,3}} \\
q_{_{2,1}} & 0 & -q_{_{2,1}} \\
q_{_{3,1}} & -q_{_{1,2}} & q_{_{3,3}} \end{array}\right)\]
and get
\[q_{_{3,3}}=-q_{_{3,1}}+q_{_{1,2}}=q_{_{1,1}}+q_{_{1,2}}+q_{_{2,1}}\]
which implies
\[\begin{array}{lcl}
(q_{_{1,1}},q_{_{2,1}}) & = & 
\frac{1}{2}(\|q_{_{1,1}}+q_{_{2,1}}\|_{_{2}}^{^{2}}-\|q_{_{1,1}}\|_{_{2}}^{^{2}}-\|q_{_{2,1}}\|_{_{2}}^{^{2}})\\[0.25cm]
 & = &
\frac{1}{2}(\|q_{_{3,1}}\|_{_{2}}^{^{2}}-\|q_{_{1,1}}\|_{_{2}}^{^{2}}-\|q_{_{2,1}}\|_{_{2}}^{^{2}})\\[0.25cm]
& = &\frac{1}{2}(b^{^{2}}-a^{^{2}}-d^{^{2}}).
\end{array}\]
Similarly we get
\[\begin{array}{lcl}
(q_{_{1,1}},q_{_{1,2}}) & =  &\frac{1}{2}(b^{^{2}}-a^{^{2}}-d^{^{2}})
\end{array}\]
and we have
\[(q_{_{1,1}},q_{_{2,1}}+q_{_{1,2}}) = b^{^{2}}-a^{^{2}}-d^{^{2}} \mbox{ and } (q_{_{1,1}},q_{_{2,1}}-q_{_{1,2}}) =0.\]
Also $q_{_{3,3}}= q_{_{1,1}}+(q_{_{2,1}}+q_{_{1,2}}),$ hence
\[\|q_{_{3,3}}\|_{_{2}}^{^{2}}=
\|q_{_{1,1}}\|_{_{2}}^{^{2}}+\|q_{_{2,1}}+q_{_{1,2}}\|_{_{2}}^{^{2}}+
2(q_{_{1,1}},q_{_{2,1}}+q_{_{1,2}}),\]
and we get
\begin{equation}
\label{lign1}
\begin{array}{lcl}\|q_{_{2,1}}+q_{_{1,2}}\|_{_{2}}^{^{2}} & = &
c^{^{2}}-a^{^{2}}-2(b^{^{2}}-a^{^{2}}-d^{^{2}})\\[0.25cm]
 & = & a^{^{2}}+c^{^{2}}-2b^{^{2}}+2d^{^{2}}
\end{array}\end{equation}
and
\begin{equation}
\label{lign2}
\begin{array}{lcl}\|q_{_{2,1}}-q_{_{1,2}}\|_{_{2}}^{^{2}} & = &
2\|q_{_{2,1}}\|_{_{2}}^{^{2}} +2\|q_{_{1,2}}\|_{_{2}}^{^{2}} - \|q_{_{2,1}}+q_{_{1,2}}\|_{_{2}}^{^{2}} \\[0.25cm]
& = & 4d^{^{2}}-a^{^{2}}-c^{^{2}}+2b^{^{2}}+2d^{^{2}}\\
 & = & 2d^{^{2}}+2b^{^{2}}-a^{^{2}}-c^{^{2}}.
\end{array}
\end{equation}
This now implies 
\[\|q_{_{3,3}}-q_{_{1,1}}\|_{_{2}}^{^{2}}=\|q_{_{2,1}}+q_{_{1,2}}\|_{_{2}}^{^{2}}=a^{^{2}}+c^{^{2}}-2b^{^{2}}+2d^{^{2}}\]
and hence
\[(q_{_{1,1}},q_{_{3,3}})=\frac{1}{2}(\|q_{_{1,1}}\|_{_{2}}^{^{2}}+\|q_{_{3,3}}\|_{_{2}}^{^{2}}-\|q_{_{1,1}}-q_{_{3,3}}\|_{_{2}}^{^{2}})=
b^{^{2}}-d^{^{2}}.\]
Cauchy--Schwartz gives
\[|b^{^{2}}-d^{^{2}}|\leq{}\|q_{_{1,1}}\|_{_{2}}\|q_{_{3,3}}\|_{_{2}} = ac.\]
From (\ref{lign1}) and (\ref{lign2}) we must have 
\[2d^{^{2}}+2b^{^{2}}\geq{}a^{^{2}}+c^{^{2}}.\]

To show that the condition is also sufficient we argue as follows. 

Assume the two conditions are satisfied. Choose $q_{_{1,1}}$ and $q_{_{3,3}}$ in ${\Bbb C}^{3}$ with the right length and inner product, and pick $h\in{\Bbb R}^{3}$ with $(h,q_{_{1,1}})=0=(h,q_{_{3,3}})$ and $\|h\|_{_{2}}^{^{2}}=
2d^{^{2}}+2b^{^{2}}-a^{^{2}}-c^{^{2}}.$ Put
\[q_{_{2,1}}=\frac{1}{2}(q_{_{3,3}}-q_{_{1,1}}+h) \mbox{ and }
 q_{_{1,2}}=\frac{1}{2}(q_{_{3,3}}-q_{_{1,1}}-h).\]
Put also
\[q_{_{1,3}}= -q_{_{1,1}}-q_{_{2,1}} \mbox{ and } 
q_{_{3,1}}= -q_{_{1,1}}-q_{_{1,2}}\]
Then the defined vectors have the right length and inner products.

Now identify ${\Bbb R}^{3}$ with the $2\times{}2$ matrices with trace 0, and pick $e_{_{i,j}}$ such that 
\[\begin{array}{lclclcl}
q(e_{_{1,1}}) & = & q_{_{1,1}} &, &q(e_{_{1,2}})& = & q_{_{1,2}}\\ 
q(e_{_{1,3}}) & = & q_{_{1,3}} &, &q(e_{_{2,1}}) & = & q_{_{2,1}}\\
q(e_{_{2,3}}) & = & q_{_{2,3}} &, &q(e_{_{3,1}}) & = & q_{_{3,1}}\\
q(e_{_{3,2}}) & = & q_{_{3,2}} &, &q(e_{_{3,3}}) & = & q_{_{3,3}}
\end{array}\]
\end{proof}
If we apply the above proposition to our construction for 
$S_{_{6}}(T(1,4,\infty)),$ we have
$a=\rho(2-\rho),$\\[0.2cm]
 $b=\rho^{2}-1,$ $c=\frac{(\rho^{2}-1)(\rho+1)}{\rho^{3}}$ and $d=1.$ In numerical entities we have $a\cong{}0.71468,$\\[0.2cm]
$b\cong{}1.35364$ and $c\cong{}0.95001$ and we see that the first criterion of  proposition \ref{prop1t14} is satisfied, but 
$|b^{^{2}}-d^{^{2}}|\cong{}0.83234$ and $ac\cong{}0.67895,$ so the second condition is not satisfied.
\newpage{}
\subsection{A Successful Attempt}
Since the first attempt to construct a commuting square with $T(1,3,\infty)$ as the index defining side, did not succed, we did a little experimenting and found another polynomial which will do the job. We shall be loking at the polynomial $S_{_{6}}(t)+1 = t^{^{6}}-6t^{^{4}}+9t^{^{2}}-1.$\\
$(S_{_{6}}+ 1)(T(1,4,\infty))$ is given by

\[\left(\mbox{\tiny$
\begin{array}{cccccccccccccccccccc}
   1 &    0 &    0 &    0 &    0 &    0 &    0 &    1 &    0 &    0 &    0 &    0 &    0 &    0 &    0 &    0 &    0 &    0 &    0 &    \\
   0 &    1 &    0 &    0 &    0 &    0 &    1 &    0 &    1 &    0 &    0 &    0 &    0 &    0 &    0 &    0 &    0 &    0 &    0 &    \\
   0 &    0 &    1 &    0 &    1 &    0 &    0 &    1 &    0 &    1 &    0 &    0 &    0 &    0 &    0 &    0 &    0 &    0 &    0 &    \\
   0 &    0 &    0 &    2 &    0 &    1 &    1 &    0 &    1 &    0 &    1 &    0 &    0 &    0 &    0 &    0 &    0 &    0 &    0 &    \\
   0 &    0 &    1 &    0 &    3 &    0 &    0 &    1 &    0 &    1 &    0 &    1 &    0 &    0 &    0 &    0 &    0 &    0 &    0 &    \\
   0 &    0 &    0 &    1 &    0 &    1 &    1 &    0 &    0 &    0 &    1 &    0 &    0 &    0 &    0 &    0 &    0 &    0 &    0 &    \\
   0 &    1 &    0 &    1 &    0 &    1 &    2 &    0 &    1 &    0 &    0 &    0 &    1 &    0 &    0 &    0 &    0 &    0 &    0 &    \\
   1 &    0 &    1 &    0 &    1 &    0 &    0 &    2 &    0 &    0 &    0 &    0 &    0 &    1 &    0 &    0 &    0 &    0 &    0 &    \\
\end{array}$}\right) \]
and the boxes corresponding to entries of $u$ or $v$ are given as
\begin{center}
\setlength{\unitlength}{0.75cm}
\thicklines
%
\end{center}
\noindent{}Where the number of dots in a box denotes the dimension of the respective element.

\noindent{}The south-east sloping rows of boxes, starting in the boxes corresponding to the cycles \\[0.2cm]
\setlength{\unitlength}{0.25mm}
%
 \raisebox{0.5cm}{, all correspond to $1\times{}1$ blocks of both $u$ and $v.$}\\ 
The scalars defining the transition from $u$ to $v$ are all equal to $1$ for these cycles, so a solution of $u$ and $v$ in these boxes, is given by putting all the scalars equal to $1.$ Again the problem of finding a solution to $u$ and $v,$ is reduced to finding a solution in the upper left corner of the diagram. Here we look at 
\begin{center}
\setlength{\unitlength}{1.45cm}
\thicklines
%
\end{center}
Again the moduli, 2--norms and Hilbert-Schmidt norms of the elements in $u$ and $v$ can  be determined using the block structure of $u$ and $v$ and the scalars in the above table. The block structure of $u$ and $v$ is indicated by the thick lines in the diagrams.

\begin{center}
\setlength{\unitlength}{1.45cm}
\thicklines
%
\end{center}
\newpage{}
We will first show that there is a solution to the parts of $u$ and $v$  involving
\setlength{\unitlength}{1.5cm}
\begin{center}
%
\end{center}

If we rescale, as in the previous computation, the moduli and norms of the involved entities become
\setlength{\unitlength}{1.5cm}
\begin{center}
%
\end{center}

\noindent{}Put $a=\rho(3-\rho)\cong{}2.249,$ $b=\rho^{2}-1\cong{}1.354,$ $c=\frac{1}{\rho}(3+3\rho-\rho^{3})\cong{}2.602$ and $d=1.$ Then 
if we look for a symmetric solution for $u$, the left-hand and right-hand columns above are part of matrices of the form 
\begin{equation}
\label{T14eqn666}
\left(\begin{array}{cc}
(2\times{}1)_{_{1}} & (2\times{}3)_{_{1}}\\
\xi & e \\
\eta & f
\end{array}\right)\;\;\;\;\;
\left(\begin{array}{cc}
f & \nu\\
h & \mu\\
(2\times{}3)_{_{2}}& (2\times{}1)_{_{2}}
\end{array}\right)\end{equation}
where $e,f,h\in{}{\Bbb C}^{3}$ with $\|e\|_{_{2}}^{^{2}} = d,$ 
$\|f\|_{_{2}}^{^{2}} = b,$ $\|h\|_{_{2}}^{^{2}} = d$ and $\xi,\eta,\nu,\mu\in{\Bbb C}.$ The left-hand matrix\\[0.1cm] 
is of the form $\sqrt{\frac{a+b+d}{3}}\mbox{unitary}$ and the right-hand matrix is of the form $\sqrt{\frac{c+b+d}{3}}\mbox{unitary}.$ 

\noindent{}If we can find $e,f,h,\xi,\eta,\nu$ and $\mu$ with the desired properties, the rest of the above matrices can be determined by extending to an orthonormal basis of ${\Bbb C}^{4}$ in either case.

\noindent{}Since a rescaling of the vectors $e$ and $h$ gives columns in a unitary in $v,$ we have $e\perp{}h.$

\noindent{}We also have
\[\begin{array}{lclclcll}
|\xi|^{^{2}} & = & \frac{a+b+d}{3}-d & = &\frac{a+b-2d}{3}&\geq{}&0&\mbox{ if } a\geq{}b-2d\\[0.2cm]
|\eta|^{^{2}} & = & \frac{a+b+d}{3}-b & = &\frac{a-2b+d}{3}&\geq{}&0&\mbox{ if } a\geq{}d-2b
\end{array}.\]
The criteria for positivity are both satisfied. Hence
\[|(e,f)| = |\xi\eta| = \mbox{$\frac{1}{3}$}\sqrt{(a+b-2d)(a-2b+d)}=\mbox{$\frac{1}{3}$}\sqrt{-2(b-d)^{2}+(a-b)(a-d)}.\]
Similarly we get
\[\begin{array}{lclclcll}
|\nu|^{^{2}} & = & \frac{c+b+d}{3}-d & = &\frac{c+b-2d}{3}\geq{}&0&\mbox{ if } c\geq{}b-2d\\[0.2cm]
|\mu|^{^{2}} & = & \frac{c+b+d}{3}-b & = &\frac{c-2b+d}{3}\geq{}&0&\mbox{ if } c\geq{}d-2b.\end{array}\]
Again both criteria for positivity are satisfied, and we get
\[|(f,h)| = |\mu\nu| = \mbox{$\frac{1}{3}$}\sqrt{(c+b-2d)(c-2b+d)}=\mbox{$\frac{1}{3}$}\sqrt{-2(b-d)^{2}+(c-b)(c-d)}.\]
In particular, a necessary condition for a solution to the bi--unitary problem is 
\[a\geq{}\max\{b-2d,d-2b\}\;\;\;\mbox{ and }\;\;\;c\geq{}\max\{b-2d,d-2b\},\]
which is seen to be satisfied by the values of $a,$ $b,$ $c,$ and $d$ above.

\noindent{}Let $k_{_{1}},$ $k_{_{2}}$ and $k_{_{3}}$ be an orthonormal,basis for ${\Bbb C}^{3},$ and put
\[e=dk_{_{1}},\;\;\;h=dk_{_{2}}\;\;\mbox{ and }\;f=\gamma_{_{1}}k_{_{1}}+\gamma_{_{2}}k_{_{2}}+\gamma_{_{3}}k_{_{3}},\]
where $\gamma_{_{1}}=\frac{1}{3}\sqrt{-2(b-d)^{2}+(a-b)(a-d)}$ and $\gamma_{_{2}}=\frac{1}{3}\sqrt{-2(b-d)^{2}+(c-b)(c-d)}.$

\noindent{}If $\gamma_{_{3}}$ can be chosen such that $|\gamma_{_{1}}|^{2}+|\gamma_{_{2}}|^{2}+|\gamma_{_{3}}|^{2}=b$ then
\[\|e\|^{2}=\|h\|^{2}=d,\;\;\;\|f\|^{2}=b,\;\;\;(e,h)=0,\]
\[(e,f)=\mbox{$\frac{1}{3}$}\sqrt{-2(b-d)^{2}+(a-b)(a-d)} \mbox{ and }(h,f)=\mbox{$\frac{1}{3}$}\sqrt{-2(b-d)^{2}+(c-b)(c-d)}.\]
Computing, we get
\[b-|\gamma_{_{1}}|^{2}-|\gamma_{_{2}}|^{2}=\mbox{$\frac{1}{9}$}(9b-(a-b)(a-d)-(c-b)(c-d)+4(b-d)^{2})\cong{}1.031>0,\]
so $\gamma_{_{3}}=\mbox{$\frac{1}{3}$}\sqrt{9b-(a-b)(a-d)-(c-b)(c-d)+4(b-d)^{2}}$ will do the job.

\noindent{}Since the above solution to $e,$ $f$ and $h$ is real, we can obtain a solution to the part of $u$ corresponding to (\ref{T14eqn666}) as follows.
Put
\[\mu'=\sqrt{\frac{\rho}{3(\rho+1)}}\mu,\;\;\;\nu'=\sqrt{\frac{\rho}{3(\rho+1)}}\nu,\;\;\;\eta'=\frac{1}{\sqrt{3\rho}}\eta,\]
\[\xi'=\frac{1}{\sqrt{3\rho}}\xi,\;\;\;\delta_{_{i}}=\frac{1}{\sqrt{3\rho}}\gamma_{_{i}},\;\;\;\epsilon_{_{i}}=\sqrt{\frac{\rho}{3(\rho+1)}}\gamma_{_{i}}.\]
Then, with $A_{_{1}},A_{_{2}},B_{_{1}}$ and $B_{_{2}}$ obtained by extending to orthonormal bases of ${\Bbb C}^{3},$ we have the following solution to (\ref{T14eqn666})
\[\left(\begin{array}{cc}
A_{_{1}}&B_{_{1}}\\[0.2cm]
\xi'&(\frac{1}{\sqrt{3\rho}},0,0)\\[0.2cm]
\eta'&(\delta_{_{1}},\delta_{_{2}},\delta_{_{3}})
\end{array}\right)\;\;\;\;\;
\left(\begin{array}{ccc}
(\epsilon_{_{1}},\epsilon_{_{2}},\epsilon_{_{3}})&\nu'\\[0.2cm]
(\sqrt{\frac{\rho}{3(\rho+1)}},0,0)&\mu'\\[0.2cm]
B_{_{2}}&A_{_{2}}
\end{array}\right).\]
Extend the above solution to a larger part of $u,$ by reflecting it in the main diagonal of the diagram and then rescale to get the right Hilbert--Schmidt norms. Then the ``directions'' of the corresponding summands of $v$ is given by:
\[\left(\begin{array}{ccc}
A_{_{1}}^{t}&\xi&\eta\\[0.2cm]
B_{_{1}}^{t}&e^{t}&f^{t}
\end{array}\right)\;\;\;\;\;
\left(\begin{array}{ccc}
f^{t}&h^{t}&B_{_{2}}^{t}\\[0.2cm]
\nu&\mu&A_{_{2}}^{t}
\end{array}\right),
\]
and hence, with the right scaling, unitary.

\noindent{}To show that there is a solution to the rest of $u$ and $v$ we argue as on page \pageref{argument}.

\newpage{}



\begin{thebibliography}{XXXXXX}
\bibitem[Bra]{bratteli}O.\ Bratteli:\\Inductive Limits of Finite Dimensional $C^{*}-$algebras\\Trans.\ Amer.\ Math.\ Soc. 171 (1972) pp.  195--234.
\bibitem[Dix]{dixmier}J.\ Dixmier:\\Von Neumann Algebras\\North Holland 1981.

\bibitem[G.H.J.]{HGJ}F.\ Goodman, P.\ de la Harpe \& V.\ Jones:\\Coxeter Graphs and Towers of Algebras.\\Springer Verlag 1989.


\bibitem[Hof]{Hof}A.\ J.\ Hoffmann:\\On Limit Points of Spectral Radii of Non-negative Symmetric Integral Matrices.\\Springer Lecture Notes in Mathematics vol. 303, pp. 165-172.

\bibitem[HW]{wenzlharpe}P.\ de la Harpe \& H. Wenzl:\\Op\'{e}rations sur les rayons spectraux de matrices sym\'{e}triques enti\`{e}res positives.\\C.\ R. Acad. Paris, Ser. I 305, 1987, pp. 733-736.

\bibitem[Jo]{jones}V. Jones:\\Index for Subfactors\\Inventiones  mathematicae 72 (1983) pp. 1--25. 
\bibitem[O]{Ocn}Adrian Ocneanu:\\Private communications. Fall 1988.

\bibitem[PH]{Pierre}Pierre de la Harpe:\\Private communications. Fall 1990.

\bibitem[Po1]{popa1}S.\ Popa:\\Markov Traces on Universal Jones Algebras and Subfactors of Finite Index \\Preprint, IHES 1990.

\bibitem[Po2]{popa2}S.\ Popa:\\Private communications.  Fall 1990.

\bibitem[S]{S}E.\ Seneta:\\Non-negative Matrices and Markov Chains.\\Springer Verlag 1981.

\bibitem[Sh]{shearer}James B.\ Shearer:\\On the Distribution of the Maximum Eigenvalue of Graph.\\Linear Algebra and its Applications 114/115, 1989, pp. 17-20.

\bibitem[ST]{ST}I.\ N.\ Stewart and D.\ O.\ Tall:\\Algebraic Number Theory.\\Chapman \& Hall 1979.

\bibitem[Wen1]{Wen1}H.\ Wenzl:\\Representations of Hecke Algebras and Subfactors, thesis\\University of Pennsylvania 1985.

\bibitem[Wen2]{Wen2}H.\ Wenzl:\\Hecke Algebras of Type $A_{n}$ and Subfactors.\\Inventiones  mathematicae 92, 1989, pp. 349--383.


\end{thebibliography}
 \end{document}